\documentclass[report]{elsarticle}
\makeatletter
\def\ps@pprintTitle{%
	\let\@oddhead\@empty
	\let\@evenhead\@empty
	\def\@oddfoot{}%
	\let\@evenfoot\@oddfoot}
\long\def\pprintMaketitle{\clearpage
	\iflongmktitle\if@twocolumn\let\columnwidth=\textwidth\fi\fi
	\resetTitleCounters
	\def\baselinestretch{1}%
	\printFirstPageNotes
	\begin{center}%
		\thispagestyle{pprintTitle}%
		\def\baselinestretch{1}%
		\Large\@title\par\vskip18pt
		\normalsize\elsauthors\par\vskip10pt
		\footnotesize\itshape\elsaddress\par\vskip36pt
		\ifvoid\absbox\else\unvbox\absbox\par\vskip10pt\fi
		\ifvoid\keybox\else\unvbox\keybox\par\vskip10pt\fi
	\end{center}%
	\gdef\thefootnote{\arabic{footnote}}%
}
\makeatother
\usepackage{amsmath, amsthm} 
\usepackage{amsfonts,ulem,amssymb,amsmath, xcolor}
\usepackage{bbold}
\usepackage{enumitem}
\usepackage{tikz}
\usepackage{tabularx}
\usepackage{appendix}
\usepackage[margin=1in]{geometry}
\newtheorem{theorem}{Theorem}[section]
\newtheorem{assumption}[theorem]{Assumption}

\newtheorem{proposition}[theorem]{Proposition}
\newtheorem{definition}[theorem]{Definition}
\newtheorem{corollary}[theorem]{Corollary}
\newtheorem{lemma}[theorem]{Lemma}
\newtheorem{remark}[theorem]{Remark}

\numberwithin{theorem}{section}
\usepackage{mathabx}
\usetikzlibrary{
	angles,
	quotes,
}
\usepackage[bbgreekl]{mathbbol}

\newcommand{\bx}{\mathbf{x}}

\newcommand{\tbx}{\widetilde{\mathbf{x}}}

\newcommand{\bpsi}{\boldsymbol{\psi}}

\renewcommand{\Im}{\operatorname{Im}}
\renewcommand{\Re}{\operatorname{Re}}

\renewcommand{\H}{\mathcal{H}}
\renewcommand{\vec}[1]{\mathbf{#1}}

\newcommand{\re}{\varepsilon}

\newcommand{\isreport}{1}

\newcommand{\citeappendixmodel}{
	\ifnum\isreport=1
	Appendix~\ref{appendix:derivation}:
	\else
	\cite[Appendix A]{mk_peillon}:
	\fi
}

\newcommand{\citeappendixregularity}{
	\ifnum\isreport=1
	Appendix~\ref{appendix:regularity_proof}
	\else
	\cite[Appendix D]{mk_peillon}
	\fi
}

\begin{document}
	\begin{frontmatter}
			\title{Limiting absorption principle for a hybrid resonance in a two-dimensional cold plasma}
		\author{Maryna Kachanovska\corref{cor1}} %% Author name
	\cortext[cor1]{Corresponding author}
	\ead{maryna.kachanovska@inria.fr}
	\author{\'Etienne Peillon}
	\address{{POEMS, CNRS, Inria, ENSTA, Institut Polytechnique de Paris, 91120 Palaiseau, France}}
	

		\begin{abstract}
			We study a limiting absorption principle for the boundary-value  problem describing a hybrid plasma resonance, with a regular coefficient in the principal part of the operator that vanishes on a curve inside the domain and changes its sign across this curve. We prove the limiting absorption principle by establishing a priori bounds on the solution in certain weighted Sobolev spaces. Next, we show that the solution can be decomposed into  regular and  singular parts. A peculiar property of this decomposition enables us to introduce a radiation-like condition in a bounded domain and to state a well-posed problem satisfied by the limiting absorption solution. 
		\end{abstract}
		\begin{keyword}
			Limiting absorption principle \sep hybrid resonance \sep weighted Sobolev space \sep singular solutions 
			
			%\MSC 35B65, 35D40, 35B34, 35M12, 35Q61
			
			
			
		\end{keyword}
		
	\end{frontmatter}

\section{Introduction}
Time-harmonic electromagnetic wave propagation in a cold plasma is described by the Maxwell's equations with a frequency- and space-dependent tensor of dielectric permittivity. Various degeneracies of this tensor lead to plasma resonances, which, mathematically, is described by the occurance of singular solutions to the underlying PDEs. Much attention in the last decade was devoted to the mathematical and numerical analysis of the situation of a hybrid plasma resonance in two dimensions, where the diagonal of plasma tensor vanishes on a given spatial curve, but the off-diagonal entries are bounded away from zero, see the recent works \cite{MR3192426, nicolopoulos} and references therein. Inside a domain $D\subset \mathbb{R}^2$, the time-harmonic magnetic field $D\ni\vec{x}\mapsto B_3(\vec{x})$ satisfies the following PDE, see \cite{nicolopoulos,ciarlet_mk_peillon} or \citeappendixmodel:
\begin{align}
	\label{eq:B3}
	\operatorname{div}\big((\alpha\mathbb{N}+i\nu\mathbb{H})\nabla B_3^{\nu}\big)-\omega^2 B_3^{\nu}=0 \text { in }D,
\end{align}
where $\omega\geq 0$ is a given fixed frequency, $\nu>0$ is an absorption parameter, and the tensors $\mathbb{N}, \mathbb{H}: \, \overline{D}\rightarrow \mathbb{C}^{2\times 2}$ are Hermitian positive definite. The behaviour of the coefficient $\alpha: \, D\rightarrow \mathbb{C}$ is responsible for an unusual behaviour of solutions to \eqref{eq:B3}.  In particular, in the situation of a hybrid resonance, $\alpha=0$ on a loop $I\subset D$. In this work we concentrate on the case when, in the vicinity of $I$,  
\begin{align*}
	\alpha(\bx)=\operatorname{dist}(\bx,I),\quad  \text{ 	where $\operatorname{dist}(\bx,I)$ is a signed distance from $\bx$ to $I$.}
\end{align*}
We are interested in establishing a limiting absorption principle for the problem \eqref{eq:B3} equipped with appropriate boundary conditions and a sufficiently regular right-hand side. Studies of \eqref{eq:B3}, up to our knowledge, were initiated by B. Despr\'es and his many co-workers (L.-M. Imbert-Gerard, M.-C. Pinto, R. Weder, A. Nicolopoulos, O. Lafitte, P. Ciarlet Jr., cf. \cite{MR3192426,DESPRES20161284,imbertgerard2013Mathematical,MR3959808,nicolopoulos,MR3705789,MR3517445,MR3611102,nicolopoulos_phd}).  In these references, with an exception of \cite{nicolopoulos}, the first-order Maxwell system leading to \eqref{eq:B3} is considered. The following results are available in the existing literature. 

The limiting absorption principle has been proven in (a) a slab geometry with $\alpha$, $\mathbb{N}$ depending on a single variable; (b) in the 1D case, (c) in a very peculiar 2D situation where the separation of variables was possible.  Singularities of the obtained solutions are quite well-understood in these cases; in 1D, an important connection between \eqref{eq:B3} and the Bessel equation has been established. For these results, please see  \cite{MR3192426} for a thorough analysis of (a) with the third-kind integral equations,  \cite{MR3611102} for (b), and \cite{phd_thesis_peillon} for (c).

In particular, all these works seem to indicate that limiting absorption solutions to \eqref{eq:B3} ($B_3^+=\lim\limits_{\nu\rightarrow 0+}B^{\nu}_3$ in a certain topology) posses in particular a logarithmic and a jump singularity across the interface $I$ (i.e. are not in $H^1(D)$). 
%	
%	
%	
%The questions whether such singularities occur in  time-domain wave propagation in the context of hybrid plasma resonances were studied in \cite{DESPRES20161284} and numerically examined in \cite{MR3517445}.
%	
Well-posed problems satisfied by the limiting absorption solution have been suggested in  \cite{MR3705789} and \cite{nicolopoulos}, with a full theoretical justification available in one dimension only. An improvement over \cite{nicolopoulos} was proposed in the work  \cite{ciarlet_mk_peillon}, which relaxed the regularity requirements necessary for the formulation of \cite{nicolopoulos} and has shown that it is injective even without the penalization terms. Moreover, under some technical assumptions,  the limiting absorption solution appears to satisfy this formulation.

From the above discussion, we see that in what concerns \eqref{eq:B3}, the following is missing:
\begin{enumerate}
	\item proof of the limiting absorption principle for \eqref{eq:B3};
	\item a well-posed problem satisfied by the limiting absorption solution;
	\item regularity results for limiting absorption solutions, especially in the case when $\mathbb{N}$ is a matrix.
\end{enumerate}
In this work we fill in these gaps. %Moreover, we propose a well-posed formulation for the problem \eqref{eq:B3} with $\nu=0$ satisfied by the limiting absorption solution, inspired by previous investigations in \cite{ciarlet_mk_peillon}. 
Unlike in the existing papers, we are able to treat the case when $\mathbb{H}, \, \mathbb{N}$ are no longer scalars, and depend on both variables $x, y$. We also prove the corresponding results for a large class of sufficiently regular domains. For the moment we concentrate our efforts around the case $\omega=0$.

Because we felt that the problem under consideration is already quite complicated, we decided to present the summary and proofs of the results of the paper for a simplified geometry first, and next argue that their extension to more general geometries is quite trivial. Thus, we refer the readers interested in the final results of the paper to the last section of the manuscript, namely, Section \ref{sec:general_geometries}; a presentation of these results to a simplified geometry, as well as more detailed comments can be found in Section \ref{sec:simplified}.  

This article is organized as follows. In Section \ref{sec:problem_setting} we introduce a simplified geometry, for which we will perform most of the computations, and outline the key results of the paper. Section \ref{sec:subdomains} is dedicated to preliminary results, namely, studies of \eqref{eq:B3} posed in subdomains with $\alpha>0$ (resp. $\alpha<0$). Next, we establish the limiting absorption principle in Section \ref{sec:LAP}. Section \ref{sec:LAP_Problem} is dedicated to a formulation of a well-posed problem satisfied by the limiting absorption solution. In Section \ref{sec:general_geometries} we show how all the arguments presented in previous sections can be altered to consider more general geometries and comment how $\omega\neq 0$ can be treated. In particular, we present the related results for a sufficiently regular domain with a hole.

\section{A simplified problem, notation, principal results}
\label{sec:problem_setting}
%\subsection{Motivation from physics}
\subsection{A simplified problem on a rectangle}
\label{sec:simpl}
\subsubsection{The geometry and the boundary-value problem}
Let $\Omega$ be a rectangle divided in two sub-rectangles and an interface between them:
\begin{align}
	\label{eq:dom_def}
	\begin{split}
		&\Omega=(-a,a)\times (-\ell,\ell),\quad \Omega_p=(0,a)\times (-\ell,\ell), \quad \Omega_n=(-a, 0)\times(-\ell,\ell), \quad \Sigma=\{0\}\times(-\ell,\ell), \quad a>0. 
	\end{split}
\end{align}
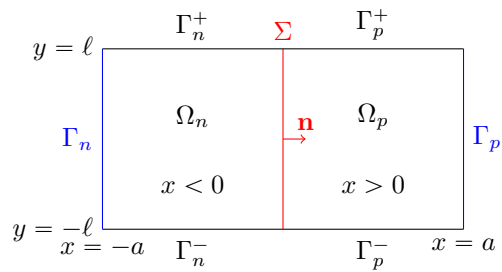
\begin{figure}[htb!]
	\centering
	\begin{tikzpicture}
		\begin{scope}[scale=1.2]
			\draw (-2,-1) -- +(4,0);
			\draw[red] (0,-1) -- +(0,2);
			\draw (-2,1) -- +(4,0);
			\draw[blue] (-2,-1) -- + (0,2);
			\draw[blue] (2,-1) -- + (0,2);
			\draw (-2,-1) node[below] { $x = - a$}
			node[left] { $y = -\ell$};
			\draw (2,-1) node[below] { $x = a$};
			\draw (-2,1) node[left] { $y = \ell$};
			
			\draw (1,0.25) node {$\Omega_p$};
			\draw (1, -0.5) node {$ x >0$};
			
			\draw (-1,0.25) node {$\Omega_n$};
			\draw (-1, -0.5) node {$ x <0$};
			
			\node[below] at (-1, -1) {$\Gamma_n^{-}$};
			\node[below] at (1, -1) {$\Gamma_p^{-}$};
			
			\node[above] at (-1, 1) {$\Gamma_n^{+}$};
			\node[above] at (1, 1) {$\Gamma_p^{+}$}; 
			
			\draw[red] (0,1) node[above] {$\Sigma$};
			\draw[blue] (-2,0) node[left] {$\Gamma_n$};
			\draw[blue] (+2,0) node[right] {$\Gamma_p$};
			
			\draw[red] (0.25,0) node[above] {$\mathbf{n}$};
			\draw[red,->] (0,0)--(0.25,0);
		\end{scope}
	\end{tikzpicture}
	\caption{An illustration to the simplified domain considered in Section \ref{sec:simpl}.}
	\label{fig:ill_geom}
\end{figure}
%	The assumption $a<1$ allows to shorten some of the expressions and is not necessary. 

We denote by $\vec{x}=(x,y)$ an element of $\mathbb{R}^2$. 
The part of the boundary of $\Omega$ (resp. $\Omega_{\lambda}, \, \lambda\in\{p,n\}$) intersecting the lines $y=\pm\ell$ is denoted by $\Gamma^{\pm}$ (resp. $\Gamma^{\pm}_{\lambda}$, $\Gamma^{\pm}_{\lambda}:=\partial\Omega_{\lambda}\cap \{y=\pm\ell\}, \, \lambda\in\{p,n\}$). Let us set $\Gamma:=\partial\Omega$. We define the unit normal $\vec{n}$ to  $\partial\Omega\cup\Sigma$. It is directed into the exterior of $\partial\Omega$; and, when considering $\Sigma$, from $\Omega_n$ into $\Omega_p$, see Figure \ref{fig:ill_geom}.

Next, let us define two matrix-valued functions $\mathbb{A}, \, \mathbb{T}: \, \overline{\Omega}\rightarrow \mathbb{C}^{2\times2}$ that satisfy the following assumptions.
%	\begin{assumption}
	%		\label{assumption:alpha}
	%		\begin{enumerate}
		%			\item $ x  \in C^{3}(\overline{\Omega})$;
		%			\item $ x $ is periodic in $y$-direction, meaning that $ x = x (x,y)$, $(x,y)\in\overline{\Omega}$, satisfies $\partial_y^k x (.,\ell)=\partial_y^k x (.,-\ell)$, $k=0,..,3$. 
		%			\item 	$ x $ changes its sign inside $\overline{\Omega}$, in particular, $ x >0$ in $\Omega_p$, $ x <0$ in $\Omega_n$, and $\left.  x \right|_{\Sigma}=0$. Moreover, $\left.\partial_x  x \right|_{\Sigma}=r$, $r=r(y)>0$ for all $y\in [-\ell,\ell]$.
		%		\end{enumerate}
	%		These conditions imply that $ x $ writes
	%		\begin{align*}
		%			 x (x,y)=r(y)x+x^2 R(x,y), \quad (x,y)\in \overline{\Omega}, 
		%		\end{align*}
	%		where the $y-$periodic functions $r$ and $R$ satisfy: $r\in C^3([-\ell,\ell])$, $R\in C^1(\overline{\Omega})$ and s.t. $\partial_y^2 R\in C(\overline{\Omega})$. 
	%	\end{assumption}
\begin{assumption}
	\label{assump:matrices}
	\begin{itemize}
		\item $\mathbb{A}, \, \mathbb{T}\in {C}^{1,1}(\overline{\Omega}; \mathbb{C}^{2\times 2})$. 
		\item For all $\bx\in \overline{\Omega}$, $\mathbb{A}(\bx), \mathbb{T}(\bx)$ are both Hermitian, positive-definite matrices.

		In particular,
		defining for $\vec{p}, \vec{v}\in \mathbb{C}^2$, 
		\begin{align*}
			\vec{p}\cdot \vec{v}=p_1v_1+p_2v_2,\quad |\vec{p}|^2=\|\vec{p}\|_{\mathbb{C}^2}^2=\vec{p}\cdot\overline{\vec{p}},
		\end{align*}
		it holds, for all $\vec{p}\in \mathbb{C}^2$, all $\vec{x}\in \overline{\Omega}$,  $
		\mathbb{T}(\bx)\vec{p}\cdot\overline{\vec{p}}\geq c_{\mathbb{T}}\|\vec{p}\|^2_{\mathbb{C}^2}, \quad 			\mathbb{A}(\bx)\vec{p}\cdot\overline{\vec{p}}\geq c_{\mathbb{A}}\|\vec{p}\|^2_{\mathbb{C}^2}, $ where $c_{\mathbb{T}}, c_{\mathbb{A}}>0$. 
		\item  Moreover, $\mathbb{A}$ and $\mathbb{T}$ satisfy the following periodicity constraints: 
		\begin{align*}
			\partial_y^k\mathbb{A}(.,\ell)=	\partial_y^k\mathbb{A}(.,-\ell), \qquad 	\partial_y^k\mathbb{T}(.,\ell)=	\partial_y^k\mathbb{T}(.,-\ell), \quad k=0,\,1. 
		\end{align*}
	\end{itemize}
	We will use the following notation for the values of $\mathbb{A}$ and $\mathbb{T}$ on $\Sigma$:
	\begin{align*}
		\left.	\mathbb{A}\right|_{\Sigma}=\left(
		\begin{matrix}
			a_{11} & a_{12}\\
			\overline{a_{12}} & a_{22}
		\end{matrix}
		\right), \qquad \left.	\mathbb{T}\right|_{\Sigma}=\left(
		\begin{matrix}
			t_{11} & t_{12}\\
			\overline{t_{12}} & t_{22}
		\end{matrix}
		\right).
	\end{align*}
\end{assumption}
An immediate corollary of the above assumption reads.
\begin{corollary}
	\label{cor:M}
	The matrix-valued function  $\vec{x}\mapsto\mathbb{M}_{\nu}(\vec{x}):=x\mathbb{A}(\vec{x})+i\nu\mathbb{T}(\vec{x})$ satisfies: 
	$\operatorname{\Im}\left(\mathbb{M}_{\nu}(\vec{x})\vec{p}\cdot\overline{\vec{p}}\right)\geq \nu c_{\mathbb{T}}\|\vec{p}\|^2_{\mathbb{C}^2}$, $\operatorname{\Re}\left(\mathbb{M}_{\nu}(\vec{x})\vec{p}\cdot \overline{\vec{p}}\right)=x \mathbb{A}(\vec{x})\vec{p}\cdot\overline{\vec{p}}$, for all $\vec{p}\in \mathbb{C}^2$. 
\end{corollary}%
We study the following family of well-posed problems, parametrized by the absorption parameter $\nu>0$: given $f\in L^2(\Omega)$, find $u^{\nu}\in H^2(\Omega)$, s.t. 
\begin{align}
	\label{eq:unu}
	\begin{split}
		& \operatorname{div}((x\mathbb{A}+i\nu\mathbb{T})\nabla u^{\nu})=f, \\
		&u^{\nu}=0 \text{ on }\Gamma_p\cup\Gamma_n, \\
		&u^{\nu}(x,\ell)=u^{\nu}(x, -\ell), \quad \text{ a. e. }x\in (-a, a),\\
		&\partial_y u^{\nu}(x,\ell)=\partial_y u^{\nu}(x, -\ell), \quad \text{a. e. }x\in (-a, a).
	\end{split}
\end{align}
%\begin{remark}
%	The interested reader can consult Appendix \ref{appendix:derivation} for the physical origin of the present model. 
%\end{remark}
Remark that in the above problem the matrix-valued function $\vec{x}\mapsto x\mathbb{A}(\vec{x})$ in the principal part of the operator degenerates on $\Sigma$, is positive definite in $\Omega_p$ and negative definite in $\Omega_n$. The problem is regularized by adding the elliptic viscosity term $i\nu\operatorname{div}(\mathbb{T}\nabla .)$. The goal of this manuscript is to show that $u^{\nu}$ converges in a given topology to a function $u^+$ and write a well-posed problem for $u^+$.

%On the other hand, it is still unclear how to prove the well-posedness of the resulting problem.

%Since the PDE is sign-definite in $\Omega_p$ (resp. $\Omega_n$), we are able to define a self-adjoint operator associated to the problem.

\subsubsection{Notation}
\paragraph{\textbf{Sobolev spaces with periodic and homogeneous boundary conditions}} Recall that $L^2(\Omega)$ is a space of complex-valued square-integrable functions on $\Omega$. We will use the following notation:
\begin{align*}
	\|v\|^2=\|v\|^2_{L^2(\Omega)}=\int_{\Omega}|v(\vec{x})|^2d\vec{x} \quad (u, v)=(u,v)_{L^2(\Omega)}=\int_{\Omega}u(\vec{x})\overline{v(\vec{x})}d\vec{x}.
\end{align*} 
We also have, for $m\geq 1$,  $H^m(\Omega):=\{v\in L^2(\Omega): \, \|v\|^2_{H^m(\Omega)}:=\|v\|^2+\sum\limits_{|\vec{\boldsymbol{\beta}}|_{1}\leq m}\|D^{\boldsymbol{\beta}} v\|^2<+\infty\}. $ In a similar manner, we define the spaces  $L^2(\Omega_{\lambda})$ and $H^m(\Omega_{\lambda})$, $\lambda\in \{n,p\}$. We will use notation $\|.\|_{\mathcal{O}}:=\|.\|_{L^2(\mathcal{O})}$; the meaning of the scalar product $(.,.)$ as a scalar product of $L^2(\Omega)$ or $L^2(\mathcal{O})$ will be clear from the context.  	

\textit{Typically all the functions we consider are periodic in the direction $y$ and have vanishing traces on $\Gamma_p$ (resp. $\Gamma_n$). 
	The associated spaces will be denoted by calligraphic letters ($\mathcal{C}^m$, $\mathcal{H}^s$ etc.). We avoid putting the indices $0$ in the definitions of such spaces (cf. e.g. $H^1_0$), since the functions we consider, in general, do not vanish on $\Sigma$.}
\paragraph{\textbf{Standard trace operators and their restrictions}}
For a piecewise-regular $u$,  i.e. $\left. u\right|_{\Omega_{\lambda}}\in C^1(\overline{\Omega_{\lambda}})$, $\lambda\in \{n,p\}$,  we define associated trace operators:
\begin{align}
	\label{eq:nmtraces}
	\gamma_0^\mathcal{D}u=\left. u\right|_{\mathcal{D}},   \qquad
	\gamma_n^\mathcal{D}u=\left.x \mathbb{A}\nabla u\cdot \vec{n}\right|_{\mathcal{D}}, \qquad\gamma_{n, \nu}^{\mathcal{D}}u=\left.(x \mathbb{A}+i\nu\mathbb{T})\nabla u\cdot \vec{n}\right|_{\mathcal{D}}\qquad  \mathcal{D}\subset \Sigma\cup\Gamma.
\end{align}
We introduce additionally, for $\vec{x}\in \Sigma$, right and left traces and a trace jump
\begin{align}
	\label{eq:trace_left_right}
	\begin{split}
		&\gamma_0^{\Sigma,p}u:=\lim\limits_{\Omega_p\ni \tilde{\vec{x}}\rightarrow \vec{x}}u(\tilde{\vec{x}}), \qquad	\gamma_0^{\Sigma,n}u:=\lim\limits_{\Omega_n\ni \tilde{\vec{x}}\rightarrow \vec{x}}u(\tilde{\vec{x}}),\qquad [\gamma_0^{\Sigma}u]:=	\gamma_0^{\Sigma,p}u-\gamma_0^{\Sigma,n}u.
	\end{split}	
\end{align}	
\textit{\textbf{Sobolev spaces of functions on $\Omega_{\lambda}$, $\lambda\in \{n,p\}$. }}	Given $\lambda\in \{n,p\}$, with an obvious abuse of notation where the spaces $H^{1/2}(\Gamma_{\lambda}^+)$ and $H^{1/2}(\Gamma_{\lambda}^{-})$ are identified, we define:
\begin{alignat*}{2}
	&\mathcal{H}^1(\Omega_{\lambda}) 
	&\;:=\; &\left\{ u \in H^1(\Omega_{\lambda}) : \gamma_{0}^{\Gamma_{\lambda}^+}u - \gamma_{0}^{\Gamma_{\lambda}^{-}}u = 0, \quad \gamma_{0}^{\Gamma_{\lambda}} u = 0 \right\}, \\
	&\mathcal{H}^1(\Omega) 
	&\;:=\; &\left\{ u \in H^1(\Omega) : \left. u \right|_{\Omega_{\lambda}} \in \mathcal{H}^1(\Omega_{\lambda}), \, \lambda \in \{n, p\} \right\},\\
	&\mathcal{H}^1(\Omega\setminus\Sigma) 
	&\;:=\; &\left\{ u \in L^2 (\Omega) : \left. u \right|_{\Omega_{\lambda}} \in \mathcal{H}^1(\Omega_{\lambda}), \, \lambda \in \{n, p\} \right\}.
	%	&\mathcal{H}^1_{0,\Sigma}(\Omega_{\lambda}) 
	%	&\;:=\; &\left\{ u \in \mathcal{H}^1(\Omega_{\lambda}) : \gamma_{0}^{\Sigma} u = 0 \right\}, \\
	%	&\mathcal{H}^1_{0,\Sigma}(\Omega) 
	%	&\;:=\; &\left\{ u \in \mathcal{H}^1(\Omega) : \gamma_{0}^{\Sigma} u = 0 \right\}, \\
\end{alignat*}
%	
%	\begin{alignat*}{2}
	%		&\mathcal{H}^1(\Omega_{\lambda})&:= \{u\in H^1(\Omega_{\lambda}): \, \gamma_{0}^{\Gamma_{\lambda}^+}u-\gamma_{0}^{\Gamma_{\lambda}^{-}}u=0, \quad \gamma_{0}^{\Gamma_{\lambda}} u=0\},\\
	%		&\mathcal{H}^1(\Omega)&:=\{u\in H^1(\Omega): \,\left. u\right|_{\Omega_{\lambda}}\in \mathcal{H}^1(\Omega_{\lambda}), \, \lambda\in \{n,p\} \}, \\
	%		&\mathcal{H}^1_{0,\Sigma}(\Omega_{\lambda})&:=\{u\in \mathcal{H}^1(\Omega_{\lambda}): \gamma_{0}^{\Sigma}u=0\},\\
	%		&\mathcal{H}^1_{0,\Sigma}(\Omega)&:=\{u\in \mathcal{H}^1(\Omega): \gamma_{0}^{\Sigma}u=0\},\\
	%		&\mathcal{H}^{1/2}(\Sigma)&:=\{v\in L^2(\Sigma): \, \exists V\in \mathcal{H}^1(\Omega), \text{ s.t. } \,v=\gamma_{0}^{\Sigma} V\},\quad \|v\|_{\mathcal{H}^{1/2}(\Sigma)}^2:=\inf\limits_{V\in \mathcal{H}^1(\Omega): \, \gamma_{0}^{\Sigma}V=v}\|V\|^2_{H^1(\Omega)},\\
	%		&\mathcal{H}^{-1/2}(\Sigma)&:=(\mathcal{H}^{1/2}(\Sigma))'. 
	%	\end{alignat*}
Additionally, we have, with  $\lambda\in \{n,p\}$, $k\in \mathbb{N}\cup\{+\infty\}$, 
\begin{alignat*}{2}
	&\mathcal{C}^k(\overline{\Omega}_{\lambda}) & := &\{u \in C^k(\overline{\Omega}_{\lambda}): \, \gamma_{0}^{\Gamma_{\lambda}}u=0, \, \gamma_0^{\Gamma_{\lambda}^+}D^{\vec{\beta}}u=\gamma_0^{\Gamma_{\lambda}^-}D^{\vec{\beta}}u, \quad |\vec{\beta}|\leq k\},  \quad,\\
	&\mathcal{C}^k_{comp}(\overline{\Omega}_{\lambda}) & := &\{u \in \mathcal{C}^k(\overline{\Omega}_{\lambda}): \operatorname{dist}(\operatorname{supp}u,\Sigma)>0\},\\	&\mathcal{C}^k(\overline{\Omega})&:=&\{u \in C^k(\overline{\Omega}):  \left. u\right|_{\Omega_{\lambda}}\in \mathcal{C}^k(\overline{\Omega_{\lambda}}), \quad \lambda\in \{n,p\}\}.
\end{alignat*}
For $k\geq 1$, we also define $\mathcal{H}^k(\Omega_{\lambda})=\overline{\mathcal{C}^{\infty}(\overline{\Omega}_{\lambda})}^{\|.\|_{H^1(\Omega_{\lambda})}}=\overline{\mathcal{C}^{k+1}(\overline{\Omega}_{\lambda})}^{\|.\|_{H^1(\Omega_{\lambda})}},$
and, for $s\in [0,1]$, 
\begin{align*}
	\mathcal{H}^s(\Omega_{\lambda})=\overline{\mathcal{C}^{\infty}(\overline{\Omega}_{\lambda})}^{\|.\|_{H^s(\Omega_{\lambda})}}=\overline{\mathcal{C}^k(\overline{\Omega}_{\lambda})}^{\|.\|_{H^s(\Omega_{\lambda})}}, \quad k\geq 1.
\end{align*}
In the above $\|v\|_{H^s(\Omega_{\lambda})}^2=\|v\|^2_{L^2(\Omega_{\lambda})}+|v|_{H^s(\Omega_{\lambda})}^2$, with $|.|_{H^s}$ being the usual Sobolev-Slobodeckii seminorm (cf. \cite[(3.18)]{mclean} for the respective definition). Recall that, for $0<s<1/2$, $\mathcal{H}^s(\Omega_{\lambda})=H^s(\Omega_{\lambda})$. 	

The above definitions extend naturally to $\Omega$ instead of $\Omega_{\lambda}$. 

\textit{\textbf{Sobolev spaces with weights.}} For $\delta\leq 2$, $\lambda\in \{n,p\}$, we define the family of Hilbert spaces (see  \cite[Theorem 1.3, Theorem 1.11 and its proof]{kufner1984define})
\begin{align}
	\label{eq:wsp}
	\begin{split}
		&L^2_{\delta}(\Omega_{\lambda}):=\{v\in L^2_{loc}(\Omega_{\lambda}): \, \|v\|_{L^2_{\delta}(\Omega_{\lambda})}:=\||x|^{\delta/2}v\|_{L^2(\Omega_{\lambda})}<\infty\},\\
		&\mathcal{H}^1_{\delta}(\Omega_{\lambda})=\{v\in L^2(\Omega_{\lambda}): \, \|u\|_{\mathcal{H}^1_{\delta}(\Omega_{\lambda})}^2:=\|u\|^2_{L^2(\Omega_{\lambda})}+\|\nabla u\|^2_{L^2_{\delta}(\Omega_{\lambda})}<\infty, \; \gamma_0^{\Gamma_{\lambda}}u=0,\;
		\gamma_0^{\Gamma_{\lambda}^+}u-\gamma_0^{\Gamma_{\lambda}^{-}}u=0\}.
	\end{split}
\end{align}
\begin{remark}
	To facilitate the distinction between these spaces, let us consider behaviour of functions from these spaces close to the interface $\Sigma$:
	\begin{itemize}
		\item for $\delta<1$, the trace operator $\gamma_0^{\Sigma,\lambda}\in \mathcal{L}(\mathcal{H}^1_{\delta}(\Omega_{\lambda}); L^2(\Sigma))$, cf. Corollary \ref{cor:tr} in Appendix \ref{appendix:weighted}.
		\item for $1\leq \delta<2$, this is not true. In particular,  $\mathcal{C}^{\infty}_{comp}(\Omega_{\lambda})$ are dense in $\mathcal{H}^1_{\delta}(\Omega_p)$, cf. Proposition \ref{prop:density_vreg}.  
	\end{itemize}	
\end{remark}
%
%
%\begin{remark}
%	Since a priori the functions $\mathcal{H}^1_{\delta}$ for $\delta$ sufficiently large may not admit $L^2$-traces, the periodic boundary conditions in the above definitions are understood in the following sense. 
%	Then the periodicity constraints in \eqref{eq:wsp} are understood as  $(\gamma_0^{\Gamma_{\lambda}^+}-\gamma_0^{\Gamma_{\lambda}^{-}})(1-\varphi_{\re})u=0$ for all $\re>0$. 
%\end{remark} 
We single out two spaces of functions that do not admit traces on $\Sigma$: 
$
\mathcal{V}_{reg}(\Omega_{\lambda}):=\mathcal{H}^1_{1}(\Omega_{\lambda}), \, \mathcal{V}_{sing}(\Omega_{\lambda}):=\mathcal{H}^1_{2}(\Omega_{\lambda}), 
$
and define
\begin{align}
	\label{eq:two_spaces}
	\mathcal{V}_{reg}=\mathcal{V}_{reg}(\Omega):=\mathcal{V}_{reg}(\Omega_{n})\times \mathcal{V}_{reg}(\Omega_p), \quad \mathcal{V}_{sing}=\mathcal{V}_{sing}(\Omega):=\mathcal{V}_{sing}(\Omega_{n})\times \mathcal{V}_{sing}(\Omega_p).
\end{align}	
The introduction of these two spaces will be motivated further in the paper, see Section \ref{sec:homog}.

For $0\leq\delta<1$, we define the following space (remark that this space is defined globally on $\Omega$, unlike \eqref{eq:wsp}; it is Hilbert as argued in \cite[Theorem 1.11]{kufner1984define}): 
\begin{align*}
	\mathcal{H}^1_{\delta}(\Omega):=\{v\in L^2(\Omega): \, \|v\|_{\mathcal{H}^1_{\delta}(\Omega)}^2:=\|v\|^2+\||x|^{\delta/2}\nabla v\|^2<\infty,\quad \gamma_0^{\Gamma_p\cup\Gamma_n}u=0,  \quad \gamma_0^{\Gamma^+}u-\gamma_0^{\Gamma^-}u=0\},
\end{align*}
and for $\delta\geq 1$, we will make use solely of 
$\mathcal{H}^1_{\delta}(\Omega\setminus\Sigma)=\mathcal{H}^1_{\delta}(\Omega_{n})\times \mathcal{H}^1_{\delta}(\Omega_{p}).$  Importantly, for $\delta<1$,  $\mathcal{H}^1_{\delta}(\Omega\setminus\Sigma)\neq \mathcal{H}^1_{\delta}(\Omega)$. We will also need $
L^2_{\delta}(\Omega)=L^2_{\delta}(\Omega_n)\times L^2_{\delta}(\Omega_p).$\\
%
%The restriction $0<\delta<1$ allows to simplify the definition of the weighted space, as discussed in \cite{zhikov}, and is sufficient for our needs.	
%
\textit{\textbf{Spaces on the interface $\Sigma$. }} 
Fractional Sobolev spaces $\mathcal{H}^s(\Sigma)$  are defined via 
\begin{align*}
	\mathcal{H}^s(\Sigma):=\{v\in L^2(\Sigma): \, \exists V\in \mathcal{H}^{s+1/2}(\Omega), \text{ s.t. }\, v=\gamma_0^{\Sigma}V\},
\end{align*}
for $s>0$ (cf. also Theorem 3.37 of \cite{mclean}), with the standard induced norm. 
We will also need
\begin{alignat*}{2}
	&\mathcal{H}^{-1/2}(\Sigma) 
	&\;:=\; &\left( \mathcal{H}^{1/2}(\Sigma) \right)',
\end{alignat*}
the dual space of linear forms. Let us define $\langle.,.\rangle_{V',V}$ the duality bracket \textit{linear} in both arguments:
\begin{align*}
	\langle g, h\rangle_{\mathcal{H}^{-1/2}(\Sigma), \mathcal{H}^{1/2}(\Sigma)}=\langle g, h\rangle_{\Sigma}:=\int_{\Sigma}g(y)h(y)dy, \quad \text{ when  }V=\mathcal{H}^{-1/2}(\Sigma),\quad g\in L^2(\Sigma).
\end{align*}
Similarly, $\langle g, h\rangle_{L^2(\Sigma)}=\int_{\Sigma}g(y)h(y)dy$.\\ 
\textit{\textbf{The Neumann trace on $\Sigma$}. }
We will need the conormal trace defined on $\Sigma$ for the problem \eqref{eq:unu} with $\nu=0$, whose strong counterpart is $\lim\limits_{\Omega_{\lambda}\ni \vec{x}\rightarrow \vec{x}_0\in \Sigma}({x}\mathbb{A}(\vec{x})\nabla u(\vec{x}))\cdot \vec{n}(\vec{x}_0)$. Remark that for $\mathcal{C}^1(\overline{\Omega_{\lambda}})$-functions this quantity vanishes, and we will be interested in the classes of functions where this is no longer the case. We will thus heavily use its variational characterization, which we recall for the convenience of the reader. Let us define the weighted space, $\delta\geq 0$, 
\begin{align}
	\label{eq:hdelta}
	\mathcal{H}_{\delta}(\operatorname{div}(x\mathbb{A}\nabla.); \Omega_{\lambda}):&=\{v\in \mathcal{V}_{sing}(\Omega_{\lambda}): \, \operatorname{div}(x\mathbb{A}\nabla v)\in L^2_{\delta}(\Omega_{\lambda}), \quad (\gamma_{n}^{\Gamma^+_{\lambda}}+\gamma_n^{\Gamma^{-}_{\lambda}})v=0\}, \quad \lambda\in \{n,p\}.
\end{align}
Given $u\in \mathcal{H}_0(\operatorname{div}(x\mathbb{A}\nabla.); \Omega_{\lambda})$, the conormal trace  $\gamma_n^{\Sigma,\lambda}u$ is well-defined via the generalized integration by parts formula (Theorem 2.2 in \cite{girault_raviart}), e.g. for $\lambda=p$,
\begin{align*}
	\langle \gamma_n^{\Sigma,p}u, \varphi\rangle_{\mathcal{H}^{-1/2}(\Sigma), \mathcal{H}^{1/2}(\Sigma)}:=-\int_{\Omega_p}\operatorname{div}(x\mathbb{A}\nabla u)\, {\Phi}d\vec{x}-\int_{\Omega_p}x\mathbb{A}\nabla u\, {\nabla \Phi}d\vec{x}, \quad \forall \Phi\in \mathcal{H}^1(\Omega_p) \text{ with }\gamma_0^{\Sigma}\Phi=\varphi.
\end{align*}	
The above definition of the conormal trace $\gamma_n^{\Sigma,p}u\in \mathcal{H}^{-1/2}(\Sigma)$ extends verbatim to the space   $\mathcal{H}_{\delta}(\operatorname{div}(x\mathbb{A}\nabla .); \Omega_{\lambda})$, provided that  $0<\delta<1$, since the expression 
$$\int_{\Omega_p}\operatorname{div}(x\mathbb{A}\nabla u)\, {\Phi}=\int_{\Omega_p}x^{\delta/2}\operatorname{div}(x\mathbb{A}\nabla u)\, x^{-\delta/2}{\Phi}d\vec{x}$$ is well-defined as the Lebesgue's integral, see Lemma \ref{lem:f_belongs_vreg}. Thus defined conormal trace satisfies $$\gamma_n^{\Sigma,p}\in \mathcal{L}(\,\mathcal{H}_{\delta}(\operatorname{div}(x\mathbb{A}\nabla.); \,\Omega_{\lambda}); \mathcal{H}^{-1/2}(\Sigma)), \quad 0\leq\delta<1.$$
Cf. also Theorem 2.2.22 of \cite{ciarlet_book}, as well as the discussion after Lemma 4.3 in \cite{mclean}. We will use these facts in the paper without referring to this discussion.

We will also need  $[\gamma_n^{\Sigma}u]:=\gamma_n^{\Sigma,p}u-\gamma_n^{\Sigma,n}u$, as well as the conormal trace for the problem \eqref{eq:unu} with $\nu>0$: 
\begin{align*}
	\gamma_{n,\nu}^{\Sigma,\lambda}u:=\lim\limits_{\Omega_{\lambda}\ni \vec{x}\rightarrow \vec{x}_0\in \Sigma}(({x}\mathbb{A}(\vec{x})+i\nu\mathbb{T}(\vec{x})\nabla u(\vec{x}))\cdot \vec{n}(\vec{x}_0), \quad\lambda\in \{n,p\}
\end{align*}

\textbf{\textit{Auxiliary notation. }}We will use the following notation, for the domain $\mathcal{O}$ being one of the domains $\Omega_{\lambda}$,  $\lambda\in \{n,p\}$, $\Omega$ or $\Omega\setminus\Sigma$:
\begin{align*}
	\mathcal{H}^{s-}(\mathcal{O}):=\bigcap_{0<\varepsilon\leq s}\mathcal{H}^{s-\varepsilon}(\mathcal{O}), \quad 	\mathcal{H}^{t-}(\Sigma):=\bigcap_{0<\varepsilon\leq t}\mathcal{H}^{t-\varepsilon}(\Sigma), \quad 0<t<1/2. 
\end{align*}
We will say that a sequence $v_k$ converges to $v$ in $\mathcal{H}^{s-}(\mathcal{O})$ if $v$ converges in $\mathcal{H}^{s-\re}(\mathcal{O})$ for all $\re>0$. 

Let us define a special set
\begin{align}
	\label{eq:omegaresigma}
	\Omega_{\Sigma}^{\delta}:=\{\bx\in \Omega: \, |\operatorname{dist}(\vec{x},\Sigma)|<\delta\},
\end{align}
and the family of cutoff functions, parametrized by the parameter $\varepsilon>0$, and supported in ${\Omega_{\Sigma}^{\re}}$:
\begin{align}
	\label{eq:cutoff_phi}
	\Omega\ni (x,y)\mapsto \varphi_{\varepsilon}(x):=\varphi_1\left(\frac{x}{\varepsilon}\right), \quad \varphi_1(x)=\left\{
	\begin{array}{ll}
		1, & |x|\leq 1/2,\\
		0, & |x|\geq 1,\\
		\in (0,1), & |x|\in (1/2,1),
	\end{array}
	\right.\quad\varphi_1\in C^{\infty}(\mathbb{R}).
\end{align}
In what follows, we will write $a\lesssim b$ to indicate that $a\leq Cb$, for some constant $C>0$, independent of the absorption parameter $\nu>0$ (cf. \eqref{eq:unu}) and data (traces/right-hand sides) of the problem, but possibly dependent on $\Omega$, $\Sigma$, $\mathbb{T}$, $\mathbb{A}$.

For brevity, we will sometimes write 
\begin{align*}
	&\int_{\Omega} f \text{ for }\int_{\Omega} f(\vec{x})d\vec{x},\quad 
	\int_{\Sigma} f\text { for }\int_{\Sigma}f(y)dy,\qquad\text{ BCs for boundary conditions}.
\end{align*}
By $\mathbb{A}^t$, $\mathbb{T}^t$ and so on we will denote transposes of matrices $\mathbb{A}$, $\mathbb{T}$ etc. 
\subsubsection{A preliminary well-posedness result and motivation}
For all $\nu>0$, the problem \eqref{eq:unu} is well-posed. Indeed, let the form $a_{\nu}: \, \mathcal{H}^1(\Omega)\times \mathcal{H}^1(\Omega)\rightarrow \mathbb{C}$ be defined by
\begin{align}
	\label{eq:anu}
	a_{\nu}(u, v):=(( x \mathbb{A}+i\nu\mathbb{T})\nabla u, \nabla v), \text{ so that } a_{\nu}(u^{\nu}, v)=-\int_{\Omega}f\, \overline{v}, \quad \forall v\in \H^1(\Omega). 
\end{align}
\begin{lemma}
	\label{lem:pb_abs_wp}
	For each $f\in L^2(\Omega)$, $\nu>0$, the problem \eqref{eq:unu} admits a unique solution $u^{\nu}\in \mathcal{H}^1(\Omega)$. Also, 
	\begin{align}
		\label{eq:est_main}
		&\nu^{1/2}\|\nabla u^{\nu}\|\lesssim  \sqrt{\|f\|\|u^{\nu}\|},\\
		\label{eq:est_main2}
		& \|u^{\nu}\|_{H^1(\Omega)}\lesssim \nu^{-1}\|f\|. 
	\end{align}
	The solution $u^{\nu}$ belongs to  $\mathcal{H}^2(\Omega)$ for all $\nu>0$. 
\end{lemma}
\begin{proof}
	Consider \eqref{eq:anu} and remark that, by Corollary \ref{cor:M}, for any $u\in \mathcal{H}^1(\Omega)$, it holds that  
	\begin{align}
		\label{eq:anubound}
		\Im a_{\nu}(u, u)\geq c_{\mathbb{T}}\nu\|\nabla u\|^2\gtrsim \nu \|u\|^2_{H^1(\Omega)},
	\end{align}
	where the last bound follows by the Poincar\'e inequality in $\mathcal{H}^1(\Omega)$ (valid since functions from $\mathcal{H}^1(\Omega)$ vanish on $\Gamma_{p}\cup\Gamma_n$). The well-posedness of \eqref{eq:anu} in $\mathcal{H}^1(\Omega)$ follows by continuity of $a_{\nu}$ and the Lax-Milgram lemma.
	
	The stability estimate \eqref{eq:est_main} is obtained by taking the imaginary part of both sides of \eqref{eq:anu} and using  \eqref{eq:anubound}:
	\begin{align*}
		c_{\mathbb{T}}\nu\|\nabla u^{\nu}\|^2\leq \Im a_{\nu}(u^{\nu}, u^{\nu})=-\Im (f, u^{\nu})\leq \|f\|\|u^{\nu}\|.
	\end{align*}
	The bound \eqref{eq:est_main2} follows from the above, using the second inequality in \eqref{eq:anubound} and the Poincar\'e inequality:
	\begin{align*}
		\nu\| u^{\nu}\|^2_{H^1(\Omega)}\lesssim \|f\|\|u^{\nu}\|\lesssim \|f\|\|u^{\nu}\|_{H^1(\Omega)}.
	\end{align*}
	The fact that $u^{\nu}\in \mathcal{H}^2(\Omega)$ follows by elliptic regularity, cf. e.g. the proof of \cite[Theorem 4.18]{mclean}. 
\end{proof}
While the above estimate shows the well-posedness of \eqref{eq:anu}, it does not indicate any convergence properties of the sequence $(u^{\nu})_{\nu>0}$ as $\nu\rightarrow 0$.  The principal goal of this paper is to investigate this question in detail. We present the principal results of this paper in the following section. 
\subsection{Principal results for the simplified problem}
\label{sec:simplified}
%
%
%	:=\overline{\mathcal{C}^{\infty}(\overline{\Omega}_{\lambda})}^{\|.\|_{\mathcal{H}^1_{\delta}(\Omega_{\lambda})}}, \quad \|u\|_{\mathcal{H}^1_{\delta}(\Omega_{\lambda})}^2:=\|u\|_{L^2(\Omega_{\lambda})}^2+\int_{\Omega_{\lambda}}|x|^{\delta}|\nabla u(\vec{x})|^2d\vec{x}.
%\end{align}
%
Recall the definition of $\mathcal{V}_{sing}$ in \eqref{eq:two_spaces}. The following holds true. 
\begin{theorem}[Limiting absorption principle]
	\label{theorem:LAP}
	Given $f\in L^2(\Omega)$, consider the family of solutions  $(u^{\nu})_{\nu>0}\subset \mathcal{H}^1(\Omega)$ to \eqref{eq:unu}. Then, as $\nu\rightarrow 0+$, $u^{\nu}\rightarrow u^+\in \mathcal{V}_{sing}$ strongly in  $H^{1/2-}(\Omega)$. 
\end{theorem}
\begin{definition}
	The function $u^+$ defined in Theorem \ref{theorem:LAP} is called a 'limiting absorption solution'. 
\end{definition}
Next, it can be shown that in a weak sense $u^+$ satisfies $\operatorname{div}( x  \mathbb{A}\nabla u^+)=f$. To state what we mean by this, we start with the following observation: a function $u\in \mathcal{V}_{sing}(\Omega)$ necessarily satisfies $x\mathbb{A}\nabla u\in L^2(\Omega)$, as argued in Proposition \ref{prop:property2}. This enables us to introduce the following definition.
\begin{definition}
	\label{def:st}
	We will say that $u\in \mathcal{V}_{sing}(\Omega)$ satisfies  $\operatorname{div}(x\mathbb{A}\nabla u)=f$, $f\in L^2(\Omega),$ if and only if 
	\begin{align*}
		\int_{\Omega}x\mathbb{A}\nabla u\,\nabla \varphi=-\int_{\Omega}f\,\varphi, \quad  \text{for all }\varphi\in C_0^{\infty}(\Omega).
	\end{align*}
	%	where the integral in the left-hand side is well-defined in the sense of Lebesgue by Proposition \ref{prop:property2}. 
	
	The above is equivalent to requiring that
	$\operatorname{div}(x\mathbb{A}\nabla u)=f$ in $\Omega_p\cup\Omega_n$ and $[\gamma_n^{\Sigma}u]=0$. 
\end{definition}
Unfortunately, the weak solution to $\operatorname{div}( x \mathbb{A} \nabla u)=f$, when considered in the space $\mathcal{V}_{sing}$ and equipped with appropriate boundary conditions on $\partial\Omega$, appears to be non-unique. In particular, if the absorption in \eqref{eq:unu} is taken negative, i.e. $\nu<0$, then  $u^{\nu}\rightarrow u^{-}$ with $u^{-}\neq u^+$, and the limit $u^{-}$ satisfies $\operatorname{div}(x\mathbb{A}\nabla u)=f$.

This shows that to state a well-posed problem for $u^+$, we need to restrict the space of solutions. Such a space cannot be singled out by imposing the regularity constraints, as it is typical in elliptic PDEs. The reason for this is that the statement of Theorem \ref{theorem:LAP} holds true for the limit from the left $\nu\rightarrow 0-$, and the two limits $u^+$ and $u^-$ have the same regularity but do not coincide, see the discussion after Theorem \ref{theorem:main_result}. Thus, the restriction is done by introducing a radiation-like condition, similarly to how it is done for the Helmholtz equation in unbounded domains, which allows to distinguish between the two limits. 

In order to state such a radiation-like condition, we take inspiration from \cite{ciarlet_mk_peillon}. We will define a Neumann and a Dirichlet trace of a singular solution $u\in \mathcal{V}_{sing}(\Omega)$, and, as we will see, it is a relation between these traces that will ensures uniqueness of the solution to our problem. To do so, we start with the following decomposition for functions from the subspace of  \eqref{eq:hdelta}, $\delta=0$, with $\mathcal{H}^{1/2}(\Sigma)$-conormal derivatives:
\begin{align}
	\label{eq:defvsing}
	\mathcal{V}_{sing}(\operatorname{div}(x\mathbb{A}\nabla.); \Omega)=\{v\in \mathcal{V}_{sing}(\Omega): \, \operatorname{div}(x\mathbb{A}\nabla v)\in L^2(\Omega), \, (\gamma_n^{\Gamma_+}+\gamma_n^{\Gamma_-})v=0,\, \gamma_n^{\Sigma}v\in \mathcal{H}^{1/2}(\Sigma)\}, 
\end{align}
equipped with the norm $
\|v\|_{	\mathcal{V}_{sing}(\operatorname{div}(x\mathbb{A}\nabla.); \Omega)}^2:=\|v\|_{\mathcal{V}_{sing}}^2+\|\gamma_n^{\Sigma}v\|^2_{\mathcal{H}^{1/2}(\Sigma)}.$
\begin{proposition}
	\label{prop:decomp1}
	Let $u\in 	\mathcal{V}_{sing}(\operatorname{div}(x\mathbb{A}\nabla.); \Omega)$. 
	Then $u$ can be decomposed in a unique manner as follows: 
	\begin{align}
		\label{eq:decomp_pi0}
		u=u_{reg}+u_{sing},\quad u_{sing}=u_h\log|x|,
	\end{align}
	where $u_{reg}\in \mathcal{H}^{1-}(\Omega\setminus\Sigma)$ and $u_h\in \mathcal{H}^1(\Omega)$ is a piecewise-$\mathbb{A}$-harmonic function that satisfies the following decoupled boundary-value problem (see Assumption \ref{assump:matrices} for the definition of $a_{11}$):
	\begin{align*}
		&\operatorname{div}(\mathbb{A}\nabla u_h)=0 \text{ in }\Omega\setminus \Sigma, \\
		&\gamma_0^{\Sigma}u_h=a_{11}^{-1}\gamma_n^{\Sigma}u,\\
		& \gamma_0^{\Gamma_p\cup\Gamma_n}u_h=0,\qquad
		\text{ periodic BCs at $\Gamma^{\pm}_p\cup\Gamma^{\pm}_n$}. %&\gamma_{n}^{\Gamma^+_{\lambda}}u_h+\gamma_{n}^{\Gamma^{-}_{\lambda}}u_h=0,\quad \lambda\in \{n,p\},
	\end{align*}
	In the above decomposition,  $\gamma_n^{\Sigma}u=\gamma_n^{\Sigma}u_{sing}$. 
\end{proposition}
The above proposition shows that, in general, solutions $u\in\mathcal{V}_{sing}(\Omega)$ to $\operatorname{div}(x\mathbb{A}\nabla u)\in L^2(\Omega)$ are not regular in the vicinity of the interface, and posses a logarithmic singularity on $\Sigma$. 
\begin{remark}
	The appearance of the logarithmic term can be understood by studying the 1D counterpart of \eqref{eq:decomp_pi0}. For a more general case, we refer the  interested reader to the work \cite{mazzeo}, which introduces the calculus of elliptic edge operators, and, more generally, to  $b$-calculus techniques \cite{melrose, grieser}. 
\end{remark} 
For $u$ as in the above proposition, we can define two types of traces. The first one is a classical conormal trace $\gamma_n^{\Sigma}u\in \mathcal{H}^{-1/2}(\Sigma)$, which is 'carried' by $\gamma_n^{\Sigma}u_{sing}$. The second trace, namely the one-sided Dirichlet trace for $u$, taken from $\Omega_p$ or $\Omega_n$, is not defined in a classical sense, since the singular term $u_h\log|x|$ obviously blows up in the vicinity of $\Sigma$. Nonetheless, we can define it for the regular part of the decomposition \eqref{eq:decomp_pi0}. Since $u_{reg}$ is only piecewise-regular, let us introduce the associated notation for its restrictions to $\Omega_{\lambda}$. Namely, for $v\in L^2(\Omega)$, we define 
\begin{align*}
	v_{\lambda}:=\left. v\right|_{\Omega_{\lambda}}\in L^2(\Omega_{\lambda}), \quad \lambda\in \{n,p\}.
\end{align*}
\begin{definition}
	\label{definition:one_sided_trace}
	Let $u$ be like in Proposition \ref{prop:decomp1}. We define the one-sided trace of $u$ on $\Sigma$ as a trace of its regular part: $
	\gamma_{0}^{\Sigma,\lambda}u:=\gamma_0^{\Sigma}u_{reg,\lambda}\in \mathcal{H}^{1/2-}(\Sigma), \; \lambda\in \{n,p\}.$ The jump of the traces is then defined via 
	$$
	[\gamma_{0}^{\Sigma}u]:=\gamma_0^{\Sigma,p}u_{reg}-\gamma_0^{\Sigma,n}u_{reg}.$$
\end{definition}
The notion of trace enables us to reformulate the problem satisfied by the limiting absorption solution $u^+$ essentially as a transmission problem between $\Omega_p$ and $\Omega_n$. We will single out the limiting absorption solution among all the solutions to
\begin{align}
	\label{eq:pb1}
	\begin{split}
		&u\in \mathcal{V}_{sing}(\operatorname{div}(x\mathbb{A}\nabla .); \Omega),\qquad \text{s.t. }\operatorname{div}(x\mathbb{A}\nabla u)=f \text{ in }\Omega.
	\end{split}
\end{align}
The following result is the second main result of this paper. 
\begin{theorem}
	\label{theorem:main_result}
	Given $f\in L^2(\Omega)$, the limiting absorption solution $u^+$ as defined in Theorem \ref{theorem:LAP} is a unique solution to the following well-posed problem: find $u$ that satisfies \eqref{eq:pb1},  and 
	\begin{align}
		\label{eq:traces}
		[\gamma_0^{\Sigma}u]=-i\pi a_{11}^{-1}\gamma_n^{\Sigma}u.
	\end{align}
\end{theorem}
\begin{remark}
	\label{rem:wp}
	In the above, the well-posedness is meant in the sense of Hadamard: \eqref{eq:pb1} combined with \eqref{eq:traces} admits a unique solution in $\mathcal{V}_{sing}(\operatorname{div}(x\mathbb{A}\nabla .); \Omega)$, and this solution satisfies the bound $\|u\|_{\mathcal{V}_{sing}(\operatorname{div}(x\mathbb{A}\nabla.); \Omega)}\leq C\|f\|_{L^2(\Omega)}$, with $C>0$ depending on $\Omega, \mathbb{A}, \mathbb{T}$ only.
\end{remark}
The above theorem shows that the limiting absorption solution satisfies a very peculiar relation between the jump of its traces and the co-normal trace. This relation indeed resembles the Sommerfeld radiation condition: it appears that the family of solutions to \eqref{eq:unu}  $(u^{\nu})_{\nu<0}$, as $\nu\rightarrow 0-$ admits a limit $u^{-}$, which satisfies \eqref{eq:pb1} and the condition \eqref{eq:traces} taken with the opposite sign $[\gamma_0^{\Sigma}u^{-}]=i\pi a_{11}^{-1}\gamma_n^{\Sigma}u^{-}$. This is seen later in the paper, in the proof of Theorem \ref{theorem:convergence_decomposition}, cf. Remark \ref{rem:lap_wrong_sign}. 

Proposition \ref{prop:decomp1} is proven in the end of Section \ref{sec:wp_irreg}, Theorems \ref{theorem:LAP} and \ref{theorem:main_result} are proven in Section \ref{sec:LAP_Problem}.
\begin{remark}
	The constraint $\gamma_n^{\Sigma}u\in \mathcal{H}^{1/2}(\Sigma)$ embedded into the space $\mathcal{V}_{sing}(\operatorname{div}(x\mathbb{A}\nabla.);\Omega)$ is of technical nature, since it allows us to define the notions of the Dirichlet and Neumann trace through the decomposition of Proposition \ref{prop:decomp1}. We believe that similar results hold for $\gamma_n^{\Sigma}u\in \mathcal{H}^{-1/2}(\Sigma)$, but we postpone the development of the corresponding argument to future works. 
\end{remark}
\begin{remark}
	\label{remark:periodicBCs}
	Periodic boundary conditions at $\Gamma^{\pm}$ are not essential for the analysis and can be replaced by homogeneous Dirichlet or Neumann boundary conditions.
\end{remark}
\subsection{A road-map to the proofs of the results of the paper}
Let us explain how the paper is organized in more detail. First of all, we will discuss the question of well-posedness of the problem \eqref{eq:unu} without the  absorption term. More precisely, we consider 
\begin{align}
	\label{eq:pb_two_domains}
	\begin{split}
		&\operatorname{div}(x\mathbb{A}\nabla u)=f\text{ in }\Omega_p\cup\Omega_n,\\
		&u=0 \text{ on }\Gamma_p\cup\Gamma_n,\qquad \text{ periodic BCs at }\Gamma^{\pm}_p\cup\Gamma^{\pm}_n. 
	\end{split}
\end{align}
Of course, we need to be precise on the definition of the spaces in which we will look for $u$. Because the above problem is sign-indefinite, we start by considering the above problem in one of the subdomains $\Omega_p$. 
One sees that perhaps one lacks a boundary condition at $\Sigma$. It is more natural to start with the Neumann boundary condition, since its well-definiteness relies solely on the requirement that  $x\mathbb{A}\nabla u\in L^2(\Omega_p)$ (contrary to the Dirichlet boundary condition which requires a regularity of $u$ itself). Thus, we first study the homogeneous Neumann problem (Section \ref{sec:homog}), and next the heterogeneous one (Section \ref{sec:heterog_neumann}). These studies lead us, on one hand, to Proposition \ref{prop:decomp1} about the decomposition of the fields, and, on the other had, pave the way to the proof of Theorem \ref{theorem:stability_estimate} about the boundedness of $\|u^{\nu}\|_{L^2(\Omega)}$ uniformly in $\nu$.

Next, we get back to the original problem with the absorption \eqref{eq:unu}. We will prove two facts. 
First, it is the fact that  the family $(u^{\nu})_{\nu>0}$ is uniformly bounded in $\nu$ in  $\mathcal{V}_{sing}(\Omega)$; this will show that the sequence $(u^{\nu})_{\nu>0}$ admits a weakly convergent subsequence. At this point there are potentially infinitely many such convergent subsequences. Therefore, we additionally need to prove that $u^{\nu}$ has a single limit point $u^+$. This is done indirectly, by showing that the conormal trace on $\Sigma$ is uniformly bounded in $\nu\rightarrow 0$:
\begin{align*}
	\text{ for all }0<\nu<\nu_0, \qquad  	\|\gamma_{n,\nu}^{\Sigma}u^{\nu}\|_{\mathcal{H}^{1/2}(\Sigma)}\lesssim \|f\|_{L^2(\Omega)},\qquad  \text{ see Theorem \ref{theorem:gnubound_improved}. }
\end{align*}
This will enable us to apply the decomposition of Proposition \ref{prop:decomp1} to the weak $L^2$-limits $u$ of subsequences of $(u^{\nu})_{\nu>0}$, and reveal that all such  limits satisfy \eqref{eq:traces} (see Theorem \ref{theorem:convergence_decomposition}). To prove that this condition ensures the uniqueness of the problem, will follow the idea of \cite{ciarlet_mk_peillon} by exploiting a non-self-adjoint nature of the limiting operator, which is expressed through the suitable Green's formula.

\section{Problems in sub-domains}
\label{sec:subdomains}

%One of the key difficulties associated with the analysis of the problem $\operatorname{div}( x  \mathbb{A}\nabla u)=f$ in $\Omega$ is the non-sign-definiteness of the associated sesquilinear form. In this section we will limit our considerations to one of the sub-domains $\Omega_p, \, \Omega_n$, where the problem becomes sign-definite. 
%
%This section is organized as follows. First of all, we formulate the homogeneous Neumann problem in $\Omega_p$ in Section \ref{sec:homog}, and prove the uniqueness and existence of the solution in the space $\mathcal{V}_{sing}(\Omega_p)$ (Theorem \ref{theorem:fl2}). We show that this solution belongs to the space $\mathcal{V}_{reg}(\Omega_p)$, defined in  \eqref{eq:two_spaces}, and thus known results from \cite{baouendi_goulaouic} can be used to deduce elliptic regularity estimates, cf. Theorem \ref{theorem:regularity}. Such estimates will play a crucial role in many of the proofs of the paper. 
%
%Next, we consider a heterogeneous Neumann problem in $\Omega_p$ in Section \ref{sec:heterog_neumann}, and prove that it is well-posed, provided sufficiently regularity on the Neumann data.  In particular, this result allows to  prove Proposition \ref{prop:decomp1}, see the end of Section \ref{sec:wp_irreg}. Moreover, we obtain the third Green's formula in Section \ref{sec:green_formula}, which allows to justify the definition of the Dirichlet trace on $\Sigma$ for functions from $\mathcal{V}_{sing}$ s.t. $\operatorname{div}(x\mathbb{A}\nabla .)\in L^2$, see Definition \ref{definition:one_sided_trace}. 
\subsection{A homogeneous Neumann problem}
\label{sec:homog}
First of all, we start by considering the following auxiliary homogeneous BVP: 
\begin{align}
	\label{eq:main_problem}
	\begin{split}
		&\operatorname{div}( x \mathbb{A}\nabla u)=f \text{ in }\Omega_p, \\
		%	&u(x,0)=u(x,L), \;  x \nabla u\cdot n(x,0)= x \nabla u\cdot n(x,L), \\
		&\gamma_n^{\Sigma}u=0, \\ &\gamma_0^{\Gamma_p}u=0, \text{  periodic BCs at }\Gamma^{\pm}.
	\end{split}
\end{align}
Assume that $f\in L^2(\Omega_p)$. We single out two spatial frameworks:
\begin{align}
	\label{eq:rp}
	\tag{RP}
	&\text{Find }u\in \mathcal{V}_{reg}(\Omega_p) \text{ that satisfies }\eqref{eq:main_problem}.\\
	\label{eq:sp}
	\tag{SP}
	&\text{Find }u\in \mathcal{V}_{sing}(\Omega_p) \text{ that satisfies }\eqref{eq:main_problem}.   
\end{align}
The first choice appears when considering the variational formulation associated to \eqref{eq:main_problem} and looking for the largest space in which the corresponding skew-symmetric bilinear form is continuous: 
\begin{align*}
	a_{r}(u,v):=\int_{\Omega_p} x \mathbb{A} \nabla u\,\cdot \overline{\nabla v}.
\end{align*}
Interestingly, as it was shown in \cite[Section 1.1]{nicolopoulos}, see also Lemma \ref{cor:uv}, the boundary condition on $\Sigma$ holds automatically true for all functions from $\mathcal{V}_{reg}$ satisfying \eqref{eq:main_problem} with $f\in L^2(\Omega_p)$.  

The second choice is motivated by remarking that, upon setting $\boldsymbol{v}:= x \mathbb{A}\nabla u$, the first equation in \eqref{eq:main_problem} implies that $\operatorname{div}\boldsymbol{v}\in L^2(\Omega_p)$. Therefore, if, additionally, $\boldsymbol{v}\in L^2(\Omega_p)$, then the boundary condition on $\Sigma$ can be understood in the sense of equality in $\mathcal{H}^{-1/2}(\Sigma)$. 
%
%In general, the analysis of such problems is subtle, due to the presence of the degenerate weight, see the detailed discussion in \cite{zhikov}; one of the key issues being that a priori it is unclear whether  $\mathcal{V}_{reg}(\Omega_p) = \overline{C^{\infty}(\Omega_p)}^{\|.\|_{\mathcal{V}_{reg}(\Omega_p)}}$ (the main difficulty arises for the case when $\Omega$ is replaced by $\Omega_p$ in the above, since such spaces may admit a more pathological types of behaviour, e.g. $\nabla u\notin L^1_{loc}(\Omega)$). We avoid this kind of problems by considering broken spaces $\mathcal{V}_{reg}$ and $\mathcal{V}_{sing}$, which is entirely sufficient for our needs.
%Nonetheless, in our case much of such problems do not arise, in particular since
%Nonetheless, many of the issues raised in the above paper  

The key result of this section is that the problems \eqref{eq:rp} and \eqref{eq:sp} are both well-posed  (and thus coincide). 
\begin{theorem}
	\label{theorem:fl2}
	Let $f\in L^2(\Omega_p)$. Then \eqref{eq:rp} and \eqref{eq:sp} both admit a unique (identical) solution. Moreover, $u\in \mathcal{H}^1(\Omega_p)$ and satisfies the following bound: 
	$\|u\|_{\mathcal{H}^1(\Omega_p)}\leq C \|f\|_{L^2(\Omega_p)}$, with some $C>0$ independent of $f$.
\end{theorem}
\begin{proof}
	See Theorem \ref{theorem:well_posedness_regular}, Theorem \ref{theorem:regularity} and Theorem \ref{thm:sp_wp}. 
\end{proof}	
\begin{remark}
	Remarkably, the problem \eqref{eq:rp}, well-posed in $\mathcal{V}_{reg}(\Omega_p)$, i.e. in the spaces of functions that do not have an $L^2(\Sigma)$-trace, admits a solution in the space $\mathcal{H}^1(\Omega_p)$. This is reminiscent of the elliptic regularity results for the classical Laplacian.
\end{remark}
We start with proving some facts about the problem \eqref{eq:rp}. 
\subsubsection{Regular problem \eqref{eq:rp}: well-posedness in $\mathcal{V}_{reg}(\Omega_p)$}
The key result of this section is given below. 
\begin{theorem}
	\label{theorem:well_posedness_regular}
	The problem \eqref{eq:rp} is well-posed, and, for all $f\in L^2(\Omega_p)$, $\|u\|_{\mathcal{V}_{reg}(\Omega_p)}\lesssim \|f\|_{L^2(\Omega_p)}$.
\end{theorem}
\paragraph{Auxiliary results needed to prove Theorem \ref{theorem:well_posedness_regular}}
We start by recalling some facts about the space $\mathcal{V}_{reg}(\Omega_p)$. First of all, the same argument as in \cite[Theorem 1.1]{grisvard}, see Proposition \ref{prop:density_vreg}, shows 
\begin{align}
	\label{eq:vreg_density}
	\mathcal{V}_{reg}(\Omega_p)=\overline{\mathcal{C}_{comp}(\Omega_p)}^{\|.\|_{\mathcal{V}_{reg}(\Omega_p)}}.
\end{align}
We also have the Poincar\`e inequality: for all $u\in \mathcal{V}_{reg}(\Omega_p)$, it holds that (Proposition \ref{prop:hardy} in Appendix \ref{appendix:weighted}):
\begin{align}
	\label{eq:poincare_vreg}
	\|u\|_{L^2(\Omega_p)}\leq C(\Omega_p)|u|_{\mathcal{V}_{reg}(\Omega_p)}.
\end{align}
Let us now state several auxiliary results for the proof of Theorem \ref{theorem:well_posedness_regular}. We will make use of the sign-definiteness of the problem \eqref{eq:main_problem} and write a corresponding variational formulation, which will appear to be coercive. We start with the following observation. 
\begin{proposition}
	\label{prop:rp_var}
	Let $f\in L^2(\Omega_p)$. Assume that $u\in \mathcal{V}_{reg}(\Omega_p)$ satisfies \eqref{eq:main_problem}. Then, necessarily, $u$ satisfies the following variational formulation: 
	\begin{align}
		\label{eq:var_form1}
		a_r(u, v)=-\int_{\Omega_p}f\overline{v}, \quad \text{for all} \quad v\in \mathcal{V}_{reg}(\Omega_p). 
	\end{align}
	And vice versa, if $u\in \mathcal{V}_{reg}(\Omega_p)$ satisfies the above variational formulation, it satisfies \eqref{eq:main_problem}. 
\end{proposition}
To prove this result, we need the following lemma, which is a generalization of a similar result in  \cite[p.70]{nicolopoulos_phd} or \cite[Section 1.1]{nicolopoulos}, and in Appendix \ref{appendix:conormal_trace}, see Proposition \ref{cor:uv_appx}. At the moment we need the result below for $\re=1/2$ only, however, we will make use of its extended form later.
\begin{lemma}
	\label{cor:uv}
	Any function $u\in \mathcal{V}_{reg}(\Omega_p)$, s.t., with some $\re>0$, $\vec{x}\mapsto x^{1/2-\varepsilon}\operatorname{div}(x\mathbb{A}(\vec{x})\nabla u(\vec{x}))\in L^2(\Omega_p)$, satisfies $
	\gamma_n^{\Sigma}u=0 \text{ in }\mathcal{H}^{-1/2}(\Sigma).$
\end{lemma}
%\begin{remark}
%	\label{rem:var}
%The co-normal trace in the above is understood variationally. More precisely, for all $\Phi\in \mathcal{H}^{1}(\Omega_p)$, we set
%	\begin{align*}
	%		\langle \gamma_n^{\Sigma}u, \gamma_0^{\Sigma}\overline{\Phi}\rangle_{\mathcal{H}^{-1/2}(\Sigma), \mathcal{H}^{1/2}(\Sigma)}:=-\left(x^{1/2-\varepsilon}\operatorname{div}(x\mathbb{A}\nabla u), x^{-1/2+\varepsilon}\Phi\right)_{L^2(\Omega_p)}-\left(x\mathbb{A}\nabla u,\, \nabla \Phi\right)_{L^2(\Omega_p)},
	%	\end{align*}
%	with the corresponding Lebesgue's integrals well-defined by Lemma \ref{lem:f_belongs_vreg} in Appendix \ref{appendix:conormal_trace}. In general, the definition of a co-normal derivative for the distributional (not $L^2(\Omega_p)$) data is non-unique \cite{mikhailov}, however, this is not the case here due to the well-definiteness of the above Lebesgue's integrals; compare with the analogous result  \cite[Theorem 2.2.22]{ciarlet_book} (stated without proof), where one argues that the normal trace operator is well-defined for $L^2(\Omega)$-fields with their divergence belonging to $H^{-s}(\Omega)$, $s<1/2$. 
%\end{remark}
The above lemma enables us to prove Proposition \ref{prop:rp_var}. 
\begin{proof}[Proof of Proposition \ref{prop:rp_var}]
	The fact that $u$ as in \eqref{eq:rp} satisfies \eqref{eq:var_form1} is standard and follows from integration by parts and the density of $\mathcal{C}^{\infty}_{comp}(\overline{\Omega_p})$ in $\mathcal{V}_{reg}(\Omega_p)$, cf. \eqref{eq:vreg_density}.

	The fact that $u$ solving \eqref{eq:var_form1} satisfies \eqref{eq:main_problem} again follows immediately, by testing \eqref{eq:main_problem} with $v\in \mathcal{D}(\Omega_p)$ which shows that 
	$\operatorname{div}(x\mathbb{A}\nabla u)=f$ in $\Omega_p$. Next we employ Lemma \ref{cor:uv} to see that $\gamma_n^{\Sigma}u=0$. Finally, testing with $v\in \mathcal{H}^1(\Omega)\subset \mathcal{V}_{reg}(\Omega)$ and using the variational definition of the co-normal trace on $\partial\Omega_p$ yields the periodicity of the co-normal derivatives in $y$-direction. 
\end{proof}
Thus, we have (classically) reduced the question of the well-posedness of \eqref{eq:rp} to the question of the well-posedness of the variational formulation \eqref{eq:var_form1}. We have the following result. 
\begin{theorem}
	\label{prop:main}
	Let $f\in \left(\mathcal{V}_{reg}(\Omega_p)\right)'$.  Then the following problem: find $u\in \mathcal{V}_{reg}(\Omega_p)$, s.t.
	\begin{align*}
		a_r(u_{reg}, v)=-\langle f, \overline{v}\rangle_{(\mathcal{V}_{reg}(\Omega_p))', \mathcal{V}_{reg}(\Omega_p)}, \quad \text{for all} \quad v\in \mathcal{V}_{reg}(\Omega_p), 
	\end{align*}
	admits a unique solution in $\mathcal{V}_{reg}(\Omega_p)$, and, moreover,  $\|u\|_{\mathcal{V}_{reg}(\Omega_p)}\lesssim \|f\|_{(\mathcal{V}_{reg}(\Omega_p))'}$.  
\end{theorem}
\begin{proof} 
	Evidently, $a_r: \mathcal{V}_{reg}(\Omega_p)\times \mathcal{V}_{reg}(\Omega_p)\rightarrow \mathbb{C}$ is continuous. Next, using Assumption  \ref{assump:matrices} on $\mathbb{A}$, 
	\begin{align}
		\Re a_r(u, u)=(x^{1/2}\mathbb{A}\nabla u,x^{1/2}\nabla u)_{L^2(\Omega_p)}\geq c_{\mathbb{A}}\|x^{1/2}\nabla u\|_{L^2(\Omega_p)}^2=c_{\mathbb{A}}|u|_{\mathcal{V}_{reg}(\Omega_p)}^2,\quad \forall u\in \mathcal{V}_{reg}(\Omega_p).	
	\end{align}
	With \eqref{eq:poincare_vreg}, this shows that $\Re a_r(u,u)\gtrsim \|u\|_{\mathcal{V}_{reg}(\Omega_p)}^2$, and we conclude using the Lax-Milgram lemma.
\end{proof}
\paragraph{Proof of Theorem \ref{theorem:well_posedness_regular}}It is an immediate corollary of Proposition \ref{prop:rp_var} and Theorem \ref{prop:main}. 
\subsubsection{Regular problem \eqref{eq:rp}: regularity estimates for regular data}
When the right-hand side $f\in L^2(\Omega_p)$, it appears that the solution to \eqref{eq:rp} possesses more regularity than predicted by Theorem \ref{prop:main}. This is reminiscent of the standard elliptic regularity for the Laplace equation. The key result of this section is Theorem \ref{theorem:regularity}, given below. 
For $\mathbb{A}\in C^{\infty}(\overline{\Omega_p}; \mathbb{C}^{2\times 2})$, the corresponding result was proven in the work \cite{baouendi_goulaouic}, modulo the boundary conditions, using pseudo-differential calculus techniques. Their proof extends quite straightforwardly to the case of less regular coefficients, and therefore we omit it here; see \citeappendixregularity for the details.
%	For convenience of the reader, we include the respective proof for $ x $, $\mathbb{A}$ satisfying Assumption \ref{assump:matrices} in  \ref{appendix:regularity}.
\begin{theorem}[Theorem 1 in \cite{baouendi_goulaouic}]
	\label{theorem:regularity}
	Let $f\in L^2(\Omega_p)$. Then the unique solution $u$ to  \eqref{eq:rp} belongs to $\mathcal{H}^1(\Omega_p)$, and satisfies the following stability bound: $\|u\|_{H^1(\Omega_p)}+\| x  u\|_{H^2(\Omega_p)}\lesssim \|f\|_{L^2(\Omega_p)}$. 
	
	If, moreover, $f\in \mathcal{H}^1(\Omega_p)$, then the unique solution to $u$ to  \eqref{eq:rp} satisfies $u\in \mathcal{H}^2(\Omega_p)$, and the following stability bound holds true: $\|u\|_{H^2(\Omega_p)}\lesssim \|f\|_{H^1(\Omega_p)}$. 
\end{theorem}
%\begin{remark}
%	It is to state this latter result about the $H^2(\Omega_p)$-regularity of $u$, useful in a sequel, that we made use of an assumption $\mathbb{A}\in C^3(\overline{\Omega})$ rather than $C^2(\overline{\Omega})$.  
%\end{remark}
\subsubsection{Singular problem \eqref{eq:sp}: proof of Theorem \ref{theorem:fl2}}
Now we have all ingredients that enable us to work with \eqref{eq:sp}. Let us address first of all the question of the existence. Since $\mathcal{V}_{reg}(\Omega_p)\subset \mathcal{V}_{sing}(\Omega_p)$, the existence follows from Theorem \ref{theorem:well_posedness_regular}.  Unfortunately, we cannot conclude with the uniqueness by using Fredholm-type arguments, since the space  $\mathcal{V}_{sing}(\Omega_p)$ is not compactly embedded into $L^2(\Omega_p)$. 
Nonetheless, the following holds true. 
\begin{theorem}
	\label{thm:sp_wp}
	There exists a unique solution to \eqref{eq:sp} (and it is also a unique solution to \eqref{eq:rp}).
\end{theorem}
Before proving this theorem, we recall that, see Propositions \ref{prop:density_smooth}, \ref{prop:density_vreg} in Appendix \ref{appendix:weighted},  
\begin{align}
	\label{eq:vsing_density}
	\mathcal{V}_{sing}(\Omega_p)=\overline{\mathcal{C}^{\infty}(\overline{\Omega}_p)}^{\|.\|_{\mathcal{V}_{sing}(\Omega_p)}}=\overline{\mathcal{C}^{\infty}_{comp}(\overline{\Omega}_p)}^{\|.\|_{\mathcal{V}_{sing}(\Omega_p)}}.
\end{align}
\begin{proof}
	As argued before the statement of this theorem, we need to prove injectivity only. Assume that $u\in \mathcal{V}_{sing}(\Omega_p)$ satisfies \eqref{eq:sp} with $f=0$. Testing \eqref{eq:sp} with $v\in \mathcal{H}^1(\Omega_p)$, and integrating by parts yields 
	\begin{align}
		\label{eq:asinguv}
		0=-\int_{\Omega_p}x\mathbb{A}\nabla u(\vec{x})\cdot\overline{\nabla v(\vec{x})}d\vec{x}, \quad \forall v\in \mathcal{H}^1(\Omega_p).
	\end{align}
	Next, we will choose $v\in \mathcal{H}^1(\Omega_p)$ that satisfies (see Theorems \ref{theorem:well_posedness_regular} and \ref{theorem:regularity} for existence/uniqueness)
	\begin{align*}
		&\operatorname{div}(x\mathbb{A}\nabla v)=u \text{ in }\Omega_p,\\ 
		&\gamma_n^{\Sigma}v=0, \\ &\gamma_0^{\Gamma_p}v=0,\qquad
		\text{ periodic BCs at $\Gamma_p^{\pm}$. }
	\end{align*}
	By Proposition \ref{prop:rp_var} and using that $\mathbb{A}$ is Hermitian, $v\in \mathcal{H}^1(\Omega_p)$ satisfies the following variational formulation:
	\begin{align*}
		\int_{\Omega_p}\nabla v\cdot \overline{x\mathbb{A}\nabla q}=	\int_{\Omega_p}x\mathbb{A}\nabla v\cdot\, \overline{\nabla q}=-\int_{\Omega_p}u\, \overline{q}, \quad \forall q\in \mathcal{C}^{\infty}(\overline{\Omega}_p).
	\end{align*} 
	By density \eqref{eq:vsing_density} we can replace $\mathcal{C}^{\infty}(\overline{\Omega}_p)$ by $\mathcal{V}_{sing}(\Omega_p)$, so that
	\begin{align*}
		\int_{\Omega_p}\nabla v\cdot \overline{x\mathbb{A}\nabla q}=-\int_{\Omega_p}u\,\overline{q}, \quad \forall q\in \mathcal{V}_{sing}(\Omega_p).
	\end{align*} 
	This enables us to choose $q=u$, which gives  
	$
	\int_{\Omega_p}\nabla v\cdot\overline{x\mathbb{A}\nabla u}=-\int_{\Omega_p}|u|^2,
	$
	and comparing the above with \eqref{eq:asinguv} we conclude that $u=0$. 
\end{proof}
\subsection{Heterogeneous Neumann problem}
\label{sec:heterog_neumann}
%\subsubsection{The first well-posedness result}
In what follows, we will need to understand the behaviour of a variant of \eqref{eq:main_problem}  when  $\gamma_{n}^{\Sigma}u$ does not vanish. Evidently, in view of Lemma \ref{cor:uv}, it does not make sense to consider $u\in \mathcal{V}_{reg}(\Omega_p)$. In this section we are interested in the well-posedness of the following problem (the choice of the spaces for the data will be made clear later): given $f\in L^2(\Omega_p)$, $g\in \mathcal{H}^{1/2}(\Sigma)$, find $u\in \mathcal{V}_{sing}(\Omega_p)$, s.t. 
\begin{align}
	\label{eq:vsing_heter}
	\tag{N}
	\begin{split}
		&\operatorname{div}( x \mathbb{A}\nabla u)=f \text{ in }\Omega_p, \\ &\gamma_{n}^{\Sigma}u=g,\\
		&\gamma_{0}^{\Gamma_p}u=0, \qquad  \text{ periodic BCs at }\Gamma_p^{\pm}.
	\end{split}
\end{align}
%\begin{remark}
%	In the above, the co-normal derivative is defined in the variational manner, namely, as a unique element	of $\mathcal{H}^{-1/2}(\Sigma)$, s.t. 
%	\begin{align}
	%		\label{eq:conormal_definition}
	%		\langle \gamma_{n}^{\Sigma}u, \gamma_0^{\Sigma}\overline{\varphi}\rangle_{\mathcal{H}^{-1/2}(\Sigma), \mathcal{H}^{1/2}(\Sigma)}=-\int_{\Omega_p}\operatorname{div}( x \mathbb{A}\nabla u)\overline{\varphi}d\bx-\int_{\Omega_p} x \mathbb{A}\nabla u\cdot\overline{\nabla \varphi}d\bx, \quad \forall \varphi\in \mathcal{H}^1(\Omega_p).
	%	\end{align}
%\end{remark}
This problem is equivalent to the following problem: find $u\in \mathcal{V}_{sing}(\Omega_p)$, s.t. 
\begin{align*}
	( x \mathbb{A}\nabla u, \nabla \varphi)=-(f, \varphi)-\langle g, \gamma_0^{\Sigma}\overline{\varphi}\rangle_{\mathcal{H}^{-1/2}(\Sigma), \mathcal{H}^{1/2}(\Sigma)}, \quad \forall \varphi \in \mathcal{H}^1(\Omega_p).
\end{align*}
The uniqueness of a solution to \eqref{eq:vsing_heter} follows from Theorem \ref{theorem:fl2}. To prove the existence, we will construct the solution using an appropriate lifting of the data $g$, which in this case will be treated like essential boundary conditions, similarly to the Dirichlet data for the non-weighted Laplacian. 
\subsubsection{Well-posedness}
\label{sec:wp_irreg}
We have the following well-posedness result. 
\begin{theorem}
	\label{theorem:reg_well_posedness}
	Let $f\in L^2(\Omega_p)$, $g\in \mathcal{H}^{1/2}(\Sigma)$. Then the problem \eqref{eq:vsing_heter} admits a unique solution $u\in \mathcal{V}_{sing}(\Omega_p)$. This solution admits the following decomposition:
	\begin{align}
		\label{eq:harmonic}
		u(x,y)={u}_{reg}(x,y)+u_{sing}(x,y), \quad u_{sing}(x,y)=u_h(x,y)\log|x|,
	\end{align}
	where ${u}_{reg}\in \mathcal{V}_{reg}(\Omega_p)\cap \bigcap\limits_{\varepsilon>0}\mathcal{H}^{1-\varepsilon}(\Omega_p)$, and $u_h\in \mathcal{H}^1(\Omega_p)$ is a piecewise-$\mathbb{A}$-harmonic function that satisfies
	\begin{align}
		\label{eq:uh}
		\begin{split}
			&\operatorname{div}(\mathbb{A}\nabla u_h)=0 \text{ in }\Omega_p, \\ &\gamma_{0}^{\Sigma}u_h = a_{11}^{-1}g, \\			
			&\gamma_0^{\Gamma_p}u_h =0, \qquad \text{ periodic BCs at }\Gamma^{\pm}_p.
		\end{split}
	\end{align} 	 
	This decomposition satisfies additional properties: \\
	{1. Stability in fractional and weighted Sobolev spaces. }For all $\re>0$, there exists $C_{\re}>0$, s.t. $$\|u_h\|_{H^1(\Omega_p)}+ \|u_{reg}\|_{H^{1-\varepsilon}(\Omega_p)}\leq C_{\varepsilon}\left( \|g\|_{\mathcal{H}^{1/2}(\Sigma)}+\|f\|_{L^2(\Omega_p)}\right).$$
	Moreover, $u_{reg}\in\bigcap\limits_{\varepsilon>0}\mathcal{H}^1_{\varepsilon}(\Omega_p)$ and, for all $\varepsilon>0$, there exists $C_{\varepsilon}>0$, s.t.  $$\|u_{reg}\|_{H^{1}_{\varepsilon}(\Omega_p)}\leq C_{\varepsilon}\left( \|g\|_{\mathcal{H}^{1/2}(\Sigma)}+\|f\|_{L^2(\Omega_p)}\right).$$ 
	{2. Property of the conormal trace.}	 It holds that $\gamma_n^{\Sigma}u=\gamma_n^{\Sigma}u_{sing}$ and $\gamma_n^{\Sigma}u_{reg}=0$.\\
	%	\item 	Additionally, it holds that $\operatorname{div}( x \mathbb{A}\nabla u_{reg})\in L^2(\Omega_p\setminus \Omega_{\Sigma}^{\delta})$, for all $\delta>0$, and so is $\operatorname{div}( x \mathbb{A}\nabla u_{sing})$.		 
	{3. Direct sum.} The decomposition \eqref{eq:harmonic} is unique, in other words, if $u=u_{reg}^j+u_h^j\log|x|$, for $j=1,2$, $u_h^j$ solves \eqref{eq:uh} and $u_{reg}^j\in \mathcal{V}_{reg}(\Omega_p)$, then, necessarily $u_{reg}^1=u_{reg}^2$ and $u_h^1=u_h^2$. 
\end{theorem}
To prove this result, we state the following auxiliary lemma.  
\begin{lemma}[Proposition \ref{prop:regularity2} in Appendix \ref{sec:wp_rough}]
	\label{lem:propreg2}
	Let $0\leq \delta<1$, and let $f\in L^2_{\delta}(\Omega_p)$. 
	Then the homogeneous Neumann problem \eqref{eq:rp} admits a unique solution $u\in \mathcal{V}_{reg}(\Omega_p)$. It satisfies $$u\in \bigcap_{0<\re<1}\mathcal{H}^1_{\delta+\re}(\Omega_p)\subset \bigcap_{0<\re\leq 1-\frac{\delta}{2}}\mathcal{H}^{1-\delta/2-\re}(\Omega_p).$$ Additionally,  $\|u\|_{\mathcal{H}^{1-\delta/2-\re}(\Omega_p)}\leq C_{\delta,\re} \|u\|_{\mathcal{H}^1_{\delta+2\re}(\Omega_p)}\leq \tilde{C}_{\delta,\re} \|x^{\delta/2}f\|_{L^2(\Omega_p)},$ for all $0<\re<1-\delta/2$. 
\end{lemma}

\begin{proof}[Proof of Theorem \ref{theorem:reg_well_posedness}]
	In view of the linearity of the problem and of the well-posedness/regularity result of Theorem \ref{theorem:fl2}, it suffices to prove the corresponding result for the case $f=0$.  The uniqueness is a corollary of Theorem \ref{theorem:fl2}, and we need to prove the existence only. This will be done by constructing the solution.
	
	\textbf{Definition of $u_{sing}$. }
	The main idea is to treat the boundary condition in \eqref{eq:vsing_heter} as an \textit{essential} boundary condition, in other words we will construct an appropriate lifting of $g$. To explain how to do so, let us first consider a simplified case $\mathbb{A}=\operatorname{Id}$ and a regular $g\in \mathcal{H}^1(\Sigma)$. The first obvious choice (cf. \cite{nicolopoulos}) is $u_{sing}(x,y)=g(y)\log|x|$, which would satisfy $x\partial_xu_{sing}=g$ in $\Omega_p$ and in particular on $\Sigma$ (thus $\gamma_{n}^{\Sigma}u_{sing}=g$). Decomposing the solution $u$ into the regular and the singular part $u=u_{sing}+u_{reg}$, we see that 
	\begin{align}
		\label{eq:divxureg}
		&\operatorname{div}(x\nabla u_{reg})=-\operatorname{div}(x\nabla u_{sing})=-x\log|x|\partial_y^2 g,\\
		\nonumber
		&\left.x\partial_x u_{reg}\right|_{\Sigma}=0.
	\end{align}
	It is possible to verify that the right-hand side in \eqref{eq:divxureg}, since $g\in \mathcal{H}^{1}(\Sigma)$, defines an element of $\left(\mathcal{V}_{reg}(\Omega_p)\right)'$; in this case \eqref{eq:divxureg} admits a unique solution in $\mathcal{V}_{reg}(\Omega_p)$, cf. Theorem \ref{prop:main}. When the regularity of $g$ is only $\mathcal{H}^{1/2}(\Sigma)$, it is unclear whether the right-hand side is still of regularity $\left(\mathcal{V}_{reg}(\Omega_p)\right)'$, and thus the theory developed in Section \ref{sec:homog} does not seem to be applicable. 
	%
	%
	% is a priori of regularity $\left(\mathcal{H}^{3/2}(\Omega_p)\right)'$, rather %than $(\mathcal{V}_{reg}(\Omega_p))'$; therefore, we cannot expect the regular part to belong to $\mathcal{V}_{reg}(\Omega_p)$. %Therefore, instead, we will choose $u_{sing}=u_g(x,y)\log|x|$, where $u_g$ is a well-chosen lifting of $g$.  
	%
	To overcome this technical difficulty and consider the case of arbitrary $ \mathbb{A}$, instead of $g(y)\log|x|$ we take $$u_{sing}(x,y)=v_{g} (x,y)\log|x|,$$ where $v_{g}\in \mathcal{H}^1(\Omega_p)$ is a well-chosen lifting of $a_{11}^{-1}(y)g(y)\in \mathcal{H}^{1/2}(\Sigma)$. The first (natural) choice is taking a piecewise  $\operatorname{div}(\mathbb{A}\nabla.)$-harmonic lifting $v_{g}=u_h$, with $u_h\in \mathcal{H}^1(\Omega\setminus\Sigma)$ being a unique solution to \eqref{eq:uh}. Before proceeding, we need to verify that  $u_{sing}$ defined as above indeed belongs to the space $\mathcal{V}_{sing}(\Omega_p)$. For this we use the fact that $u_h\in \mathcal{H}^1(\Omega_p)$, and   
	\begin{align*}
		&\vec{x}\mapsto u_h(\vec{x})\log|x| \in L^2(\Omega_p) \text{ by the Hardy-type  inequality, cf. Proposition \ref{prop:hardy} in Appendix  \ref{appendix:weighted}},\\	
		&\vec{x}\mapsto x\nabla (u_h(\vec{x})\log|x|)=x\log|x|\nabla u_h+\partial_x u_h \in L^2(\Omega_p) \text{ by $\|x\log|x|\|_{L^{\infty}(\Omega_p)}<\infty$.}
	\end{align*}
	The stated bound $\|u_{sing}\|_{\mathcal{V}_{sing}(\Omega_p)}\lesssim \|g\|_{\mathcal{H}^{1/2}(\Sigma)}$ follows from the above and \eqref{eq:uh}.
	
	We will verify a posteriori that $\gamma_n^{\Sigma}u=\gamma_n^{\Sigma}u_{sing}$,and for the moment will concentrate on constructing $u_{reg}$.

	\textbf{Definition of $u_{reg}$. }	Let us write the problem to be satisfied by $u_{reg}\in \mathcal{V}_{sing}(\Omega_p)$, which we will equip additionally with $\gamma_n^{\Sigma}u_{reg}=0$:
	\begin{align}
		\label{eq:uregpb}
		\begin{split}
			\operatorname{div}(x\mathbb{A}\nabla u_{reg})&=-\operatorname{div}(x\mathbb{A}\nabla u_{sing})=-\operatorname{div}(x\log|x|\mathbb{A}\nabla u_h)-\operatorname{div}(u_h\mathbb{A}\cdot\vec{e}_x)\\
			&=-(1+\log|x|)\vec{e}_x\cdot \mathbb{A}\nabla u_h -\operatorname{div}(u_h\mathbb{A}\cdot\vec{e}_x)=:f_h,\\
			&\gamma_0^{\Gamma_p}u=0,\text{  periodic BCs at }\Gamma^{\pm}_p.
		\end{split}
	\end{align}
	With Lemma \ref{lem:propreg2} we fix $u_{reg}$ as a unique solution from $\mathcal{V}_{reg}(\Omega_p)$ to the above problem. Indeed, since $\nabla u_h\in L^2(\Omega_p)$, the right-hand side $f_h\in L^2_{\delta}(\Omega_p)$ for all $\delta>0$. Hence $u_{reg}\in \bigcap_{0<\re<1}\mathcal{H}^{1}_{\re}(\Omega_p)\subset\mathcal{H}^{1-}(\Omega_p)$, and the control on the norms of $u_{reg}$ follows from the statement of the same lemma.

	\textbf{Verification that $u_{reg}+u_{sing}$ is indeed a solution. } We need to verify that $u=u_{reg}+u_{sing}$ as defined above satisfies \eqref{eq:uh} with $f=0$. By construction of $u_{reg}$ and $u_{sing}$ it is sufficient to check that $\gamma_n^{\Sigma}u=g$ only. By Lemma \ref{cor:uv}, it remains to verify that $\gamma_n^{\Sigma}u_{sing}=g$. With the  variational definition of the conormal derivative,  using the second line \eqref{eq:uregpb}, the fact that $f_h\in L^2_{\delta}(\Omega_p)$ and Lemma \ref{lem:f_belongs_vreg} to justify that the below integrals are defined as Lebesgue's integrals, we arrive at the following identity for all $\varphi\in \mathcal{H}^1(\Omega_p)$:
	\begin{align*}
		\langle \gamma_n^{\Sigma}u_{sing}, \gamma_0^{\Sigma}\varphi\rangle_{\mathcal{H}^{-1/2}(\Sigma), \mathcal{H}^{1/2}(\Sigma)}&= -\int_{\Omega_p}\operatorname{div}(x\mathbb{A}\nabla u_{sing}) \varphi - \int_{\Omega_p}x\mathbb{A}\nabla u_{sing}\cdot \nabla \varphi \\
		&=-\int_{\Omega_p}(1+\log|x|)\vec{e}_x\cdot \mathbb{A}\nabla u_h\, \varphi -\int_{\Omega_p}\operatorname{div}(u_h\mathbb{A}\cdot \vec{e}_x)\, \varphi-\int_{\Omega_p}x\mathbb{A}\nabla u_{sing}\cdot \nabla \varphi.
	\end{align*}
	Integrating by parts the second integral and replacing $x\mathbb{A}\nabla u_{sing}$ by
	%Recall that 
	%	\begin{align*}
		%		\operatorname{div}(x\mathbb{A}\nabla u_{sing})=(1+\log|x|)\vec{e}_x\cdot \mathbb{A}\nabla u_h+\operatorname{div}(u_h\mathbb{A}\cdot \vec{e}_x)
		%%	\end{align*}
	%	We use the variational characterization of the conormal derivative, and test $\operatorname{div}(x\mathbb{A}\nabla(u_{sing}+u_{reg}))=0$ with $\varphi\in \mathcal{H}^1(\Omega_p)$:	
	%	\begin{align}
		%		\label{eq:ggamma0}
		%		\langle \gamma_n^{\Sigma}u, \gamma_{0}^{\Sigma}\overline{\varphi}\rangle_{\mathcal{H}^{-1/2}(\Sigma), \mathcal{H}^{1/2}(\Sigma)}=-\int_{\Omega_p} x \mathbb{A}\left(\nabla u_{sing}+\nabla u_{reg}\right)\cdot\overline{\nabla \varphi}.
		%	\end{align}
	%We rewrite
	\begin{align*}
		x \mathbb{A}\nabla u_{sing}= x \mathbb{A}\left(\begin{matrix}
			x^{-1}u_h+\log|x|\partial_x u_h \\
			\log|x|\partial_y u_h
		\end{matrix}\right)=\mathbb{A}\left(\begin{matrix}
			u_h + x \log|x|\partial_x u_h\\
			x \log|x|\partial_y u_h
		\end{matrix}\right)=\mathbb{A}\left(\begin{matrix}
			u_h\\
			0	
		\end{matrix}
		\right)+ x \log|x|\mathbb{A}\nabla u_h,
	\end{align*}
	we conclude that 
	\begin{align*}
		\langle \gamma_n^{\Sigma}u_{sing}, \gamma_0^{\Sigma}\varphi\rangle_{\mathcal{H}^{-1/2}(\Sigma), \mathcal{H}^{1/2}(\Sigma)}-\int_{\Sigma}a_{11}u_h\varphi &= -\int_{\Omega_p}(1+\log|x|)\vec{e}_x\cdot\mathbb{A}\nabla u_h \, \varphi  -\int_{\Omega_p}x\log|x|\mathbb{A}\nabla u_h \, \cdot \nabla\varphi\\
		&=-\int_{\Omega_p}\mathbb{A}\nabla u_h \nabla (x\log|x|\varphi).
	\end{align*}
	Since $\varphi\in \mathcal{H}^1(\Omega_p)$,  $\vec{x}\mapsto x\log|x|\varphi(\vec{x}) \in \mathcal{H}^1(\Omega_p)$; in particular,  $
	\nabla(x\log|x|\varphi)=	(1+\log|x|)\varphi+x\log|x|\nabla \varphi \in L^2(\Omega_p)$ by the Hardy-type inequality of Proposition \ref{prop:hardy}. We also have that $\gamma_0^{\Sigma}(x\log|x|\varphi)=0$. Therefore, by definition of $u_h$, the right-hand side vanishes, and it follows that $\gamma_n^{\Sigma}u_{sing}=a_{11}\gamma_0^{\Sigma}u_h = g =\gamma_n^{\Sigma }u$. 
	
	This also proves the property of the conormal derivative stated in the theorem.

	%\textbf{Step 4. }
	
	%On the other hand, $x^{\varepsilon}(-\log|x|\vec{e}_x\cdot \mathbb{A}\nabla u_h+\widetilde{f})\in L^2(\Omega_p)$
	%	
	%	Finally, the stability bounds on $u_h, \, u_{reg}$ are inherited from stability bounds of the problems satisfied by these functions. 
	%	
	%From \eqref{eq:fu} we conclude that $\operatorname{div}( x \mathbb{A}\nabla u_{reg})\in L^2(\Omega_p\setminus\Omega_{\Sigma}^{\delta})$, for all $\delta>0$. 
	%	
	%	\textbf{Step 4. Proof that $\gamma_n^{\Sigma}u=\gamma_n^{\Sigma}u_{sing}$. }
	%	Since $u=u_{sing}+u_{reg}$, it holds that $\gamma_n^{\Sigma}u=\gamma_n^{\Sigma}u_{sing}+\gamma_n^{\Sigma}u_{reg}$. As $u_{reg}\in \mathcal{V}_{reg}(\Omega_p)$ satisfies \eqref{eq:fu}, we verify that $\operatorname{div}(x\mathbb{A}\nabla u_{reg})$ satisfies conditions of Lemma \ref{cor:uv}, and this shows that $\gamma_n^{\Sigma}u_{reg}=0$. Hence the conclusion. 
	%	
	\textbf{Uniqueness of the decomposition \eqref{eq:harmonic}. }Uniqueness of the decomposition follows by uniqueness of the solution to \eqref{eq:uh} and next to \eqref{eq:uregpb} in $\mathcal{V}_{reg}(\Omega_p)$ as argued in Lemma \ref{cor:uv}.
\end{proof}
\begin{remark}
	We require that $g\in \mathcal{H}^{1/2}(\Sigma)$ in \eqref{eq:vsing_heter} in order to be able to construct its lifting in a simple manner.	We believe that the well-posedness result of Theorem \ref{theorem:well_posedness_regular} holds true also for less regular data $g\in\mathcal{H}^{-1/2}(\Sigma)$, however, the decomposition \eqref{eq:harmonic} is no longer explicit. This is postponed to the future work. 
\end{remark}
Theorem \ref{theorem:reg_well_posedness} shows that the singular part of the solution to \eqref{eq:vsing_heter} has a very peculiar behaviour in the vicinity of $\Sigma$: the singularity is necessarily of a logarithmic type, while of course the space $\mathcal{V}_{sing}(\Omega_p)$ contains functions with stronger singularities in the vicinity of $x=0$.%, while the regular part is from the space $\mathcal{H}^{1-}(\Omega_p)$.  
\begin{remark}
	As seen from the proof of Theorem \ref{theorem:reg_well_posedness}, the decomposition in \eqref{eq:harmonic} is not stable in the following sense: $\operatorname{div}( x \mathbb{A}\nabla u_{reg}), \operatorname{div}( x \mathbb{A}\nabla u_{sing})\notin L^2(\Omega_p)$, but rather in a larger space $\bigcap_{\delta>0}L^2_{\delta}(\Omega_p)$. This may seem not entirely satisfactory. One way to avoid this is to change the definition of the singular term in \eqref{eq:harmonic}, by incorporating a well-chosen weight into the PDE satisfied by $u_h$. However, the resulting decomposition appeared to be more difficult to work with, and that is why we abandoned this idea.
\end{remark}
Theorem \ref{theorem:reg_well_posedness} also allows to prove Proposition \ref{prop:decomp1}, announced in Section \ref{sec:simplified}. 
\begin{proof}[Proof of Proposition \ref{prop:decomp1}]
	The function $u$ as in Proposition \ref{prop:decomp1} satisfies \eqref{eq:vsing_heter} for some $f\in L^2(\Omega)$, separately in $\Omega_p$ and in $\Omega_n$. Also,  $\gamma_n^{\Sigma,p}u=\gamma_n^{\Sigma,n}u=\gamma_n^{\Sigma}u\in \mathcal{H}^{1/2}(\Sigma)$. The result is immediate with Theorem \ref{theorem:reg_well_posedness}.
\end{proof}
\subsubsection{The third Green's formula}
\label{sec:green_formula}
This section is dedicated to the derivation of the third Green's formula for functions satisfying \eqref{eq:vsing_heter}. This construction is inspired by \cite{ciarlet_mk_peillon}, where the Green's formula was used to define a weak jump of the Dirichlet trace of a regular part of the limiting absorption solution.  

While the notion of the Neumann trace for \eqref{eq:vsing_heter} is inherited from the definition of the normal trace of functions from  ${H}(\operatorname{div};\Omega_p)$, it seems impossible to define the Dirichlet trace for solutions of \eqref{eq:vsing_heter}, due to the presence of the logarithmic singularity, see Theorem \ref{theorem:well_posedness_regular}. Nonetheless, it appears that the third Green's formula holds true, provided a new definition of the trace (Definition \ref{definition:one_sided_trace}), re-stated below. Let (cf. \eqref{eq:hdelta})
\begin{align*}
	\mathcal{V}_{sing}(\operatorname{div}(x\mathbb{A}\nabla.); \Omega_{\lambda})&=\{v\in \mathcal{H}_0(\operatorname{div}(x\mathbb{A}\nabla .); \Omega_{\lambda}): \, \gamma_n^{\Sigma}v\in \mathcal{H}^{1/2}(\Sigma)\},\quad \lambda\in\{n,p\}, 
\end{align*}
equipped with the norm $(\|.\|_{\mathcal{V}_{sing}(\Omega_{\lambda})}^2+\|\operatorname{div}(x\mathbb{A}\nabla .)\|_{\Omega_{\lambda}}^2+\|\gamma_n^{\Sigma}.\|_{\mathcal{H}^{1/2}(\Sigma)}^2)^{1/2}$.
\begin{definition}
	\label{def:trace}
	Let $u\in	\mathcal{V}_{sing}(\operatorname{div}(x\mathbb{A}\nabla.); \Omega_{\lambda})$. Given the decomposition \eqref{eq:harmonic} of $u$, we define its trace as 
	\begin{align*}
		\gamma_{0}^{\Sigma}u:=\gamma_{0}^{\Sigma}u_{reg}\in \mathcal{H}^{1/2-}(\Sigma).
	\end{align*}
\end{definition}
The above definition of the trace takes into account the regular part of $u$ only, and the trace is well-defined, due to the uniqueness of the decomposition in Theorem \ref{theorem:reg_well_posedness}. It appears in the third Green's formula, as made precise below.
\begin{theorem}
	\label{theorem:green}
	Let $u, \, v\in \mathcal{V}_{sing}(\operatorname{div}(x\mathbb{A}\nabla.); \Omega_{p})$, and  $u_{reg}, u_h$,  $v_{reg}, v_h$ be defined in Theorem \ref{theorem:reg_well_posedness}, so that 
	\begin{align*}
		u=u_{reg}+u_h\log|x|, \quad v=v_{reg}+v_h\log|x|.
	\end{align*}
	Then the following integration by parts formula holds true: 
	\begin{align}
		\label{eq:ibp}
		\begin{split}
			\int_{\Omega_p}\operatorname{div}( x  \mathbb{A}\nabla u)\overline{v}-\int_{\Omega_p}\overline{\operatorname{div}( x \mathbb{A}\nabla v)}{u}&=-\langle a_{11}\gamma_{0}^{\Sigma} u_h, \overline{\gamma_{0}^{\Sigma} v_{reg}}\rangle_{L^2(\Sigma)}+\overline{\langle a_{11}\gamma_{0}^{\Sigma} v_h, \overline{\gamma_{0}^{\Sigma} u_{reg}}\rangle}_{L^2(\Sigma)}\\
			&=-\langle \gamma_{n}^{\Sigma}u, \, \overline{\gamma_{0}^{\Sigma}v}\rangle_{L^2(\Sigma)} +\overline{\langle \gamma_{n}^{\Sigma}v, \, \overline{\gamma_{0}^{\Sigma}u}\rangle}_{L^2(\Sigma)}
		\end{split}
	\end{align}
\end{theorem}
\begin{proof}
	See Appendix \ref{appendix:green_formula_to_submit}.
\end{proof}

%We will need the following auxiliary lemma. 
%\begin{lemma}
%	For all $p\in \mathcal{C}^{\infty}(\overline{\Omega}_p)$, it holds that 
%	\begin{align*}
	%		\|p\|^2_{H^{\theta}(\Omega_p)}\lesssim 
	%	\end{align*}
%\end{lemma}
\subsubsection{On the Dirichlet trace of singular functions}
As introduced in Definition \ref{def:trace}, the notion of the Dirichlet trace $\gamma_0^{\Sigma}$ depends on $\Omega_p$ through the chosen lifting of the Neumann trace, which we fixed to be $\mathbb{A}-$harmonic. More precisely, it is a priori unclear whether  $\gamma_0^{\Sigma,p}u=\gamma_0^{\Sigma,p}(\varphi_{\re}u)$, $0<\re<1$, where $\varphi_{\re}$ is a cut-off function localized in the vicinity of the interface, as defined in \eqref{eq:cutoff_phi}. The answer to this question appears to be positive, see Lemma \ref{lem:deftraces}. Another interesting question is whether the trace is preserved under a change of coordinates (the answer appears to be negative). We present however a related result, which will be used in a sequel. 
\begin{lemma}[Definition of traces]
	\label{lem:deftraces}
	Assume that $u\in L^2(\Omega_p)$ writes, for some $0<\delta<1$, 
	\begin{align*}
		u=u_{s,1}\log|x|+u_{r,1}=u_{s,2}\log|x|+u_{r,2}, \text{ where }u_{s,j}\in \mathcal{H}^1(\Omega_p), \; u_{r,j}\in \mathcal{H}^1_{\delta}(\Omega_p),\, j=1,p.
	\end{align*}
	Then  $\gamma_0^{\Sigma}u_{s,1}=\gamma_0^{\Sigma}u_{s,2}$ and  $\gamma_0^{\Sigma}u_{r,1}=\gamma_0^{\Sigma}u_{r,2}.$\\
	If, additionally, $u\in \mathcal{V}_{sing}(\operatorname{div}(x\mathbb{A}\nabla.), \Omega_p)$, then  $\gamma_0^{\Sigma}u=\gamma_0^{\Sigma}u_{r,j}$ and $\gamma_n^{\Sigma}u=a_{11}\gamma_0^{\Sigma}u_{h,j}$. 
\end{lemma}
\begin{proof}
	Let us start by proving that $\gamma_0^{\Sigma}u_{s,1}=\gamma_0^{\Sigma}u_{s,2}$. To see this, we rewrite 
	\begin{align}
		\label{eq:urdiff}
		u_{r,1}-u_{r,2}=d_s\log|x|, \quad d_s:=u_{s,2}-u_{s,1},\text{ and }	\nabla (u_{r,1}-u_{r,2})=\frac{d_s}{x}+\nabla d_{s}\log|x|.
	\end{align}
	Since $u_{r,j}\in \mathcal{H}^1_{\delta}(\Omega_p)$, and $d_s\in \mathcal{H}^1(\Omega_p)$, it holds that $\vec{x}\mapsto\frac{d_s(\vec{x})}{x}\in L^2_{\delta}(\Omega_p)$, in other words, $d_s\in L^2_{\delta-2}(\Omega_p)$. In particular, $d_s\in \mathcal{H}^1_{\delta}(\Omega_p)\cap L^2_{\delta-2}(\Omega_p)$. By \cite[Theorem 1.2, Proposition 1.2]{grisvard} and \eqref{eq:h1delta} (which accounts for different conventions in the definition of spaces in this article and in \cite{grisvard}), we conclude that $\gamma_0^{\Sigma}d_s=0$.

	It remains to prove the equality of traces $\gamma_0^{\Sigma}u_{r,j}$. We proceed in a similar manner, by using now that $\gamma_0^{\Sigma}d_s=0$ and $d_s\in \mathcal{H}^1(\Omega_p)$.

	In particular, by Hardy's inequality \cite[p.313]{brezis2010functional}, the operator 	$M:\, \varphi(\vec{x})\mapsto\frac{\varphi(\vec{x})}{x}$ \text{ is a continuous operator }
	\begin{align}
		\label{eq:multhardy}
		M\in \mathcal{L}(\mathcal{H}^1_{\Sigma,0}(\Omega_p), L^2(\Omega_p)), \quad\text{ where } \mathcal{H}^1_{\Sigma,0}(\Omega_p)=\{u\in \mathcal{H}^1(\Omega_p): \, \gamma_0^{\Sigma}u=0\}.
	\end{align}
	Therefore, $\vec{x}\mapsto d_s(\vec{x})/x\in L^2(\Omega_p)$, thus $d_r:=(u_{r,1}-u_{r,2})\in L^2_{-2+\varepsilon}(\Omega_p)$, for all $\varepsilon>0$. Therefore, $d_r\in \mathcal{H}^1_{\delta}(\Omega_p)\cap L^2_{\delta-2}(\Omega_p)$, and the conclusion $\gamma_0^{\Sigma}d_r=0$ follows using the same argument as before.
	
	The final assertion of the lemma for $u\in \mathcal{V}_{sing}(\operatorname{div}(x\mathbb{A}\nabla.);\Omega_p)$ is an immediate corollary of the previous statements, the decomposition defined in Theorem \ref{theorem:reg_well_posedness} and Definition \ref{def:trace}. 
\end{proof}
An immediate corollary of the above lemma and Theorem \ref{theorem:green} reads. 
\begin{corollary}
	\label{cor:green}
	Let $u, v\in \mathcal{V}_{sing}(\operatorname{div}(x\mathbb{A}\nabla.); \Omega_{p})$,  be s.t. $u=u_s\log|x|+u_r$ and $v=v_s\log|x|+v_r$, where
	\begin{eqnarray}
		\label{eq:cond}
		\begin{split}
			&u_s, v_s\in \mathcal{H}^1(\Omega_p), \quad u_r, \, v_r\in \mathcal{H}^1_{\delta}(\Omega_p),\qquad \text{ for some }0<\delta<1.	
		\end{split}
	\end{eqnarray}	
	Then $
	\int_{\Omega_p}\operatorname{div}( x  \mathbb{A}\nabla u)\overline{v}-\int_{\Omega_p}\overline{\operatorname{div}( x \mathbb{A}\nabla v)}{u}=-\langle a_{11}\gamma_{0}^{\Sigma} u_s, \overline{\gamma_{0}^{\Sigma} v_{r}}\rangle_{L^2(\Sigma)}+\overline{\langle a_{11}\gamma_{0}^{\Sigma} v_s, \overline{\gamma_{0}^{\Sigma} u_{r}}\rangle}_{L^2(\Sigma)}.		$
\end{corollary}
\section{The limiting absorption principle and properties of solution to \eqref{eq:unu}}
\label{sec:LAP}
Recall the family of problems \eqref{eq:unu}: given $f\in L^2(\Omega)$, $\nu>0$, find $u^{\nu}\in \mathcal{H}^1(\Omega)$ that satisfies 
\begin{align}
	\label{eq:unu_orig}
	\begin{split}
		& \operatorname{div}(( x \mathbb{A}+i\nu\mathbb{T})\nabla u^{\nu})=f \text{ in }\Omega, \\
		&\gamma_0^{\Gamma_p\cup\Gamma_n}u^{\nu}=0, \text{ periodic BCs at }\Gamma^{\pm}.
	\end{split}
	\tag{P$\nu$}
\end{align}
As argued in Lemma \ref{lem:pb_abs_wp}, the above problem is well-posed for all $\nu>0$. Let us discuss a strategy to prove the limiting absorption principle for \eqref{eq:unu_orig}. First of all, we see that if we are able to prove that $g^{\nu}:=\gamma_{n, \nu}^{\Sigma}u^{\nu}$ is bounded uniformly in $\nu$, namely
\begin{align}
	\label{eq:gnufirst}
	\|g^{\nu}\|_{\mathcal{H}^{1/2}(\Sigma)}\lesssim \|f\|,
\end{align}
we can argue that $g^{\nu}$ admits a weakly convergent subsequence (again denoted by $g^{\nu}$) in $\mathcal{H}^{1/2}(\Sigma)$ that would converge to some $g\in \mathcal{H}^{1/2}(\Sigma)$. This allows to decouple the problem  \eqref{eq:unu_orig} into two independent, sign-definite subproblems with the Neumann boundary condition on $\Sigma$: 
\begin{align}
	\label{eq:unu_orig2}
	\begin{split}
		& \operatorname{div}(( x \mathbb{A}+i\nu\mathbb{T})\nabla u^{\nu})=f \text{ in }\Omega_p\cup\Omega_n, \\
		&\gamma_{n,\nu}^{\Sigma}u^{\nu}=g^{\nu},\\
		&\gamma_0^{\Gamma_p\cup\Gamma_n}u^{\nu}=0, \text{ periodic BCs at }\Gamma^{\pm}_p\cup\Gamma^{\pm}_n.
	\end{split}
\end{align}
and conclude about the convergence of $u^{\nu}$, using the same argument as in Theorem \ref{theorem:reg_well_posedness}. However, derivation of the estimate \eqref{eq:gnufirst} is quite technical; thus we start with a simpler bound (even its weaker version \eqref{eq:gnubound})
$
\|g^{\nu}\|_{\mathcal{H}^{-1/2}(\Sigma)}\lesssim \|f\|.
$
It is then used to prove that $\|u^{\nu}\|_{\mathcal{V}_{sing}}$ is uniformly bounded as $\nu\rightarrow 0+$, see Theorem \ref{theorem:stability_estimate}. This result is further used in the proof of the bound \eqref{eq:gnufirst}, see Theorem \ref{theorem:gnubound_improved}. Once the bound \eqref{eq:gnufirst} is established, we proceed to proving the first part of Theorem \ref{theorem:LAP}:  more precisely, we will argue that the statement of this theorem is valid up to a subsequence, see Theorem  \ref{theorem:convergence_decomposition}. These results will be used in Section \ref{sec:LAP_Problem} to prove the statements of Theorem \ref{theorem:LAP}, Theorem \ref{theorem:main_result}
\begin{remark}
	\label{rem:stab_results}
	All stability results of this section hold true when $f$ in \eqref{eq:unu_orig} is replaced by $f^{\nu}$, s.t. there exist $C,\nu_0>0$, s.t. for $0<\nu<\nu_0$, $\|f^{\nu}\|\leq C \|f\|$.
\end{remark}
\subsection{The key stability estimate}
\begin{theorem}[The first stability estimate]
	\label{theorem:stability_estimate}
	There exist $C,\, \nu_0>0$, s.t. for all $ 0<\nu<\nu_0$, the solution to \eqref{eq:unu_orig} satisfies the following stability bound: $
	\|u^{\nu}\|_{\mathcal{V}_{sing}(\Omega)}\leq C\|f\|_{L^2(\Omega)}.
	$
\end{theorem}
To prove this theorem, we start with the following estimate, which shows that the norm of the conormal trace of $u^{\nu}$ is well-controlled in the space $\mathcal{H}^{-1/2}(\Sigma)$. 
\begin{proposition}
	\label{prop:gnubound_proof}
	Given $u^{\nu}$ as in \eqref{eq:unu_orig}, let the co-normal derivative at the interface $\Sigma$ be denoted by 
	$g^{\nu}:=\gamma_{n, \nu}^{\Sigma}u^{\nu}=\left.( x \mathbb{A}+i\nu\mathbb{T})\nabla u^{\nu}\cdot \vec{n}\right|_{\Sigma}.%=\left.i\nu \vec{e}_x\cdot \mathbb{T}\nabla u^{\nu}\right|_{\Sigma}.
	$
	There exist $C,\, \nu_0>0$, s.t. for all $0<\nu<\nu_0$, it holds that
	\begin{align}
		\label{eq:gnubound}
		\|g^{\nu}\|_{\mathcal{H}^{-1/2}(\Sigma)}\leq C\left(  \nu^{1/2}\|f\|+\sqrt{\|f\|\|u^{\nu}\|}\right).
	\end{align}
\end{proposition}
\begin{remark}
	It is quite easy to obtain a rougher version of this estimate. Indeed, since $\vec{v}:=(x\mathbb{A}+i\nu\mathbb{T})\nabla u^{\nu}\in \mathcal{H}(\operatorname{div};\Omega_p)$, it holds that $\|\gamma_{n,\nu}^{\Sigma }u^{\nu}\|_{\mathcal{H}^{-1/2}(\Sigma)}\lesssim \|\operatorname{div}\vec{v}\|+\|\vec{v}\|\lesssim\|f\|+\|x\nabla u^{\nu}\|$. By integration by parts one sees that $\|x\nabla u^{\nu}\|\lesssim \|u^{\nu}\|$, see the proof of Theorem \ref{theorem:stability_estimate}. Thus $\|g^{\nu}\|_{\mathcal{H}^{-1/2}(\Sigma)}\lesssim \|f\|+\|u^{\nu}\|$. Unfortunately, this rougher bound does suffice for our proof of Theorem \ref{theorem:stability_estimate}.
\end{remark}
The proof of this proposition relies on the following lifting lemma. To state it, let us repeat the definition \eqref{eq:omegaresigma} for the convenience of the reader: $
\Omega^{\delta}_{\Sigma}:=
\{\bx=(x,y)\in \Omega: \, |x|<\delta\},  \quad \delta>0.
$
\begin{lemma}
	\label{lemma:lifting}
	For all $\phi\in \mathcal{H}^{1/2}(\Sigma)$, all $0<\delta<\ell$, there exists $\Phi^{\delta}\in \mathcal{H}^1(\Omega_p)$, s.t. 
	\begin{align}
		\nonumber
		&\gamma_0^{\Sigma}\Phi^{\delta}=\phi,\qquad 
		\operatorname{supp}\operatorname{\Phi}^{\delta}\subseteq \overline{\Omega_p\cap\Omega^{\delta}_{\Sigma}}, \qquad \Phi^{\delta}(\delta,\vec{y})=0,	\text{ and }\\
		\label{eq:bound_data}
		&\|\Phi^{\delta}\|_{L^2(\Omega_p)}+\delta \|\nabla \Phi^{\delta}\|_{L^2(\Omega_p)}\leq C \delta^{1/2}\|\phi\|_{\mathcal{H}^{1/2}(\Sigma)}, \quad 
	\end{align}
	where $C>0$ is independent of $\delta$, $\phi$.
\end{lemma}
\begin{proof}
	See Lemma \ref{lemma:lifting_lemma} in Appendix  \ref{appendix:conormal_trace}.
\end{proof}
\begin{proof}[Proof of Proposition \ref{prop:gnubound_proof}]
	Given $\phi\in \H^{1/2}(\Sigma)$ and $\delta>0$, let $\Phi^{\delta}$ be its lifting from Lemma \ref{lemma:lifting}. Recall 
	\begin{align*}
		\|g^{\nu}\|_{\H^{-1/2}(\Sigma)}&=\sup\limits_{\phi\in \H^{1/2}(\Sigma): \, \|\phi\|_{\mathcal{H}^{1/2}(\Sigma)}=1}\left|\langle g^{\nu}, \phi\rangle_{\H^{-1/2}(\Sigma), \H^{1/2}(\Sigma)}\right|. 
	\end{align*}
	With the variational definition of the conormal derivative, it holds that 
	\begin{align*}
		-\langle g^{\nu}, \phi\rangle_{\H^{-1/2}(\Sigma), \H^{1/2}(\Sigma)}=\int_{\Omega_{\Sigma}^{\delta}\bigcap\Omega_p}\operatorname{div}\left(( x \mathbb{A}+i\nu\mathbb{T})\nabla u^{\nu}\right)\, {\Phi^{\delta}}+\int_{\Omega_{\Sigma}^{\delta}\bigcap\Omega_p}( x \mathbb{A}+i\nu\mathbb{T})\nabla u^{\nu}\cdot \nabla{\Phi}^{\delta}.
	\end{align*}
	The Cauchy-Schwarz inequality, Assumption  \ref{assump:matrices} on $\mathbb{A}$, $\mathbb{T}$ and the inequality$| x |<\delta$ in $\Omega_{\Sigma}^{\delta}$ yield
	\begin{align*}
		|\langle g^{\nu},\phi\rangle_{\Sigma}|\leq 
		\|f\|_{L^2(\Omega_p)}\|\Phi^{\delta}\|_{L^2(\Omega_p)}+(\delta+\nu)\|\nabla u^{\nu}\|_{L^2(\Omega_p)}\|\nabla \Phi^{\delta}\|_{L^2(\Omega_p)}. 
	\end{align*}
	Next, we use \eqref{eq:bound_data} for $\Phi^{\delta}$ and \eqref{eq:est_main} for $u^{\nu}$, more precisely $\nu^{1/2}\|\nabla u^{\nu}\|\leq c  \sqrt{\|f\|\|u^{\nu}\|}$:
	\begin{align*}
		|\langle g^{\nu},\phi\rangle_{\Sigma}|&\leq C		 \delta^{1/2}\|f\|\|\phi\|_{\mathcal{H}^{1/2}(\Sigma)}+C\delta^{-1/2}(\delta+\nu)\|\nabla u^{\nu}\|\|\phi\|_{\mathcal{H}^{1/2}(\Sigma)}\\
		&\leq C\left( \|f\|\delta^{1/2}+c\delta^{-1/2}(\delta+\nu)\nu^{-1/2}\sqrt{\|f\|\|u^{\nu}\|}\right)\|\phi\|_{\mathcal{H}^{1/2}(\Sigma)}.
	\end{align*}
	Choosing $\delta=\nu$ proves the bound in the statement of the proposition.
\end{proof}
Once we have the result of Proposition \ref{prop:gnubound_proof}, we can prove Theorem \ref{theorem:stability_estimate}. 
We use  duality-type techniques. 
\begin{proof}
	\textbf{Proof that $\|u^{\nu}\|\lesssim \|f\|$. }
	Let $p^{\nu}\in \mathcal{H}^1(\Omega\setminus\Sigma)$ be a unique solution to 
	\begin{align*}
		&\operatorname{div}\left(x \mathbb{A}\nabla p^{\nu}\right)=u^{\nu} \text{  in }\Omega_p\cup\Omega_n,\\
		&\gamma_{n}^{\Sigma}p^{\nu}=0,\\ &\gamma_0^{\Gamma_p\cup\Gamma_n}p^{\nu}=0,\qquad \text{ periodic BCs at }\Gamma_p^{\pm}\cup\Gamma_n^{\pm},
	\end{align*}
	namely, the homogeneous decoupled Neumann problem, 
	cf. Theorem \ref{theorem:fl2} for the well-posedness. Thus defined $p^{\nu}\in \mathcal{H}^1(\Omega\setminus\Sigma)$ satisfies the following variational formulation: 
	\begin{align}
		\label{eq:omegap}
		\int_{\Omega_p}\nabla p^{\nu}\cdot\overline{x \mathbb{A}\nabla v^{\nu}}+\int_{\Omega_n}\nabla p^{\nu}\cdot\overline{ x \mathbb{A}\nabla v^{\nu}}=-\int_{\Omega}u^{\nu}\overline{v^{\nu}},\quad \forall v^{\nu}\in \mathcal{H}^1(\Omega\setminus \Sigma).
	\end{align}
	On the other hand, testing \eqref{eq:unu_orig} with $p^{\nu}$ yields the following identity:
	\begin{align*}
		\int_{\Omega_p}( x \mathbb{A}+i\nu\mathbb{T})\nabla u^{\nu}\cdot\overline{\nabla p^{\nu}}+	\int_{\Omega_n}( x \mathbb{A}+i\nu\mathbb{T})\nabla u^{\nu}\cdot\overline{\nabla p^{\nu}}+\int_{\Sigma} \gamma_{n, \nu}^{\Sigma}u^{\nu}\, \overline{\left(\gamma_{0}^{\Sigma,p}p^{\nu}-\gamma_{0}^{\Sigma,n}p^{\nu}\right)}=-\int_{\Omega}f\overline{p^{\nu}}.
	\end{align*}
	Replacing the terms that involve $x\mathbb{A}\nabla u^{\nu}\cdot\overline{\nabla p^{\nu}}$ in the first two integrals in the above by the conjugated right-hand-side of \eqref{eq:omegap} where $v^{\nu}=u^{\nu}$ yields the identity:
	\begin{align*}
		-\int_{\Omega}|u^{\nu}|^2=-\int_{\Omega_p\cup\Omega_n}i\nu\mathbb{T}\nabla u^{\nu}\cdot\overline{\nabla p^{\nu}}-\int_{\Omega}f\overline{p}^{\nu}-\int_{\Sigma} \gamma_{n, \nu}^{\Sigma}u^{\nu}\, \overline{[\gamma_0^{\Sigma}p^{\nu}]}.
	\end{align*}
	Applying the Cauchy-Schwarz inequality we obtain the following bound:
	\begin{align*}
		\|u^{\nu}\|^2&\lesssim \nu\| \nabla u^{\nu}\|(\|\nabla p^{\nu}\|_{L^2(\Omega_p)}+\|\nabla p^{\nu}\|_{L^2(\Omega_n)})+\|\gamma_{n,\nu}^{\Sigma}u^{\nu}\|_{\mathcal{H}^{-1/2}(\Sigma)}(\|\gamma_0^{\Sigma,p}p^{\nu}\|_{\mathcal{H}^{1/2}(\Sigma)}+\|\gamma_0^{\Sigma,n}p^{\nu}\|_{\mathcal{H}^{-1/2}(\Sigma)})\\
		&\lesssim   \nu\|\nabla u^{\nu}\|\|p^{\nu}\|_{\mathcal{H}^1(\Omega\setminus \Sigma)}+(\nu^{1/2}\|f\|+\sqrt{\|f\|\|u^{\nu}\|})\|p^{\nu}\|_{\mathcal{H}^1(\Omega\setminus \Sigma)},
	\end{align*}
	where in the last inequality we used the continuity of the trace operator in $\mathcal{H}^1(\Omega_p)$ (resp. $\mathcal{H}^1(\Omega_n)$ and the bound on the conormal trace \eqref{eq:gnubound}. Next, we apply Theorem \ref{theorem:regularity} to bound  $\|p^{\nu}\|_{\mathcal{H}^1(\Omega\setminus \Sigma)}\lesssim \|u^{\nu}\|$ in the right-hand side of the above, and use the bound $\nu\|\nabla u^{\nu}\|\lesssim \|f\|$ from \eqref{eq:est_main2}. Dividing both sides of the obtained expression by $\|u^{\nu}\|$ leads to the following bound: 
	\begin{align*}
		\|u^{\nu}\|\lesssim \|f\|+(\nu^{1/2}\|f\|+\sqrt{\|f\|\|u^{\nu}\|}). 
	\end{align*}
	Using the Young inequality $\sqrt{\|f\|\|u^{\nu}\|}\leq \frac{1}{2}(\varepsilon^{-1}\|f\|+\varepsilon\|u^{\nu}\|)$  with $\varepsilon$ sufficiently small independent of $\nu$ yields
	\begin{align}
		\label{eq:unubound}
		\|u^{\nu}\|\lesssim \|f\|. 
	\end{align}
	\textbf{Proof that $\| x  u^{\nu}\|\lesssim \|f\|$. }
	We test \eqref{eq:unu_orig} with $ x  u^{\nu}$. This yields 
	\begin{align*}
		-\int_{\Omega} x ( x \mathbb{A}+i\nu\mathbb{T})\nabla u^{\nu}\cdot \overline{\nabla u^{\nu}}-\int_{\Omega}\overline{u^{\nu}}( x \mathbb{A}+i\nu\mathbb{T}) \nabla u^{\nu}\cdot \mathbf{e}_x=\int_{\Omega}f\overline{ x  u^{\nu}}.
	\end{align*}
	Taking the real part of the above and using that $\mathbb{A}, \, \mathbb{T}$ are Hermitian (cf. Corollary \ref{cor:M}), we obtain 
	\begin{align*}
		-\int_{\Omega}x^2 \mathbb{A}\nabla u^{\nu}\cdot \overline{\nabla u^{\nu}}-\Re \int_{\Omega}\overline{u^{\nu}}( x \mathbb{A}+i\nu\mathbb{T}) \nabla u^{\nu}\cdot \mathbf{e}_x=\Re \int_{\Omega}f\overline{xu^{\nu}}.
	\end{align*}
	Together with the Cauchy-Schwarz inequality this shows that
	\begin{align*}
		\int_{\Omega} x ^2|\nabla u^{\nu}|^2\lesssim   \left(\| x  \nabla u^{\nu}\|+\nu\|\nabla u^{\nu}\|\right)\|u^{\nu}\|+\|f\|\|u^{\nu}\|.
	\end{align*}
	With \eqref{eq:unubound} and \eqref{eq:est_main2} (namely $\nu\|\nabla u^{\nu}\|\lesssim \|f\|$) we conclude that 
	\begin{align*}
		\|x\nabla u^{\nu}\|\lesssim \|f\|\|x\nabla u^{\nu}\|+\|f\|^2.	
	\end{align*}
	Using the Young's inequality, we prove that $\|xu^{\nu}\|\lesssim \|f\|^2$. 
\end{proof}
Quite a rough result of Theorem \ref{theorem:stability_estimate} paves the way to proving a series of regularity estimates. We will present the corresponding results step-by-step in the list of propositions below. Before continuing, we state an immediate corollary of Theorem \ref{theorem:stability_estimate} and Lemma \ref{lem:pb_abs_wp}, in particular of \eqref{eq:est_main}: for all $0<\nu<\nu_0$, 
\begin{align}
	\label{eq:est_main_v2}
	&\|u^{\nu}\|_{L^2(\Omega)}+\|x\nabla u^{\nu}\|_{L^2(\Omega)}+\nu^{1/2}\|\nabla u^{\nu}\|_{L^2(\Omega)}\lesssim \|f\|_{L^2(\Omega)}.
	%\tag{StH1}
\end{align}
\subsection{Refined stability estimates on the conormal trace}
The key result of this section is an improved regularity of the conormal derivative of $u^{\nu}$. As we will see further, this result will play a crucial role in constructing the limiting absorption solution.
\begin{theorem}
	\label{theorem:gnubound_improved}
	Given $u^{\nu}$ as in \eqref{eq:unu_orig}, let $g^{\nu}:=\gamma_{n, \nu}^{\Sigma}u^{\nu}$. Then $g^{\nu}\in \mathcal{H}^{1/2}(\Sigma)$, and there exist $C, \nu_0>0$, s.t. 
	\begin{align}
		\label{eq:gnubound_key}
		\|g^{\nu}\|_{\mathcal{H}^{1/2}(\Sigma)}\leq C\|f\|, \quad \text{ for all }0<\nu<\nu_0.
	\end{align}
\end{theorem}
From Proposition \ref{prop:gnubound_proof} combined with Theorem \ref{theorem:stability_estimate} we already know that $
\|g^{\nu}\|_{\mathcal{H}^{-1/2}(\Sigma)}\leq C\|f\|.	$
Thus the result of Theorem \ref{theorem:gnubound_improved} improves this regularity by one order. The proof of Theorem \ref{theorem:gnubound_improved} relies on two auxiliary results. The first one is a counterpart of Proposition \ref{prop:gnubound_proof}, and is given below.  
\begin{proposition}
	\label{prop:gnubound_improved}
	Let $g^{\nu}$ be like in Theorem \ref{theorem:gnubound_improved}. There exist $C_1, C_2, \, \nu_0>0$, s.t. for all $0<\nu<\nu_0$,
	\begin{align}
		\label{eq:gnubound_improved}
		\|\partial_y g^{\nu}\|_{\mathcal{H}^{-1/2}(\Sigma)}\leq C_1 
		\|g^{\nu}\|_{\mathcal{H}^{1/2}(\Sigma)}\leq C_2\left(\|f\|+\sqrt{\|f\|\|\partial_y u^{\nu}\|}\right).
	\end{align}
\end{proposition}
The second results indicates an improved regularity of $u^{\nu}$ in the direction tangential to the interface.
\begin{proposition}
	\label{proposition:regularity_dy}
	There exist $C, \, \nu_0>0$, s.t. for all $0<\nu<\nu_0$, the solution to \eqref{eq:unu_orig} satisfies:
	\begin{align*}
		\|\partial_y u^{\nu}\|\leq C\|f\|.  
	\end{align*}
\end{proposition}
\begin{proof}
	Once Proposition \ref{prop:gnubound_improved} is proven, the result  follows by an argument resembling the proof of Theorem \ref{theorem:stability_estimate}, thus can be found in Appendix \ref{sec:prop_reg_dy}.
\end{proof}
As soon as we have these two Propositions proven, we can prove Theorem \ref{theorem:gnubound_improved}. 
\begin{proof}[Proof of Theorem \ref{theorem:gnubound_improved}]
	Follows at once from Propositions  \ref{prop:gnubound_improved} and \ref{proposition:regularity_dy}.
\end{proof}
The remainder of the section is dedicated to the proof of Proposition \ref{prop:gnubound_improved}. 
Since this proposition  is an improvement over the earlier result of Proposition \ref{prop:gnubound_proof}, we pursue a similar strategy in its proof. Looking through the proof of Proposition \ref{prop:gnubound_proof}, we see that it relies on the following two components: 
\begin{itemize}
	\item the bound \eqref{eq:est_main} for $\|\nabla u^{\nu}\|$, namely $\nu^{1/2}\|\nabla u^{\nu}\|\lesssim \sqrt{\|f\|\|u^{\nu}\|}$; 
	\item a lifting Lemma \ref{lemma:lifting} with well-tuned parameters.
\end{itemize}
To prove Proposition \ref{prop:gnubound_improved}, it would be natural to require a bound of the type \eqref{eq:est_main} with $u^{\nu}$ replaced by $\partial_y u^{\nu}$; however, such a bound does not seem to hold true. The key observation is the following: if we split the solution into high- and low-frequency components (with respect to the variable $y$, tangential to the interface $\Sigma$), with a well-chosen cutoff frequency, a necessary bound can be shown to hold true for the high-frequency component. On the other hand, low-frequency components have a high regularity, and the control on the corresponding conormal derivative can be obtained easier. This is the starting idea of the approach currently hidden in technical details. The suitable definition of a low- and high- frequency appears to be $\nu$-dependent, and thus more care should be given to the construction of an appropriate lifting, as well as obtaining low-frequency bounds. In particular, we will require fine regularity estimates on the solution in fractional Sobolev spaces; this is inspired by the earlier paper of Baouendi and Goulaouic \cite{baouendi_goulaouic}. 
%\begin{itemize}
%	\item introduction of extension operators, which enable us to work with high- and low-frequency components of the solution separately;
%	\item introduction of Bessel-like potentials to obtain regularity estimates in fractional spaces;
%	\item stability bounds on the solution which will be used in the low- and high- frequency cases; in the low-frequency case we will make use of the fractional regularity estimates, while in the high-frequency case the usual elliptic regularity estimates will suffice;
%	\item well-adapted lifting lemmas, for low- and high-frequency components of the solution.  
%\end{itemize}
\subsubsection{{Extension operators and Bessel potentials}}
\label{sec:extension_bessel}
In what follows, we will use Fourier analysis/pseudo-differential calculus techniques, in the spirit of \cite{baouendi_goulaouic}, to relate weighted and fractional Sobolev norms through appropriate embeddings. 

We start by rewriting the problem \eqref{eq:unu_orig} in the strip $(-a,a)\times \mathbb{R}$, by using periodic boundary conditions, next localize the new extended solution around the strip $\Omega_{\Sigma}^{\varepsilon}$, and, finally, will make use of the Fourier transform in the $y$-direction. To do so, we need to introduce extension, localization and restriction operators. In the case of periodic functions on $\Omega$ those operators take an extremely simple form.

\textit{\textbf{Extension operators. }}Let us fix $0<\delta<\ell/2$ (recall that $\Omega=(-a,a)\times (-\ell,\ell)$), and define a truncation function in the tangential (i.e. $y-$) direction  $\chi_{\ell,\delta}\in C^{\infty}( \mathbb{R}^2;\mathbb{R})$ as follows:
\begin{align}
	\label{eq:childelta}
	\mathbb{R}^2\ni \vec{x}=(x,y)\mapsto	\chi_{\ell,\delta}(x,y):=\left\{
	\begin{array}{ll}
		1, & |y|\leq \ell+\delta/2, \\
		0,& |y|>\ell+\delta, \\
		\in (0,1), & \text{ otherwise. }
	\end{array}	\right.
\end{align}
Sometimes we will write $\chi_{\ell,\delta}(y)$ instead of $\chi_{\ell,\delta}(x,y)$, to underline that $\chi_{\ell,\delta}$ is independent of $x$.  

Let us set $\mathbb{R}^2_a=\{(x,y)\in \mathbb{R}^2: \, x\in (-a, a)\}$. Then we define the operator $\mathcal{E}_{\delta}$ as a product of a truncation and the periodic extension operator: 
\begin{align*}
	\mathcal{E}_{\delta}&: \,L^2(\Omega)\rightarrow L^2(\mathbb{R}^2_a), \quad \mathcal{E}_{\delta}u(x,y)=\chi_{\ell,\delta}(x,y)\mathcal{E}u(x,y),\\[10pt]
	\mathcal{E}&: \,L^2(\Omega)\rightarrow L^2_{loc}(\mathbb{R}^2_a),\quad \mathcal{E}u(x,y)=\left\{
	\begin{array}{ll}
		u(x,y),& |y|<\ell, \\[5pt]
		u(x,y-2k\ell), & y\in (\ell(2k-1),\, \ell(2k+1)).  
	\end{array}
	\right.
\end{align*}
The following observations will be of use: there exists $C_{\delta}>0$, s.t. for all $u$ sufficiently regular, 
\begin{align}
	\label{eq:important_mappings}
	\begin{split}
		&\|u\|_{\Omega}\leq \|\mathcal{E}_{\delta} u\|_{\mathbb{R}^2_a}\leq C_{\delta}\|u\|_{\Omega}, \quad \|\partial_y u\|_{\Omega}\leq \|\partial_y \mathcal{E}_{\delta} u\|_{\mathbb{R}^2_a}\leq C_{\delta}(\|\partial_y u\|_{\Omega}+\|u\|_{\Omega}),\\
		&\|\partial_y (\mathbb{P}\nabla u)\|_{\Omega}\leq \|\partial_y(\mathbb{P} \nabla\mathcal{E}_{\delta}u)\|_{\mathbb{R}^2_a}\leq C_{\delta,\mathbb{P}}( \|\mathbb{P}\partial_y\nabla u\|_{\Omega}+\|\partial_y u\|_{\Omega}+\|u\|_{\Omega}),\quad \mathbb{P}\in \mathbb{R}^{2\times 2},\\
		&\|x\nabla \mathcal{E}_{\delta}u\|_{\mathbb{R}^2_a}\leq C_{\delta} \|u\|_{\mathcal{V}_{sing}(\Omega)}.
	\end{split}
\end{align} 
\textit{\textbf{Definition of $\mathcal{J}$.}} Next, let us define a Bessel potential in the direction $y$. Let us introduce the space of functions compactly supported in the direction $y$: 
$${C}^{\infty}_{0,y}(\mathbb{R}^2_a):=\{v\in C^{\infty}(\mathbb{R}^2_a): \, \operatorname{supp}v \subseteq [-a,a]\times [-L, L], \text{ for some }L>0\}.$$
Given $u\in {C}_{0,y}^{\infty}(\mathbb{R}^2_a)$, we define the (uni-directional) Bessel potential $\mathcal{J}: \, {C}_{0,y}^{\infty}(\mathbb{R}^2_a)\rightarrow {C}^{\infty}(\mathbb{R}^2_a)$ as follows:
\begin{align*}
	\mathcal{J}u:=\mathcal{F}_y^{-1}\left((1+\xi_y^2)^{1/4}\mathcal{F}_y u\right), \quad
	\mathcal{F}_yv(x,\xi_y):=\frac{1}{\sqrt{2\pi}}\int_{\mathbb{R}}\mathrm{e}^{-i\xi_y y}v(x,y)dy. 
\end{align*}
The operator $\mathcal{J}$ extends by density to a bounded linear operator from $H^{1/2}(\mathbb{R}^2_a)$ into $L^2(\mathbb{R}^2_a)$. 

By the Plancherel identity together with the Fubini theorem, we have that 
\begin{align}
	\label{eq:Ju}
	\|\mathcal{J}^2 u\|^2_{L^2(\mathbb{R}^2_a)}=\int_{-a}^a\int_{\mathbb{R}}(1+\xi_y^2)|\mathcal{F}_yu(x,\xi_y)|^2d\xi_y\, dx=\|u\|^2_{L^2(\mathbb{R}^2_a)}+\|\partial_y u\|^2_{L^2(\mathbb{R}^2_a)},\quad \forall u\in H^1(\mathbb{R}^2_a).
\end{align}
%which allows to extend $\mathcal{J}^2$ (and $\mathcal{J}$) to a bounded linear  operator from  the anisotropic Hilbert space $\{v\in L^2(\mathbb{R}^2_a): \, \|v\|^2_{L^2(\mathbb{R}^2_a)}+\|\partial_y v\|^2_{L^2(\mathbb{R}^2_a)}<\infty\}$ into $L^2(\mathbb{R}^2_a)$.
\subsubsection{{Extending the problem \eqref{eq:unu_orig} to $\mathbb{R}^2_a$ and relating the conormal derivatives.}}
Let us introduce
\begin{align*}
	U^{\nu}(x,y):=\mathcal{E}u^{\nu}(x,y), \quad U^{\nu}_{\delta}(x,y):=\mathcal{E}_{\delta}u^{\nu}(x,y),
\end{align*}
and extend $\mathbb{A}, \mathbb{T}$ by periodicity to the full strip $\mathbb{R}_a^2$ preserving the notation. %Then \eqref{eq:unu_orig} implies that 
%\begin{align*}
%	\operatorname{div}\left((x\mathbb{A}+i\nu\mathbb{T})\nabla U^{\nu}\right)=\mathcal{E}f \text{ in }\mathbb{R}^2_a.
%\end{align*}
% On the other hand, 
The truncated function $U^{\nu}_{\delta}=\chi_{\ell,\delta}U^{\nu}$ satisfies the new problem:
\begin{align}
	\label{eq:Unudelta}
	\operatorname{div}\left((x\mathbb{A}+i\nu\mathbb{T})\nabla U^{\nu}_{\delta}\right)&=F^{\nu}_{\delta} \text{ in }\mathbb{R}^2_a, \\
	\nonumber
	F^{\nu}_{\delta}&=\mathcal{E}_{\delta}f+U^{\nu}\operatorname{div}((x\mathbb{A}+i\nu\mathbb{T})\nabla \chi_{\ell,\delta})+\nabla U^{\nu}\cdot (x\mathbb{A}+i\nu\mathbb{T})\nabla \chi_{\ell,\delta}+\nabla \chi_{\ell,\delta}\cdot (x\mathbb{A}+i\nu\mathbb{T})\nabla U^{\nu}.
\end{align}
Remark that $F^{\nu}_{\delta}$ is supported on  $|y|\leq \ell+\delta$, and  $U^{\nu}=\chi_{\ell,2\delta}U^{\nu}$ on $|y|\leq \ell+\delta$. Therefore,
\begin{align}
	\|F^{\nu}_{\delta}\|_{L^2(\mathbb{R}^2_a)}\leq \|\mathcal{E}_{\delta}f\|_{L^2(\mathbb{R}^2_a)}+c_{1,\delta}\left(	\|U^{\nu}_{2\delta}\|_{L^2(\mathbb{R}^2_a)}+\|x\nabla U^{\nu}_{2\delta}\|_{L^2(\mathbb{R}^2_a)}+\nu\|\nabla U^{\nu}_{2\delta}\|_{L^2(\mathbb{R}^2_a)}\right),
\end{align}
for some $c_{1,\delta}>0$.
The following stability bound follows from \eqref{eq:important_mappings} and \eqref{eq:est_main_v2} (remark that the bound below is stronger than needed, cf. the power of $\nu$):
\begin{align}
	\label{eq:stabU}
	\|U^{\nu}_{2\delta}\|_{L^2(\mathbb{R}^2_a)}+\|x\nabla U^{\nu}_{2\delta}\|_{L^2(\mathbb{R}^2_a)}+\nu^{1/2}\|\nabla U^{\nu}_{2\delta}\|_{L^2(\mathbb{R}^2_a)}&\leq   c_{2,\delta}\|f\|_{L^2(\Omega)},
\end{align}
with $c_{2,\delta}>0$. Since  $\|\mathcal{E}_{\delta}f\|_{L^2(\mathbb{R}^2_a)}\leq C_{\delta}\|f\|_{L^2(\Omega)}$ (see \eqref{eq:important_mappings}), we conclude that there exists $c_{3,\delta}>0$, s.t. 
\begin{align}
	\label{eq:stabFdelta}
	\|F^{\nu}_{\delta}\|_{L^2(\mathbb{R}^2_a)}\leq c_{3,\delta} \|f\|_{L^2(\Omega)}, \quad \forall 0<\nu<1. 
\end{align}
In what follows, we will also need  
\begin{lemma}
	\label{lem:ellreg}
	For all $\nu>0$, $U^{\nu}_{\delta}\in H^2(\mathbb{R}^2_a)$; also $\gamma_{0}^{\partial\mathbb{R}^2_a}\left(\partial_y U^{\nu}_{\delta}\right)=0$, $\lambda\in \{n,p\}$.
\end{lemma}
\begin{proof}
	Follows by elliptic regularity ($u^{\nu}\in  \mathcal{H}^2(\Omega)$ by Lemma \ref{lem:pb_abs_wp}) and definition of $\mathcal{E}_{\delta}$. 
\end{proof}
The rewriting \eqref{eq:Unudelta} of the problem \eqref{eq:unu} will enable us to apply the Fourier transform in the direction $y$ and obtain more precise regularity estimates. To prove Proposition \ref{prop:gnubound_improved} we will resort to proving an analogous result for $U^{\nu}_{\delta}$. For this, let us introduce
\begin{equation}
	\label{eq:sigmainf}
	\Sigma_{\infty}:=\{(0,y), \, y\in \mathbb{R}\}.
\end{equation}
The result below is standard and links norms of  $\gamma_{n,\nu}^{\Sigma}u^{\nu}$ and $\gamma_{n,\nu}^{\Sigma_{\infty}}U^{\nu}_{\delta}:=\left.(\mathbb{A}+i\nu\mathbb{T})\nabla U^{\nu}_{\delta}\cdot \vec{n}\right|_{\Sigma^{\infty}}$ (see in particular \cite[Theorem 4.2.1]{wendland_hsiao} on equivalence of different $\mathcal{H}^{1/2}(\Sigma)$-norms, and use e.g. a standard trace norm based on parametrization of $\Gamma$, cf. \cite[p.176]{wendland_hsiao}).   
\begin{proposition}
	\label{prop:traces}
	Let $\delta>0$. There exists $C_{\delta}, \nu_0>0$, s.t. for all $0<\nu<\nu_0$, 
	\begin{align*}	
		%	\|\partial_y\gamma_{n,\nu}^{\Sigma}u^{\nu}\|_{\mathcal{H}^{-1/2}(\Sigma)}\leq C_1
		\|\gamma_{n,\nu}^{\Sigma}u^{\nu}\|_{\mathcal{H}^{1/2}(\Sigma)}\leq C_{\delta} \|\gamma_{n,\nu}^{\Sigma_{\infty}}U^{\nu}_{\delta}\|_{H^{1/2}(\mathbb{R})}.
	\end{align*}
	%	where $\|v\|_{H^{1/2}(\mathbb{R})}=\int_{\mathbb{R}}(1+\xi^2)^{1/2}|\mathcal{F}v(\xi)|^2d\xi$. 
\end{proposition}
%\begin{proof}
%	See Proposition \ref{prop:traces_apx2} in  \ref{appendix:auxiliary_technical_results}.
%\end{proof}
\begin{remark}
	In what follows, we will never need to make $\delta\rightarrow 0$ or $\delta\rightarrow +\infty$. Therefore, where appropriate, we will not indicate the dependence of the bounds on $\delta$, but rather consider that we fix $\delta>0$.  
\end{remark}
\subsubsection{{Auxiliary regularity bounds}}
In what follows we will need two stability bounds. The first one is a counterpart of the bound \eqref{eq:est_main}, namely $\nu\|\nabla u^{\nu}\|^2\lesssim \|f\|\|u^{\nu}\|$.  Informally, the result below says that $\nu^{1/2}\|(-\partial_y^{2})^{1/4}\nabla u^{\nu}\|\lesssim \|f\|+\sqrt{\|f\|\|\partial_y u^{\nu}\|}$. 
\begin{proposition}
	\label{proposition:jbound}
	Let $(u^{\nu})_{\nu>0}$ solve \eqref{eq:unu_orig},  $U^{\nu}_{\delta}=\mathcal{E}_{\delta}u^{\nu}$ satisfy \eqref{eq:Unudelta}, and let the matrix-valued function $\mathbb{B}\in C^{1,1}([-a,a]\times \mathbb{R}; \mathbb{C}^{2\times 2})$. Then there exist  $C,\,\nu_0>0$, s.t. for all $0<\nu<\nu_0$, 
	\begin{align*}
		\nu\|\mathcal{J}(\mathbb{B}\nabla U^{\nu}_{\delta})\|^2_{L^2(\mathbb{R}^2_a)}\leq C\left( \|f\|^2_{L^2(\Omega)}+\|f\|_{L^2(\Omega)}\|\partial_y u^{\nu}\|_{L^2(\Omega)}\right).
	\end{align*}
\end{proposition}
%\begin{remark}
%	In the above we implicitly extended the definition of $\mathcal{J}$ to allow its acting on the functions of the space $\left(H^1(\mathbb{R}^2_a)\right)^2$ by associating to $\vec{v}=(v_1,v_2)\in \left(H^1(\mathbb{R}^2_a)\right)^2$ $\mathcal{J}\vec{v}=(\mathcal{J}v_1,\mathcal{J}v_2)$.
%\end{remark}
\begin{proof}
	See Appendix \ref{sec:prop_jbound}. 
\end{proof}
The second bound is a standard elliptic regularity bound, explicit in $\nu>0$.  
\begin{proposition}
	\label{proposition:bounds_nu}
	Let $U^{\nu}_{\delta}$ and $\mathbb{B}\in C^{0,1}([-a,a]\times \mathbb{R};\mathbb{C}^{2\times 2})$. Then there exist  $C,\,\nu_0>0$, s.t.
	\begin{align}
		\label{eq:unubtmp}
		\nu\| \partial_y(\mathbb{B}\nabla U^{\nu}_{\delta})\|_{L^2(\mathbb{R}^2_a)}\leq  C \|f\|_{L^2(\Omega)}, \text{ for all }0<\nu<\nu_0.
	\end{align}
\end{proposition}
\begin{proof}
	See Appendix \ref{sec:proof_bounds_nu}.
\end{proof}
\subsubsection{{Frequency filters and corresponding results}} 
As discussed before, our proof of Proposition \ref{prop:gnubound_improved} relies on the decomposition of the solution in high- and low-frequency components in the direction tangential to the interface, and working with them in a separate manner. The originality of the approach is that the definition of the 'low-' and 'high-' frequency components is now $\nu$-dependent; as we will see later, we will need more sophisticated stability estimates for the low-frequency components of the solution, which make use of an improved regularity in the tangential direction. This is due to the definition of the lifting operator, which allows to control well high frequencies, but is less efficient at low frequencies, cf. Lemma \ref{lem:lifting_lem_important}.

\textit{Interface filters. }Let us introduce high- and low-frequency filters in the $y$-direction on the line ${\Sigma}_{\infty}$, cf. \eqref{eq:sigmainf}, and next on $\mathbb{R}^2_a$. For a fixed $w>0$, we define  (where $\mathcal{F}$ is the Fourier transform):
\begin{align}
	\label{eq:filter_operators}
	\begin{split}
		&\mathcal{L}_{w}: \, L^2(\mathbb{R})\rightarrow L^2(\mathbb{R}), \quad \mathcal{L}_w q:=\mathcal{F}^{-1}\left(\mathbb{1}_{|\xi|<w}\mathcal{F}q(\xi)\right),\quad \mathcal{H}_w:=\operatorname{Id}-\mathcal{L}_w.
	\end{split}
\end{align}
It is straightforward to verify that a priori $\|\mathcal{L}_w u\|_{H^s(\mathbb{R})}\leq C\|u\|_{H^s(\mathbb{R})}$, for all $u\in H^s(\mathbb{R})$, thus both filters are continuous on the space $H^s(\mathbb{R})$, $s\in \mathbb{R}$. Moreover, we also have that $\|\mathcal{L}_w u\|_{H^p(\mathbb{R})}\leq C_{s,p}\|u\|_{H^s(\mathbb{R})}$, for all $p\in \mathbb{R}$. Additionally, these operators commute with $\partial_y$, i.e. $[\partial_y, \mathcal{L}_w]=0$, and similarly for $\mathcal{H}_w$. 

\textit{Volume filters. }In a similar manner, we define the filter operators on the strip $\mathbb{R}_a^{2}$:
\begin{align}
	\label{eq:filter_operatorsR2}
	\begin{split}
		&\boldsymbol{\mathcal{L}}_{w}: \, L^2(\mathbb{R}^{2}_a)\rightarrow L^2(\mathbb{R}^{2}_a), \quad \boldsymbol{\mathcal{L}}_{w} q=\mathcal{F}^{-1}_y\left(\mathbb{1}_{|\xi_y|<w}\mathcal{F}_yq(.,\xi_y)\right),\\
		&\boldsymbol{\mathcal{H}}_{w}=\operatorname{Id}-\boldsymbol{\mathcal{L}}_{w}, \quad w>0.
	\end{split}
\end{align}
From the Plancherel identity, it follows that these operators are $L^2(\mathbb{R}^2_a)$-orthogonal projectors. Again, we have the following commutation relations: $\nabla \boldsymbol{\mathcal{L}}_{w}=\boldsymbol{\mathcal{L}}_{w}\nabla$, and, of course,  $f(x) \boldsymbol{\mathcal{L}}_{w}=\boldsymbol{\mathcal{L}}_{w}(f(x)\cdot)$, for any $f\in L^{\infty}(-a,a)$ (i.e. $[\boldsymbol{\mathcal{L}}_{w}, f(x)]=0$).
\subsubsection{{Lifting lemmas}} 
Another important component of our approach are lifting lemmas, cf. the proof of the less precise counterpart of Proposition \ref{prop:gnubound_improved}, namely Proposition \ref{prop:gnubound_proof}. Below we will present a lifting lemma for functions defined on a real line $\Sigma_{\infty}$. 
Let us also denote by $\mathbb{R}^{2,+}_{a}:=(0,a)\times \mathbb{R}$. %We use the same notation for the filter operators defined on $L^2(\mathbb{R}^2_a)$ and on $L^2(\mathbb{R}^{2,+}_a)$. 
\begin{lemma}
	\label{lem:lifting_lem_important}
	Let $0<\nu<a$. There exists a bounded linear operator 
	\begin{align*}
		L^{\nu}: \, H^{1/2}(\Sigma_{\infty})\rightarrow H^{1}(\mathbb{R}^{2,+}_a),\text{ s.t. }\gamma_0^{\Sigma_{\infty}}L^{\nu}\psi=\psi,\text{ and } \operatorname{supp}L^{\nu}\psi\subseteq [0,\nu]\times\mathbb{R}.
	\end{align*}
	This operator commutes with $\partial_y$: for all $\psi\in H^{3/2}(\Sigma_{\infty})$, $\partial_y L^{\nu}\psi=L^{\nu}\partial_y \psi$, and satisfies the following bounds, valid for all $0<\re<1/2$, $\psi\in H^{1/2}(\Sigma_{\infty})$, with $C_{\re}>0$ independent of $\nu$, but depending on $\re$:
	%This operator has the following properties: 
	%\begin{enumerate}
	%	\item for all $s\geq 1/2$, $L^{\nu}\in \mathcal{L}(H^{s}(\Sigma_{\infty}), H^{s+1/2}(\mathbb{R}^{2,+}_{\nu}))$.
	%	\item For all $w>0$, $L^{\nu}\mathcal{L}_{w}= \boldsymbol{\mathcal{L}}_{w}L^{\nu}$, $L^{\nu}\mathcal{H}_{w}= \boldsymbol{\mathcal{H}}_{w}L^{\nu}$. 
	%	\item For all $\psi\in H^{3/2}(\Sigma_{\infty})$, $\partial_y L^{\nu}\psi=L^{\nu}\partial_y \psi$. 
	%\end{enumerate}
	%Let $0<\re\leq 1/2$ be fixed. Then, for all $\psi\in H^{1/2}(\Sigma_{\infty})$, with $C>0$ independent of $\nu$, $\re$, it holds that
	\begin{align}
		\label{eq:psibound_start0_mt}
		\|\partial_yL^{\nu}\psi\|_{L^2(\mathbb{R}^{2,+}_{\nu})}\leq C_{\re}\|\psi\|_{H^{1/2}(\Sigma_{\infty})},\\
		\label{eq:psibound_start_dx_mt}
		\nu^{1/2}\|\mathcal{J}\partial_x L^{\nu}\mathcal{L}_{\re\nu^{-1}}\psi\|_{L^2(\mathbb{R}^{2,+}_{\nu})}\leq C_{\re}\|\psi\|_{H^{1/2}(\Sigma_{\infty})},\\
		\label{eq:partialy_psibound_mt}
		\|\partial_x L^{\nu}\mathcal{H}_{\re\nu^{-1}}\psi\|_{L^2(\mathbb{R}^{2,+}_{\nu})}\leq C_{\re}\|\psi\|_{H^{1/2}(\Sigma_{\infty})}.
	\end{align}
\end{lemma}
\begin{proof}
	See Appendix \ref{appendix:lifting_lemmas}.
\end{proof}
We see that the above lifting lemma distinguishes between the low- and high-frequency cases. We are able to construct a lifting that is supported on $(0,\nu)$, and whose $H^1$-seminorm is controlled for 'high' frequencies, however, we are not able to do this for the 
'low' frequency case, where only derivatives in the tangential to $\Sigma$ direction are bounded uniformly in $\nu$. 
\subsubsection{{Proof of Proposition \ref{prop:gnubound_improved}}}
We start by remarking that
\begin{align}
	\label{eq:unusimple}
	\|\partial_y\gamma_{n, \nu}^{\Sigma}u^{\nu}\|_{\mathcal{H}^{-1/2}(\Sigma)}\leq C_1\|\gamma_{n, \nu}^{\Sigma}u^{\nu}\|_{\mathcal{H}^{1/2}(\Sigma)},
\end{align}
which follows by using the Fourier characterization of the norms on closed regular curves \cite[Section 4.2.1]{wendland_hsiao}. To bound the above further, we use Proposition \ref{prop:traces}, and thus it remains to establish the corresponding bound for  $\|\gamma_{n,\nu}^{\Sigma_{\infty}}U^{\nu}_{\delta}\|_{{H}^{1/2}(\Sigma_{\infty})}$. The Fourier definition of the $H^{s}(\mathbb{R})$-norm and $(1+\xi^2)^{1/2}= (1+\xi^2)^{-1/2}+\xi^2(1+\xi^2)^{-1/2}$ yield
\begin{align}
	\label{eq:fc}
	\|\gamma_{n,\nu}^{\Sigma_{\infty}}U^{\nu}_{\delta}\|_{H^{1/2}(\Sigma_{\infty})}^2= \|\gamma_{n,\nu}^{\Sigma_{\infty}}U^{\nu}_{\delta}\|_{H^{-1/2}(\Sigma_{\infty})}^2+\|\partial_y \gamma_{n,\nu}^{\Sigma_{\infty}}U^{\nu}_{\delta}\|_{H^{-1/2}(\Sigma_{\infty})}^2.
\end{align}
The $H^{-1/2}(\Sigma_{\infty})$-part can be controlled using 
\begin{align*}
	\|\gamma_{n,\nu}^{\Sigma_{\infty}}U^{\nu}_{\delta}\|_{H^{-1/2}(\Sigma_{\infty})}\lesssim \|\operatorname{div}((x\mathbb{A}+i\nu)\nabla U^{\nu}_{\delta})\|_{L^2(\mathbb{R}^2_a)}+\|(x\mathbb{A}+i\nu)\nabla U^{\nu}_{\delta}\|_{L^2(\mathbb{R}^2_a)}\lesssim \|f\|_{L^2(\Omega)},
\end{align*} 
see \eqref{eq:stabU} and \eqref{eq:stabFdelta}.The estimate \eqref{eq:unusimple}, Proposition \ref{prop:traces}, \eqref{eq:fc} and the estimate above yield
\begin{align}
	\label{eq:unu_vs_Unu}
	\|\partial_y\gamma_{n, \nu}^{\Sigma}u^{\nu}\|_{\mathcal{H}^{-1/2}(\Sigma)}\leq C_1 \|\gamma_{n,\nu}^{\Sigma}u^{\nu}\|_{\mathcal{H}^{1/2}(\Sigma)}\leq C(\|f\|+	\|\partial_y\gamma_{n, \nu}^{\Sigma_{\infty}}U^{\nu}_{\delta}\|_{\mathcal{H}^{-1/2}(\Sigma_{\infty})}),
\end{align}
with some $C>0$. It remains to control the following quantity:
\begin{align}
	\label{eq:unudelta_trace}
	\|\partial_y \gamma_{n, \nu}^{\Sigma_{\infty}}U^{\nu}_{\delta}\|_{H^{-1/2}(\Sigma_{\infty})}=\sup\limits_{0\neq\varphi\in C_0^{\infty}(\Sigma_{\infty})}\frac{\left|\langle \partial_y \gamma_{n, \nu}^{\Sigma_{\infty}}U^{\nu}_{\delta}, \overline{\varphi}\rangle_{H^{-1/2}(\Sigma_{\infty}), H^{1/2}(\Sigma_{\infty})}\right| }{\|\varphi\|_{H^{1/2}(\Sigma_{\infty})}},
\end{align}
where we used the density of $C_0^{\infty}(\Sigma_{\infty})$ in $H^{1/2}(\Sigma_{\infty})$.  By definition of the distributional derivative 
\begin{align*}
	\langle \partial_y \gamma_{n, \nu}^{\Sigma_{\infty}}U^{\nu}_{\delta}, \overline{\varphi}\rangle_{H^{-1/2}(\Sigma_{\infty}), H^{1/2}(\Sigma_{\infty})}=-\langle  \gamma_{n, \nu}^{\Sigma_{\infty}}U^{\nu}_{\delta}, \partial_y\overline{\varphi}\rangle_{H^{-1/2}(\Sigma_{\infty}), H^{1/2}(\Sigma_{\infty})}.
\end{align*} 
To estimate the above, we use the variational definition of the conormal trace, with $\Phi^{\nu}:=L^{\nu}\varphi$, see Lemma \ref{lem:lifting_lem_important}, and the fact that $[\partial_y, L^{\nu}]=0$: 
\begin{align*}
	\langle \partial_y \gamma_{n, \nu}^{\Sigma_{\infty}}U^{\nu}_{\delta}, \overline{\varphi}\rangle_{H^{-1/2}(\Sigma_{\infty}), H^{1/2}(\Sigma_{\infty})}&=\int_{\mathbb{R}^{2,+}_{\nu}}\operatorname{div}((x\mathbb{A}+i\nu\mathbb{T})\nabla U^{\nu}_{\delta})\overline{\partial_y\Phi^{\nu}}+\int_{\mathbb{R}^{2,+}_{\nu}}(x\mathbb{A}+i\nu\mathbb{T})\nabla U^{\nu}_{\delta}\,\overline{\partial_y\nabla \Phi^{\nu}}\\
	&=\int_{\mathbb{R}^{2,+}}F^{\nu}_{\delta}\overline{\partial_y \Phi^{\nu}}-\int_{\mathbb{R}^{2,+}_{\nu}}\partial_y\left((x\mathbb{A}+i\nu\mathbb{T})\nabla U^{\nu}_{\delta}\right)\cdot \overline{\nabla \Phi^{\nu}},
\end{align*}
where we also used in the integration by parts that $U^{\nu}_{\delta}\in H^2(\mathbb{R}^2_a)$, cf. Lemma \ref{lem:ellreg}. To obtain the required estimates, we will exploit the fact that $\|\partial_y \Phi^{\nu}\|_{L^2(\mathbb{R}^{2,+}_{\nu})}/\|\varphi\|_{H^{1/2}(\Sigma_{\infty})}$ is uniformly bounded in $\nu$, cf. \eqref{eq:psibound_start0_mt}, and single out the respective terms:
\begin{align}
	\label{eq:dyy}
	&\langle \partial_y \gamma_{n, \nu}^{\Sigma_{\infty}}U^{\nu}_{\delta}, \overline{\varphi}\rangle_{H^{-1/2}(\Sigma_{\infty}), H^{1/2}(\Sigma_{\infty})}=\sum\limits_{j=1}^3 \mathcal{I}_{j},\quad \text{ where } \mathcal{I}_1=\int_{\mathbb{R}^{2,+}}F^{\nu}_{\delta}\overline{\partial_y \Phi^{\nu}},\\
	\nonumber
	&\mathcal{I}_2=-\int_{\mathbb{R}^{2,+}_{\nu}}\vec{e}_y\cdot\partial_y\left((x\mathbb{A}+i\nu\mathbb{T})\nabla U^{\nu}_{\delta}\right) \overline{\partial_y \Phi^{\nu}}, \qquad 
	\mathcal{I}_3=-
	\int_{\mathbb{R}^{2,+}_{\nu}}\vec{e}_x \cdot\partial_y\left((x\mathbb{A}+i\nu\mathbb{T})\nabla U^{\nu}_{\delta}\right)\overline{\partial_x \Phi^{\nu}}.
\end{align}
\textbf{A bound on $\mathcal{I}_1$. }
With the Cauchy-Schwarz inequality we have 
\begin{align}
	\label{eq:i1}
	\left|\mathcal{I}_1\right|\leq \|F_{\delta}^{\nu}\|_{L^2(\mathbb{R}_{\nu}^{2,+})}\|\partial_y\Phi^{\nu}\|_{L^2(\mathbb{R}_{\nu}^{2,+})}\leq C_{\delta} \|f\|\|\varphi\|_{H^{1/2}(\Sigma_{\infty})},\quad C_{\delta}>0,
\end{align}
where the last bound follows from the bound  \eqref{eq:stabFdelta} on $\|F_{\delta}^{\nu}\|$ and \eqref{eq:psibound_start0_mt}.\\
\textbf{A bound on $\mathcal{I}_2$. }By the same argument as above, it holds that
\begin{align*}
	|\mathcal{I}_2|&\leq \|\vec{e}_y\cdot\partial_y\left((x\mathbb{A}+i\nu\mathbb{T})\nabla U^{\nu}_{\delta}\right)\|_{L^2(\mathbb{R}^{2,+}_{\nu})}\|\varphi\|_{H^{1/2}(\Sigma_{\infty})}\\
	&\leq \left(\|(x\partial_y \mathbb{A}+i\nu\partial_y\mathbb{T})\nabla U^{\nu}_{\delta}\|_{L^2(\mathbb{R}^{2,+}_{\nu})}+
	\|(x\mathbb{A}+i\nu\mathbb{T})\partial_y\nabla U^{\nu}_{\delta}\|_{L^2(\mathbb{R}^{2,+}_{\nu})}\right)\|\varphi\|_{H^{1/2}(\Sigma_{\infty})}.
\end{align*}
We bound $\|(x\partial_y \mathbb{A}+i\nu\partial_y\mathbb{T})\nabla U^{\nu}_{\delta}\|_{L^2(\mathbb{R}^{2,+}_{\nu})}\lesssim \|f\|$, cf. \eqref{eq:stabU}; as for the second term, we use $|x|<\nu$:
\begin{align*}
	\|(x\mathbb{A}+i\nu\mathbb{T})\partial_y\nabla U^{\nu}_{\delta}\|_{L^2(\mathbb{R}^{2,+}_{\nu})}\lesssim \nu\|\partial_y \nabla U^{\nu}_{\delta}\|_{L^2(\mathbb{R}^{2,+}_{\nu})}\lesssim \|f\|_{L^2(\Omega)},
\end{align*}
as argued in \eqref{eq:unubtmp}.  This proves that 
\begin{align}
	\label{eq:i2}
	|\mathcal{I}_2|\lesssim \|f\|\|\varphi\|_{H^{1/2}(\Sigma_{\infty})}.
\end{align}
\textbf{A bound on $\mathcal{I}_3$. }We split the term $\partial_x\Phi^{\nu}$ into low- and high- frequencies:
\begin{align*}
	\mathcal{I}_3&=\mathcal{I}_{3,l}+\mathcal{I}_{3,h},\qquad 
	\mathcal{I}_{3,l}=-\int_{\mathbb{R}^{2,+}_{\nu}}\vec{e}_x\cdot (x\partial_y(\mathbb{A}\nabla U^{\nu}_{\delta})+i\nu\partial_y (\mathbb{T}\nabla U^{\nu}_{\delta}))\cdot\overline{\partial_x L^{\nu}{\mathcal{L}}_{(2\nu)^{-1}}\varphi},\\ 
	\mathcal{I}_{3,h}&=-\int_{\mathbb{R}^{2,+}_{\nu}}\vec{e}_x\cdot (x\partial_y(\mathbb{A}\nabla U^{\nu}_{\delta})+i\nu\partial_y (\mathbb{T}\nabla U^{\nu}_{\delta}))\cdot\overline{\partial_x L^{\nu}{\mathcal{H}}_{(2\nu)^{-1}}{\varphi}}.	
\end{align*}
We start by bounding $\mathcal{I}_{3,h}$. Using the Cauchy-Schwarz inequality and $|x|<\nu$ we obtain 
\begin{align}
	\label{eq:i3h}
	|\mathcal{I}_{3,h}|\leq \nu (\|\partial_y (\mathbb{A}\nabla U^{\nu}_{\delta})\|_{L^2(\mathbb{R}^{2,+}_{\nu})}+\|\partial_y (\mathbb{T}\nabla U^{\nu}_{\delta})\|_{L^2(\mathbb{R}^{2,+}_{\nu})})\|\partial_x L^{\nu}\mathcal{H}_{(2\nu)^{-1}}\varphi\|_{L^2(\mathbb{R}^{2,+}_{\nu})}\leq C\|f\|\|\varphi\|_{H^{1/2}(\Sigma_{\infty})},
\end{align}
where the desired bound follows by combining \eqref{eq:unubtmp} and \eqref{eq:partialy_psibound_mt} with $\re=1/2$. 

The quantity $\mathcal{I}_{3,l}$ is further rewritten using the Plancherel identity ($\xi$ is the Fourier variable in $y$-direction):
\begin{align*}
	\mathcal{I}_{3,l}=-\int_{\mathbb{R}^{2,+}_{\nu}}i\xi\left(x\mathcal{F}_y(\mathbb{A}\nabla U^{\nu}_{\delta})+i\nu \mathcal{F}_y(\mathbb{T}\nabla U^{\nu}_{\delta})\right)\, \overline{\mathcal{F}_y(\partial_x L^{\nu}\mathcal{L}_{(2\nu)^{-1}}\varphi)}dx d\xi,
\end{align*}
and with the Cauchy-Schwarz inequality and $|\xi|\leq (1+\xi^2)^{1/2}$, we obtain 
\begin{align*}
	|\mathcal{I}_{3,l}|\leq \|(1+\xi^2)^{1/4}\left(x\mathcal{F}_y(\mathbb{A}\nabla U^{\nu}_{\delta})+i\nu \mathcal{F}_y(\mathbb{T}\nabla U^{\nu}_{\delta})\right)\|_{L^2(\mathbb{R}^{2,+}_{\nu})}\|(1+\xi^2)^{1/4}{\mathcal{F}_y(\partial_x L^{\nu}\mathcal{L}_{(2\nu)^{-1}}\varphi)}\|_{L^2(\mathbb{R}^{2,+}_{\nu})}.
\end{align*}
With the definition of $\mathcal{J}$ (see p. \pageref{eq:Ju}), the Plancherel identity, and the bound $|x|<\nu$, we conclude that   
\begin{align*}
	|\mathcal{I}_{3,l}|\leq \nu(\|\mathcal{J}\mathbb{A}\nabla U^{\nu}_{\delta}\|_{L^2(\mathbb{R}^{2,+}_{\nu})}+\|\mathcal{J}\mathbb{T}\nabla U^{\nu}_{\delta}\|_{L^2(\mathbb{R}^{2,+}_{\nu})})\,\|\mathcal{J}\partial_x L^{\nu}\mathcal{L}_{(2\nu)^{-1}}\varphi\|_{L^2(\mathbb{R}^{2,+}_{\nu})}.
\end{align*}
We use Proposition \ref{proposition:jbound} to bound the terms involving $U^{\nu}_{\delta}$ and Lemma \ref{lem:lifting_lem_important}, more precisely, \eqref{eq:psibound_start_dx_mt} with $\re=1/2$, for the term involving $\varphi$. This gives the upper bound  
\begin{align}
	\label{eq:i3l}
	|\mathcal{I}_{3,l}|\leq C (\|f\|+\sqrt{\|f\|\|\partial_y u^{\nu}\|})\|\varphi\|_{H^{1/2}(\Sigma_{\infty})},
\end{align}
with $C>0$ independent of $\nu$. 
Combining \eqref{eq:i3h} and \eqref{eq:i3l} we obtain 
\begin{align}
	\label{eq:i3}
	|\mathcal{I}_{3}|\lesssim (\|f\|+\sqrt{\|f\|\|\partial_y u^{\nu}\|})\|\varphi\|_{H^{1/2}(\Sigma_{\infty})}.
\end{align}
\textbf{The final bound. }Combining the bounds \eqref{eq:i3}, \eqref{eq:i2}, \eqref{eq:i1} into \eqref{eq:dyy} and next \eqref{eq:unudelta_trace} yields
\begin{align*}
	\|\partial_y\partial_{n,\nu}^{\Sigma_{\infty}}U^{\nu}_{\delta}\|_{H^{-1/2}(\Sigma_{\infty})}\lesssim \|f\|+\sqrt{\|f\|\|\partial_y u^{\nu}\|}.
\end{align*}
The desired bound in the statement of the proposition is immediate from the above and \eqref{eq:unu_vs_Unu}. \qed

\subsection{An important jump property of the solutions to \eqref{eq:unu_orig}: a weakened form of Theorem \ref{theorem:LAP}}
\label{sec:result_decomposition}
Now that we have the result of Theorem \ref{theorem:gnubound_improved} on the regularity of the Neumann trace of $u^{\nu}$ on $\Sigma$, we can consider separately the problem satisfied by $u^{\nu}$ in $\Omega_p$ and $\Omega_n$, with the boundary condition $\gamma_{n}^{\Sigma}u^{\nu}=g^{\nu}$. As $\nu\rightarrow 0+$, the couple $(u^{\nu}, g^{\nu})$ admits, up to a subsequence, a weak limit in the topology $\mathcal{V}_{sing}(\Omega)\times \mathcal{H}^{1/2}(\Sigma)$. At this point we are not able to conclude about the uniqueness of the limit.
Nonetheless, we will be able apply the result of Theorem \ref{theorem:reg_well_posedness} to the $u^+_k$, which would yield the decomposition of $u^+_k$ into a regular and singular parts. This approach will reveal one important property of the limiting absorption solution, namely, the relation between the jump across $\Sigma$ of the regular part and the conormal trace. Recall \eqref{eq:defvsing}: 
\begin{align}
	\label{eq:vsingdiv}
	\mathcal{V}_{sing}(\operatorname{div}(x\mathbb{A}\nabla .); \Omega)=\{{v}\in \mathcal{V}_{sing}(\Omega): \operatorname{div}(x\mathbb{A}\nabla v)\in L^2(\Omega), \, \gamma_n^{\Sigma}v\in \mathcal{H}^{1/2}(\Sigma), \, (\gamma_n^{\Gamma_+}+\gamma_n^{\Gamma_{-}})v=0\},
\end{align}
and Definition \ref{definition:one_sided_trace} of the trace of functions from the above space.
\begin{theorem}
	\label{theorem:convergence_decomposition}
	Let $(u^{\nu})_{\nu>0}\subset \mathcal{H}^1(\Omega)$ be a sequence of solutions to \eqref{eq:unu_orig}. Then there exists a weakly convergent, as $\nu\rightarrow 0+$,  subsequence in $L^2(\Omega)$. For any such subsequence  $(u^{\nu_k})_{k\in\mathbb{N}}$, the following holds true. \\
	As $\nu_k\rightarrow 0$, it converges strongly in ${H}^{1/2-\varepsilon}(\Omega)$ for all $\varepsilon>0$ to a limit $\widetilde{u}\in \mathcal{V}_{sing}(\operatorname{div}(x\mathbb{A}\nabla .); \Omega)$, which satisfies
	\begin{align}
		\label{eq:bvpu}
		\begin{split}
			&\operatorname{div}(x\mathbb{A}\nabla \widetilde{u})=f \text{ in }\Omega,\\
			&[\gamma_0^{\Sigma}\widetilde{u}]=-i\pi a_{11}^{-1} \gamma_{n}^{\Sigma}\widetilde{u}.
		\end{split}
	\end{align}
\end{theorem}
The above convergence statement is a weakened version of Theorem \ref{theorem:LAP}, since it is valid up to a subsequence. The key, unusual property, is of course $[\gamma_0^{\Sigma} \widetilde{u}]=-i\pi a_{11}^{-1}\gamma_n^{\Sigma}\widetilde{u}$, which, as we will see later, ensures the uniqueness of the limit of $(u^{\nu})_{\nu>0}$. 
\begin{remark}
	Up to our knowledge, this property was first observed in the work \cite{ciarlet_mk_peillon}; moreover, it was proven to hold true under the assumption that, as $\nu\rightarrow 0+$, $u^{\nu}\rightarrow u^+$ in $L^2(\Omega)$, where $u^+(x,y)=g(y)(\log|x|+i\pi\mathbb{1}_{x<0})+u^+_{cont}(x,y)$, with $g\in \mathcal{H}^1(\Sigma)$ and $u^+_{cont}\in \mathcal{V}_{reg}(\Omega)$. Unlike in quite an implicit proof of the respective result in \cite{ciarlet_mk_peillon}, in the approach of the present work this property of the traces comes out naturally, as a corollary of the fact that the solution with the absorption can be split as $u^{\nu}=u^{\nu}_{s}+u^{\nu}_{cont}$, with $u^{\nu}_{s}$ converging to a function with in particular a logarithmic and jump singularities, and $u^{\nu}_{cont}\in \mathcal{H}^1(\Omega)$ to a regular function in the topology of  $\mathcal{H}^{1-}(\Omega)$. See the proof of Theorem \ref{theorem:convergence_decomposition} for more details. 
\end{remark}
The proof of Theorem \ref{theorem:convergence_decomposition} is based on two auxiliary results, summarized in the following sesction.
\subsubsection{Auxiliary results}
The first result is a counterpart of Proposition \ref{prop:decomp1} for the problem with absorption \eqref{eq:main_problem}.
\begin{proposition}
	\label{prop:aux}
	There exists $\nu_0>0$, s.t. for any $0<\nu<\nu_0$, the solution to \eqref{eq:unu_orig} writes
	\begin{align*}
		u^{\nu}=u^{\nu}_h\log\left(x+i\nu r\right)+u^{\nu}_{cont},\qquad r:=\mathbb{T}_{11}(\mathbb{A}_{11})^{-1},
	\end{align*}
	where, for all $0<\re<1/2$ there exists $C_{\re}>0$ independent of $\nu$, s.t.    $\|u^{\nu}_{cont}\|_{\mathcal{H}^{1-\re}(\Omega)}+ \|u^{\nu}_{cont}\|_{\mathcal{H}^{1}_{\re}(\Omega)}\leq C_{\re}\|f\|$, and there exists $C>0$ independent of $\nu, \re,$ s.t. $\|u^{\nu}_h\|_{\mathcal{H}^1(\Omega)}\leq C\|f\|$. 
\end{proposition}
This result, in turn, relies on the following limiting absorption result, proven in Appendix \ref{appendix:lap_auxiliary}. 
\begin{proposition}
	\label{prop:lap_reg}
	Let $f\in L^2_{\delta}(\Omega_p)$ for some $0<\delta<1$, and  let $(v^{\nu})_{\nu>0}\subset \mathcal{H}^1(\Omega_p)$ be a family of the solutions to the well-posed problems parametrized by $\nu>0$: find $v^{\nu}\in \mathcal{H}^1(\Omega_p)$, s.t.
	\begin{align*}
		&\operatorname{div}((x+i\nu r)\mathbb{A})\nabla v^{\nu})=f \text{ in }\Omega_p,\qquad r=\mathbb{T}_{11}(\mathbb{A}_{11})^{-1},\\
		&\widetilde{\gamma}_{n,\nu}^{\Sigma}v^{\nu}:=\left.(x+i\nu r)\mathbb{A}\nabla v^{\nu}\cdot \vec{n}\right|_{\Sigma}=0,\\
		&\gamma_0^{\Gamma_p}v=0, \text{ periodic BCs at }\Gamma^{\pm}_p.
	\end{align*}
	Then there exists $\nu_0>0$, s.t. for all $0<\nu<\nu_0$, all $0<\re<1$,  $\|v^{\nu}\|_{\mathcal{H}^1_{\delta+\re}(\Omega_p)}\leq C_{\re} \|f\|_{L^2_{\delta}(\Omega_p)}$, with $C_{\re}>0$ independent of $\nu$. 
\end{proposition}
\begin{proof}[Proof of Proposition \ref{prop:aux}]
	\textbf{Rewriting of the problem.}	Before proving the desired result, we will rewrite the problem satisfied by $u^{\nu}$ in a more convenient for us form. For this we will alter the original formulation:
	\begin{align*}
		\operatorname{div}((x\mathbb{A}+i\nu\mathbb{T})\nabla u^{\nu})=\operatorname{div}((x+i\nu r)\mathbb{A}\nabla u^{\nu})+i\nu \operatorname{div}((\mathbb{T}-r\mathbb{A})\nabla u^{\nu}).
	\end{align*}
	We set $
	f_{1}^{\nu}:=\operatorname{div}((\mathbb{T}-r\mathbb{A})\nabla u^{\nu})$, and remark that $(\mathbb{T}-r\mathbb{A})_{11}=0$, therefore 
	\begin{align}
		\label{eq:fcontnu}
		\|f_{1}^{\nu}\|\leq C \nu(\|\partial_y \nabla u^{\nu}\|+\|\nabla u^{\nu}\|+\|u^{\nu}\|)\lesssim \|f\|,
			\end{align}
	as follows from Proposition \ref{proposition:bounds_nu} and \eqref{eq:est_main_v2}. This shows that 
	\begin{align}
		\label{eq:unur}
		\operatorname{div}((x+i\nu r)\mathbb{A}\nabla u^{\nu})=f-f_1^{\nu}, \quad \|f-f_1^{\nu}\|_{L^2(\Omega)}\leq C\|f\|.
	\end{align}
	Then all the previous results apply to this new problem, see also Remark \ref{rem:stab_results}, with the adapted conormal derivative $\widetilde{\gamma}_{n,\nu}^{\Sigma}u^{\nu}:=\left.(x+i\nu r)\mathbb{A}\nabla u^{\nu}\cdot \vec{n}\right|_{\Sigma}=:\widetilde{g}^{\nu}$.
	
	\textbf{Decomposition.}	We make an ansatz (cf. the decomposition of Proposition \ref{prop:decomp1}):
	\begin{align*}
		u^{\nu}(x,y):=u^{\nu}_{s}+u^{\nu}_{cont},\qquad u^{\nu}_{s}(x,y):=u^{\nu}_{h}(x,y)\log\left(x+i\nu r(x,y)\right), 
	\end{align*}
	with $u^{\nu}_h\in \mathcal{H}^1(\Omega)$ being a unique solution to the decoupled boundary-value problem: 
	\begin{align}
		\label{eq:unuh}
		\begin{split}
			&\operatorname{div}(\mathbb{A}\nabla u^{\nu}_h)=0 \text{ in }\Omega\setminus \Sigma,\\
			&\gamma_{0}^{\Sigma}u^{\nu}=a_{11}^{-1}\widetilde{g}^{\nu},\\
			&\gamma_0^{\Gamma_p\cup\Gamma_n}u^{\nu}_h=0,  \text{ periodic BCs at }\Gamma^{\pm}_p\cup\Gamma^{\pm}_n,
		\end{split}
	\end{align}
	and $u^{\nu}_{cont}:=u^{\nu}-u^{\nu}_h$. The stated upper bound on $u^{\nu}_h$ is a corollary of Theorem \ref{theorem:gnubound_improved} applied to \eqref{eq:unur}. 
	
	It remains to prove the bound on $u^{\nu}_{cont}$. For this we will write the problem satisfied by $u^{\nu}_{cont}$.  To do so, we will need the one-sided conormal trace of $u^{\nu}_h$, associated to the problem \eqref{eq:unuh}, and defined for sufficiently regular functions via $\partial_{n}^{\Sigma,\lambda}v:=\gamma_0^{\Sigma,\lambda}(\mathbb{A}\nabla v)\cdot \vec{n}$, $\lambda\in \{n,p\}$. Remark that from the stated upper bound on $\|u^{\nu}_h\|_{\mathcal{H}^1(\Omega)}$ it follows that
	\begin{align}
		\label{eq:conorm_der}
		q^{\nu}_{\lambda}:=\partial_{n}^{\Sigma,\lambda}u^{\nu}_h\quad \text{ satisfies }\quad
		\|q^{\nu}_{\lambda}\|_{\mathcal{H}^{-1/2}(\Sigma)}\lesssim \|f\|, \quad \lambda\in \{n,p\}.
	\end{align}
	Let us now study the \textbf{problem satisfied by $u^{\nu}_{cont}$.} A priori, $\operatorname{div}((x+i\nu r)\mathbb{A}\nabla u^{\nu}_{cont})\notin L^2(\Omega)$, since the jump of the conormal trace of $u^{\nu}_{cont}$ across $\Sigma$ may not vanish. Therefore, we will write a decoupled problem satisfied by $u^{\nu}_{cont}$, on $\Omega_p$ and $\Omega_n$ separately. In particular, on $\Omega\setminus \Sigma$, using that $\operatorname{div}(\mathbb{A}\nabla u^{\nu}_h)=0$, 
	\begin{align}
		\nonumber
		\operatorname{div}((x+i\nu r)\mathbb{A}\nabla u^{\nu}_{cont})&=f-f^{\nu}_{1}-\operatorname{div}((x+i\nu r)\mathbb{A}\nabla u^{\nu}_{s})\\
		\label{eq:unuc}
		&=f-f^{\nu}_{1}-\operatorname{div}(\mathbb{A}(\vec{e}_x+i\nu \nabla r)u^{\nu}_h)-\nabla((x+i\nu r)\log(x+i\nu r))\cdot \mathbb{A}\nabla u^{\nu}_h=:f^{\nu}.
	\end{align}
	By a straightforward computation, with  $\|u^{\nu}_h\|_{\mathcal{H}^1(\Omega)}\lesssim \|f\|_{L^2(\Omega)}$ and \eqref{eq:fcontnu}, it follows that for all $\delta>0$, \begin{align}
		\label{eq:rhs_bound}
		\|f^{\nu}\|_{L^2_{\delta}(\Omega_{\lambda} )}\leq C_{\delta}\|f\|_{L^2(\Omega)}, \quad \lambda\in \{n,p\}.
	\end{align}
	Next, since $u^{\nu}_{cont}$ satisfies the decoupled problem \eqref{eq:unuc} in $\Omega\setminus \Sigma$, we need to provide it with the boundary condition on $\Sigma$, which would render it well-posed. We choose the Neumann boundary condition, adapted to \eqref{eq:unuc}. Remark that 
	\begin{align*}
		(x+i\nu r)\mathbb{A}\nabla u^{\nu}_{s}=\mathbb{A}(\vec{e}_x+i\nu \nabla r)u^{\nu}_{h}+(x+i\nu r)\log(x+i\nu r)\mathbb{A}\nabla u^{\nu}_h, 
	\end{align*}
	hence, using the fact that $\partial_{n}^{\Sigma,\lambda}u^{\nu}_h\in \mathcal{H}^{-1/2}(\Sigma)$, one can check that the following identity holds in  $\mathcal{H}^{-1/2}(\Sigma)$:
	\begin{align*}
		\widetilde{\gamma}_{n,\nu}^{\Sigma,\lambda}u^{\nu}_{s}&=\mathbb{A}(\vec{e}_x+i\nu \nabla r)\cdot \vec{e}_x\gamma_0^{\Sigma,\lambda}u^{\nu}_{h}+i\nu r\log(i\nu r)\partial_n^{\Sigma,\lambda}u^{\nu}_h  \\
		&=\widetilde{g}^{\nu}(1+i\nu \gamma)+i\nu r\log(i\nu r)q^{\nu}_{\lambda}, \quad \gamma:=a_{11}^{-1}\left.\vec{e}_x\cdot\mathbb{A}\nabla r\right|_{\Sigma}.
	\end{align*}
	Since $u^{\nu}_{cont}=u^{\nu}-u^{\nu}_{s}$, we  conclude that $u^{\nu}_{cont}$ satisfies the following two-sided Neumann BC on $\Sigma$, $0<\nu<1/2$:
	\begin{align}
		\label{eq:two_sided_neumann}
		\tilde{\gamma}_{n,\nu}^{\Sigma,\lambda}u^{\nu}_{cont}=\nu |\log\nu| e^{\nu}_{\lambda}, \text{ where } e^{\nu}_{\lambda}=-i|\log\nu|^{-1} \left(\gamma \widetilde{g}^{\nu}-i r\log(i\nu r)q^{\nu}_{\lambda}\right)\in \mathcal{H}^{-1/2}(\Sigma), \quad \lambda\in \{n,p\}.
	\end{align}
	Crucially, from Theorem \ref{theorem:gnubound_improved} and \eqref{eq:conorm_der}, it follows that for all $\nu<1/2$,  
	\begin{align}
		\label{eq:enu}
		\|e^{\nu}_{\lambda}\|_{\mathcal{H}^{-1/2}(\Sigma)}\lesssim \|f\|.
	\end{align}
	Thus, we have shown that $u^{\nu}_{cont}\in \mathcal{H}^1(\Omega\setminus\Sigma)$ solves the following decoupled Neumann problem: with $\lambda\in \{n,p\}$, 
	\begin{align*}
		&\operatorname{div}((x+i\nu r)\mathbb{A}u^{\nu}_{cont})=f^{\nu} \text{ in }\Omega_{\lambda},\\
		&\tilde{\gamma}_{n,\nu}^{\Sigma,\lambda}u^{\nu}_{cont}=\nu |\log\nu| e^{\nu}_{\lambda}, \\ &\gamma_0^{\Gamma_{\lambda}}u=0, \quad\text{periodic BCs at }\Gamma^{\pm}_{\lambda}. 
	\end{align*}
	The above problem is well-posed in $\mathcal{H}^1(\Omega\setminus\Sigma)$, by the same argument as used in Lemma \ref{lem:pb_abs_wp}. 
	
	\textbf{The bound on $u^{\nu}_{cont}$: splitting. }
	We split $u^{\nu}_{cont}=u^{\nu}_{cont,0}+s^{\nu}$, and bound the two terms separately. Here $u^{\nu}_{cont,0},\,  s^{\nu}\in \mathcal{H}^1(\Omega\setminus\Sigma)$ are unique solutions to the problems: for $\lambda\in \{n,p\}$, 
	\begin{align}
		\label{eq:unucont0pb}
		\begin{split}
			&\operatorname{div}((x+i\nu r)\mathbb{A}\nabla u^{\nu}_{cont,0})=f^{\nu} \text{ in }\Omega_{\lambda},\\
			&\widetilde{\gamma}_{n,\nu}^{\Sigma,\lambda}u^{\nu}_{cont}=0, \qquad{\gamma_0^{\Gamma_{\lambda}}u^{\nu}_{cont,0}=0, } \text{ periodic BCs at }\Gamma^{\pm}_{\lambda},
		\end{split}
	\end{align}
	and 
	\begin{align}
		\label{eq:neumann}
		\begin{split}
			&\operatorname{div}((x+i\nu r)\mathbb{A}\nabla s^{\nu})=0 \text{ in }\Omega_{\lambda},\\
			&\widetilde{\gamma}_{n,\nu}^{\Sigma,\lambda}s^{\nu}=\nu|\log\nu|e^{\nu}_{\lambda}, \qquad{\gamma_0^{\Gamma_{\lambda}}s^{\nu}=0, } \text{ periodic BCs at }\Gamma^{\pm}_{\lambda}.
		\end{split}
	\end{align}
	We will obtain the bounds on $u^{\nu}_{cont,0}$ and $s^{\nu}$ separately in $\Omega_n$ and $\Omega_p$, and next argue that they imply the global bound on $u^{\nu}_{cont}$. \\
	\textbf{A bound on $u^{\nu}_{cont,0}$. }From Proposition \ref{prop:lap_reg} and \eqref{eq:rhs_bound}, it follows that for all $\re>0$, there exists $C_{\re}>0$, 
	\begin{align}
		\label{eq:boununu0}
		\|u^{\nu}_{cont,0}\|_{\mathcal{H}^1_{\re}(\Omega_p)}+	\|u^{\nu}_{cont,0}\|_{\mathcal{H}^1_{\re}(\Omega_n)}\leq C_{\re}\|f\|.
	\end{align}
	\textbf{A bound on $s^{\nu}$. }We start by proving the following two auxiliary statements, and next proceed by a very simple, interpolation-like argument:
	\begin{align}
		\label{eq:snu1}
		(a)\; 
		\|\nabla s^{\nu}\|_{\Omega_{\lambda}}\lesssim \log\nu^{-1}\|f\|_{\Omega}, \quad 
		(b)\;\||x|^{1/2}\nabla s^{\nu}\|_{\Omega_{\lambda}}\lesssim \nu^{1/2}\log\nu^{-1}\|f\|_{\Omega},\qquad  \lambda\in \{n,p\}.  
	\end{align}
	It suffices to prove these estimates in $\Omega_p$. We test \eqref{eq:neumann} with $s^{\nu}$ and integrate by parts. This yields 
	\begin{align}
		\label{eq:mid}
		\begin{split}
			\int_{\Omega_p}(x+i\nu r)\mathbb{A}\nabla s^{\nu}\cdot\overline{\nabla s^{\nu}}=-\nu|\log\nu|\langle e^{\nu}_p, \overline{\gamma_{0}^{\Sigma,p}s^{\nu}}\rangle_{\mathcal{H}^{-1/2}(\Sigma), \mathcal{H}^{1/2}(\Sigma)}.
		\end{split}
	\end{align}
	Taking the imaginary part of the above shows that 
	\begin{align*}
		\nu\|\nabla s^{\nu}\|^2_{L^2(\Omega_p)}\lesssim \nu |\log\nu|\|e^{\nu}\|_{\mathcal{H}^{-1/2}(\Sigma)}\|\gamma_0^{\Sigma,p}s^{\nu}\|_{\mathcal{H}^{1/2}(\Sigma)}\lesssim \nu|\log\nu| \|f\|\|\nabla s^{\nu}\|_{L^2(\Omega_p)}, 
	\end{align*}
	where the last bound follows from \eqref{eq:enu},   continuity of $\gamma_0^{\Sigma,p}$ on $\mathcal{H}^1(\Omega_p)$ and the Poincar\'e inequality in $\mathcal{H}^1(\Omega_p)$. This yields \eqref{eq:snu1}(a). The bound \eqref{eq:snu1}(b) is obtained by taking the real part of \eqref{eq:mid} and proceeding like before: 
	\begin{align*}
		\|x^{1/2}\nabla s^{\nu}\|^2\lesssim \nu |\log\nu|\|e^{\nu}\|_{\mathcal{H}^{-1/2}(\Sigma)}\|\gamma_0^{\Sigma,p}s^{\nu}\|_{\mathcal{H}^{1/2}(\Sigma)}\lesssim \nu|\log\nu| \|f\|\|\nabla s^{\nu}\|_{L^2(\Omega_p)}\lesssim \nu \log^2\nu \|f\|^2, 
	\end{align*}
	where the last bound follows from \eqref{eq:snu1}(a).
	The bounds \eqref{eq:snu1}(a), (b) imply similar bounds in the weighted space $\mathcal{H}^1_{\re}(\Omega_p)$. Indeed, for $0<\delta<1$, with $\Omega_{\Sigma}^{\nu}=\{(x,y)\in \Omega: \, |x|<\nu\}$, and for $\nu<1/2$,  
	\begin{align*}
		\|x^{\delta/2}\nabla s^{\nu}\|^2_{\Omega_p}=\int_{\Omega_{\Sigma}^{\nu}\cap \Omega_p}x^{\delta}|\nabla s^{\nu}|^2+\int_{\Omega_p\setminus \overline{\Omega}_{\Sigma}^{\nu}}x^{\delta-1}x|\nabla s^{\nu}|^2\leq \nu^{\delta}\|\nabla s^{\nu}\|^2_{\Omega_p}+\nu^{\delta-1}\|x^{1/2}\nabla s^{\nu}\|^2_{\Omega_p}
		\lesssim \nu^{\delta}\log^2\nu\|f\|^2,
	\end{align*}
	as follows from \eqref{eq:snu1}(a), (b). By the Poincar\'e inequality in weighted spaces, cf. e.g.  Proposition \ref{prop:hardy}, and with the above bound, we conclude that 
	\begin{align}
		\label{eq:boundsnu}
		\|s^{\nu}\|_{\mathcal{H}_{\delta}^1(\Omega_p)}+	\|s^{\nu}\|_{\mathcal{H}_{\delta}^1(\Omega_n)}\leq C_{\delta} \nu^{\delta/2}\log\nu^{-1}\|f\|.
	\end{align}
	\textbf{Final bound on $u^{\nu}_{cont}$. }We combine the bounds \eqref{eq:boununu0} and \eqref{eq:boundsnu} to conclude that, since, in particular, for all $\nu>0$, $u^{\nu}_{cont}=u^{\nu}-u^{\nu}_h\log(x+i\nu r)\in \mathcal{H}^1(\Omega)$ (as $u^{\nu}_h\in \mathcal{H}^1(\Omega)$), it holds that
	\begin{align*}
		\|u^{\nu}_{cont}\|_{\mathcal{H}^1_{\varepsilon}(\Omega)}\leq 	\|u^{\nu}_{cont}\|_{\mathcal{H}^1_{\varepsilon}(\Omega_p)}+\|u^{\nu}_{cont}\|_{\mathcal{H}^1_{\varepsilon}(\Omega_{n})}\leq  \sum\limits_{\lambda\in\{n,p\}}(\|u^{\nu}_{cont,0}\|_{\mathcal{H}^1_{\varepsilon}(\Omega_{\lambda})}+\|s^{\nu}\|_{\mathcal{H}^1_{\varepsilon}(\Omega_{\lambda})})\leq C_{\re} \|f\|_{L^2(\Omega)}, 
	\end{align*}
	for all $\re>0$, with some $C_{\re}>0$ independent of $\nu$ but depending on $\re$. This proves one part of the bound. To prove the fractional space bound, we resort 
	to  Lemma \ref{lem:embedding}; for all $0<\re<1/2$, there exists $c_{\re}>0$, s.t.
	\begin{align*}
		\|u^{\nu}_{cont}\|_{\mathcal{H}^{1-\varepsilon}(\Omega_p)}+	\|u^{\nu}_{cont}\|_{\mathcal{H}^{1-\varepsilon}(\Omega_n)}\leq c_{\re} \|f\|.
	\end{align*}
	Because additionally $u^{\nu}_{cont}\in \mathcal{H}^{1}(\Omega)$, we have that $[\gamma_0^{\Sigma}u^{\nu}_{cont}]=0$. By \cite[Lemma 1.5.1.8]{grisvard_book} or adapting the proof of a more detailed result \cite[Theorem 2.4]{jakovlev}, we conclude that $
	\|u^{\nu}_{cont}\|_{\mathcal{H}^{1-\varepsilon}(\Omega)}\leq C_{\re}  \|f\|,$ 
	with some $C_{\re}>0$ depending on $\re>0$ but independent of $\nu$. 
\end{proof}
A rather immediate corollary of the bounds of  Proposition \ref{prop:aux} is 
\begin{corollary}
	\label{cor:cv}
	There exists $\nu_0>0$, s.t. for all $0<\nu<\nu_0$, 
	\begin{align*}
		\|u^{\nu}\|_{H^{1/2-\re/2}(\Omega)}\leq c_{\re}	\|u^{\nu}\|_{\mathcal{H}^1_{1+\re}(\Omega\setminus\Sigma)}\leq C_{\re}\|f\|_{L^2(\Omega)}, \quad \text{ for all }0<\re<1,
	\end{align*}
	where $c_{\re},\,C_{\re}>0$ depend on $\re>0$ but are independent of $\nu$. 
\end{corollary}
\begin{proof}
	By \cite[Theorem 2.4]{jakovlev} it is sufficient to prove the corresponding fractional Sobolev estimates separately in subdomains $\Omega_p$ and $\Omega_n$, and by Lemma \ref{lem:embedding}, it is sufficient to show the stated estimates in the weighted Sobolev spaces only. In $\Omega_p$, using the decomposition of Proposition \ref{prop:aux},
	\begin{align*}
		x^{\frac{1+\re}{2}}\nabla u^{\nu}=\frac{x^{\frac{1+\re}{2}}}{x+i\nu r}(\vec{e}_x+i\nu\nabla r)u^{\nu}_h+	x^{\frac{1+\re}{2}}\log(x+i\nu)\nabla u^{\nu}_h+x^{\frac{1+\re}{2}}\nabla u^{\nu}_{cont}.
	\end{align*}
	By the same result, the last two terms in the right-hand side are bounded  in $L^2(\Omega_p)$ uniformly in $\nu$. The first term is bounded pointwise by $Cx^{-(1-\re)/2}|u^{\nu}_h|$, with $C$ independent of $\nu$,  
	and the desired inequality follows by Proposition \ref{prop:hardy}. 
\end{proof}
\subsubsection{Proof of Theorem \ref{theorem:convergence_decomposition}}
%%%
%%%MK: the bounded sequence $u^{\nu_k}$ admits a single limit point, thus any subsequence that converges in any stronger topology actually has a single limit point, thus the full sequence has a single limit point in a stronger topology
%%%
\textbf{Convergence. }The existence of the weakly convergent in $L^2(\Omega)$ subsequence follows from Theorem \ref{theorem:stability_estimate}. Assume now that $(u^{\nu_k})_{k\in \mathbb{N}}$ converges weakly in $L^2(\Omega)$ to the limit $\widetilde{u}$. Its weak convergence in $\mathcal{H}^{1/2-\re}(\Omega)$ for any $\re>0$ is immediate from Corollary \ref{cor:cv} and the fact that $u^{\nu_k}$ admits a single limit point $\widetilde{u}$ in $L^2(\Omega)$.

The stated strong convergence in $\mathcal{H}^{1/2-\re}(\Omega)$ for any $\re>0$ follows by the compact embedding of Sobolev spaces $\mathcal{H}^{1/2-\re}(\Omega)$ into $\mathcal{H}^{1/2-\re/2}(\Omega)$.  

Let us now prove that $\widetilde{u}\in \mathcal{V}_{sing}(\operatorname{div}(x\mathbb{A}\nabla .); \Omega)$, cf. \eqref{eq:vsingdiv}. We will prove this alongside with showing that $\widetilde{u}$ satisfies the first identity in \eqref{eq:bvpu}. In virtue of \eqref{eq:unu_orig2}, we also have that  $u^{\nu}\rightharpoonup \widetilde{u}$ in $\mathcal{V}_{sing}(\Omega)$. The  standard argument about convergence of distributional derivatives, separately in $\Omega_p$ and $\Omega_n$, yields $$x\mathbb{A}\nabla u^{\nu_k}\rightharpoonup x\mathbb{A}\nabla \widetilde{u}\quad \text{in}\quad  L^2(\Omega_p) \text{ and }L^2(\Omega_n).$$ 
Since $\tilde{u}\in \mathcal{V}_{sing}(\Omega)$, by Proposition \ref{prop:property2}, $x\mathbb{A}\nabla \widetilde{u}\in L^2(\Omega)$. Therefore, using \eqref{eq:unu_orig2}, as $\nu_k\rightarrow 0+$, 
\begin{align}
	\label{eq:vnuk}
	\vec{v}^{\nu_k}:=(x\mathbb{A}+i\nu\mathbb{T})\nabla u^{\nu_k}\overset{L^2(\Omega)}{\rightharpoonup} \widetilde{\vec{v}}:=x\mathbb{A}\nabla \widetilde{u}.
\end{align}
Since $\operatorname{div}\vec{v}^{\nu_k}=f$ in $\Omega$, and using \eqref{eq:vnuk}, necessarily,  $\operatorname{div}\vec{v}^{\nu_k}= \operatorname{div}\widetilde{\vec{v}}.$

Next, by continuity of the normal trace mapping $\mathcal{H}(\operatorname{div};\Omega)\ni\vec{h}\mapsto \left.\vec{h}\cdot \vec{n}\right|_{\Sigma}\in \mathcal{H}^{-1/2}(\Sigma)$, the above weak convergence results, and Theorem \ref{theorem:gnubound_improved} we conclude $\gamma_{n,\nu_k}^{\Sigma}u^{\nu_k}\rightharpoonup \gamma_{n}^{\Sigma}\widetilde{u}$ in $\mathcal{H}^{1/2}(\Sigma)$.  

The above considerations and continuity of the normal trace prove that $\widetilde{u}\in \mathcal{V}_{sing}(\Omega)$ satisfies
\begin{align*}
	&\operatorname{div}(x\mathbb{A}\nabla \widetilde{u})=f \text{ in }\Omega,\qquad\gamma_{n}^{\Sigma}\widetilde{u}\in \mathcal{H}^{1/2}(\Sigma), \\
	&\gamma_0^{\Gamma_p\cup\Gamma_n}\widetilde{u} =0, \qquad\text{periodic BCs at }\Gamma^{\pm}.
\end{align*}
This shows that $\widetilde{u}\in \mathcal{V}_{sing}(\operatorname{div}(x\mathbb{A}\nabla.); \Omega)$. 

\textbf{The jump property of $\widetilde{u}$. }It remains to argue that $[\gamma_0^{\Sigma}\widetilde{u}]=-i\pi a_{11}^{-1}\gamma_n^{\Sigma}\widetilde{u}$. Recalling Definition  \ref{def:trace} of the Dirichlet trace, to prove this property we need to construct the decomposition of $\widetilde{u}$ into the singular and regular parts, as defined in Theorem \ref{theorem:reg_well_posedness}. On the other hand, we have at hand the decomposition of Proposition \ref{prop:aux}, which was constructed in a slightly different manner compared to \eqref{eq:harmonic}, see the proofs of the corresponding results. Nonetheless, Lemma \ref{lem:deftraces} enables us to make use of this modified decomposition.

We remark that, with notation of Proposition \ref{prop:aux}, $\|u^{\nu_k}_h\|_{\mathcal{H}^1(\Omega)}$, and $\|u^{\nu_k}_{cont}\|_{\mathcal{H}^1_{\re}(\Omega)}$ are uniformly bounded in $\nu_k\rightarrow 0$ (for all $0<\re<1/2$). This shows that, up to a subsequence, as $\nu_k\rightarrow 0$, for all $0<\re<1/2$, 
\begin{align}
	\label{eq:limitsd}
	u^{\nu_k}_h\overset{\mathcal{H}^{1}(\Omega)}{\rightharpoonup} \widetilde{u}_h, \qquad u^{\nu_k}_{cont}\overset{\mathcal{H}^{1-\re}(\Omega)}{\rightharpoonup}\widetilde{u}_{cont},\qquad u^{\nu_k}_h\overset{\mathcal{H}^{1-\re}(\Omega)}{\rightarrow} \widetilde{u}_h, \qquad u^{\nu_k}_{cont}\overset{\mathcal{H}^{1-\re}(\Omega)}{\rightarrow}\widetilde{u}_{cont},
\end{align}
where we used in the last statements the compactness of the embeddings $H^s(\Omega)\subset H^{s+\re}(\Omega)$, for $\re>0$. 

Since, additionally, $u^{\nu_k}$, $u^{\nu_k}_h$ and $u_h^{\nu_k}$ admit subsequences that converge pointwise (see \cite[Theorem 4.9]{brezis2010functional}) to their $L^2(\Omega)$-strong limits, we pass to the pointwise limit in the decomposition of Proposition \ref{prop:aux}: 
\begin{align}
	\label{eq:utildlog}
	\widetilde{u}=\widetilde{u}_h\lim\limits_{\nu\rightarrow 0+}\log(x+i\nu r)+\widetilde{u}_{cont}=\widetilde{u}_h(\log|x|+i\pi\mathbb{1}_{x<0})+\widetilde{u}_{cont}.
\end{align}
By \eqref{eq:limitsd} and uniqueness of the limits, we have obtained the decomposition of $\widetilde{u}$ whose restrictions to $\Omega_p$ (resp. $\Omega_n$) satisfy conditions of Lemma \ref{lem:deftraces}. Application of Lemma \ref{lem:deftraces} shows that 
\begin{align*}
	\gamma_n^{\Sigma,\lambda}\widetilde{u}=a_{11}\gamma_0^{\Sigma,\lambda}\widetilde{u}_h,\quad\lambda\in\{n,p\},\quad \gamma_0^{\Sigma,p}\widetilde{u}=\gamma_0^{\Sigma,p}\widetilde{u}_{cont}, \qquad \gamma_0^{\Sigma,n}\widetilde{u}=\gamma_{0}^{\Sigma,n}\widetilde{u}_{cont}+i\pi \gamma_0^{\Sigma,n}\widetilde{u}_h.
\end{align*}
The regularity of the limits \eqref{eq:limitsd} implies that, cf. Lemma 1.5.1.8 of \cite{grisvard_book} for fractional Sobolev spaces,  
$
[\gamma_0^{\Sigma}\widetilde{u}_h]=0, \, [\gamma_0^{\Sigma}\widetilde{u}_{cont}]=0, 
$
and hence the assertion of the theorem. 
\begin{remark}
	\label{rem:lap_wrong_sign}
	Remark that \eqref{eq:utildlog} written for $\nu\rightarrow 0-$ yields the decomposition $\widetilde{u}_h(\log|x|-i\pi\mathbb{1}_{x<0})+\widetilde{u}_{cont}$, and the jump condition  $[\gamma_0^{\Sigma}\tilde{u}]=i\pi \gamma_n^{\Sigma}\tilde{u}$. 
\end{remark}
\section{The limiting problem. Proofs of Theorems \ref{theorem:LAP} and \ref{theorem:main_result}}
\label{sec:LAP_Problem}
We start by defining a well-posed problem satisfied by the limiting absorption solution. To state it, we will make use of the trace jump property observed in Theorem \ref{theorem:convergence_decomposition}. The first result states that the jump property ensures the uniqueness of the associated transmission problem. 
\begin{proposition}[Uniqueness]
	\label{theorem:Aplus}
	Assume that $u\in \mathcal{V}_{sing}(\operatorname{div}(x\mathbb{A}\nabla.);\Omega)$ satisfies
	\begin{align}
		\label{eq:alphaAnablaU}
		\begin{split}
			&\operatorname{div}(x\mathbb{A}\nabla u)=0 \text{ in }\Omega,\\
			&[\gamma_0^{\Sigma}u]=-i\pi a_{11}^{-1}\gamma_n^{\Sigma}u.
		\end{split}
	\end{align}
	Then $u=0$.
	%admits a unique solution. Moreover, there exists $C>0$ independent of $f$, s.t. $$\|u\|_{\mathcal{V}_{sing}(\Omega)}+\|\gamma_n^{\Sigma}u\|_{\mathcal{H}^{1/2}(\Sigma)}\leq C\|f\|_{L^2(\Omega)}.$$
\end{proposition}
\begin{remark}
	In the above, we require that $\gamma_n^{\Sigma}u\in \mathcal{H}^{1/2}(\Sigma)$ (cf. the definition of the space \eqref{eq:vsingdiv}), since the notion of $[\gamma_0^{\Sigma}.]$ introduced in Definition \ref{definition:one_sided_trace} relies on the decomposition of Proposition \ref{prop:decomp1}, which we have proven only in the case when $\gamma_n^{\Sigma}u\in \mathcal{H}^{1/2}(\Sigma)$. 
\end{remark}
The results of Theorems \ref{theorem:LAP} and \ref{theorem:main_result} follow from Proposition \ref{theorem:Aplus} and Theorem \ref{theorem:convergence_decomposition} at once. 
\begin{proof}[Proof of Theorem \ref{theorem:LAP}]
	By Theorem \ref{theorem:convergence_decomposition} $(u^{\nu})_{\nu>0}$ admits a subsequence that converges $\mathcal{H}^{1/2-\re}(\Omega)$ strongly, as $\nu\rightarrow 0+$, to some $u^+\in \mathcal{V}_{sing}(\operatorname{div}(x\mathbb{A}\nabla.); \Omega)$.  
	
	Assume now that $(u^{\nu})_{\nu>0}$ admits two  $L^2(\Omega)$-weakly convergent as $\nu\rightarrow 0$ subsequences that converge to different limits $u_1\neq u_2$. In virtue of Theorem \ref{theorem:convergence_decomposition}, we see that $u_j\in \mathcal{V}_{sing}(\operatorname{div}(x\mathbb{A}); \Omega)$ satisfies 
	\begin{align*}
		\operatorname{div}(x\mathbb{A}\nabla u_j)=f, \qquad [\gamma_0^{\Sigma}u_j]=-i\pi a_{11}^{-1}\gamma_n^{\Sigma}u_j,\quad j=1,2.
	\end{align*}
	Therefore, $d=u_1-u_2$ solves the problem \eqref{eq:alphaAnablaU}, and by Proposition \ref{theorem:Aplus}, $d=0$.  This shows that, as $\nu\rightarrow 0$, the whole family $u^{\nu}$ necessarily converges to $u^+$ in $\mathcal{H}^{1/2-\re}(\Omega)$.
\end{proof}
\begin{proof}[Proof of Theorem \ref{theorem:main_result}]
	Follows by the application of Theorem \ref{theorem:LAP},  Theorem \ref{theorem:convergence_decomposition} and the uniqueness result of Proposition \ref{theorem:Aplus}. The stated stability result, cf. Remark \ref{rem:wp}, follows from these results as well. 
\end{proof}
It remains to prove Proposition \ref{theorem:Aplus}. 
%
%Since any $L^2(\Omega)$-weak limiting point of the sequence $(u^{\nu})_{\nu>0}$ satisfies the problem defined in Proposition \ref{theorem:Aplus}, as argued in Theorem \ref{theorem:convergence_decomposition}, we see that such a limiting point must be unique, and this will prove Theorem \ref{theorem:LAP}.
Our proof is inspired by a similar result in \cite{ciarlet_mk_peillon}, where uniqueness was proven for a non-standard variational formulation of \cite{nicolopoulos}. We use a similar approach to prove uniqueness for \eqref{eq:alphaAnablaU}. 
%The key idea is the following observation: solving \eqref{eq:alphaAnablaU} amounts to finding $\gamma_n^{\Sigma}u$, and next employing the result of Theorem \ref{theorem:reg_well_posedness}. Therefore, we are going to establish an isomorphism between the data $f=\operatorname{div}(x\mathbb{A}\nabla u)$ and $\gamma_n^{\Sigma}u$ for $u\in D(\mathcal{A}^+)$. This will be possible due to the special condition on 
%%
%Our proof is based on the third Green's formula, which we will state for the elements of the broken space  
%\begin{align*}
%	&	\mathcal{V}_{sing}(\operatorname{div}(x\mathbb{A}\nabla.); \Omega\setminus\Sigma):=	\mathcal{V}_{sing}(\operatorname{div}(x\mathbb{A}\nabla.); \Omega_n)\times  \mathcal{V}_{sing}(\operatorname{div}(x\mathbb{A}\nabla.); \Omega_p), 
%\end{align*}
%with $\mathcal{V}_{sing}(\operatorname{div}(x\mathbb{A}\nabla.); \Omega_{\lambda})$ defined as follows, cf. \eqref{eq:vsingdiv}, 
%\begin{align*}
%	\mathcal{V}_{sing}(\operatorname{div}(x\mathbb{A}\nabla.); \Omega_{\lambda}):=&	\big\{v\in \mathcal{V}_{sing}(\Omega_{\lambda}): \, \operatorname{div}(x\mathbb{A}\nabla v)\in L^2(\Omega_{\lambda}), \\ &(\gamma_{n}^{\Gamma^+_{\lambda}}+\gamma_n^{\Gamma^{-}_{\lambda}})v=0, \, \gamma_n^{\Sigma,\lambda}v\in \mathcal{H}^{1/2}(\Sigma)\big\},\quad \lambda \in \{n,p\}.
%\end{align*}
The following result follows at once from the Green's formula of Theorem \ref{theorem:green}.
\begin{corollary}
	For all $u, v\in \mathcal{V}_{sing}(\operatorname{div}(x\mathbb{A}\nabla.); \Omega)$, the following integration by parts formula holds true: 
	\begin{align}
		\label{eq:key_green}
		(\operatorname{div}(x\mathbb{A}\nabla u), v)_{L^2(\Omega)}-(u, \operatorname{div}(x\mathbb{A}\nabla v))_{L^2(\Omega)}=-\langle \gamma_n^{\Sigma}u, \overline{[\gamma_0^{\Sigma}v]}\rangle_{L^2(\Sigma)}
		+\overline{\langle \gamma_n^{\Sigma}v, \overline{[\gamma_0^{\Sigma}u]}\rangle}_{L^2(\Sigma)}.
	\end{align}
\end{corollary}
\begin{proof}
	This result follows by a direct computation from Theorem \ref{theorem:green}: we apply it to $\Omega_p$ and $\Omega_n$ and sum up the obtained identities; recall that for the elements $u\in \mathcal{V}_{sing}(\operatorname{div}(x\mathbb{A}\nabla.); \Omega)$, $[\gamma_n^{\Sigma}u]=0$. 
\end{proof}
\begin{proof}[Proof of Proposition \ref{theorem:Aplus}]
	Applying \eqref{eq:key_green} with $v=u$ yields 
	\begin{align*}
		2i\Im	\langle \gamma_n^{\Sigma}u, [\gamma_0^{\Sigma}u]\rangle_{L^2(\Sigma)} =0.
	\end{align*}
	Since $[\gamma_0^{\Sigma}u]=-i\pi a_{11}^{-1}\gamma_n^{\Sigma}u$ the above implies  $
	\int_{\Sigma}|\gamma_n^{\Sigma}u|^2 a_{11}^{-1}=0, $
	hence $\gamma_n^{\Sigma}u=0$. Then $\left.u\right|_{\Omega_{\lambda}}\in \mathcal{V}_{sing}(\Omega_{\lambda})$, $\lambda\in \{n,p\}$, satisfies the homogeneuos Neumann problem \eqref{eq:main_problem} with vanishing data.  
	By Theorem \ref{theorem:fl2}, $u=0$.
\end{proof}
\section{Results for a domain with a hole and extension to more general cases}
\label{sec:general_geometries}
In this section we argue how the results of the previous sections can be extended to different geometries, first for a particular case of geometries with a hole and $\omega=0$, and next comment on the results for $\omega\neq 0$ and a more general class of geometries (see Section \ref{sec:other}).  
%We start with a particular class of geometries (domains with a hole); the reason for this choice will be clear later. 
%\subsection{Domain with a hole, $\omega=0$}
%\label{sec:domain_hole}
\subsection{A domain with a hole, $\omega=0$}
\label{sec:domhole}
We are given $\Omega_{ext}$ and $\Omega_{int}$, two bounded Lipschitz simply connected domains in $\mathbb{R}^2$, s.t. $\overline{\Omega}_{int}\subset\Omega_{ext}$; we assume $\Omega_{int}\neq \emptyset$ (see Remark \ref{rem:assumptionomegaint}). With their help, we define a domain $D:=\Omega_{ext}\setminus \overline{\Omega_{int}}$. Let $I$ be a $C^{3}$-loop inside $D$ without self-intersections, which surrounds $\Omega_{int}$, as shown in Figure \ref{fig:loop}. This loop divides $D$ into two $C^{2,1}$-subdomains $D_p$ and $D_{n}$. Let $I_p, I_n$ be the outer and inner boundaries of $D_p$ and $D_n$ respectively:
\begin{align}
	I_p:=\partial D_p\setminus I, \quad I_n:=\partial D_n\setminus I.
\end{align}
On the loop $I$ we define a unit normal $\vec{n}=\vec{n}_{I}$, which is oriented from $D_n$ into $D_p$. 
\begin{figure}
	\begin{tikzpicture}
		\begin{scope}[scale=1.2]
			\begin{scope}[scale = 1.5]
				\coordinate (A) at (0,1);
				\coordinate (B) at (1,0);
				\coordinate	(C) at (-1,-1);
				\draw[blue,fill=gray!10] (A) .. controls	+(0.7,0) and +(0,1) ..
				(B) .. controls +(0,-1) and +(1,-1) .. (C) .. controls +(-2,2) and
				+(0.2,0) .. (A);
			\end{scope}
			
			\coordinate (A) at (0,1);
			\coordinate (B) at (1,0);
			\coordinate (C) at (-1,-1);
			\draw[red] (A) .. controls +(0.7,0) and +(0,1) .. (B) .. controls +(0,-1)
			and +(1,-1) .. (C) .. controls +(-2,2) and +(0.2,0) .. (A);
			
			\begin{scope}[scale = 0.3]
				\coordinate (A) at (0,1);
				\coordinate (B) at (1,0);
				\coordinate (C) at (-1,-1);
				\draw[blue, fill=white] (A) .. controls +(0.7,0) and +(0,1) ..
				(B) .. controls    +(0,-1) and +(1,-1) .. (C) .. controls +(-2,2) and
				+(0.2,0) .. (A);
			\end{scope}
			
			\coordinate (A) at (0,1);
			\coordinate (B) at (1,0);
			\coordinate (C) at (-1,-1);
			
			\draw[red] (1,1) node {$I$};
			
			\draw[black] (-1.75,-0.9) node {$D_n$};
			\draw[black] (-0.8,-0.5) node {$D_p$};
		\end{scope}
	\end{tikzpicture}
	\caption{Illustration to the geometric configuration of Section \ref{sec:domhole}.} 
	\label{fig:loop}
\end{figure}
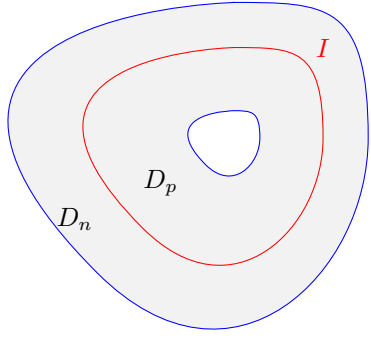

Now that the necessary geometrical preliminaries have been defined, let us introduce the coefficients of the model. In particular, let $\alpha: \, D\rightarrow \mathbb{R}$ satisfy the following assumptions. 
\begin{assumption}
	\label{assump:alpha}
	\begin{itemize}
		\item $\alpha\in C^{1,1}(\overline{D}; \mathbb{R})$
		\item $\alpha>0$ in $\overline{D_p}\setminus I$ and $\alpha<0$ in $\overline{D_n}\setminus I$;
		\item there exists an open neighborhood $U_I$ of $I$, s.t. 
		\begin{align}
			\label{eq:signed_distance}
			\alpha(\vec{x}):=\left\{
			\begin{array}{ll}
				\left|\operatorname{dist}(\vec{x},I)\right|,  & \vec{x}\in D_p\cap U_I,\\
				-\left|\operatorname{dist}(\vec{x},I)\right|, & \vec{x}\in D_{n}\cap U_I.
			\end{array}
			\right.
		\end{align} 
		Let us remark that the above definition is compatible with the regularity of $\alpha$, as argued in \cite{krantz_parks} or \cite[the proof of Lemma 14.6]{gilbrag_trudinger}. 
	\end{itemize}
\end{assumption}
We are given two matrix-valued functions that satisfy the following assumptions. 
\begin{assumption}
	\label{assump:matrices_general}
	\begin{itemize}
		\item $\mathbb{H}, \, \mathbb{N}\in {C}^{1,1}(\overline{D}; \mathbb{C}^{2\times 2})$. 
		\item For all $\bx\in \overline{D}$, $\mathbb{H}(\bx), \mathbb{N}(\bx)$ are both Hermitian, positive-definite matrices. In particular, 
		\begin{align*}
			\text{ for all }\bx\in \overline{D}, \, \vec{p}\in \mathbb{C}^2, \quad 	\mathbb{N}(\bx)\vec{p}\cdot \overline{\vec{p}}\geq c_{\mathbb{N}}|\vec{p}|^2, \quad 		\mathbb{H}(\bx)\vec{p}\cdot\overline{\vec{p}}\geq c_{\mathbb{H}}|\vec{p}|^2, \quad c_{\mathbb{N}}, \, c_{\mathbb{H}}>0.  
		\end{align*}
		\item by $\mathbb{H}_{I}$, resp. $\mathbb{N}_I$, we will denote the values of $\mathbb{H}$, resp. $\mathbb{N}$ on the interface $I$.
	\end{itemize}
\end{assumption}
We do not repeat the rest of the notations, recalling however that $\gamma_{n}^{I}u=\left.\alpha \mathbb{H}\nabla u\right|_{I}\cdot \vec{n}$ and $\gamma_{n,\nu}^{I}u=\left.(\alpha \mathbb{H}+i\nu \mathbb{N})\nabla u\right|_{I}\cdot \vec{n}.$
%Then the results of Section \ref{sec:simplified} hold true for the above problem, modulo adaptation to the domain $D$.
We consider the following BVP. Given $f\in L^2(D)$, let $(v ^{\nu})_{\nu>0}\subset  H^1_0(D)$ solve
\begin{align}
	\label{eq:B32}
	\begin{split}
		&\operatorname{div}((\alpha\mathbb{H}+i\nu\mathbb{N})\nabla v^{\nu})=0 \text { in }D,\\
		&\gamma_0 v^{\nu}=0.
	\end{split}
\end{align}
The following result is proven just like Lemma \ref{lem:pb_abs_wp}. 
\begin{lemma}
	\label{lem:pb_abs_wp_general}
	For each $f\in L^2(\Omega)$, $\nu>0$, the problem \eqref{eq:B32} admits a unique solution $v^{\nu}\in {H}^1_0(D)$, which also belongs to $H^2(D)$. %Additionally, it holds that 
	%	\begin{align}
		%		\nonumber
		%		&\nu^{1/2}\|\nabla v^{\nu}\|\lesssim  \sqrt{\|f\|\|v^{\nu}\|},\\
		%		\label{eq:est_main_general}
		%		&\|v^{\nu}\|_{H^1(D)}\lesssim \nu^{-1}\|f\|.
		%	\end{align}
	%	Also, $v^{\nu}\in {H}^2(D)$ for all $\nu>0$. 
\end{lemma}
To state a generalization of the results of Section \ref{sec:simplified}, let us introduce the following spaces, for $\lambda\in \{n,p\}$,
\begin{alignat*}{2}
	&\mathcal{H}^1_{\delta}(D_{\lambda})&:=&\{v\in L^2(D_{\lambda}): \||\alpha|^{\delta/2}\nabla v\|<\infty, \gamma_0^{I_{\lambda}} v=0\},\\
	&\mathcal{V}_{reg}(D_{\lambda})&:=&\{v\in L^2(D_{\lambda}): \||\alpha|^{1/2}\nabla v\|<\infty, \, \gamma_0^{I_{\lambda}} v=0\}, \\ &\mathcal{V}_{sing}(D_{\lambda})&:=&\{v\in L^2(D_{\lambda}): \|\alpha\nabla v\|<\infty, \gamma_0^{I_{\lambda}} v=0\},
\end{alignat*}	
and set
\begin{align*}
	\mathcal{V}_{reg}(D)=\mathcal{V}_{reg}(D_{n})\times \mathcal{V}_{reg}(D_p), \quad \mathcal{V}_{sing}(D)=\mathcal{V}_{sing}(D_{n})\times \mathcal{V}_{sing}(D_p).
\end{align*}	
Additionally, let 
\begin{align*}
	\mathcal{V}_{sing}(\operatorname{div}(\alpha\mathbb{H}\nabla.); D):=\{v\in \mathcal{V}_{sing}(D): \, \operatorname{div}(\alpha\mathbb{H}\nabla v)\in L^2(D), \, \gamma_n^I v\in H^{1/2}(I)\}.
\end{align*}
Then the following counterparts of the results of Section \ref{sec:simplified} hold true.
\begin{theorem}[Limiting absorption principle]
	\label{theorem:LAP2}
	Given $f\in L^2(D)$, consider the family of solutions to \eqref{eq:B32} $(v^{\nu})_{\nu>0}\subset H^1_0(D)$. Then, as $\nu\rightarrow 0+$, $v^{\nu}\rightarrow v^+$ in  $\bigcap_{\varepsilon>0}H^{1/2-\varepsilon}(D)$, and $v^+\in \mathcal{V}_{sing}(\operatorname{div}(\alpha\mathbb{H}\nabla.); D)$. 
\end{theorem}
Like before, the function $v^+$ as defined in the above theorem will be referred to as a 'limiting absorption solution'.  
Let us state a counterpart of Proposition \ref{prop:decomp1}. For this let us define following function on $I$:
\begin{align}
	\label{eq:Hi}
	h_I:=\left.\vec{n}_I\cdot\mathbb{H}\nabla \alpha\right|_{I}=\vec{n}_I\cdot \mathbb{H}_I\vec{n}_I,
\end{align}
where we used that, since $\alpha$ is a signed distance in the vicinity of the interface, necessarily, \begin{align}\label{eq:alphaprop}\left.\nabla \alpha\right|_{I} =\vec{n}_I,
\end{align} see \cite [proof of Lemma 14.16]{gilbrag_trudinger}. The function $h_I$ plays a role of $a_{11}$ in the previous sections. 
\begin{proposition}
	\label{prop:decomp2}
	Let $v\in \mathcal{V}_{sing}(\operatorname{div}(\alpha\mathbb{H}\nabla.);D)$. 
	Then $u$ can be decomposed in a unique manner as follows: 
	\begin{align}
		\label{eq:decomp_pi0_2}
		v=v_h\log|\alpha|+v_{reg},
	\end{align}
	where $v_h\in H^1_0(D)$ is a piecewise-$\mathbb{H}$-harmonic function that satisfies 
	\begin{align*}
		&\operatorname{div}(\mathbb{H}\nabla v_h)=0 \text{ in }D\setminus I, \\
		&\gamma_0^I v_h=h_{I}^{-1}\gamma_n^I v,\\
		&\gamma_0^{I_p\cup I_n}v_h=0,
	\end{align*}
	and $v_{reg}\in {H}^{1-}(D\setminus I)\bigcap \bigcap_{\varepsilon>0}\mathcal{H}^1_{\re}(D\setminus I)$. 
\end{proposition}
Next, for functions like in Proposition \ref{prop:decomp2}, we denote by $v_{reg,\lambda}=\left. v_{reg}\right|_{D_{\lambda}}$. The regular part of such functions carries the Dirichlet trace. 
\begin{definition}
	\label{def:trace_general}
	Let $v$ be like in Proposition \ref{prop:decomp2}. We define the one-sided trace of $\left. v\right|_{D_{\lambda}}$, $\lambda\in \{n,p\}$, on $I$ as a trace of its regular part:
	\begin{align*}
		\gamma_{0}^{I,\lambda}v:=\gamma_0^{I,\lambda}v_{reg,\lambda}\in H^{1/2-}(I), \quad \lambda\in \{n,p\}.
	\end{align*}
	The jump of the traces is then defined via $
	[\gamma_{0}^{I}v]:=\gamma_0^{I,p}v-\gamma_0^{I,n}v.$
\end{definition}
Like before, the notion of trace enables us to single out the limiting absorption solution among all the functions $v\in \mathcal{V}_{sing}(\operatorname{div}(\alpha\mathbb{H}\nabla.); D)$ that satisfy 
\begin{align}
	\label{eq:pb1_m}
	\begin{split}
		&\operatorname{div}(\alpha\mathbb{H}\nabla v)=f \text{ in }D.
	\end{split}
\end{align}
The following result is the second main result of this section. 
\begin{theorem}
	\label{theorem:main_result2}
	Given $f\in L^2(D)$, the limiting absorption solution $v^+$ as defined in Theorem \ref{theorem:LAP2} is a unique solution to the following well-posed problem: find $v\in \mathcal{V}_{sing}(\operatorname{div}(\alpha\mathbb{H}\nabla .); D)$ that satisfies \eqref{eq:pb1_m} and 
	\begin{align}
		\label{eq:traces2}
		[\gamma_0^{I}v]=-i\pi h_{I}^{-1}\gamma_n^{I}v. 
	\end{align}
\end{theorem}
The proofs of the statements of Propositions \ref{prop:decomp1}, Theorem \ref{theorem:LAP}, \ref{theorem:main_result} in the domain $\Omega$ can be quite easily adapted to a more general geometric setting. An argument on why this can be done is the following. 

First of all, our previous $L^2(\Omega)$-stability estimates on $u^{\nu}$ are based on the integration by parts arguments, well-posedness and regularity results for the homogeneous decoupled Neumann problem and quite delicate lifting lemmas, cf. the proof of Theorem \ref{theorem:stability_estimate}. The lifting lemmas are easily adapted by a change of coordinates to the curve $I$; well-posedness and regularity for the homogeneous Neumann problem are taken from \cite{baouendi_goulaouic}. This enables us to obtain the estimate  $\|v^{\nu}\|_{\mathcal{V}_{sing}(D)}\lesssim \|f\|$.

Once this estimate is obtained, we localize the problem in the vicinity of the interface $I$ with the help of the cut-off function $\chi$ to $\operatorname{div}(\alpha\nabla (\chi v))=f_{\chi}$, with $\|f_{\chi}\|\lesssim \|v\|_{\mathcal{V}_{sing}}+\|f\|\lesssim \|f\|$, and next perform an appropriate change of coordinates to get back to the problem of Section \ref{sec:simpl}, and repeat all the regularity derivations for this problem, including decompositions of the functions from $\mathcal{V}_{sing}(\operatorname{div}(x\mathbb{A}\nabla.); \Omega)$ into a regular and a singular parts.  Since, compared to the results of the previous sections, only technical modifications are necessary, all the required derivations and modifications are moved to Appendix \ref{appendix:proof_general}. 
\begin{remark}
	\label{rem:assumptionomegaint}
	The assumption on $\Omega_{int}\neq \emptyset$ is important, since it ensures that the homogeneous Neumann problem for $\operatorname{div}(\alpha\mathbb{N}\nabla u)=f$ in $D_p$ is well-posed, due to the additional Dirichlet boundary condition on $I_p$.
\end{remark}
\subsection{Case $\omega\neq 0$, other types of domains}
\label{sec:other}
The results of the previous section translate verbatim to the case when one replaces $\operatorname{div}((\alpha\mathbb{H}+i\nu \mathbb{N})\nabla v^{\nu})=f$ by   $\operatorname{div}((\alpha\mathbb{H}+i\nu \mathbb{N})\nabla v^{\nu})+\omega^2 v^{\nu}=f$ in \eqref{eq:B32}, and the domain $D_I$ by a simply connected domain, provided that $\omega$ satisfies the following assumption.
\begin{assumption}
	\label{remark:assumption}
	Assume that $\omega\in \mathbb{R}$ is s.t. the following problem is well-posed, for $\lambda\in \{n,p\}$. Given $f\in L^2(D_{\lambda})$, find $u_{\lambda}\in \mathcal{V}_{reg}(D_{\lambda})$ s.t. $\operatorname{div}(\alpha\mathbb{A}\nabla u_{\lambda})+\omega^2 u_{\lambda}=f \text{ in }D_{\lambda},$ and $\gamma_0^{\partial D_{\lambda}\setminus I}u_{\lambda}=0$. 
\end{assumption}
Because the space $\mathcal{V}_{reg}(D_{\lambda})$ is compactly embedded into $L^2(D_{\lambda})$ (this follows e.g. from Lemma \ref{lem:embedding} and compact embedding of $H^s$ into $L^2$, $s>0$), and the problem can be written in a self-adjoint form, there exists only a discrete number of frequencies, for which it is not well-posed. See \cite{baouendi_goulaouic} for more details.

The above assumption ensures in particular that Proposition \ref{prop:decomp2} holds true provided a non-vanishing zero order term, since the  corresponding problem with the frequency $\omega\neq 0$ admits a unique solution, see Remark \ref{rem:assumptionomegaint}.

 %Theorem \ref{theorem:sp_wp_general}. The rest of the derivations are the same, since Theorem \ref{theorem:stability_estimate_general} can be proven by relying on the result of Theorem \ref{theorem:sp_wp_general}. Once this result is proven, we can rewrite the problem  $\operatorname{div}((\alpha\mathbb{H}+i\nu\mathbb{N})u^{\nu})=f^{\nu}$, with $f^{\nu}=f-\omega^2 u^{\nu}$ bounded in $L^2$, and reduce all computations to the previous case, see Remark \ref{rem:stab_results}. %It also enables us to prove Theorem \ref{theorem:stability_estimate} (see Section \ref{prop:d2} and Section \ref{sec:thmstab}). 
%\subsubsection{Other types of domains, $\omega\neq 0$ }
%As deduced from Remark \ref{rem:assumptionomegaint} and Remark \ref{remark:assumption}, when we consider simply connected $D$ with a loop $I$ inside the domain $D$, but $\omega\neq 0$, provided that it satisfies Remark \ref{remark:assumption}, the result of Theorem \ref{theorem:sp_wp_general} holds true, and thus we can extend all the results of previous sections to this case. 
%
%An interesting open question is what happens when the loop $I$ touches the boundary of the domain $D$, or when $I$ does not have a desired regularity (e.g. it is a Lipschitz curve). 

\section*{Acknowledgements}
The authors are grateful to Patrick Ciarlet (ENSTA Paris, France) and Lyonell Boulton (Heriot-Watt University, UK) for many fruitful discussions. We thank as well to Daniel Grieser (University of Oldenburg) for his interest in this work and for pointing out a possibility of use of the Melrose' $b-$calculus techniques.

	\renewcommand{\appendix}{\par
		\setcounter{section}{0}
		\setcounter{subsection}{0}
		\gdef\thesection{\Alph{section}}
	}

	\appendix
	\numberwithin{theorem}{section}

\renewcommand{\appendix}{\par
	\setcounter{section}{0}
	\setcounter{subsection}{0}
	\gdef\thesection{\Alph{section}}
}
\numberwithin{theorem}{section}
\appendix
\section{Derivation of the model \eqref{eq:B3}}
\label{appendix:derivation}
The following presentation is due to Stix \cite{stix1992waves}, and Freidberg \cite[Sections 15.5, 15.6]{freidberg2008plasma}. 
Time-harmonic electromagnetic wave propagation in one-species cold magnetized plasma under the magnetic field $\boldsymbol{B}_0=(0,0,B_0)$ is described by the Maxwell's equations (with an appropriate rescaling, so that $\varepsilon_0=\mu_0=c=1$)
\begin{align*}
&	\operatorname{\textbf{curl}}\vec{E}-i\omega \vec{B}=0,\\
&	\operatorname{\textbf{curl}}\vec{B}+i\omega \uuline{\varepsilon}(\omega)\vec{E}=0,\qquad \omega>0, 
\end{align*}
with the cold plasma dielectric tensor 
\begin{align*}
	\uuline{\varepsilon}=\left(
	\begin{matrix}
	S & -iD & 0 \\
	iD & S & 0\\
	0 & 0 & P
	\end{matrix}
\right), \text{ where }
S=1-\frac{\omega_p^2}{\omega^2-\omega_c^2}, \quad D=\frac{\omega_c}{\omega}\frac{\omega_p^2}{\omega^2-\omega_c^2}, \; P=1-\frac{\omega_p^2}{\omega^2}.
\end{align*}
Here $\omega_c=c_cB_0$ is an algebraic cyclotron frequency, while $\omega_p=c_p\mathcal{N}$ is a plasma frequency;  the constants $c_c, \, c_p\in \mathbb{R}$ depend only on the nature of the particles the plasma is comprised of, while $\mathcal{N}=\mathcal{N}(\vec{x})$ is the plasma density. Variations in the plasma density in space $\vec{x}$ can lead to various degeneracies of the above model. 

The above PDE is posed in a bounded smooth domain $\mathcal{D}_{3D}$, and equipped with appropriate boundary conditions. One could think of $\mathcal{D}_{3D}$ being a tore (tokamak).

The model presented above does not take into account damping due to collisions, which occur with the collision frequency $\nu>0$ (i.e. we took $\nu=0$ to derive the above); to account for them, it suffices to replace in $\uuline{\varepsilon}$ the following quantities, see \cite[p.38]{stix1992waves},
\begin{align*}
	\omega_p^2\rightarrow \omega_p^2 \frac{\omega}{\omega_{\nu}}, \quad 	\omega_c\rightarrow \omega_c \frac{\omega}{\omega_{\nu}}, \qquad  \omega_{\nu}:=\omega+i\nu.
\end{align*}
We will denote the corresponding tensor by $\uuline{\varepsilon}^{\nu}$ and use the index $\nu$ in $S^{\nu}$, $D^{\nu}$, $P^{\nu}$.

We are interested in the situation when the fields do not depend on the $z$-variable. The field $(E_1, \, E_2, \, B_3)$ then satisfies the 2D Maxwell equations in a domain $\mathcal{D}$  (e.g. a cut of $\mathcal{D}_{3D}$ with a plane $z=0$):
\begin{align*}
	&\operatorname{curl}\mathbf{E}_{\perp}-i\omega B_3=0,\\
	&\operatorname{\mathbf{curl}}_{\perp}B_3+i\omega \uuline{\varepsilon_{\perp}}\mathbf{E}_{\perp}=0,
\end{align*}
with 
\begin{align*}
	\uuline{\varepsilon_{\perp}}=\left(
\begin{matrix}
	S & -iD \\
	iD & S 
\end{matrix}
\right).
\end{align*}
As usual in the 2D case, we can rewrite the above in a more convenient, scalar, form
\begin{align}
	\label{eq:divB}
	\operatorname{div}_{\perp}(	\uuline{\varepsilon_{\perp}}^{-1}\nabla B_3)+\omega^2 B_3=0,
\end{align}
with 
\begin{align}
	\label{eq:uualpha}
\uuline{\alpha}:=\uuline{\varepsilon_{\perp}}^{-1}=\frac{1}{S^2-D^2}\left(
	\begin{matrix}
		S & iD \\
		-iD & S
	\end{matrix}
\right).
\end{align}
In what follows, we will assume that $\uuline{\alpha}$ is well-defined everywhere in the domain $\mathcal{D}$.
\begin{assumption}[$\uuline{\alpha}$ is well-defined]
	\label{assump:SD}
It holds that $S^2\neq D^2 \text{ in }\mathcal{D}$. 
\end{assumption}
Recall now that all the coefficients in the above depend on frequency, as well as on the spatial variable $\vec{x}$, through the plasma density. 
An interesting situation occurs when $S=S(\vec{x})=0$ on a given curve $I\subset \mathbb{R}^3$ and $D=D(\vec{x})\neq 0$. 
Because the dependence on $\vec{x}$ manifests through the plasma frequency $\omega_p$ only, the level sets of $S$ coincide with the level sets of $D$; let us denote $\left. D\right|_{I}:=D_I=\operatorname{const}$. Let us remark at this point that $D=\frac{\omega_c}{\omega}(1-S)=D_I(1-S)$, and it holds that 
\begin{align}
	\label{eq:SD}
	S^{-1}(D-D_{I})=-\frac{\omega_c}{\omega}=-D_I.
\end{align}
The above rewriting allows to decompose
\begin{align*}
&\uuline{\alpha}=\frac{1}{S^2-D^2}
\left(
\begin{matrix}
S & iD_I-iD_IS\\
	-iD_I+iD_IS & S
\end{matrix}
\right)=
-\frac{1}{D_I^2}\left(
\begin{matrix}
	0 & iD_I\\
	-iD_I & 0
\end{matrix}
\right)+\uuline{r}, \\ &\uuline{r}:=\frac{S}{S^2-D^2}\left(
\begin{matrix}  
1 & i\frac{S+DD_I}{D_I}\\
-i\frac{S+DD_I}{D_I} & 1
\end{matrix}
\right),
\end{align*}
where we used, cf \eqref{eq:SD}, 
\begin{align*}
&\frac{1}{S^2-D^2}+\frac{1}{D_I^2}=	\frac{D_I^2+S^2-D^2}{(S^2-D^2)D_I^2}=\frac{S(S+D_I(D+D_I))}{(S^2-D^2)D_I^2}, \quad \text{ and  }\\
&\frac{S+D_I(D+D_I)}{D_I}-D_I=\frac{S+D_I D}{D_I}.
\end{align*}
Due to the anti-symmetry of the first term in the above expansion of $\uuline{\alpha}$, and the fact that this term is constant, we have that $\operatorname{div}(\uuline{\alpha}\nabla u)=\operatorname{div}(\uuline{r}\nabla u)$, which allows to rewrite  \eqref{eq:divB} as follows: 
\begin{align}
	\label{eq:divB2}
	\operatorname{div}\left(\uuline{r}\nabla B_3\right)+\omega^2 B_3=0.
\end{align}
Up to now, the derivation had been rather formal. We aim at giving a more precise meaning to the above expression. Indeed, in the vicinity of $I$, where $S(\vec{x})$ vanishes and $D=D_I=\operatorname{const}$, we have that 
\begin{align*}
	\uuline{r}=-\frac{S(\vec{x})}{D^2_I}\left(
	\begin{matrix}  
		1 & iD_I\\
		-iD_I& 1
	\end{matrix}
	\right)+O(S^2(\vec{x})),
\end{align*}
which shows that the equation \eqref{eq:divB2} has a degenerate coefficient in the principal part of the operator.

If we want to exclude a possible difficulty of different signs of eigenvalues of $S^{-1}\uuline{r}$, we can make the following assumption.
\begin{assumption}[Piecewise-ellipticity of the degenerate coefficient]
	\label{assumption:positive_definite_coplex}
	The matrix $S^{-1}\uuline{r}$ is hermitian, sign-definite. In other words, $\uuline{r}(S^2-D^2)$ is positive definite (its trace is $>0$), i.e
	\begin{align*}
		&1-\left(\frac{S+DD_I}{D_I}\right)^2>0 \text{ in }\mathcal{D}.
	\end{align*}
The latter condition is a condition on a relative variation of $D$ and $S$. 
It rewrites 
\begin{align*}
	(D_I-S-D_ID)(D_I+S+D_ID)>0.
\end{align*}
When evaluated on $I$, it requires that $(D_I-D_I^2)(D_I+D_I^2)>0\iff (D_I^2-D_I^4)>0\iff |D_I|<1$.

\end{assumption}
To ensure the above, we can impose a stricter assumption:
\begin{align*}
	D_I(1-D)-S>0 \text{ and }D_I(1+D)+S>0. 
\end{align*}
On the other hand, recalling that $D=D_I(1-S)$, cf. \eqref{eq:SD}, this implies the following restriction on $S$:
\begin{align*}
	S<\frac{D_I(D_I-1)}{(D_I^2-1)}<\frac{D_I}{D_I+1}, \quad S>-\frac{D_I+D_I^2}{1-D_I^2}=\frac{D_I}{D_I-1}.
\end{align*}
In other words, we can impose the following sufficient conditions on the coefficients.
\begin{assumption}[Simplification of Assumption \ref{assumption:positive_definite_coplex}]
	\label{assumption:positive_definite}
	We assume that 
	\begin{align}
	0<D_I<1 \text{ and }-\frac{D_I}{1-D_I}<S<\frac{D_I}{1+D_I},
	\end{align}
which implies that $\uuline{r}(\vec{x})=S(\vec{x})\uuline{a}(\vec{x})$, with $(S^2-D^2)\uuline{a}(\vec{x})$ being a Hermitian positive definite matrix. Remark that actually this assumption implies that $\uuline{a}$ is a Hermitian negative definite matrix. 

Indeed, for $D_I, S$ as above $S\neq \pm D$; this holds because $S=\pm D \iff S=\pm D_I(1-S)\iff S=\pm\frac{D_I}{1\pm D_I}$, which is excluded by our assumption. Therefore, $\operatorname{sign}(S^2-D^2)=-\operatorname{sign}D_I^2<0$ on $\mathcal{D}$. 
\end{assumption}

Next, we formalize \eqref{eq:divB2} more, by adding a source term $F$ and equipping it with appropriate boundary conditions: 
\begin{align}
		\label{eq:divB32}
		\operatorname{div}\left(\uuline{r}\nabla B_3\right)\equiv \operatorname{div}\left(\uuline{\alpha}\nabla B_3\right)=-\omega^2 B_3+F,\qquad \text{ in }\mathcal{D},\qquad 
		\gamma_0 B_3=0.
\end{align}
To give a meaning to a 'physical' solution to the above, which is far from being evident due to the degeneracy in $\uuline{r}$, one can pursue at least two approaches: 
\begin{itemize}
	\item  consider a non-vanishing collision frequency, which results in $\uuline{r}$ replaced by $\uuline{r}^{\nu}$ (resp. $\uuline{\alpha}$ by $\uuline{\alpha}^{\nu}$); hoping that the resulting problem is well-posed, denote the solution by $B_3^{\nu}$,  and next study the limit, if exists, as $\nu\rightarrow 0+$, of $B_3^{\nu}$. 
	\item consider $\nu=0$, add absorption to the frequency $\omega$ (i.e. replace $\omega\rightarrow \omega+i\eta$), and next study the limit of the obtained solution, which we, with an abuse of notation, again denote by $B_3^{\eta}$, as  $\eta\rightarrow 0+$. 
\end{itemize}
In general, it is unclear whether these two limits commute. We start by considering the first approach. Without going to actual numbers, we start by studying a general case when $S$ and $D$ are perturbed with a small absorption parameter, and the  tensor $\uuline{\alpha}$ is replaced by its perturbation.

\textit{General computations. }Given a small parameter $\mu>0$ and two expansions
\begin{align*}
	{S}^{\mu}(\vec{x}):=S(\vec{x})+i\mu\delta_S(\vec{x})+O_{L^{\infty}(\mathcal{D})}(\mu^2), \quad {D}^{\mu}(\vec{x}):=D(\vec{x})+i\mu \delta_D(\vec{x}) +O_{L^{\infty}(\mathcal{D})}(\mu^2),
\end{align*}
we can rewrite the perturbed tensor 
\begin{align}
\label{eq:perturbed_tensor}
	\uuline{\alpha}^{\mu}&=\frac{1}{(S^{\mu})^2-(D^{\mu})^2}\left(
	\begin{matrix}
		S^{\mu} & iD^{\mu}\\
		-iD^{\mu} & S^{\mu}
	\end{matrix}
\right)=\uuline{\alpha}+i\mu \uuline{t}+O(\mu^2),\\
\nonumber
\uuline{t}&:=-\frac{2S\delta_S-2D\delta_D}{(S^2-D^2)^2}\left(\begin{matrix}
	S & iD\\
	-iD & S
\end{matrix}\right)+\frac{1}{S^2-D^2}\left(
\begin{matrix}
	\delta_S & i\delta_D\\
	-i\delta_D & \delta_S
\end{matrix}
\right)\\
\label{eq:t_matrix_absorption}
&=\frac{1}{(S^2-D^2)^2}\left(
\begin{matrix}
2DS\delta_D-\delta_S(S^2+D^2) & i(\delta_D(S^2+D^2)-2SD\delta_S)\\
-i(\delta_D(S^2+D^2)-2SD\delta_S)&	2DS\delta_D-\delta_S(S^2+D^2)
\end{matrix}
\right).
\end{align}
In particular, on the interface $\vec{x}\in I$, 
\begin{align}
	\label{eq:tensorReg}
	\uuline{t}(\vec{x})&=\frac{1}{D^2_I}\left(
	\begin{matrix}
		-\delta_S(\vec{x}) & i\delta_D(\vec{x})\\
		-i\delta_D(\vec{x}) & -\delta_S(\vec{x})
	\end{matrix}
	\right).
\end{align}
The tensor $\uuline{t}$ will be responsible for a 'regularization' of the problem. Indeed, \eqref{eq:divB32} with a perturbed tensor $\uuline{\alpha}^{\mu}$ (where we use the same argument as before to replace $\uuline{\alpha}$ by $\uuline{r}$) rewrites  
\begin{align}
	\label{eq:divB3}
\operatorname{div}(\uuline{\alpha}^{\mu}\nabla \widetilde{B}_3^{\mu})=	\operatorname{div}((\uuline{r}+i\mu\uuline{t}+O(\mu^2))\nabla \widetilde{B}_3^{\mu})=-\omega^2 \widetilde{B}_3^{\mu}+F.
\end{align}
Because $\uuline{r}$ vanishes on $I$, we would like $\uuline{t}$ be sign-definite in the vicinity of $I$; in particular, this requires that 
\begin{align}
	\label{eq:sign_definite}
	\delta_S^2-\delta_D^2>0 \text{ on }I.
\end{align}
On the other hand, if $S$, $D$, $\delta_S$, $\delta_D$ do not vary drastically in the domain, $\uuline{t}$ remains sign-definite. Thus, a generalization of \eqref{eq:sign_definite} reads. 
\begin{assumption}[Ellipticity of $\uuline{t}$]
	\label{assump:coefficients_positivity}
The coefficients $S, \, \delta_S, \, D, \delta_D$ satisfy in $\mathcal{D}$: 
	\begin{align*}
\left(2DS\delta_D-\delta_S(S^2+D^2)\right)^2-(\delta_D(S^2+D^2)-2SD\delta_S)^2>0.
	\end{align*}
The above is equivalent to
\begin{align*}
	-(\delta_S+\delta_D)(S-D)^2(\delta_D-\delta_S)(S+D)^2>0 \iff \delta_S^2-\delta_D^2>0 \text{ and }S\neq \pm D\text{ in }\mathcal{D}. 
\end{align*}
The above implies that $\uuline{t}$ is a strictly positive (or negative)-definite matrix in $\mathcal{D}$. 
\end{assumption}
\paragraph{Collision limit}
When $\nu>0$ is indeed a collision frequency, we can compute the perturbations of the coefficients $S, \, D$ explicitly; in particular, 
	\begin{align*}
S^{\nu}&=1-\frac{\omega_p^2\omega_{\nu}}{\omega(\omega_{\nu}^2-\omega_c^2)}=S+i\nu\frac{(\omega^2+\omega_c^2)\omega_p^2}{\omega(\omega^2-\omega_c^2)^2}+O_{L^{\infty}(\mathcal{D})}(\nu^2),\\
D^{\nu}&=\frac{\omega_c\omega_p^2}{\omega(\omega_{\nu}^2-\omega_c^2)}=D-i\nu\frac{2\omega_c\omega_p^2}{(\omega^2-\omega_c^2)^2}+O_{L^{\infty}(\mathcal{D})}(\nu^2).
	\end{align*}
We then have 
\begin{align}
	\label{eq:deltas}
	\delta_S=\frac{\omega^2+\omega_c^2}{\omega(\omega^2-\omega_c^2)}\left(1-S\right), \quad \delta_D=-\frac{2\omega_c}{(\omega^2-\omega_c^2)}\left(1-S\right). 
\end{align}
For Assumption \ref{assump:coefficients_positivity} to be fulfilled, we want that 
\begin{align*}
	|\delta_S|>|\delta_D|,
\end{align*}
and as $S<1$ by Assumption \ref{assumption:positive_definite}, the above is equivalent to requiring that $|\omega|\neq |\omega_c|$ and $\omega^2+\omega_c^2>2|\omega|\omega_c$, which is always true.

%
%
%
%
%&=\frac{\delta_D}{S^2-D^2}\left(-\frac{2\omega_c S}{\omega(S^2-D^2)}\left(\begin{matrix}
%	1& -i\frac{\omega_c}{\omega}\\
%	i\frac{\omega_c}{\omega} & 1
%\end{matrix}\right)+\left(\begin{matrix}
%	\frac{\omega_c}{\omega} & i-\frac{2i\omega_c^2}{\omega^2(S^2-D^2)}\\
%	-i+\frac{2i\omega_c^2}{\omega^2(S^2-D^2)} & \frac{\omega_c}{\omega} 
%\end{matrix}\right)\right).
With these new expressions, we have the following important result. 
\begin{corollary}
	\label{cor:expansion}
Assume that $\omega>\omega_c>0$, and the following holds true:
	\begin{align}
		\label{eq:assump_var}
	 -\frac{D_I}{1-D_I}<S<\frac{D_I}{1+D_I}.
	\end{align}
Then, for all sufficiently regular $u$, it holds that  
\begin{align*}
	\operatorname{div}(\uuline{\alpha}^{\nu}\nabla u)=	\operatorname{div}(\left(S(\vec{x})\uuline{a}+i\nu\uuline{t}+\nu^2\uuline{c}^{\nu}\right)\nabla u), 
\end{align*}
where 
\begin{align*}
	\uuline{a}=\frac{1}{S^2-D^2}\left(
	\begin{matrix}
		1 & i\frac{S+DD_I}{D_I}\\
		-i\frac{S+DD_I}{D_I}& 1	
	\end{matrix}
	\right)
\end{align*}
is a Hermitian, negative definite in $\mathcal{D}$ matrix; the matrix $\uuline{t}$ is given by \eqref{eq:t_matrix_absorption} with the coefficients defined in \eqref{eq:deltas} and is Hermitian, negative definite in $\mathcal{D}$.

Finally, for sufficiently regular $S, \, D$, the matrix $\uuline{c}^{\nu}(\vec{x})$ is bounded in $\mathcal{D}$ uniformly in $\nu$, in other words, $\sup\limits_{0\leq \nu<1}\|\uuline{c}^{\nu}\|_{L^{\infty}}<\operatorname{const}$.
\end{corollary}
\begin{proof}
	We use the expansion \eqref{eq:perturbed_tensor}, and the same argument as before to pass from \eqref{eq:divB} to \eqref{eq:divB2}; we have $\uuline{r}=S\uuline{a}$. Remark that $\omega>\omega_c>0$ implies that $0<D_I=\frac{\omega_c}{\omega}<1$, cf. \eqref{eq:SD}. Then 
	 the assumption \eqref{eq:assump_var} of the corollary is exactly Assumption \ref{assumption:positive_definite}, which ensures that $\uuline{a}$ is a Hermitian negative definite matrix. 
	
	Next, let us show that \eqref{eq:assump_var} implies Assumption \ref{assump:coefficients_positivity}, which, in turn ensures, that $\uuline{t}$ is a strictly positive- (or negative-)definite matrix in $\mathcal{D}$. As argued before the statement of the corollary, since Assumption \ref{assumption:positive_definite} holds true, we have that 
	\begin{align*}
		\delta_S^2-\delta_D^2>0 \text{ and }|S|< |D| \text{ in }\mathcal{D}.
	\end{align*}

	%As for the first part, we use explicit expressions \eqref{eq:deltas}, which allow us to write the equivalence, see also \eqref{eq:SD} for the explicit expression of $D_I$:
%	\begin{align*}
%		\delta_S^2-\delta_D^2>0 \iff (\omega^2+\omega_c^2)^2-4\omega_c^2>0 \text{ and }S\neq 1 \iff |D_I|=\left|\frac{\omega_c}{\omega}\right|<1 \text{ and }S\neq 1.
%	\end{align*}
Thus, $\uuline{t}$ is (strictly) sign-definite in $\mathcal{D}$.  As its both eigenvalues do not change their signs in $D^{\nu}$, it suffices to examine what happens on $I$. Using the expression \eqref{eq:tensorReg}, we see that $\operatorname{sign}\operatorname{Tr}\uuline{t}=-\operatorname{sign}\left.\delta_S\right|_{I}=-\operatorname{sign}\frac{(1-S)}{1-\omega^{-2}\omega_c^2}=-\operatorname{sign}\frac{1}{1-D_I^2}<0,$ where we remarked that $S<1$ by \eqref{eq:assump_var}, recalled \eqref{eq:SD}, namely $D_I=\omega_c/\omega$ and used $\omega>\omega_c$.

	Finally, the uniform bound on $\uuline{c}^{\nu}$ follows from the regularity of $\uuline{\varepsilon}^{\nu}(\vec{x})$, both in $\nu$ and $\vec{x}$. 
\end{proof}
We are thus interested in the following boundary-value problem: 
\begin{align*}
	\operatorname{div}\left(\left(S(\vec{x})(-\uuline{a})+i\nu(-\uuline{t})\right)\nabla B_3^{\nu}\right)-\omega^2 B_3^{\nu}=F, \quad \text{ in }\mathcal{D}, \quad \gamma_0 B_3^{\nu}=0.
\end{align*}
We omit the term $O(\nu^2)$ compared \eqref{eq:divB3}, since, as we will see later in the course of the article, for $\nu\rightarrow 0$, it does not seem to play a role in any of the conclusions of the paper. 
\paragraph{Limiting absorption limit}
When $\nu=\eta$ is an absorption added to the frequency, we can compute the perturbations of the coefficients $S, \, D$ explicitly; in particular, 
\begin{align*}
	S^{\eta}&=1-\frac{\omega_p^2}{(\omega_{\eta}^2-\omega_c^2)}=S+\frac{2\omega_p^2\omega}{(\omega^2-\omega_c^2)^2}i\eta+O_{L^{\infty}(\mathcal{D})}(\eta^2),\\
	D^{\eta}&=\frac{\omega_c\omega_p^2}{\omega_{\eta}(\omega_{\eta}^2-\omega_c^2)}=D-i\eta \frac{\omega_c\omega_p^2(\omega^2+\omega_c^2)}{\omega^2(\omega^2-\omega_c^2)^2}
+O_{L^{\infty}(\mathcal{D})}(\eta^2).
\end{align*}
We then have 
\begin{align}
	\label{eq:deltaseta}
	\delta_S=\frac{2\omega}{\omega^2-\omega_c^2}\left(1-S\right), \quad \delta_D=-\frac{\omega_c(\omega^2+\omega_c^2)}{\omega^2(\omega^2-\omega_c^2)}\left(1-S\right). 
\end{align}
For Assumption \ref{assump:coefficients_positivity} to be fulfilled, we want that 
\begin{align*}
	|\delta_S|>|\delta_D|,
\end{align*}
and as $S<1$ by Assumption \ref{assumption:positive_definite}, the above is equivalent to requiring that, for $\omega\neq \pm \omega_c$,  $\omega\neq 0$,
\begin{align*}
2|\omega|^3>|\omega_c|(\omega^2+\omega_c^2),
\end{align*}
and for the above to hold true it is sufficient that 
\begin{align*}
	|\omega|>|\omega_c|. 
\end{align*}
We then have the following result, proven like Corollary \ref{cor:expansion}.
\begin{corollary}
	Assume that $\omega>\omega_c>0$, and the following holds true:
	\begin{align}
		\label{eq:assump_var2}
		 -\frac{D_I}{1-D_I}<S<\frac{D_I}{1+D_I}.
	\end{align}
	Then, for all sufficiently regular $u$, it holds that  
	\begin{align*}
		\operatorname{div}(\uuline{\alpha}^{\eta}\nabla u)=	\operatorname{div}(\left(S(\vec{x})\uuline{a}+i\eta\uuline{t}+\eta^2\uuline{c}^{\nu}\right)\nabla u), 
	\end{align*}
	where 
	\begin{align*}
		\uuline{a}=\frac{1}{S^2-D^2}\left(
		\begin{matrix}
			1 & i\frac{S+DD_I}{D_I}\\
			-i\frac{S+DD_I}{D_I}& 1	
		\end{matrix}
		\right)
	\end{align*}
	is a Hermitian, negative definite in $\mathcal{D}$ matrix; the matrix $\uuline{t}$ is given by \eqref{eq:t_matrix_absorption} with the coefficients defined in \eqref{eq:deltaseta} and is Hermitian, negative definite in $\mathcal{D}$. Finally, for sufficiently regular $S, \, D$, the matrix $\uuline{c}^{\eta}(\vec{x})$ is bounded in $\mathcal{D}$ uniformly in $\nu$, in other words, $\sup\limits_{0\leq \nu<1}\|\uuline{c}^{\eta}\|_{L^{\infty}}<\operatorname{const}$.
\end{corollary}

	\section{Properties of weighted spaces}
\label{appendix:weighted}
We will state the results for $\Omega_p$ only; they extend in a trivial manner to $\Omega_n$. 
\subsection{Basic properties and inequalities}
Let us introduce an auxiliary Hilbert space
\begin{align*} 
			\mathcal{H}^1_{\nu,\delta}(\Omega_p)=\{v\in L^2_{loc}(\Omega_p): \, \int_{\Omega_p}(x^{\nu}|u|^2+x^{\delta}|\nabla u|^2)d\vec{x}<\infty, \quad \gamma_0^{\Gamma_p}u=0, \quad \gamma_0^{\Gamma_p^+}u=\gamma_0^{\Gamma_p^{-}}u\},
\end{align*}
so that $\mathcal{H}^1_{\delta}(\Omega_p)=\mathcal{H}^1_{0,\delta}(\Omega_p)$. We start with the Poincar\'e inequality. 
\begin{proposition}%[The Poincar\'e inequality]
	\label{prop:poincare_appendix}
	Let $0\leq \delta\leq 2$. There exists $C_{\delta}>0$, s.t. for all $u\in \mathcal{H}^1_{\delta,\delta}(\Omega_p)$, 
	it holds that 
	\begin{align*}
		\|u\|_{L^2(\Omega_p)}^2\leq C_{\delta}	\int_{\Omega_p}x^{\delta}|\nabla u|^2d\vec{x}.
	\end{align*}
\end{proposition}
\begin{proof}
	By repeating the argument of \cite[Theorem 7.2]{kufner}, $\mathcal{C}^{\infty}(\overline{\Omega}_p)$ functions are dense in $\mathcal{H}^1_{\delta,\delta}(\Omega_p)$, hence it suffices to prove the desired result for $u\in \mathcal{C}^{\infty}(\overline{\Omega}_p)$. Without loss of generality, we assume that $u$ is real-valued. By integration by parts we have
	\begin{align*}
		u^2(x,y)=-\int_{x}^a 2\partial_{x'}u(x',y)u(x',y)dx', \quad (x,y)\in \Omega_p,
	\end{align*}
	where we used $u^2(a, y)=0$. The above yields 
	\begin{align*}
		\int_0^a u^2(x,y)dx&=-2\int_0^a\int_x^a \partial_{x'} u(x',y)u(x',y)dx' dx%\\
	%	&=-2\int_0^a \partial_{x'}u(x',y)u(x',y) \int_0^{x'}dx\, dx'
	=-2\int_0^a x\partial_x u(x,y)u(x,y)dx.
	\end{align*}
	Using the Young inequality implies that, for all $\varepsilon>0$,  there exists $C_{\varepsilon}>0$, s.t. 
	\begin{align*}
		\int_0^a u^2(x,y)dx\leq C_{\varepsilon}\int_0^a x^2(\partial_x u)^2 dx +\varepsilon\int_0^a u^2 dx. 
	\end{align*}
	Integrating the above in $y\in (-\ell,\ell)$ and using $x^2\lesssim x^{\delta}$, $0\leq\delta\leq 2$,  yields the desired statement. 
\end{proof}
With the above, we obtain
\begin{proposition}
	\label{prop:density_smooth}
	Let $0\leq \delta\leq 2$. Then ${\mathcal{H}}^1_{\delta}(\Omega_p)=\overline{\mathcal{C}^{\infty}(\overline{\Omega}_p)}^{\|.\|_{{\mathcal{H}}^1_{\delta}(\Omega_p)}}$.% space $\mathcal{H}^1_{\delta}(\Omega_p)$ can be alternatively defined as $\{v\in L^2(\Omega_p): \, \|v\|_{\mathcal{H}^1_{w,\delta}(\Omega_p)}<\infty, \, \gamma_0^{\Gamma_p}v=0, \, \gamma_0^{\Gamma_p^+}v-\gamma_0^{\Gamma_p^{-}}v=0\}$. 
\end{proposition}
\begin{proof}
		This result follows from \cite[Theorem 7.2]{kufner} about  the density of $C^{\infty}(\overline{\Omega}_p)$ functions in $\mathcal{H}^1_{\delta,\delta}(\Omega_p)$, once we argue that 
		\begin{align}
			\label{eq:h1delta}
		\text{for $0\leq \delta\leq 2$, }	\mathcal{H}^1_{\delta}(\Omega_p)=\mathcal{H}^1_{\delta,\delta}(\Omega_p) \text{ with equivalent norms. }
		\end{align}
The inclusion $\subseteq$ is evident, while $\supseteq$ follows from the Poincar\'e inequality \ref{prop:poincare_appendix}.
\end{proof}
For $\delta\geq 1$, we have a stronger result. 
\begin{proposition}
	\label{prop:density_vreg}
	Let $1\leq \delta\leq 2$. Then $\mathcal{H}^1_{\delta}(\Omega_p)=\overline{\mathcal{C}^{\infty}_{comp}(\overline{\Omega}_p)}^{\|.\|_{\mathcal{H}^1_{\delta}(\Omega_p)}}.$
\end{proposition}
\begin{proof}
Follows by a straightforward adaptation of the proof of \cite[Theorem 1.1]{grisvard}, originally stated for $\mathcal{H}^1_{\delta,\delta}(U)$ with $U=\{(x,y)\in \mathbb{R}^2: x>0\}$, combined with Proposition \ref{prop:density_smooth}.
\end{proof}
We will also need the following Hardy inequality, which we state for the space $\mathcal{V}_{reg}(\Omega_p)$.
\begin{proposition}[Hardy inequality]
	\label{prop:hardy}
	For all $u\in \mathcal{V}_{reg}(\Omega_p)$, and all $\varepsilon>0$, it holds that 
	\begin{align*}
		\|x^{-1/2+\varepsilon}u\|_{L^2(\Omega_p)}\leq C(\Omega_p,\varepsilon)  \|u\|_{\mathcal{V}_{reg}(\Omega_p)},\quad \text{ with some }C(\Omega_p, \re)>0.
	\end{align*}
\end{proposition}
\begin{proof}
	See \cite[Lemma 6]{nguyen} and references therein, in particular \cite[Theorem 1.5]{necas}.
\end{proof}
An important corollary of Proposition \ref{prop:hardy} reads.
\begin{lemma}
	\label{lem:L1}
	Let $u\in \mathcal{H}^1(\Omega_p)$, $v\in \mathcal{H}_{\delta}^1(\Omega_p)$, for some $0<\delta<1$. Then $uv\in L^1(\Omega_p)$ and $\partial_x (uv)\in L^1(\Omega_p)$. 
\end{lemma}
\begin{proof}
	We argue by density. Let $u, v\in \mathcal{C}^{\infty}(\overline{\Omega_p})$. Then 
	\begin{align*}
		\int_{\Omega_p}|\partial_x(uv)|\lesssim \int_{\Omega_p}|\partial_x u \, v|+\int_{\Omega_p}|\partial_x v \, u|.
	\end{align*}
	The first integral is well-defined since $u\in \mathcal{H}^1(\Omega)$ and $v\in L^2(\Omega)$; as for the second integral, we employ the Cauchy-Schwarz inequality and next the Hardy inequality of Proposition \ref{prop:hardy} for $u$ (recall that $\delta<1$):
	\begin{align*}
		\int_{\Omega_p}|\partial_x v \, u|=\int_{\Omega_p}|x^{\delta/2}\partial_x v| \, |x^{-\delta/2}u|\leq \|v\|_{\mathcal{H}^1_{\delta}(\Omega_p)} \|x^{-\delta}u\|_{L^2(\Omega_p)}\leq C_{\delta} \|v\|_{\mathcal{H}^1_{\delta}(\Omega_p)} \|u\|_{\mathcal{V}_{reg}(\Omega_p)}.\qquad\qquad\qedhere
	\end{align*}
\end{proof}
Another important property concerns the space $\mathcal{V}_{sing}(\Omega)$. 
\begin{proposition}
	\label{prop:property2}
	Assume that $u\in \mathcal{V}_{sing}(\Omega)$. Then $
		x\mathbb{A}\nabla u\in L^2(\Omega).
$
\end{proposition}
\begin{proof}
Given $u\in \mathcal{V}_{sing}(\Omega)$, a straightforward computation yields that $v:=xu\in \mathcal{H}^1(\Omega_p)\times \mathcal{H}^1(\Omega_n)$. Let us argue that $v\in \mathcal{H}^1(\Omega)$. For this it is sufficient to show that $\gamma_0^{\Sigma,p}v=\gamma_0^{\Sigma,n}v=0$, or, equivalently, that $v\in \mathcal{H}^1_{\Sigma,0}(\Omega_p)\times \mathcal{H}^1_{\Sigma,0}(\Omega_n)$, where
\begin{align*}
	\mathcal{H}^1_{\Sigma,0}(\Omega_{\lambda})=\{v\in \mathcal{H}^1(\Omega_{\lambda}): \, \gamma_0^{\Sigma,\lambda}v=0\}=\overline{\mathcal{C}^{\infty}_{comp}(\Omega_{\lambda})}^{\|.\|_{H^1(\Omega_{\lambda})}}.
\end{align*} 
Let $\varphi_n\in \mathcal{C}^{\infty}_{comp}(\Omega_p)$ be s.t. $\varphi_n\rightarrow u$ in $\mathcal{V}_{sing}(\Omega_p)$ (it exists by Proposition \ref{prop:density_vreg}). Then $x\varphi_n\rightarrow xu=v$ in $H^1(\Omega_p)$. Since $x\varphi_n\in \mathcal{C}^{\infty}_{comp}(\Omega_p)$, we conclude that $v=xu\in \mathcal{H}^1_{\Sigma,0}(\Omega_p)$. Repeating the argument for $\Omega_n$ yields the desired conclusion.
\end{proof}
\subsection{Relation to fractional Sobolev spaces}The weighted spaces can be shown to be embedded into fractional Sobolev spaces.
\begin{lemma}
	\label{lem:embedding}
	For $0<\theta\leq2$, the space $\mathcal{H}^1_{ \theta}(\Omega_p)$ is continuously embedded into  ${H}^{1-\frac{\theta}{2}}(\Omega_p)$.
	%
	%In a similar manner, for $1<\theta<2$, the spaces $\mathcal{H}^1_{w, \theta}(\Omega_p)\subset H^{1-\theta}(\Omega_p)$.
\end{lemma}
\begin{proof}
	Up to our knowledge, this result dates back to the works \cite{lizorkin, uspenskii}, in the case when $\Omega_p$ is a half-space (in other words, $a=\ell=+\infty$). The corresponding result for $\Omega_p$ follows by a localization argument, cf. \cite[Theorem 3.20]{mclean}. A different version of the proof can be found in \cite[Theorems 4.1, 4.2]{jerison_kenig};  remark that while the needed result is stated for harmonic functions, its proof does not rely on this property.
\end{proof}
%	One would be tempted to proceed by interpolation for weighted spaces, however, we were not able to find related results in the literature, cf. however  \cite{cwikel_einav}. Thus, we  
	
%	By \eqref{eq:h1delta} will prove the result for the spaces $\mathcal{H}^1_{\theta,\theta}$. We use interpolation argument. The following embeddings are continuous: $$\mathcal{H}^1_{2,2}\hookrightarrow L^2(\Omega_p), \qquad \mathcal{H}^{1}_{0,0}\hookrightarrow H^1(\Omega_p). $$
%Next, for the interpolation of the weighted and non-weighted %spaces we employ the results from \cite[Theorem 3.2]{favini} (cf. \cite[Lemma 21.6]{tartar} and the definition before on the complex interpolation used in the paper \cite{favini}).
An immediate corollary of the above result and the  continuity of the trace in fractional Sobolev spaces of Theorem 3.37 in \cite{mclean} reads.
\begin{corollary}
	\label{cor:tr}
	For $0<\theta<1$, the trace operator $\gamma_0^{\Sigma}: \, \mathcal{H}^1(\Omega_p)\rightarrow \mathcal{H}^{1/2-\theta/2}(\Sigma)$ extends by density to a bounded linear operator from $\mathcal{H}^1_{\theta}(\Omega_p)$ into $\mathcal{H}^{1/2-\theta/2}(\Sigma)$.
\end{corollary}
\begin{proposition}[Prop. 1.2 in \cite{grisvard}]
	\label{prop:gamma0}
	The space $\mathcal{H}^1_{\theta;0}(\Omega_p):=\{u\in \mathcal{H}^1_{\theta}(\Omega_p), \text{ s.t. }\gamma_0^{\Sigma,p}u=0\}$ equals to $\overline{\mathcal{C}^{\infty}_{comp}(\Omega_p)}^{\|.\|_{\mathcal{H}^1_{\theta}(\Omega_p)}}$.
\end{proposition}
the latter space being well-defined, as follows by adapting the proof of \cite[Proposition 1.1']{grisvard} (remark that the norm in \cite{grisvard} is equivalent to the norm $\|.\|_{H^1_{\re}(\Omega_p)}$ we use, as shown in Proposition \ref{prop:poincare_appendix}).

	\section{Properties of a regular problem \eqref{eq:rp}}
\label{sec:apx_reg_pb}
\subsection{Vanishing conormal trace}
\label{appendix:conormal_trace}
\begin{proposition}
	\label{cor:uv_appx}
	Any function $u\in\mathcal{V}_{reg}(\Omega_p)$, s.t., for some $\varepsilon>0$, $\vec{x}\mapsto x^{1/2-\varepsilon}\operatorname{div}(x\mathbb{A}(\vec{x})\nabla u(\vec{x}))\in L^2(\Omega_p)$, satisfies: $
		\gamma_{n}^{\Sigma}u=0 \text{ in }\mathcal{H}^{-1/2}(\Sigma).$
\end{proposition}
To prove the above result, we rely on the following auxiliary statements. Let us denote $
	f:=\operatorname{div}(x\mathbb{A}(\vec{x})\nabla u(\vec{x}))$, 
and argue that $f$ as in the statement of Proposition \ref{cor:uv_appx} satisfies $f\in (\mathcal{V}_{reg}(\Omega_p))'$.
\begin{lemma}
	\label{lem:f_belongs_vreg}
	For all $0\leq \delta<1$, the space $L^2_{\delta}(\Omega_p)$ is continuously embedded into $\mathcal{V}_{reg}'(\Omega_p)$. In particular, for all $f\in L^2_{\delta}(\Omega_p)$, the Lebesgue integral $\int_{\Omega_p}fvd\vec{x}$ is well-defined for all $v\in \mathcal{V}_{reg}(\Omega_p)$. 
\end{lemma}
\begin{proof}
Let $f\in L^2_{\delta}(\Omega_p)$, for some $0\leq \delta<1$. It suffices to prove the second statement. 
In virtue of the Hardy inequality of Proposition \ref{prop:hardy} the integral is well-defined as a scalar product of two $L^2(\Omega_p)$-functions:
	\begin{align}
		\label{eq:two_funcs}
		\int_{\Omega_p}fvd\vec{x}=\int_{\Omega_p}x^{\delta/2}f\,x^{-\delta/2}v d\vec{x}=(x^{\delta/2}f, x^{-\delta/2}\overline{v})_{L^2(\Omega_p)}, \quad \forall v\in \mathcal{V}_{reg}(\Omega_p).
	\end{align}
From the above and the Hardy inequality it is immediate that $\|f\|_{\mathcal{V}'_{reg}(\Omega_p)}\lesssim \|x^{\delta}f\|_{L^2(\Omega_p)}$. 
\end{proof}
The next result is an auxiliary lifting lemma. Before stating it, let us recall the notation  \eqref{eq:omegaresigma}: $
	\Omega_{\Sigma}^{\delta}:=\{\bx\in \Omega: \, |\operatorname{dist}(\vec{x},\Sigma)|<\delta\}, \quad \delta>0.$
\begin{lemma}[Lifting lemma]
	\label{lemma:lifting_lemma}
	Given $v\in \mathcal{H}^{1/2}(\Sigma)$ and  $0<\delta<1$, let $V^{\delta}\in \mathcal{H}^1(\Omega_p)$ be a unique solution to the following boundary-value problem:
	\begin{align*}
		&-\Delta V^{\delta}=0 \text{ on }\Omega_{\Sigma}^{\delta}\cap\Omega_p,\\
		&V^{\delta}(0,y)=v(y), \qquad V^{\delta}(\delta,y)=0,\\
		&\text{ periodic BCs at }y=\pm \ell.
	\end{align*}
	Then, with a constant $C>0$ independent of $\delta$, 
	\begin{align}
		\begin{split}
		\label{eq:bound_deltaL2}
		&(a)\,\;\|V^{\delta}\|_{L^2(\Omega_p)}+\delta\|V^{\delta}\|_{H^1(\Omega_p)}\leq C \delta^{1/2}\|v\|_{\mathcal{H}^{1/2}(\Sigma)},\quad
		(b)\,\;\|V^{\delta}\|_{\mathcal{V}_{reg}(\Omega_p)}\leq C\|v\|_{\mathcal{H}^{1/2}(\Sigma)}.
		\end{split}
	\end{align}
\end{lemma}
Before proving the above result, let us recall a standard equivalent definition of the $\mathcal{H}^{1/2}(\Sigma)$-norm. 
First of all, consider the 1D periodic Laplacian operator $-\partial_y^2$ on $(-\ell,\ell)$; we define its eigenfunctions $\{\phi_m\}_{m\in\mathbb{N}}$ and the corresponding eigenvalues $\lambda_m^2$ (where $0=\lambda_0^2\leq \lambda_1^2\leq\ldots \rightarrow +\infty$) via 
\begin{align*}
		&-\partial_y^2 \varphi_m=\lambda_m^2 \varphi_m, \quad \text{ on }(-\ell,\ell),\\
		&\text{periodic BCs at }y=\pm \ell,\qquad \|\varphi_m\|_{L^2(-\ell,\ell)}=1. 
\end{align*}
Identifying $\Sigma$ with an interval $(-\ell,\ell)$, and given $v\in L^2(\Sigma)$, we can expand it into the Fourier series in $\varphi_m$. Let  $v_m:=(v,\varphi_m)_{L^2(-\ell,\ell)}$ and define the norm $
	\|v\|_{\mathcal{H}^{1/2}_{per}(\Sigma)}^2:=\sum\limits_{m=0}^{+\infty}(1+\lambda_m^2)^{1/2}|v_m|^2.$ 
The norms $\|.\|_{\mathcal{H}^{1/2}_{per}(\Sigma)}$ and $\|.\|_{\mathcal{H}^{1/2}(\Sigma)}$ are equivalent, and 
$
	\mathcal{H}^{1/2}_{per}(\Sigma)=\{v\in L^2(\Sigma): \, \sum\limits_{m=0}^{+\infty}(1+\lambda_m^2)^{1/2}|v_m|^2<\infty\}=\mathcal{H}^{1/2}(\Sigma).$
This follows by the same argument as in the proof of \cite[Lemma 2.10]{cabre_tan}. 
\begin{proof}[Proof of Lemma \ref{lemma:lifting_lemma}]
It holds that $
	V^{\delta}(x,y)=\sum\limits_{m=0}^{\infty}V^{\delta}_m(x)\phi_m(y), $
	where each $V^{\delta}_m(x)$ solves the ODE
	\begin{align*}
		-\partial_x^2 V^{\delta}_m+\lambda_m^2 V^{\delta}_m=0 \text{ on }(0,\,\delta), \quad 
		V^{\delta}_m(0)=v_m, \quad V^{\delta}_m(\delta)=0.
	\end{align*}
	Solving the above explicitly yields the identity:
	\begin{align*}
		V_m^{\delta}(x)=v_m\frac{\mathrm{e}^{\lambda_m(\delta-x)}-\mathrm{e}^{-\lambda_m(\delta-x)}}{\mathrm{e}^{\lambda_m\delta}-\mathrm{e}^{-\lambda_m\delta}}.
	\end{align*}
To bound the $H^1$-norm, we use $\Delta V^{\delta}=0$ in $\Omega_p\cap\Omega_{\Sigma}^{\delta}$ and $V^{\delta}(\delta,y)=0$ to integrate by parts:
	\begin{align*}
		\int_{\Omega_p\cap\Omega_{\Sigma}^{\delta}}|\nabla V^{\delta}|^2=-\int_{\Sigma}\partial_x V^{\delta}(0,y)\overline{V^{\delta}}(0,y)d\Gamma_y=-\int_{\Sigma}\partial_x V^{\delta}(0,y)\overline{v}(y)d\Gamma_y. 
	\end{align*}
	Then a straightforward computation yields $\partial_x V^{\delta}(0,y)=-\sum\limits_{m=0}^{\infty}\lambda_mv_m\frac{\mathrm{e}^{\lambda_m\delta}+\mathrm{e}^{-\lambda_m\delta}}{\mathrm{e}^{\lambda_m\delta}-\mathrm{e}^{-\lambda_m\delta}}\phi_m(y)$, so that 
	\begin{align}
		\label{eq:series}
		\|\nabla V^{\delta}\|^2_{L^2(\Omega_p)}=\sum\limits_{m=0}^{\infty}|v_m|^2\lambda_m\frac{1+\mathrm{e}^{-2\delta\lambda_m}}{1-\mathrm{e}^{-2\lambda_m\delta}}\leq \sum\limits_{m=0}^{\infty}|v_m|^2\frac{2\lambda_m}{1-\mathrm{e}^{-2\lambda_m\delta}},
	\end{align}
where we used $\mathrm{e}^{-2\delta\lambda_m}<1$. Next, 
for some $0<\re<1$ sufficiently small, there exists $C_{\re}>0$, s.t.
	\begin{align*}
		&1-\mathrm{e}^{-t}\geq C_{\re}t, \text{ for all }|t|<\re,\qquad \text{ and }\quad 
		1-\mathrm{e}^{-t}\geq C_{\re}\re, \text{ for all }|t|\geq \re.
	\end{align*}
We split the series in \eqref{eq:series} into two parts, and use the above bounds for a fixed $\re>0$ in both parts:
	\begin{align}
		\label{eq:final1}
		\|\nabla V^{\delta}\|^2_{L^2(\Omega_p)}\leq  \frac{1}{\delta C_{\re}}\sum\limits_{2\lambda_m\delta<\re}|v_m|^2+\frac{2}{C_{\re}\re}\sum\limits_{2\lambda_m\delta\geq \re}|v_m|^2\lambda_m\lesssim \max(1,\delta^{-1})\|v\|^2_{\mathcal{H}^{1/2}(\Sigma)}.
	\end{align}
	To estimate the $L^2$-norm, we combine the above  with the Poincar\'e inequality. Since $V^{\delta}(\delta,0)=0$, we use $$V^{\delta}(x,y)=-\int_{x}^{\delta}\partial_{\tilde{x}}V(\tilde{x},y)d\tilde{x},\text{ a.e. }y\in (-\ell,\ell),$$
	and by the Cauchy-Schwarz inequality, it holds a.e. $y\in (-\ell,\ell)$, that
	\begin{align*}
		|V^{\delta}(x,y)|^2\leq (\delta-x)\int_{x}^\delta|\partial_{\tilde{x}} V^{\delta}(\tilde{x},y)|^2 d\tilde{x} \implies \int_0^{\delta}|V^{\delta}(x,y)|^2dx\leq \delta^2 \int_0^{\delta}|\partial_x V^{\delta}(\tilde{x},y)|^2 d\tilde{x}.
	\end{align*}
	Therefore,  $
		\|V^{\delta}\|^2_{L^2(\Omega_p)}=		\|V^{\delta}\|^2_{L^2(\Omega_p\cap\Omega_{\Sigma}^{\delta})}\lesssim \delta^2\|\nabla V^{\delta}\|^2_{L^2(\Omega_p)}. $
		The bound \eqref{eq:bound_deltaL2}(a)  then follows with \eqref{eq:final1}. The remaining bound \eqref{eq:bound_deltaL2}(b) is immediate from the above: since $\operatorname{supp}V^{\delta}\subseteq \overline{\Omega}_{\Sigma}^{\delta}$, 
	\begin{align*}
		\|x^{1/2}\nabla V^{\delta}\|_{L^2(\Omega_p)}=	\|x^{1/2}\nabla V^{\delta}\|_{L^2(\Omega_p\cap \Omega_{\Sigma}^{\delta})}\leq   \delta^{1/2}\|\nabla V^{\delta}\|_{L^2(\Omega_p\cap \Omega_{\Sigma}^{\delta})}\leq C \|v\|_{\mathcal{H}^{1/2}(\Sigma)},
	\end{align*}
	as follows from \eqref{eq:final1}. Combining the above with $\|V^{\delta}\|_{L^2(\Omega_p)}\lesssim \delta^{1/2}\|v\|_{\mathcal{H}^{1/2}(\Sigma)}$ yields the sought inequality.
\end{proof}
Now we have the necessary ingredients to prove Proposition \ref{cor:uv_appx}. 
\begin{proof}[Proof of Proposition \ref{cor:uv_appx}]
Recall the variational definition of the conormal derivative: given $\varphi\in \mathcal{H}^{1/2}(\Sigma)$,
\begin{align}
	\label{eq:nderdef}
	\langle 	\gamma_{n}^{\Sigma}u, \varphi\rangle_{\mathcal{H}^{-1/2}(\Sigma), \mathcal{H}^{1/2}(\Sigma)}=-\int_{\Omega_p} f\Phi -\int_{\Omega_p}x\mathbb{A} \nabla u\, \nabla \Phi,
\end{align}
with $\Phi\in\H^1(\Omega_p)$ being such that $\gamma_0^{\Sigma}\Phi=\varphi$. The right-hand side is well-defined by Lemma \ref{lem:f_belongs_vreg}. 

To prove that $\gamma_n^{\Sigma}u=0$, we fix $\varphi\in \mathcal{H}^{1/2}(\Sigma)$, $0<\delta<a$, and choose the lifting of $\varphi$ as $\Phi=\Phi^{\delta}$ defined in Lemma \ref{lemma:lifting_lemma}. We introduce $\Omega_{p,\Sigma}^{\delta}:=\Omega_{\Sigma}^{\delta}\cap \Omega_p$, use Lemma \ref{lem:f_belongs_vreg} (see \eqref{eq:two_funcs}), the fact that $\operatorname{supp}\Phi^{\delta}\subseteq\overline{\Omega_{p,\Sigma}^{\delta}}$ and the Cauchy-Schwarz inequality (with the hidden constant depending on $\re$):
\begin{align*}
	\left|\langle 	\gamma_{n}^{\Sigma}u, \varphi\rangle_{\mathcal{H}^{-1/2}(\Sigma), \mathcal{H}^{1/2}(\Sigma)}\right|&\lesssim  \|x^{1/2-\varepsilon}f\|_{L^2(\Omega_{p,\Sigma}^{\delta})}\|\Phi^{\delta}\|_{\mathcal{V}_{reg}(\Omega_{p,\Sigma}^{\delta})}+\|x^{1/2}\nabla u\|_{L^2(\Omega_{p,\Sigma}^{\delta})}\|x^{1/2}\nabla \Phi_{\delta}\|_{L^2(\Omega_{p,\Sigma}^{\delta})}.
\end{align*}
Using the bound \eqref{eq:bound_deltaL2}(b), we  conclude that, with some $C>0$, independent of $\delta$, it holds that 
\begin{align*}
	\left|\langle 	\gamma_{n}^{\Sigma}u, \varphi\rangle_{\mathcal{H}^{-1/2}(\Sigma), \mathcal{H}^{1/2}(\Sigma)}\right|&\leq C \left(\|x^{1/2-\varepsilon}f\|_{L^2(\Omega_{p,\Sigma}^{\delta})}+\|u\|_{\mathcal{V}_{reg}(\Omega_{p,\Sigma}^{\delta})}\right)\|\varphi\|_{\mathcal{H}^{1/2}(\Sigma)}.
\end{align*}
Since $u\in \mathcal{V}_{reg}(\Omega_p)$ and  $x^{1/2-\varepsilon}f\in L^2(\Omega_p)$, we conclude that by taking $\delta\rightarrow 0$, the norm $\|\gamma_n^{\Sigma}u\|_{\mathcal{H}^{-1/2}(\Sigma)}$ can be made arbitrarily small, hence the conclusion of the proposition.
\end{proof}
\subsection{Well-posedness for non-$L^2$-data}
\label{sec:wp_rough}
%We will additionally need the following regularity result. 
%\begin{proposition}
%	\label{prop:regularity}
%	Let $0<\varepsilon<1/2$. Given $f\in H^{-\varepsilon}(\Omega_p)\subset (\mathcal{V}_{reg}(\Omega_p))'$, the problem: find $u\in \mathcal{V}_{reg}(\Omega_p)$, s.t. 
%	\begin{align}
%		\label{eq:pb_reg}
%		\int_{\Omega_p}x\mathbb{A}\nabla u\, \nabla v=-\langle f, v\rangle_{\mathcal{V}_{reg}'(\Omega_p), \mathcal{V}_{reg}(\Omega_p)} \qquad \forall v\in \mathcal{V}_{reg}(\Omega_p),
%	\end{align}
%	admits a unique solution. Moreover, $u\in H^{1-\varepsilon-\delta}(\Omega_p)$, for all $\delta>0$.
%\end{proposition}
%\begin{proof}
%	First of all, by Proposition \ref{prop:embedding}, $H^{-\varepsilon}(\Omega_p)\subset (\mathcal{V}_{reg}(\Omega_p))'$. The well-posedness result follows immediately from Theorem \ref{prop:main}. The regularity result follows by interpolation from Proposition \ref{prop:embedding} ($f\in H^{-1/2+\delta}(\Omega_p)\implies u\in H^{1/2-\delta'}(\Omega_p)$, for all $\delta,\delta'>0$) and Theorem \ref{theorem:regularity}. 
%\end{proof}
\begin{proposition}
	\label{prop:regularity2}
	Let $0\leq \delta<1$, and let $f\in L^2_{\delta}(\Omega_p)$. 
	Then the problem: find $u\in \mathcal{V}_{reg}(\Omega_p)$, s.t. 
	\begin{align}
		\label{eq:divxu}
		\begin{split}
		&\operatorname{div}(x\mathbb{A}\nabla u)=f \text{ in }\Omega_p,\\
		&\gamma_0^{\Gamma_p}u=0, \quad \text{ periodic BCs at }\Gamma_p^{\pm},
		\end{split}
	\end{align}
admits a unique solution $u\in \mathcal{V}_{reg}(\Omega_p)$. It satisfies $\gamma_n^{\Sigma}u=0$ and $$u\in \bigcap_{0<\re<1}\mathcal{H}^1_{\delta+\re}(\Omega_p)\subset \bigcap_{0<\re\leq 1-\frac{\delta}{2}}\mathcal{H}^{1-\delta/2-\re}(\Omega_p).$$
Also, $\|u\|_{\mathcal{H}^{1-\delta/2-\re/2}(\Omega_p)}\leq C_{\delta,\re} \|u\|_{\mathcal{H}^1_{\delta+\re}(\Omega_p)}\leq \tilde{C}_{\delta,\re} \|x^{\delta/2}f\|_{L^2(\Omega_p)},$ for all $0<\re<1/2$.   
\end{proposition}
\begin{proof}
By Lemma \ref{lem:f_belongs_vreg}, $f\in \mathcal{V}_{reg}'(\Omega_p)$. Testing the problem \eqref{eq:divxu} with $v\in \mathcal{C}_{comp}^{\infty}(\Omega_p)$, and using the density of $\mathcal{C}_{comp}^{\infty}(\Omega_p)$ in $\mathcal{V}_{reg}(\Omega_p)$, see Proposition \ref{prop:density_vreg},  yields the well-posed problem examined in Theorem \ref{prop:main}. By Proposition \ref{cor:uv_appx}, $\gamma_n^{\Sigma}u=0$. 
It remains to argue about the regularity of $u$, more precisely, we would like to show that $\nabla u\in L^2_{\delta+\re}(\Omega_p)$, for all $\re>0$. We proceed by interpolation. According to \cite[Theorem 3.1]{cwikel_einav} stated in the form that we need (which is based on the results from \cite{stein_weiss} and \cite{calderon}):
\begin{align}
	\label{eq:l2int}
	[ L^2(\Omega_p), L^2_{\eta}(\Omega_p)]_{\theta}=L^2_{{\theta\eta}}(\Omega_p), \quad 0<\theta<1, \quad \eta>0,
\end{align}
where $[,]_{\theta}$ stands for complex interpolation. 
Let us now consider the operator $\mathcal{S}: \, \mathcal{V}'_{reg}(\Omega_p)\rightarrow L^2_{1}(\Omega_p)$, defined by $\mathcal{S}f=\nabla u_f$, where $u_f\in \mathcal{V}_{reg}(\Omega_p)$ is a unique solution to
\begin{align*}
	(x^{1/2}\mathbb{A}\nabla u_f, x^{1/2}\nabla v)=-\langle f, \overline{v}\rangle_{\mathcal{V}_{reg}'(\Omega_p), \mathcal{V}_{reg}(\Omega_p)}, \quad \forall v\in \mathcal{V}_{reg}(\Omega_p)
\end{align*}
(remark that with an abuse of notation we use $L^2_{\eta}(\Omega_p)$ both for scalar- and vector-valued functions).

By Theorem \ref{theorem:regularity}, $\mathcal{S}\in \mathcal{L}(L^2(\Omega_p), L^2(\Omega_p))$, and, with Lemma \ref{lem:f_belongs_vreg}, for any $0<\re<1$, $$\mathcal{S}\in \mathcal{L}(L^2_{1-\re}(\Omega_p), L^2_{1}(\Omega_p)).$$ With \eqref{eq:l2int} this shows that $\mathcal{S}\in \mathcal{L}(L^2_{(1-\re)\theta}(\Omega_p), L^2_{\theta}(\Omega_p))$, for all $\theta\in (0,1)$. Choosing $\re$ sufficiently close to $0$, and $\theta=\delta/(1-\re)<1$ yields $\mathcal{S}\in \mathcal{L}(L^2_{\delta}(\Omega_p), L^2_{\frac{\delta}{1-\re}}(\Omega_p))$, for all $\re$ sufficiently close to $0$, thus all $\re\in (0,1)$. Hence the statement of the proposition for bounds in the  weighted space $\mathcal{H}^1_{\delta+\re}(\Omega_p)$, for any $\re>0$. 

On the other hand, the corresponding result on the fractional spaces stems from Lemma \ref{lem:embedding}. 
\end{proof}

	\section{Proof of Theorem \ref{theorem:regularity}}
\label{appendix:regularity_proof}
In this Appendix we provide a proof of Theorem \ref{theorem:regularity} from \cite{baouendi_goulaouic}, where we argue that we can reduce the regularity of the coefficients compared to the regularity considered in \cite{baouendi_goulaouic}; the statement of this theorem is repeated below for the convenience of the reader. The proof of this result relies on the elements  introduced in Section \ref{sec:extension_bessel}, that appear in the manuscript later compared to the statement of Theorem \ref{theorem:regularity}.

\begin{theorem}
	\label{theorem:regularity_appendix}
		Let $f\in L^2(\Omega_p)$. Then the unique solution $u$ to  \eqref{eq:rp} belongs to $\mathcal{H}^1(\Omega_p)$, and satisfies the following stability bound: $\|u\|_{H^1(\Omega_p)}+\| x  u\|_{H^2(\Omega_p)}\lesssim \|f\|_{L^2(\Omega_p)}$. 

If, moreover, $f\in \mathcal{H}^1(\Omega_p)$, then the unique solution to $u$ to  \eqref{eq:rp} satisfies $u\in \mathcal{H}^2(\Omega_p)$, and the following stability bound holds true: $\|u\|_{H^2(\Omega_p)}+\|xu\|_{H^3(\Omega_p)}\lesssim \|f\|_{H^1(\Omega_p)}$. 
	
\end{theorem}
Before proving Theorem \ref{theorem:regularity_appendix}, a couple of remarks are in order. In principle, the first part of the result ($f\in L^2(\Omega_p)$), is probably true when $\mathbb{A}$ is less regular, e.g. $C^{0,1}(\overline{\Omega}_p; \mathbb{C}^{2\times 2})$; however, in our regularity proof we used a non-optimal Lemma \ref{lem:commutator}, and did not prove the corresponding result for $\mathbb{A}$ being less regular.
Second, in the standard elliptic regularity estimates, cf. \cite[Theorem 4.18]{mclean}, for the uniform bound $\|u\|_{H^2(\Omega_p)}$, it suffices that coefficients are of regularity $C^{0,1}(\overline{\Omega}_p)$. We were not able to obtain such a result for our, degenerate, setting.

%We also think that the statement of Theorem \ref{theorem:regularity_appendix} holds true for $\mathbb{A}\in C^{1,1}(\overline{\Omega}_p; \mathbb{C})

\subsection{Proof of the statement of Theorem \ref{theorem:regularity_appendix} for the case when $f\in L^2(\Omega_p)$}
The first part of the result is proven indirectly. 
We introduce a regularized problem (see also Proposition \ref{prop:lap_reg}). This section is devoted to construction and analysis of such a regularized problem. 
\paragraph{A regularized Dirichlet problem} We look for $u^{\nu}_r\in \mathcal{H}^1(\Omega_p)$, s.t. 
\begin{align}
	\label{eq:regularized_pb}
	\begin{split}
		&\operatorname{div}((x+i\nu)\mathbb{A}\nabla u^{\nu}_r)=f \text{ in }\Omega_p,\\
		&\gamma_0^{\Sigma} u^{\nu}_r=0,\\
		&\gamma_0^{\Gamma_p}u^{\nu}_r =0, \text{ periodic BCs at }y=\pm\ell.
	\end{split}
\end{align}
Introducing $\mathcal{H}^1_{\Sigma,0}(\Omega_p):=\{u\in \mathcal{H}^1(\Omega_p), \quad \gamma_0^{\Sigma}u=0\}$, we 
can write a variational formulation for this problem, namely: find $u_r^{\nu}\in \mathcal{H}^{1}_{\Sigma,0}(\Omega)$, s.t.  
\begin{align}
	\label{eq:anur}
	a_r^{\nu}(u^{\nu}_r, v)=-(f, v), \quad \forall v\in \mathcal{H}^1_{\Sigma,0}(\Omega_p), \quad a_r^{\nu}(q,v)=\int_{\Omega_p}(x+i\nu)\mathbb{A}\nabla q\overline{\nabla v}. 
\end{align}
We have the following result, that enables us to use the solution to \eqref{eq:regularized_pb} to analyze properties of the solution to \eqref{eq:rp}. 
\begin{theorem}
	\label{theorem:convergenceNeumann}
	Let $f\in L^2(\Omega_p)$. 
	Then, for $\nu>0$, the problem \eqref{eq:regularized_pb} admits a unique solution $u^{\nu}_r\in \mathcal{H}^1(\Omega_p)$. Moreover, \begin{align}
		\label{eq:vreg}
		\|u^{\nu}_r\|_{\mathcal{V}_{reg}(\Omega_p)}\leq C\|f\|_{L^2(\Omega_p)},\\
		\label{eq:vreg2}
		\nu^{1/2}	\|u^{\nu}_r\|_{\mathcal{H}^1(\Omega_p)}\leq C\|f\|_{L^2(\Omega_p)}.
	\end{align}
	As $\nu\rightarrow 0+$, up to a subsequence, $u^{\nu}_r\rightharpoonup u_r$ in $\mathcal{V}_{reg}(\Omega_p)$ and $u^{\nu}_r\rightarrow u_r$ in $H^{1/2-\re}(\Omega_p)$, for all $0<\re\leq 1/2$.  
\end{theorem}
\begin{proof}
Cf. Lemma \ref{lem:pb_abs_wp} and Theorem \ref{theorem:well_posedness_regular} on how to obtain the results up to \eqref{eq:vreg2} including.
	
	From the bound in the statement of the theorem, it follows immediately that $u^{\nu}_r\rightharpoonup u_r$ in $\mathcal{V}_{reg}(\Omega_p)$ up to a subsequence. As for the statement about the strong convergence up to a subsequence, it follows from Lemma \ref{lem:embedding}, and the compactness of the embedding $H^s(\Omega_p)\subset H^{s+\epsilon}(\Omega_p)$, for all $s\geq 0$ and $\epsilon>0$.

	%Let us argue that the result holds true for the complete sequence $u^{\nu}_r$, rather than just its subsequence. Since $C_0^{\infty}(\Omega_p)$ is dense in $(\mathcal{V}_{reg}(\Omega_p))'$, it suffice to consider $\langle\varphi, u^{\nu}_r\rangle_{\mathcal{V}_{reg}(\Omega_p)', \mathcal{V}_{reg}(\Omega_p)}=\int_{\Omega_p}\varphi\, u^{\nu}_r$ for all $\varphi\in C_0^{\infty}(\Omega_p)$

	% To argue that $u^{\nu}_r$ converges weakly as a sequence, we remark that for all s$\varphi\in (\mathcal{V}_{reg}(\Omega_p))'$ $\langle\varphi, u^{\nu}_r\rangle_{\t is mathcal{V}_{reg}(\Omega_p)', \mathcal{V}_{reg}(\Omega_p)}$

\end{proof}
%Since $\Im a_r^{\nu}$ is coercive on $\mathcal{H}^1(\Omega_p)$, there exists a unique $u^{\nu}_r$ satisfying \eqref{eq:regularized_pb}. 

\paragraph{The regularity of $u^{\nu}_r$ for $f\in L^2(\Omega_p)$}
The key result of this section is the following statement.
\begin{theorem}
	\label{theorem:regularNeumann}
	Let $u^{\nu}_r$ be the unique solution of \eqref{eq:regularized_pb}, where the right-hand side $f\in L^2(\Omega_p)$.

	Then, for all $\nu>0$, $\|\partial_y u^{\nu}_r\|_{L^2(\Omega_p)}+\|x\partial_y \nabla u^{\nu}_r\|_{L^2(\Omega_p)}+\nu\|\partial_y \nabla u^{\nu}_r\|_{L^2(\Omega_p)}\leq C\|f\|_{L^2(\Omega_p)}$. 
	
\end{theorem}
Once we have proven the above result, the proof of the first part of the Theorem \ref{theorem:regularity_appendix} is quite straightforward. Remark that it is impossible to prove the uniform boundedness of $\|\partial_x u^{\nu}_r\|_{L^2(\Omega_p)}$ in the above setting, cf. the discussion in the beginning of Appendix \ref{appendix:lap_auxiliary}.
\begin{proof}[Proof of Theorem \ref{theorem:regularity_appendix}]
	By Theorem \ref{theorem:convergenceNeumann} and Theorem \ref{theorem:regularNeumann}, and the uniqueness of the weak $L^2(\Omega_p)$-limit of $u^{\nu}_r$, it holds that $\partial_y u_r\in \mathcal{H}^1(\Omega_p)$, and, moreover, $\partial_y (x\mathbb{A}\nabla u_r)\in L^2(\Omega_p)$. The uniform control of the corresponding norms by $\|f\|_{L^2(\Omega_p)}$ is immediate from the weak convergence properties. It remains to consider $\|\partial_x u_r\|_{L^2(\Omega_p)}$. For this we rewrite 
	\begin{align*}
		\partial_x(x\vec{e}_x\cdot\mathbb{A}\nabla u_r)=f-\partial_y(x\vec{e}_y\cdot\mathbb{A}\nabla u_r)\in L^2(\Omega_p).
	\end{align*}
The function $q:=x\vec{e}_x\cdot \mathbb{A}\nabla u_r$ is thus in $\mathcal{H}^1_{\Sigma,0}(\Omega_p)$ (see Lemma \ref{cor:uv} for $\gamma^{\Sigma}_0q=\gamma_n^{\Sigma}u_r=0$), and, by Hardy's inequality, cf. \eqref{eq:multhardy}, it holds that $\vec{x}\mapsto q(\vec{x})/x\in L^2(\Omega_p)$. This shows that $\mathbb{A}\nabla u_r\in L^2(\Omega_p)$, hence the conclusion about $u\in \mathcal{H}^1(\Omega_p)$.

The estimates of Theorem \ref{theorem:regularNeumann} imply in particular that $x\partial_y^2 u_r, x\partial_{xy}u_r\in L^2(\Omega_p)$. From this and  $\operatorname{div}(x\mathbb{A}\nabla u_r)=f$ it follows also that $\partial_x(x\mathbb{A}\nabla u_r\cdot\vec{e}_x)\in L^2(\Omega_p)$, and, in particular, $x\partial_x^2 u_r\in L^2(\Omega_p)$.

The uniform control of the corresponding norms of $u$ by $\|f\|_{L^2(\Omega_p)}$ stated in the theorem is a corollary of the above reasoning and continuity estimates of Hardy's inequality \eqref{eq:multhardy}. 

\end{proof}

The (remaining) proof of Theorem \ref{theorem:regularNeumann} is fairly tedious, and relies on several auxiliary results. We proceed as follows: 
\begin{enumerate}
	\item we start by proving the bound  
	\begin{align}
		\label{eq:jbound}
		\|\partial_y u^{\nu}_r\|_{\Omega_p}\leq C\|f\|_{\Omega_p}.
	\end{align}

	%By Lemma \ref{lem:fourier_embedding}, this will imply that $\|\mathcal{J}U^{\nu}_{r,\delta}\|_{1/2,y}=\|\partial_y U^{\nu}_{r,\delta}\|_{L^2(\mathbb{R}^2_{a})}\leq C\|f\|$ (and thus we have one part of the statement of Proposition \ref{prop:regularNeumann}). It remains to prove a similar statement on $\partial_x u^{\nu}$. This will be done in a less direct manner.  
	\item next, we prove that   
	\begin{align}
		\label{eq:partialY}
		\|\partial_y(x\mathbb{A}\nabla u^{\nu}_{r})\|_{\Omega_p}+\nu\|\partial_y (\mathbb{A}\nabla u^{\nu}_r)\|_{\Omega_p}\leq C\|f\|_{\Omega_p}.	
	\end{align} 
	\item next, we will be able to argue that the above allows us to conclude that \begin{align}
		\label{eq:dxu}
		\|\partial_x((x+i\nu)\nabla u^{\nu}\|_{\Omega_p}\leq C\|f\|_{\Omega_p}.
	\end{align} 
Remark that the stated inequality in Theorem \ref{theorem:regularNeumann}, namely,   $\|u^{\nu}_r\|_{H^1(\Omega_p)}+\|x\partial_y \nabla u^{\nu}_r\|_{L^2(\Omega_p)}+\nu\|\partial_y\nabla u^{\nu}\|_{L^2(\Omega_p)}\lesssim \|f\|_{L^2(\Omega_p)}$ follows from the above bounds and Theorem \ref{theorem:convergenceNeumann}. 
\end{enumerate}
\subsubsection{Proof of \eqref{eq:jbound}}
\label{sec:dyu}
We will use the Fourier analysis techniques, to relate the weighted norm and the fractional Sobolev norms. 
 The definitions can be found in Section \ref{sec:extension_bessel}. In particular, it suffices to consider the counterpart of \eqref{eq:regularized_pb} for $U^{\nu}_{r,\delta}=\chi_{\ell,\delta}\mathcal{E}_{\delta}u^{\nu}_r$ that satisfies, cf. \eqref{eq:Unudelta},  
\begin{align*}
	&\operatorname{div}((x+i\nu)\mathbb{A}\nabla U^{\nu}_{r,\delta})=F^{\nu}_{r,\delta} \text{ in }\mathbb{R}^{2,+}_{a},\\
	&\gamma_{0}^{\Sigma_{\infty}}U^{\nu}_{r,\delta}=0,
\end{align*}
with $\|F^{\nu}_{r,\delta}\|_{\mathbb{R}^{2,+}_a}\lesssim \|f\|_{\Omega_p}$. 

The lemma that follows will enable us to work with second derivatives of $U^{\nu}_{r,\delta}$ as $L^2$-functions. It can be proven by using a classical elliptic regularity argument, cf. \cite[Theorem 4.18]{mclean}. 
\begin{lemma}
	\label{lem:ellregc}
	Let  $u^{\nu}_r$ be the unique solution of \eqref{eq:regularized_pb}, where the right-hand side $f\in L^2(\Omega_p)$. For all $\nu>0$, $u^{\nu}_r\in H^2(\Omega_p)$, and, moreover, $\partial_x u^{\nu}_r(x,\ell)=\partial_x u^{\nu}_r(x,-\ell)$.

	As a corollary, for all $\nu>0$, $U^{\nu}_{r,\delta}\in H^2(\mathbb{R}_a)$ and also $\left.\mathcal{J}^2 U^{\nu}_{r,\delta}\right|_{x=a}=0$. 
\end{lemma}

The first result recalls that $\mathcal{V}_{reg}(\mathbb{R}^2_a)$ is embedded into the Bochner space $L^2((0, a); H^{1/2}(\mathbb{R}))$, see Lemma \ref{lem:embedding}.
\begin{lemma}
	\label{lem:fourier_embedding}
	For all $v\in \mathcal{V}_{reg}(\mathbb{R}^2_a)$, it holds that $\|\mathcal{J} v\|_{L^2(\mathbb{R}^2_a)}\leq C\|v\|_{\mathcal{V}_{reg}(\mathbb{R}^2_a)}$. 
\end{lemma}
\begin{proof}
	See the proof of \cite[Lemma 1 and 2]{baouendi_goulaouic}. Remark that in the above reference, the result is formulated for the domain $\mathbb{R}^{2,+}{+\infty}$, but extends easily to $\mathbb{R}^{2,+}_a$, $a<\infty$.
\end{proof}
The key technical result of this section is the following. It was proven in \cite{baouendi_goulaouic}, but our point is to show that it is also valid for $\mathbb{A}$ as in Assumption \ref{assump:matrices}. 
\begin{proposition}
	\label{prop:dyu}
	It holds that $\|\mathcal{J} U^{\nu}_{r,\delta}\|_{\mathcal{V}_{reg}(\mathbb{R}^{2,+}_a)}\leq C_{\delta}\|f\|_{\Omega_p}$.
\end{proposition}
\begin{proof}
	Remark that $U^{\nu}_{r,\delta}$ satisfies the following variational formulation:
	\begin{align}
		\label{eq:varf}
		((x+i\nu)\mathbb{A}_{\delta}\nabla U^{\nu}_{r,\delta}, \nabla v)_{\mathbb{R}^{2,+}_a}=-(F^{\nu}_{r,\delta}, v), \quad \forall v\in \mathcal{V}_{reg}(\mathbb{R}_a^{2,+}),
	\end{align}
where 
\begin{align*}
	\mathbb{A}_{\delta}=\mathbb{A}\chi_{\ell, 2\delta}+(1-\chi_{\ell, 2\delta})\mathbb{I},
\end{align*}
where we used that $\chi_{\ell,2\delta}=1$ on $\operatorname{supp}\chi_{\ell,2\delta}$, and $U^{\nu}_{r,\delta}=0$ for $|y|>\delta$. Remark that this is a Hermitian positive definite $C^{1,1}(\mathbb{R}^{2,+}_{a})$-matrix.

	Choosing $v=\mathcal{J}^2U^{\nu}_{r,\delta}$, with $v=0$ on $x=0$, cf. Lemma \ref{lem:ellreg},  and using the symmetry of $\mathcal{J}$ of Lemma \ref{lem:plancherel}, yields 
	\begin{align*}
		(\mathcal{J}(x+i\nu)\mathbb{A}_{\delta}\nabla U^{\nu}_{r,\delta}, \mathcal{J}\nabla U^{\nu}_{r,\delta})_{\mathbb{R}^{2,+}_a}=-(F^{\nu}_{\delta}, \mathcal{J}^2U^{\nu}_{r,\delta}).	
	\end{align*}
	The goal is to obtain a control on $\|\mathcal{J}U^{\nu}_{r,\delta}\|_{\mathcal{V}_{reg}}$, thus we rewrite the left-hand side using a commutator $[\mathcal{J},(x+i\nu)\mathbb{A}_{2\delta} )u]=[\mathcal{J},\mathbb{A}_{2\delta}](x+i\nu)$: 
	\begin{align*}
		((x+i\nu)\mathbb{A}_{2\delta}\mathcal{J}\nabla U^{\nu}_{r,\delta}, \mathcal{J}\nabla U^{\nu}_{r,\delta})_{\mathbb{R}^{2,+}_a}+	([\mathcal{J},\mathbb{A}_{2\delta}](x+i\nu)\nabla U^{\nu}_{r,\delta}, \mathcal{J}\nabla U^{\nu}_{r,\delta})_{\mathbb{R}^{2,+}_a}=-(F^{\nu}_{\delta}, \mathcal{J}^2U^{\nu}_{r,\delta})_{\mathbb{R}^{2,+}_a}.	
	\end{align*}
	Remark that $\mathcal{J}\nabla=\nabla\mathcal{J}$. Taking the real part of the above yields
	\begin{align}
		\label{eq:a2delta}
		\begin{split}
		\|x^{1/2}\nabla \mathcal{J}U^{\nu}_{r,\delta}\|^2_{\mathbb{R}^{2,+}_a}&\lesssim \left(\|x^{1/2}[\mathcal{J},\mathbb{A}_{2\delta}]\nabla U^{\nu}_{r,\delta}\|_{\mathbb{R}^{2,+}_a}\| x^{1/2}\mathcal{J}\nabla U^{\nu}_{r,\delta}\|_{\mathbb{R}^{2,+}_a}+\nu\Im ([\mathcal{J}, \mathbb{A}_{2\delta}]\nabla U^{\nu}_{r,\delta}, \mathcal{J}\nabla U^{\nu}_{r,\delta})\right)_{\mathbb{R}^{2,+}_a}\\
		&+\|F^{\nu}_{r,\delta}\|_{\mathbb{R}^{2,+}_a}\| \mathcal{J}^2U^{\nu}_{r,\delta}\|_{\mathbb{R}^{2,+}_a}. 
		\end{split}
	\end{align}
Next, we estimate each term in the above. 

With Lemma \ref{lem:commutator}, and next \eqref{eq:vreg}, which extends to $U^{\nu}_{r,\delta}$, 
\begin{align}
	\label{eq:unub}
	\|x^{1/2}[\mathcal{J},\mathbb{A}_{2\delta}]\nabla U^{\nu}_{r,\delta}\|_{\mathbb{R}^{2,+}_a}=	\|x^{1/2}[\mathcal{J},(\mathbb{A}_{2\delta}-\mathbb{I})\chi_{\ell,2\delta}]\nabla U^{\nu}_{r,\delta}\|_{\mathbb{R}^{2,+}_a}\lesssim \|x^{1/2}\nabla U^{\nu}_{r,\delta}\|_{\mathbb{R}^{2,+}_a}\lesssim \|F^{\nu}_{r,\delta}\|_{\mathbb{R}^{2,+}_a}.
\end{align}
Remark that $C^{1,1}(\overline{\Omega}_p)$-regularity of $\mathbb{A}$ is sufficient to apply Lemma \ref{lem:commutator}, cf. Remark \ref{rem:regularity}. 

By Lemma \ref{lem:self_adjoint_commutator}, using that $[\mathcal{J}, \mathbb{A}_{2\delta}]=[\mathcal{J}, (\mathbb{A}-\mathbb{I})\chi_{\ell,2\delta}]$, 
\begin{align*}
|\Im ([\mathcal{J}, \mathbb{A}_{2\delta}]\nabla U^{\nu}_{r,\delta}, \mathcal{J}\nabla U^{\nu}_{r,\delta})_{\mathbb{R}^{2,+}_a}|\lesssim \|[\mathcal{J}^2, \mathbb{A}_{2\delta}]\nabla U^{\nu}_{r,\delta}\|_{\mathbb{R}^{2,+}_a}\|\nabla U^{\nu}_{r,\delta}\|_{\mathbb{R}^{2,+}_a},
\end{align*} 
and with the use of Lemma \ref{lem:commutator}, and next \eqref{eq:vreg2} adapted to $U^{\nu}_{r,\delta}$, we conclude that 
\begin{align}
	\label{eq:imest}
|\Im ([\mathcal{J}, \mathbb{A}_{2\delta}]\nabla U^{\nu}_{r,\delta}, \mathcal{J}\nabla U^{\nu}_{r,\delta})_{\mathbb{R}^{2,+}_a}|\lesssim \|\nabla U^{\nu}_{r,\delta}\|_{\mathbb{R}^{2,+}_a}^2\lesssim \nu^{-1}\|F^{\nu}_{r,\delta}\|_{\mathbb{R}^{2,+}_a}^2.
\end{align}
Finally, by Lemma \ref{lem:fourier_embedding}, $\|\mathcal{J}^2 U^{\nu}_{r,\delta}\|_{\mathbb{R}^{2,+}_a}\lesssim \|x^{1/2}\nabla\mathcal{J}U^{\nu}_{r,\delta}\|_{\mathbb{R}^{2,+}_a}$. Combining this estimate, \eqref{eq:imest}, \eqref{eq:unub} into \eqref{eq:a2delta}, and using $\|F^{\nu}_{r,\delta}\|_{\mathbb{R}^{2,+}_a}\lesssim \|f\|_{\Omega_p}$, yields 
\begin{align*}
	\|x^{1/2}\nabla \mathcal{J}U^{\nu}_{r,\delta}\|_{\mathbb{R}^{2,+}_a}\lesssim \|f\|_{\Omega_p}.
\end{align*}
\end{proof}
An immediate corollary of Proposition \ref{prop:dyu} and Lemma \ref{lem:fourier_embedding} is the bound 
\begin{align*}
	\|\mathcal{J}^2U^{\nu}_{\delta}\|_{\mathbb{R}^{2,+}_a}\lesssim \|f\|_{\Omega_p}.
\end{align*}
From $u^{\nu}_r=U^{\nu}_{r,\delta}$ on $\Omega_p$, the above and \eqref{eq:Ju} it is immediate that 
\begin{corollary}
	\label{cor:dyu}
	Let  $u^{\nu}_r$ be the unique solution of \eqref{eq:regularized_pb}, where the right-hand side $f\in L^2(\Omega_p)$. With some $C>0$, for all $\nu>0$, it holds that $\|\partial_y u^{\nu}_r\|_{L^2(\Omega_p)}\leq C\|f\|_{L^2(\Omega_p)}$. 	
\end{corollary}
At this point, we will stop working with $U^{\nu}_{r,\delta}$, and get back to the original problem \eqref{eq:regularized_pb}. 
We now proceed by proving the bound \eqref{eq:partialY}. 
\subsubsection{Proof of \eqref{eq:partialY}}
The key auxiliary result is
\begin{proposition}
	\label{prop:nu_second}
For all $\nu>0$,\, $\nu\|\partial_y\nabla u^{\nu}_{r}\|_{\Omega_p}\lesssim \|f\|_{\Omega_p}$. 
\end{proposition}
\begin{proof}
	The proof mimics the proof of Proposition \ref{proposition:bounds_nu_appendix}, cf. in particular \eqref{eq:dhu} and the bounds of Theorem \ref{theorem:convergenceNeumann}.
\end{proof}
\begin{proposition}
	\label{prop:partialY}
	For all $\nu>0$, $\|\partial_y(x\nabla u^{\nu}_{r})\|_{\Omega_p}\leq C\|f\|_{\Omega_p}$. 
\end{proposition}
\begin{proof}	
	We start with the variational formulation \eqref{eq:anur}. We use the Nirenberg's finite difference quotient, defined as $\delta_y^hv(x,y)=\frac{v(x,y+h)-v(x,y)}{h}$, for $v\in \mathcal{H}^1_{\Sigma,0}(\Omega_p)$, where $v$ is extended by periodicity for $|y|>\ell$ (see the proof of Proposition \ref{proposition:bounds_nu_appendix}). 
	
	Let us test \eqref{eq:regularized_pb} with the Nirenberg's finite difference quotient $x\delta_y^h\delta_y^{-h} u^{\nu}_r$, which yields, cf. \eqref{eq:anur}, 
	\begin{align*}
		a_{r}^{\nu}(u^{\nu}_r, x\delta_y^h\delta_{y}^{-h} u^{\nu}_r)=-\int_{\Omega_p}fx\overline{\delta_y^h\delta_{y}^{-h} u^{\nu}_r}. 
	\end{align*} 
The same computations as in Appendix \ref{appendix:lap_auxiliary} show that
\begin{align*}
	&\int_{\Omega_p}(x+i\nu )\mathbb{A}\delta^y_{-h}\nabla u^{\nu}_r\,\overline{x\delta^y_{-h}\nabla u^{\nu}_r}+\int_{\Omega_p}(x+i\nu )\mathbb{A}\delta^y_{-h}\nabla u^{\nu}_r\vec{e}_x\overline{\delta_{-h}^y u^{\nu}_r}\\
	&+
	\int_{\Omega_p}(x+i\nu)(\delta_{-h}^y\mathbb{A})\nabla u^{\nu}_r\left(x\overline{\nabla \delta_{-h}^y u^{\nu}_r}+\overline{\vec{e}_x \delta_{-h}^y u^{\nu}_r}\right)=\int_{\Omega}f x\overline{\delta_h^y \delta_{-h}^y u^{\nu}_r}.
\end{align*}
Taking the real part of the above, and using that, in particular,  $\mathbb{A}\in C^1(\overline{\Omega_p})$, with appropriate periodicity constraints, we conclude that 
\begin{align*}
\|x\delta_{-h}^y\nabla u^{\nu}_r\|^2_{L^2(\Omega_p)}&\lesssim \|x\delta_{-h}^y\nabla u^{\nu}_r\|_{L^2(\Omega_p)}\|\delta_{-h}^y u^{\nu}_r\|_{L^2(\Omega_p)}+\nu\|\delta_{-h}^y\nabla u^{\nu}_r\|_{L^2(\Omega_p)}\|\delta_{-h}^y u^{\nu}_r\|_{L^2(\Omega_p)}\\
&+(\|x\nabla u^{\nu}_r\|_{L^2(\Omega_p)}+\nu\|\nabla u^{\nu}_r\|_{L^2(\Omega_p)})(\|x\nabla \delta_{-h}^y u^{\nu}_r\|_{L^2(\Omega_p)}+\|\delta_{-h}^y u^{\nu}_r\|_{L^2(\Omega_p)})+\|f\|_{L^2(\Omega_p)}\|x\delta_{-h}^y \nabla u^{\nu}_r\|_{L^2(\Omega_p)}.
\end{align*}
Using \eqref{eq:vreg}, \eqref{eq:vreg2}, the estimate of \cite[Lemma 4.13]{mclean} and $\|\partial_y u^{\nu}_r\|_{L^2(\Omega_p)}\lesssim \|f\|_{L^2(\Omega_p)}$ (cf. Corollary \ref{cor:dyu}), we obtain the stated estimate.  
\end{proof}
The stated inequality \eqref{eq:partialY} is a corollary of Propositions \ref{prop:nu_second} and \ref{prop:partialY} together with the estimates of Theorem \ref{theorem:convergenceNeumann}.  
%\subsubsection{Proof of \eqref{eq:dxu} and \eqref{eq:dxxu}}
%Remark that for $\vec{q}^{\nu}=(x+i\nu)\mathbb{A}\nabla u^{\nu}_r$, 
%\begin{align*}
%\|\partial_y ((x+i\nu)\mathbb{A}\nabla u^{\nu}_r)\|=\|x\partial_y(\mathbb{A}\nabla u^{\nu}_r)+i\nu\partial_y(\mathbb{A}\nabla u^{\nu}_r)\|\lesssim \|f\|,
%\end{align*}
%because of Theorem \ref{theorem:regularNeumann}, Corollary \ref{cor:dyu}, Propositions \ref{prop:nu_second}, \ref{prop:partialY}. Because, additionally, $\partial_x{q}^{\nu}_1=f-\partial_y q^{\nu}_2$, we conclude that $\|\partial_x q^{\nu}_1\|\lesssim \|f\|$. 
%
%Since, by the above,  $\|\partial_y q^{\nu}_1\|\lesssim \|f\|$, and $\gamma_0^{\Sigma}q^{\nu}_1=\gamma_n^{\Sigma}u^{\nu}_r=0$, we have that $\gamma_0^{\Sigma}q_{\nu}=0$, and, by Hardy's inequality \eqref{eq:multhardy}, $x^{-1}q^{\nu}_1\in L^2(\Omega_p)$, more precisely,  
%\begin{align*}
%	\int_{\Omega_p}(1+\frac{\nu^2}{x^2})|\mathbb{A}\nabla u^{\nu}_r|^2\lesssim \|f\|^2,
%\end{align*}
%and we conclude that $\|\nabla u^{\nu}_r\|_{L^2(\Omega_p)}\lesssim \|f\|_{L^2(\Omega_p)}$.
%
%
%Finally, let us get back to \eqref{eq:dxxu}. We have used that $\|\partial_x((x+i\nu)\mathbb{A}\nabla u^{\nu}_r)\|\lesssim \|f\|$, which, with uniform boundedness of $\nabla u^{\nu}$, implies that 
%\begin{align*}
%	\|\mathbb{A}(x+i\nu)\partial_x\nabla u^{\nu}_r\|\lesssim \|f\|, \text{ thus }\|(x+i\nu)\partial_x\nabla u^{\nu}_r\|\lesssim \|f\|, 
%\end{align*}
%and \eqref{eq:dxxu} follows from the above. 
\subsubsection{Proof of \eqref{eq:dxu}}
We start by remarking that 
\begin{align*}
	\partial_x((x+i\nu)\mathbb{A}\nabla u^{\nu}_r)=-(x+i\nu)\partial_y(\mathbb{A}\nabla u^{\nu}_r)+f, 
\end{align*}
therefore, with Theorem \ref{theorem:regularNeumann}, Corollary \ref{cor:dyu}, Propositions \ref{prop:nu_second}, \ref{prop:partialY}, and Assumption \ref{assump:matrices}, it holds that $\|\partial_x((x+i\nu)\mathbb{A}\nabla u^{\nu}_r)\|\lesssim \|f\|$. Since
\begin{align*}
\mathbb{A}\partial_x((x+i\nu)\nabla u^{\nu}_r)=	\partial_x((x+i\nu)\mathbb{A}\nabla u^{\nu}_r)-(\partial_x\mathbb{A})\cdot(x+i\nu)\nabla u^{\nu},
\end{align*}
we conclude, using Theorem \ref{theorem:regularNeumann}, about the validity of the inequality \eqref{eq:dxu}. 
%\end{align*}
%and \eqref{eq:dxxu} follows from the above. 

\subsection{Proof of Theorem \ref{theorem:regularity_appendix} for the case $f\in \mathcal{H}^1(\Omega_p)$}
\subsubsection{An auxiliary result}
The proof of this result of \cite{baouendi_goulaouic} uses a Hardy inequality similar to the one of Lemma \ref{lem:hardy_ho}, however, the latter is stated without proof. One proof can be found in Theorem (I.1) of \cite{baouendi}, however, it has to be adapted to take into account that we do not have a vanishing derivative at $0$.
\begin{lemma}
	\label{lem:hardy_ho}
	Let $q\in H^2(0,1)$, $q(0)=0$, $q(1)=0$. Then $x\mapsto q(x)/x\in H^1(0,1)$, with $\|q/x\|_{H^1(0,1)}\lesssim \|q''\|_{L^2(0,1)}$.
\end{lemma}
\begin{proof}
	Assume that $q''=f$, $f\in L^2(0,1)$, then, necessarily, it holds that 
	\begin{align*}
		q(x)=\int_0^x f(t)(x-t)dt+cx, \text{ so that }q(x)/x=\int_0^x f(t)\left(1-\frac{t}{x}\right)dt+c, \quad c=-\int_0^1f(t)(x-t)dt. 
	\end{align*}
By the standard Hardy's inequality, cf. \eqref{eq:multhardy}, $x\mapsto q(x)/x\in L^2(0,1)$. By a direct computation we obtain  
\begin{align*}
(q(x)/x)'=\frac{1}{x^2}\int_0^x f(t)tdt \implies |(q(x)/x)'|\leq \frac{1}{x}\int_0^x |f(t)|dt.
\end{align*}
Applying \eqref{eq:multhardy} to the above, we conclude that $\|(q(x)/x)'\|_{L^2(0,1)}\leq \|f\|_{L^2(0,1)}$.
\end{proof}
\begin{remark}
	By a direct computation, it can be verified that if $q\in H^m(0,1)$, $q(0)=0$, and $q^{(j)}(1)=0$, $j=0, \ldots, m-2$, then $x\mapsto q(x)/x\in H^{m-1}(0,1)$, and $\|q/x\|_{H^m(0,1)}\lesssim \|q^{(m)}\|_{L^2(0,1)}$. 
	
	The chosen boundary condition at $x=1$ is not essential for the validity of the result. 
\end{remark}
By the density of $\mathcal{C}^1_{comp}(\Omega_p)\cap\mathcal{C}^2(\Omega_p)$ in the space $\{v\in \mathcal{H}^1_{\Sigma,0}: \, \partial_x^2 v\in L^2(\Omega_p)\}$, the above extends to $\Omega_p$. 
\begin{corollary}
	\label{cor:dmp}
Assume that $q\in \mathcal{H}^1_{\Sigma,0}(\Omega_p)=\{v\in \mathcal{H}^1(\Omega_p): \, \gamma_0^{\Sigma}v=0\}$, and, additionally, $\partial_x^2 q\in L^2(\Omega_p)$. Then $(x,y)\mapsto \partial_x(q(x,y)/x)\in L^2(\Omega_p)$, moreover, $\|\partial_x (q/x)\|_{L^2(\Omega_p)}\lesssim \|\partial_x^2 q\|_{L^2(0,1)}$.
\end{corollary}

\subsubsection{Proof of Theorem \ref{theorem:regularity_appendix} for $f\in \mathcal{H}^1(\Omega_p)$}
Consider the unique solution $u_r\in \mathcal{H}^1(\Omega_p)$ (exists by the first part of Theorem \ref{theorem:regularity_appendix}) to 
\begin{align}
	\label{eq:divxA}
	\begin{split}
	\operatorname{div}(x\mathbb{A}\nabla u_r)=f, \quad f\in\mathcal{H}^1(\Omega_p), \\
	\gamma_0^{\Sigma}u_r=0, \, \gamma_0^{\Gamma_p}u_r=0, \text{ periodic BCs at }y=\pm\ell.
	\end{split}
\end{align}
The key idea is again to regularize the problem by adding an absorption term and using the Dirichlet boundary conditions on $\Sigma$, i.e. rewrite the problem in the form  \eqref{eq:regularized_pb}. Then by the standard elliptic regularity result, $u^{\nu}_r\in \mathcal{H}^3(\Omega_p)$, see \cite[Theorem 4.18]{mclean} or \cite[Theorem 2.5.1.1]{grisvard_book}. In this case it is not difficult to verify that $\partial_y u^{\nu}_r\in \mathcal{H}^1(\Omega_p)$ satisfies (remark that we used below the periodicity of $\mathbb{A}$ in the direction $y$ to conclude from the original formulation that $\partial_y u^{\nu}_r$ satisfies periodic boundary conditions at $y=\pm\ell$):
\begin{align}
	\label{eq:divxA2}
	\begin{split}
	&\operatorname{div}((x+i\nu)\mathbb{A}\nabla \partial_y u^{\nu}_r)+\operatorname{div}((x+i\nu)(\partial_y\mathbb{A})\nabla u^{\nu}_r)=\partial_y f,\\
	&\gamma_0^{\Sigma}\partial_y u^{\nu}_r=0, \qquad \gamma_0^{\Gamma_p}\partial_y u^{\nu}_r=0, \text{ periodic BCs at }y=\pm \ell. 
	\end{split}
\end{align}
With Theorems \ref{theorem:convergenceNeumann} and \ref{theorem:regularNeumann}, the term below is uniformly bounded in $\nu>0$, provided that $\partial_y\mathbb{A}\in C^{0,1}(\overline{\Omega}_p)$:
\begin{align*} \|\operatorname{div}((x+i\nu)(\partial_y\mathbb{A})\nabla u^{\nu}_r)\|\leq \|\partial_x((x+i\nu)\nabla u^{\nu}_r)\|+\|(x+i\nu)\nabla u^{\nu}\|+\|(x+i\nu)\partial_y \nabla u^{\nu}\|. 
\end{align*}
We conclude that 
\begin{align*}
	&\operatorname{div}((x+i\nu)\mathbb{A}\nabla \partial_y u^{\nu}_r)=\tilde{f}^{\nu},\qquad \|\tilde{f}^{\nu}\|_{L^2(\Omega_p)}\lesssim \|f\|_{\mathcal{H}^1(\Omega_p)},\\
	&\gamma_0^{\Sigma}\partial_y u^{\nu}_r=0, \qquad \gamma_0^{\Gamma_p}\partial_y u^{\nu}_r=0, \text{ periodic BCs at }y=\pm \ell. 
\end{align*}
By Theorem \ref{theorem:convergenceNeumann}, we see that $\partial_y u^{\nu}_r$ is uniformly bounded in $\mathcal{V}_{reg}(\Omega_p)$. Therefore, with Theorem \ref{theorem:regularNeumann}, and using the uniqueness of the weak $L^2(\Omega_p)$-limit, we conclude that $\partial_y u_r\in \mathcal{V}_{reg}(\Omega_p)$. We have extracted the necessary information from \eqref{eq:divxA2}, and now start working with the original problem \eqref{eq:divxA}.

Weakly, $\partial_y u_r\in \mathcal{V}_{reg}(\Omega_p)$ satisfies, cf. \eqref{eq:divxA}, 
\begin{align*}
	&\operatorname{div}(x\mathbb{A}\nabla \partial_y u_r)+\operatorname{div}(x(\partial_y \mathbb{A})\,\nabla u_r)=\partial_y f,\\
	&\gamma_0^{\Sigma}\partial_y u_r=0, \qquad \gamma_0^{\Gamma_p}\partial_y u_r=0, \text{ periodic BCs at }y=\pm \ell. 
\end{align*}
and, as shown in the part 1 of Theorem \ref{theorem:regularity_appendix}, if $\mathbb{A}\in C^{1,1}(\overline{\Omega_p}; \mathbb{C}^{2\times 2})$, then
\begin{align*}
	\|\operatorname{div}(x(\partial_y \mathbb{A})\,\nabla u_r)\|\lesssim \|f\|.
\end{align*}
Thus $	\operatorname{div}(x\mathbb{A}\nabla \partial_y u_r)\in L^2(\Omega_p)$, and we can apply to $\partial_y u_r$ the part 1 of Theorem \ref{theorem:regularity_appendix}. In particular, 
\begin{align}
	\label{eq:bds}
	\|\nabla \partial_y u_r\|_{L^2(\Omega_p)}+\|x \partial_y u_r\|_{H^2(\Omega_p)}\lesssim \|f\|_{H^1(\Omega_p)}. 
\end{align}
It remains to show that the result holds true when $\partial_y u_r$ is replaced in the above by $\partial_x u_r$. More precisely, we want to show that 
\begin{align}
	\label{eq:todo}
	\|\partial_x^2 u_r\|_{L^2(\Omega_p)}+\|x\partial_x^3 u_r\|_{L^2(\Omega_p)}\lesssim \|f\|_{H^1(\Omega_p)}.
\end{align}

 For this, we rewrite \eqref{eq:divxA} once again, to see that 
\begin{align*}
	\partial_x(x\vec{e}_x\cdot \mathbb{A}\nabla u_r)=f-x\vec{e}_y\cdot (\partial_y \mathbb{A})\nabla u_r-x\vec{e}_y\mathbb{A}\cdot \nabla \partial_y u_r,
\end{align*}
so that, with $q:=x\vec{e}_x\cdot \mathbb{A}\nabla u_r$,
\begin{align*}
	\partial_x^2q&=\partial_x f-\vec{e}_y\cdot (\partial_y\mathbb{A})\nabla u_r-x\vec{e}_y\cdot (\partial_{xy}\mathbb{A})\nabla u_r-x\vec{e}_y\cdot \partial_x\nabla u_r-\vec{e}_y\mathbb{A}\nabla \partial_y u_r\\
	&-x\vec{e}_y\partial_x\mathbb{A}\cdot \nabla \partial_y u_r-x\vec{e}_y\cdot \mathbb{A}\nabla \partial_{xy}u_r. 
\end{align*}
From the bounds \eqref{eq:bds}, part 1 of Theorem \ref{theorem:regularity_appendix} and regularity assumptions on $\mathbb{A}$, we conclude that  $\|\partial_x^2 q\|_{L^2(\Omega_p)}\lesssim \|f\|_{H^1(\Omega_p)}.$

Since $\gamma_n^{\Sigma}u_r=0$ by Lemma \ref{cor:uv}, and $q\in \mathcal{H}^1(\Omega_p)$ by the part 1 of Theorem \ref{theorem:regularity_appendix}, we have that $\gamma_0^{\Sigma}q=\gamma_0^{\Sigma}(x\vec{e}_x\mathbb{A}\nabla u_r)=\gamma_n^{\Sigma}u_r=0$. Therefore, $q$ satisfies conditions of Corollary \ref{cor:dmp}, and we conclude that $\|\partial_x(q/x)\|_{L^2(\Omega_p)}=\|\partial_x(\vec{e}_x\cdot \mathbb{A}\nabla u_r)\|_{L^2(\Omega_p)}\lesssim \|f\|_{H^1(\Omega_p)}$. With \eqref{eq:bds} and $\|\nabla u_r\|_{L^2(\Omega_p)}\lesssim \|f\|_{L^2(\Omega_p)}$, we conclude that 
\begin{align}
	\label{eq:partx}
	\|\partial_x^2 u_r\|_{L^2(\Omega_p)}\lesssim \|f\|_{H^1(\Omega_p)}.
\end{align}
Next, rewriting $\partial_x^2 q$ by definition of $q$, we we obtain that 
\begin{align*}
	\partial_x^2q=x\partial_x^2(\vec{e}_x\mathbb{A}\nabla u_r)+2\partial_x(\vec{e}_x\mathbb{A}\nabla u_r),
\end{align*}
and, therefore, with \eqref{eq:partx} and \eqref{eq:bds}, we conclude that 
\begin{align*}
	\|x\partial_x^3 u_r\|_{H^2(\Omega_p)}\lesssim \|f\|_{H^1(\Omega_p)},
\end{align*}
thus the desired bound \eqref{eq:todo}.

	\section{Proof of Theorem \ref{theorem:green}}
\label{appendix:green_formula_to_submit}
Let $\re>0$, $
\Omega^{\re}_{\Sigma}$ be as in \eqref{eq:omegaresigma}, 
and the family of cut-off functions $\varphi_{\varepsilon}$ be defined in \eqref{eq:cutoff_phi}; recall that $\varphi_{\re}=1$ on $\Omega^{\re/2}_{\Sigma}$ and vanishes on $\Omega\setminus\Omega^{\re}_{\Sigma}$; also, it depends on $x$ only.   
Then 
	\begin{align*}
		\int_{\Omega_p}\operatorname{div}( x \mathbb{A}\nabla u)\overline{v}=\lim\limits_{\varepsilon\rightarrow 0+}	\int_{\Omega_p}\operatorname{div}( x \mathbb{A}\nabla u)\overline{v}(1-\varphi_{\varepsilon})d\vec{x}
		=\lim\limits_{\varepsilon\rightarrow 0+}	\int_{\Omega_p\setminus \Omega_{\Sigma}^{\varepsilon/2}}\operatorname{div}( x \mathbb{A}\nabla u)\overline{v}(1-\varphi_{\varepsilon})d\vec{x}.
	\end{align*}
We integrate by parts and use that $(1-\varphi_{\varepsilon})|_{|x|=\re/2}=0$, $v=0$ on $\Gamma_p$:
	\begin{align*}
		\int_{\Omega_p\setminus \Omega_{\Sigma}^{\varepsilon/2}}\operatorname{div}&( x \mathbb{A}\nabla u)\overline{v}(1-\varphi_{\varepsilon})d\vec{x}=-\int_{\Omega_p\setminus \Omega_{\Sigma}^{\varepsilon/2}}x\mathbb{A}\nabla u\cdot \nabla \overline{v}(1-\varphi_{\re})d\vec{x}+\int_{\Omega_p\setminus \Omega_{\Sigma}^{\varepsilon/2}}x\mathbb{A}\nabla u\cdot \vec{e}_x \overline{v}\varphi_{\re}'d\vec{x}\\
		&=\int_{\Omega_p\setminus \Omega^{\re/2}_{\Sigma}}
		u\overline{\operatorname{div}(x\mathbb{A}\nabla v)}(1-\varphi_{\re})-\int_{\Omega_p\setminus \Omega^{\re/2}_{\Sigma}}xu\overline{\mathbb{A}\nabla v}\cdot \vec{e}_x\varphi'_{\re}+\int_{\Omega_p\setminus \Omega_{\Sigma}^{\varepsilon/2}}x\mathbb{A}\nabla u\cdot \vec{e}_x \overline{v}\varphi_{\re}'d\vec{x}. 
 	\end{align*}
	Taking $\lim\limits_{\re\rightarrow 0+}$ of both sides of the above yields the following identity:
	\begin{align}
		\int_{\Omega_p}\operatorname{div}( x \mathbb{A}\nabla u)\overline{v}
		&=\lim\limits_{\varepsilon\rightarrow 0+}\underbrace{\int_{\operatorname{supp}\varphi'_{\varepsilon}} x \vec{e}_x\cdot\left(
			{\mathbb{A}\nabla u}\overline{v}-\overline{\mathbb{A}\nabla v}u\right)\; \varphi'_{\varepsilon}}_{\mathcal{I}^{\varepsilon}}
		\label{eq:Iterm}+\int_{\Omega_p}u\overline{\operatorname{div}( x  \mathbb{A}\nabla v)}.
	\end{align}
	With the decomposition $u=u_h\log|x|+u_{reg}$, $v=v_h\log|x|+v_{reg}$, the term $\mathcal{I}^{\varepsilon}$ rewrites
	\begin{align*}
		\mathcal{I}^{\varepsilon}&=\sum\limits_{j}\mathcal{I}_j^{\varepsilon}, \quad \mathcal{I}^{\varepsilon}_1:=\int_{\operatorname{supp}\varphi'_{\varepsilon}} x \mathbf{e}_x\cdot\left(\mathbb{A}\nabla u_{reg}\, \overline{v_h}-\overline{\mathbb{A}\nabla v_{reg}}{u_h}\right)\log|x|\partial_{x}\varphi_{\varepsilon},\\
		&\mathcal{I}^{\varepsilon}_2:=\int_{\operatorname{supp}\varphi'_{\varepsilon}} x \mathbf{e}_x\cdot\left(\mathbb{A}\nabla u_{reg}\, \overline{v_{reg}}-\overline{\mathbb{A}\nabla v_{reg}}{u_{reg}}\right)\partial_{x}\varphi_{\varepsilon},\\
		&\mathcal{I}^{\varepsilon}_3:=\int_{\operatorname{supp}\varphi'_{\varepsilon}} x \mathbf{e}_x\cdot\left(\mathbb{A}(\log|x|\nabla u_h+\frac{\mathbf{e}_x}{x}u_h)\, \overline{v_{reg}}-\overline{\mathbb{A}(\log|x|\nabla v_h+\frac{\mathbf{e}_x}{x}v_h)}{u_{reg}}\right)\partial_{x}\varphi_{\varepsilon},\\
		&\mathcal{I}^{\varepsilon}_4:=\int_{\operatorname{supp}\varphi'_{\varepsilon}} x \mathbf{e}_x\cdot\left(\mathbb{A}(\log|x|\nabla u_h+\frac{\mathbf{e}_x}{x}u_h)\, \overline{v_{h}}-\overline{\mathbb{A}(\log|x|\nabla v_h+\frac{\mathbf{e}_x}{x}v_h)}{u_{h}}\right)\log|x|\partial_{x}\varphi_{\varepsilon}\\
		&=\int_{\operatorname{supp}\varphi'_{\varepsilon}} x \mathbf{e}_x\cdot\left(\mathbb{A}\nabla u_h \overline{v_{h}}-\overline{\mathbb{A}\nabla v_h}u_h\right)\log^2|x|\partial_{x}\varphi_{\varepsilon}.
	\end{align*}
	In the last identity we used $\mathbf{e}_x\cdot\mathbb{A}\mathbf{e}_x=\mathbf{e}_x\cdot\overline{\mathbb{A}\mathbf{e}_x}=\mathbb{A}_{11}\in \mathbb{R}$, valid since $\mathbb{A}$ is Hermitian. We will see that only $\mathcal{I}^{\re}_3$ will not converge to zero as $\re\rightarrow 0$. Therefore, let us now examine the remaining integrals. 
	
	\textbf{Proof that $\mathcal{I}_j^{\varepsilon}\rightarrow 0$, as $\varepsilon\rightarrow 0$, with $j\in\{1,2,4\}$.} We treat these integrals in a similar manner, therefore, let us combine the relevant estimates. First of all, remark that with $C$ independent on $\re$, it holds
	\begin{align}
		\label{eq:est1}
		\| x \partial_x\varphi_{\varepsilon}\|_{L^{\infty}(\operatorname{supp}\varphi'_{\re})}\leq \| x \partial_x\varphi_{\varepsilon}\|_{L^{\infty}(\Omega^{\re}_{\Sigma})}\leq C\|x/\re\|_{L^{\infty}(\Omega^{\re}_{\Sigma})}\leq C,
	\end{align}
	where we used the definition of $\varphi_{\re}$ and the fact that $\operatorname{supp}\varphi'_{\re}\subseteq \overline{\Omega_{\Sigma}^{\re}}$. 
	Also, for functions $q,p\in \mathcal{H}^1_{\delta}(\Omega_p)$, with $\delta<1$, it holds, for any $\mathbf{e}\in \mathbb{C}^2$, and $a\in \{0,1,2\}$,
	\begin{align}
		\nonumber
		&\int_{\operatorname{supp}\varphi'_{\varepsilon}}|\log^a|x|\;\mathbf{e}\cdot\nabla q \, p|d\vec{x}\lesssim 
		\left(\int_{\operatorname{supp}\varphi'_{\varepsilon}}|x^{\delta/2}\nabla q|^2dx\right)^{1/2}\left(\int_{\operatorname{supp}\varphi'_{\varepsilon}}|\log^a|x|\;x^{-\delta/2}p|^2dx\right)^{1/2}\\
		\label{eq:funeq}
		&\lesssim \left(\int_{\operatorname{supp}\varphi'_{\varepsilon}}|x^{\delta/2}\nabla q|^2dx\right)^{1/2}\|p\|_{\mathcal{V}_{reg}(\Omega_p)}\lesssim \left(\int_{\operatorname{supp}\varphi'_{\varepsilon}}|x^{\delta/2}\nabla q|^2dx\right)^{1/2}\|p\|_{\mathcal{H}^1_{\delta}(\Omega_p)}\rightarrow 0, \quad \text{ as }\varepsilon\rightarrow 0+,
	\end{align} 
	where we used the Cauchy-Schwarz inequality in the first line and the Hardy inequality of Proposition \ref{prop:hardy} in the second line (remark that $\delta<1$). 
	
	By Theorem \ref{theorem:reg_well_posedness}, we have that $u_{reg}, \, v_{reg}, \, u_h, \, v_h\in \mathcal{H}^1_{\delta}(\Omega_p)$, for any $\delta>0$. Now we have all necessary ingredients to prove the desired result. 
	First, to show that $\lim\limits_{\varepsilon\rightarrow 0+}\mathcal{I}_1^{\varepsilon}= 0$, we use \eqref{eq:est1} which yields  
	\begin{align*}
		\mathcal{I}_1^{\re}&\leq C \int_{\operatorname{supp}\varphi'_{\re}}\left(|\vec{e}_x\cdot \mathbb{A}\nabla u_{reg}\,\overline{ v_h}|+|\vec{e}_x\cdot \overline{\mathbb{A}\nabla v_{reg}}\, u_h|\right)|\log x|d\vec{x}\\
		&=C\int_{\operatorname{supp}\varphi'_{\re}}\left(| \mathbb{A}^t\vec{e}_x\cdot\nabla u_{reg}\,\overline{ v_h}|+|\mathbb{A}\vec{e}_x\cdot \overline{\nabla v_{reg}}\, u_h|\right)|\log x|d\vec{x}.
	\end{align*}
	It remains to use the bound \eqref{eq:funeq} twice, first with $\mathbf{e}=\mathbb{A}^t\mathbf{e}_x$, $a=1$, $q=u_{reg}$ and $p=\overline{v_h}$, and then with $\mathbf{e}=\mathbb{A}\mathbf{e}_x$, $a=1$, $q=\overline{v_{reg}}$ and $p=u_h$. The terms $\mathcal{I}_2^{\re}$, $\mathcal{I}_4^{\re}$ are treated similarly.
%	
%	In a similar manner, the same argument using \eqref{eq:est1} and the bound \eqref{eq:funeq} for $a=0$, $\mathbf{e}=\mathbb{A}^t\mathbf{e}_x$, $q=u_{reg}$, $p=\overline{v_{reg}}$ (resp. $\vec{e}=\mathbb{A}\mathbf{e}_x$, $q=\overline{v_{reg}}$, $p=u_{reg}$), shows that $\mathcal{I}^{\varepsilon}_2\rightarrow 0$, as $\varepsilon\rightarrow 0$. Using \eqref{eq:est1}, and taking in \eqref{eq:funeq} $a=2$, $\vec{e}=\mathbb{A}^t\vec{e}_x$, $q=u_h$, $p=\overline{v_h}$ (resp. $a=2$, $\vec{e}=\mathbb{A}\vec{e}_x$, $q=\overline{v_h}$, $p=u_h$), we conclude that $\mathcal{I}^{\varepsilon}_4\rightarrow 0$, as $\varepsilon\rightarrow 0$.

\textbf{Proof that 	$\lim\limits_{\varepsilon\rightarrow 0+}\mathcal{I}_3^{\varepsilon}$ yields the thought boundary terms. }By the same reasoning as above, some terms in the expression for $\mathcal{I}_3^{\re}$ converge to zero, and it holds that $\lim\limits_{\varepsilon\rightarrow 0}\mathcal{I}_3^{\varepsilon}=\lim\limits_{\varepsilon\rightarrow 0}\widetilde{\mathcal{I}_3^{\varepsilon}}$, where
	\begin{align*}
		\widetilde{\mathcal{I}_3^{\varepsilon}}&=\int_{\operatorname{supp}\varphi'_{\varepsilon}}\partial_{x}\varphi_{\varepsilon} \mathbb{A}_{11}(u_h\overline{v}_{reg}-u_{reg}\overline{v}_h)d\vec{x}, \text{ with } \vec{e}_x\cdot \overline{\mathbb{A}\vec{e}_x}=\vec{e}_x\cdot \mathbb{A}\vec{e}_x=\mathbb{A}_{11}.
	\end{align*}
It remains to integrate by parts the above  (the integration by parts is justified by \cite[Theorem 1.5.3.1]{grisvard_book} and Lemma \ref{lem:L1} in   Appendix \ref{appendix:weighted}, using Theorem \ref{theorem:reg_well_posedness} that states that $u_{reg}$, $v_{reg}$ are from $\bigcap_{0<\delta<1/2}\mathcal{H}^1_{\delta}(\Omega_p)$):
	\begin{align*}
		\widetilde{\mathcal{I}_3^{\varepsilon}}
		=-\int_{\Sigma}a_{11}(u_h\overline{v}_{reg}-u_{reg}\overline{v}_h)d\vec{x}-\int_{\Omega_p}\varphi_{\varepsilon}\partial_x\left(\mathbb{A}_{11}(u_h\overline{v}_{reg}-u_{reg}\overline{v}_h)\right)d\vec{x},
	\end{align*}
By the Lebesgue's dominated convergence theorem, as $\re\rightarrow 0$, the last term tends to $0$, therefore
	\begin{align*}
		\lim\limits_{\varepsilon\rightarrow 0+}	\widetilde{\mathcal{I}_3^{\varepsilon}}&=
		-\int_{\Sigma}a_{11}(u_h\overline{v}_{reg}-u_{reg}\overline{v}_h)d\vec{x}=-\int_{\Sigma}\gamma_n^{\Sigma}u\overline{\gamma_0^{\Sigma} v_{reg}}+\int_{\Sigma}\gamma_n^{\Sigma}u_{reg}\overline{\gamma_n^{\Sigma} v_{reg}}.
	\end{align*}
This, together with previous considerations and \eqref{eq:Iterm}, proves the desired result. 

	\subsection{Proof of Propositions \ref{proposition:jbound} and \ref{proposition:bounds_nu}}
\label{appendix:commutators}
The proof of Proposition \ref{proposition:jbound} relies on some auxiliary facts on commutators of $\mathcal{J}$ with multiplication operators. The result would have been easy to obtain, had we considered $\mathbb{A}=\mathbb{T}=\operatorname{Id}$, cf. the proof of Proposition \ref{proposition:jbound_basic}; this seems to be not the case when $\mathbb{A}, \, \mathbb{T}$ are matrices. 
 
Some of the results below are well-known; we chose to present them for the sake of completeness.
\subsubsection{Preliminary results: properties of the Bessel potential}
We start by remarking that $\mathcal{J}$ is a symmetric operator, as follows from the Plancherel identity. 
\begin{lemma}
	\label{lem:plancherel}
	For all $u, v\in H^1(\mathbb{R}^2_a)$, it holds that $
		(\mathcal{J}u,v)_{L^2(\mathbb{R}^2_a)}=(u,\mathcal{J}v)_{L^2(\mathbb{R}^2_a)}.	$
\end{lemma}
Also, the operator $\mathcal{J}$ commutes with multiplication by $y$-independent functions:
\begin{align*}
	\mathcal{J}(\varphi(x)h)=\varphi(x)\mathcal{J}h,\quad \forall h\in H^{1/2}(\mathbb{R}^{2}_a). 
\end{align*}
In what follows we will also need the following property, which can be verified by the density argument. 
Let the trace operator $\gamma_0^{x=\pm a}$ be defined via $\gamma_0^{x=\pm a}u(x,y)=u(\pm a, y)$ for sufficiently regular $u:\,\mathbb{R}^2_a\rightarrow \mathbb{C}$. Then for all $u\in H^2(\mathbb{R}^2_a)$, s.t. $\gamma_0^{x=\pm a}u=0$, it holds that
\begin{align}
	\label{eq:zero_trace}
	\gamma_0^{x=\pm a} \mathcal{J} u=0,\qquad \gamma_0^{x=\pm a} \mathcal{J}^2 u=0.
\end{align}
The result below is non-optimal, but sufficient for our needs.
\begin{lemma}
	\label{lem:commutator}
	Let $n\in \{1,2\}$, $\beta_0\in H^{2}(\mathbb{R}^2_a)$. Then, there exists $C>0$, s.t. for all $p\in L^2(\mathbb{R}_a^2)$, $n\in \{1,2\}$,
	\begin{align*}
		&\|[\mathcal{J}^n,\beta_0]p\|_{L^2(\mathbb{R}^2_a)}\leq C \|p\|_{L^2(\mathbb{R}^2_a)}.
	\end{align*}
%	Similarly, for all $p\in L^2_1(\mathbb{R}_a^2)$, $n\in \{1,2\}$,  
%	\begin{align*}
%		\|[\mathcal{J}^{n},x\beta_0]p\|_{L^2(\mathbb{R}^2_a)}\leq C \|p\|_{L^2_1(\mathbb{R}^2_a)}.
%	\end{align*}
\end{lemma}
\begin{remark}
	\label{rem:regularity}
We will often apply the above result in the case when $\beta_0\in C^{1,1}([-a,a]\times\mathbb{R}; \mathbb{C})$ and is compactly supported.
\end{remark}
\begin{proof}
%	We will prove the first statement only; the second statement follows immediately from the first %statement by remarking that %$[\mathcal{J}^n,x\beta_0]p=[\mathcal{J}^n,\beta_0]x%p$. 
	By the density argument, it suffices to prove that $\|[\mathcal{J}^n,\beta_0]p\|_{L^2(\mathbb{R}^2_a)}\leq C\|p\|_{L^2(\mathbb{R}^2_a)}$, for all $p\in C_0^{\infty}(\mathbb{R}^2_a)$.
	
	We denote by $\xi$ the Fourier variable in the direction $y$ and introduce  $\mu(\xi)=(1+\xi^2)^{1/4}$, so that
	\begin{align}
		\label{eq:jbeta}
		[\mathcal{J}^n,\beta_0]p=\mathcal{F}_y^{-1}\left(\mu^n\mathcal{F}_y\left(\beta_0 p\right)\right)-\beta_0\mathcal{F}_y^{-1}\left(\mu^n\mathcal{F}_y p\right).
	\end{align}
We use the notation $\hat{v}=\mathcal{F}_yv$. Also, for $u\in L^1(\mathbb{R}^2_a)$, $v\in L^2(\mathbb{R}_a^2)$, we denote the convolution in $y$-direction by
	\begin{align*}
		u\ast v(x,y):=\int_{\mathbb{R}}u(x,y-y')v(x,y')dy'.
	\end{align*} 
	With $\mathcal{F}_y(uv)=\hat{u}\ast \hat{v}$, we can rewrite the equation \eqref{eq:jbeta} (recall that we assume that $p\in C_0^{\infty}(\mathbb{R}^2_a)$):
	\begin{align*}
		\mathcal{F}_y\left([\mathcal{J}^n,\beta_0]p\right)(x,\xi)&=	\left(\mu^n(\xi)\hat{\beta}_0\ast\hat{p}-\hat{\beta}_0\ast(\mu^n(\xi)\hat{p})\right)(x,\xi)=\int_{\mathbb{R}}\hat{\beta}_0(x,\xi')(\mu^n(\xi)-\mu^n(\xi-\xi'))\hat{p}(x,\xi-\xi')d\xi'.
	\end{align*}

	%We start by employing the H\"older $C^{0,1/2-\epsilon}$-bound on $m_{\xi}$, that we derive using H\"older inequality, with $\frac{1}{q}+\frac{1}{p}=1$, 
	%\begin{align*}
	%	\left|m_{\xi}-m_{\xi-\xi'}\right|=\left|\int_{\xi-\xi'}^{\xi}m'(t)dt\right|\leq |\xi'|^{1/q}\left|\int_{\xi-\xi'}^{\xi}|m'(t)|^pdt\right|^{1/p}.
	%\end{align*}
	%Let us evaluate 
	%\begin{align*}
	%	\left|\int_{\xi-\xi'}^{\xi}|m'(t)|^pdt\right|^{1/p}=	2^{-1}\left|\int_{\xi-\xi'}^{\xi}\frac{|t|^p}{(1+t^2)^{3p/4}}dt\right|^{1/p}.
	%\end{align*}
	%
	%
	%
	Next, we employ a Lipschitz bound on $\mu^n$, so that $|\mu^n(\xi)-\mu^n(\xi-\xi')|\leq |\xi'|\sup_{t\in [\xi'-\xi, \xi]}\frac{n|t|}{2(1+t^2)^{1-n/4}}\leq C|\xi'|$, \text{ since }$n\in \{1,2\}$. This yields a.e. $(x,\xi)\in \mathbb{R}^2_a$, 
	\begin{align*}
		|\mathcal{F}_y\left([\mathcal{J}^n,\beta]p\right)(x,\xi)|\leq C\int_{\mathbb{R}}|\xi'\hat{\beta}_0(x,\xi')| |\hat{p}(x,\xi-\xi')|d\xi'.
	\end{align*}
	We recognize in the right-hand side of the above a convolution, and use the Young inequality for convolutions: 
	\begin{align}
		\label{eq:mbound}
		\|\mathcal{F}_y\left([\mathcal{J}^n,{\beta}_0]p\right)(x,.)\|_{L^2(\mathbb{R})}^2\leq C \|\hat{p}(x,.)\|_{L^2(\mathbb{R})}^2\, \left(\int_{\mathbb{R}}|\xi\hat{\beta}_0(x,\xi)|d\xi\right)^2, \quad \text{a.e. }x\in (-a,a).
	\end{align}
	Next, we use the Cauchy-Schwarz inequality:
	\begin{align*}
		\left(\int_{\mathbb{R}}|\xi\hat{\beta}_0(x,\xi)|d\xi\right)^2\leq \int_{\mathbb{R}}|\xi \hat{\beta}_0(x,\xi)|^2(1+\xi^2)d\xi \int_{\mathbb{R}}(1+\xi^2)^{-1}d\xi\leq C\|\hat{\beta}_0(x,.)\|_{H^{2}(\mathbb{R})}^2.
	\end{align*}
	Therefore, from \eqref{eq:mbound}, the Plancherel identity, and the above bound, we conclude that 
	\begin{align*}
		\|[\mathcal{J}^n,\beta_0]p\|_{L^2(\mathbb{R}^2_a)}^2=	\|\mathcal{F}_y[\mathcal{J}^n,\beta_0]p\|_{L^2(\mathbb{R}^2_a)}^2\leq C\|p\|^2_{L^2(\mathbb{R}^2_a)}\|\beta_0\|^2_{H^2(\mathbb{R}^2_a)}.
	\end{align*}
\end{proof}
We will also need a corresponding result on a commutator of a self-adjoint operator and $\mathcal{J}$.  
\begin{lemma}
	\label{lem:self_adjoint_commutator}
	Let $E:=(L^2(\mathbb{R}^2_a))^2$. 
	Let $\mathcal{A}: D(\mathcal{A})\rightarrow E$ be a self-adjoint operator, with a domain $D(\mathcal{A})\subset E$. Then, for all $\vec{v}\in E$ s.t. $\mathcal{J}\vec{v}, \mathcal{J}^2\vec{v}\in D(\mathcal{A}) \text{  and   }\mathcal{A}\vec{v}\in \left(H^1(\mathbb{R}^2_a)\right)^2,$ it holds that $$\Im ([\mathcal{J}\mathbb{I}, \mathcal{A}]\vec{v},\mathcal{J} \vec{v})_{E}=\frac{1}{2i}([\mathcal{J}^2\mathbb{I}, \mathcal{A}]\vec{v},\vec{v})_{E}.$$ 
\end{lemma}
\begin{proof}
By the definition of the commutator, for $\vec{v}$ as in the statement of the lemma, 
	\begin{align*}
		\Im ([\mathcal{J}\mathbb{I}, \mathcal{A}]\vec{v},\mathcal{J} \vec{v})_{E}=\Im (\mathcal{J}\mathcal{A}\vec{v}, \mathcal{J}\vec{v})_{E}-\Im (\mathcal{A}\mathcal{J}\vec{v}, \mathcal{J}\vec{v})_{E}.
	\end{align*}
	By self-adjointness, $(\mathcal{A}\vec{u}, \vec{u})_E=(\vec{u}, \mathcal{A}\vec{u})_E=\overline{(\mathcal{A}\vec{u}, \vec{u})_E}$, thus the second term in the above vanishes. Therefore, 
	\begin{align*}
		2i\Im ([\mathcal{J}\mathbb{I}, \mathcal{A}]\vec{v},\mathcal{J} \vec{v})_E=(\mathcal{J}\mathcal{A}\vec{v}, \mathcal{J}\vec{v})_E-(\mathcal{J}\vec{v}, \mathcal{J}\mathcal{A}\vec{v})_E=(\mathcal{J}^2\mathcal{A}\vec{v}, \vec{v})_E-(\mathcal{A}\mathcal{J}^2\vec{v}, \vec{v})_E,
	\end{align*}
	where in the last identity we used first the self-adjointness of $\mathcal{J}$, see Lemma \ref{lem:plancherel}, and next of $\mathcal{A}$.
\end{proof}
\subsubsection{Proof of Proposition \ref{proposition:jbound}}
\label{sec:prop_jbound}
Proposition \ref{proposition:jbound} is a simple corollary of its counterpart with $\mathbb{B}=\mathbb{I}$, as we argue on p. \pageref{proof_prop_jbound}.
\begin{proposition}
	\label{proposition:jbound_basic}
	Let $(u^{\nu})_{\nu>0}$ solve \eqref{eq:unu_orig},  and $U^{\nu}_{\delta}=\mathcal{E}_{\delta}u^{\nu}$ satisfy \eqref{eq:Unudelta}. Then there exists $C>0$, s.t.  
	\begin{align}
		\label{eq:lemboundnu}
		\nu\|\mathcal{J}\nabla U^{\nu}_{\delta}\|^2\leq C \left(\|f\|^2+
		\|f\|\|\partial_y u^{\nu}\|\right), \text{ for all }0<\nu<1.
	\end{align}
	%	Also, 
	%	\begin{align}
		%		\label{eq:lemboundnu2}
		%		\nu\|\mathcal{J}\nabla U^{\nu}_{\delta})\|^2\lesssim  \|f\|^2+
		%		\|f\|\|\partial_y u^{\nu}\|.
		%	\end{align}
\end{proposition}
\begin{proof}
To prove \eqref{eq:lemboundnu}, we rewrite the extended problem \eqref{eq:Unudelta} in a more convenient form. In particular, we decompose the matrix $\mathbb{A}=\mathbb{D}+\mathbb{H}$, with $\mathbb{D}=\operatorname{diag}\mathbb{A}=\operatorname{diag}(\mathbb{A}_{11}, \mathbb{A}_{22})$. Next, we define the modified matrices:
\begin{align}
	\label{eq:ptilde}
	\begin{split}
		\widetilde{\mathbb{A}}&:=\mathbb{A}_{11}^{-1}\mathbb{A}\chi_{\ell,2\delta}+\mathbb{I}(1-\chi_{\ell,2\delta})=\mathbb{I}+\mathbb{E}_{\mathbb{A}}, \quad \mathbb{E}_{\mathbb{A}}=(\mathbb{A}_{11}^{-1}\mathbb{A}-\mathbb{I})\chi_{\ell,2\delta},\\
		\widetilde{\mathbb{T}}&:=\mathbb{A}_{11}^{-1}\mathbb{T}\chi_{\ell,2\delta}+\mathbb{I}(1-\chi_{\ell,2\delta})=\mathbb{I}+\mathbb{E}_{\mathbb{T}}, \quad \mathbb{E}_{\mathbb{T}}=(\mathbb{A}_{11}^{-1}\mathbb{A}-\mathbb{I})\chi_{\ell,2\delta},\\
		\widetilde{\mathbb{D}}&:=\mathbb{A}_{11}^{-1}\mathbb{D}\chi_{\ell,2\delta}+\mathbb{I}(1-\chi_{\ell,2\delta})=\mathbb{I}+\mathbb{E}_{\mathbb{D}}, \quad \mathbb{E}_{\mathbb{D}}=(\mathbb{A}_{11}^{-1}\mathbb{D}-\mathbb{I})\chi_{\ell,2\delta}, \\
		\widetilde{\mathbb{H}}&:=\mathbb{A}_{11}^{-1}\mathbb{H}\chi_{\ell,2\delta}+\mathbb{I}(1-\chi_{\ell,2\delta})=\mathbb{I}+\mathbb{E}_{\mathbb{H}}, \quad \mathbb{E}_{\mathbb{H}}=(\mathbb{A}_{11}^{-1}\mathbb{H}-\mathbb{I})\chi_{\ell,2\delta}
	\end{split}
\end{align}  
In the above, $\chi_{\ell,2\delta}$ is the same truncation function as in \eqref{eq:childelta}. The above defined matrix-valued functions are constant and equal to $\mathbb{I}$ for $|y|\geq \ell+2\delta$. The matrices $\widetilde{\mathbb{A}}, \widetilde{\mathbb{T}}, \widetilde{\mathbb{D}}, \widetilde{\mathbb{H}}$ are Hermitian and positive definite, and are $C^{1,1}([-a,a]\times \mathbb{R}; \mathbb{C}^{2\times 2})$. 
%In particular, for each $\vec{x}\in \mathbb{R}^2_a$, $\vec{p}\in \mathbb{C}^2$, 
%\begin{align*}
%	\begin{split}
%		\widetilde{\mathbb{T}}(\vec{x})\vec{p}\cdot 	\overline{\vec{p}}&\geq \chi_{\ell,2\delta}(y)\mathbb{A}_{11}^{-1}c_{\mathbb{T}}(\vec{x})|\vec{p}|^2+(1-\chi_{\ell,2\delta}(y))|\vec{p}|^2\\
%		&\geq 
%		\min(c_{\mathbb{T}}(\vec{x}), 1)(\chi_{\ell,2\delta}(y)|\vec{p}|^2+(1-\chi_{\ell,2\delta}(y))|\vec{p}|^2)=
%		\min(\min_{\vec{x}}c_{\mathbb{T}}(\vec{x}),1)|\vec{p}|^2,
%	\end{split}
%\end{align*}
%where in the last inequality we used $\chi_{\ell,2\delta}\in [0, 1]$.
Moreover, $\widetilde{\mathbb{D}}_{11}(\vec{x})=1$ for all $\vec{x}\in \mathbb{R}^2_a$. As we will see further, this will allow us to avoid appearance of  $\|F^{\nu}_{\delta}\|_{\mathbb{R}^2_a}\|\partial_x \mathcal{J}U^{\nu}_{\delta}\|_{\mathbb{R}^2_a}$ in the right-hand side, which would have prevented us from obtaining the sharp bound \eqref{eq:lemboundnu}. 
%
%Next, for $(x,y): \, |y|>\ell+2\delta$, the matrix $\widetilde{\mathbb{P}}$ becomes constant (actually, identity). Finally, it is straightforward to see that $\widetilde{\mathbb{P}}$ preserve its hermitian and positive definiteness properties: e.g. for $\mathbb{P}=\mathbb{T}$, for each $y\in \mathbb{R}$, $\vec{p}\in \mathbb{C}^2$, 

Let us now rewrite the problem satisfied by $U^{\nu}_{\delta}$ \eqref{eq:Unudelta}. We start by remarking that, as  $\operatorname{supp}U^{\nu}_{\delta}, \operatorname{supp}F^{\nu}_{\delta}\subseteq\operatorname{supp}\chi_{\ell,\delta}$, and  $\chi_{\ell,2\delta}=1$ on $\operatorname{supp}\chi_{\ell,\delta}$, the matrices $\mathbb{A}$ and $\mathbb{T}$ in \eqref{eq:Unudelta} can be replaced by $\mathbb{A}_{11}\widetilde{\mathbb{A}}, \, \mathbb{A}_{11}\widetilde{\mathbb{T}}$:
\begin{align*}
	\operatorname{div}\left(\left(x\mathbb{A}+i\nu\mathbb{T}\right)\nabla U^{\nu}_{\delta}\right)&=	\operatorname{div}\left(\mathbb{A}_{11}\left(x\widetilde{\mathbb{A}}+i\nu\widetilde{\mathbb{T}}\right)\nabla U^{\nu}_{\delta}\right)=\nabla \mathbb{A}_{11}\cdot \left(x\widetilde{\mathbb{A}}+i\nu\widetilde{\mathbb{T}}\right)\nabla U^{\nu}_{\delta}+	\mathbb{A}_{11}\operatorname{div}\left(\left(x\widetilde{\mathbb{A}}+i\nu\widetilde{\mathbb{T}}\right)\nabla U^{\nu}_{\delta}\right).
\end{align*}
The above can be rewritten as 
\begin{align}
	\label{eq:neweq}
	\operatorname{div}\left(\left(x\widetilde{\mathbb{A}}+i\nu\widetilde{\mathbb{T}}\right)\nabla U^{\nu}_{\delta}\right)=\widetilde{F}^{\nu}_{\delta},\quad \text{ in }\mathbb{R}^2_a,
\end{align}
where the right-hand side $\widetilde{F}^{\nu}_{\delta}$ satisfies the following bound, with some $C_j>0$, $j=1,2$, independent of $\nu>0$:
\begin{align}
	\label{eq:fnudelta_bound}
	\|\widetilde{F}^{\nu}_{\delta}\|_{L^2(\mathbb{R}^2_a)}\leq C_1(\|x\nabla U^{\nu}_{\delta}\|_{L^2(\mathbb{R}^2_a)}+\nu\|\nabla U^{\nu}_{\delta}\|_{L^2(\mathbb{R}^2_a)})+\|F^{\nu}_{\delta}\|_{L^2(\mathbb{R}^2_a)}\leq C_2\|f\|_{L^2(\Omega)}, 
\end{align}
with the latter bounds following from \eqref{eq:stabU} and \eqref{eq:stabFdelta}. Moreover, $\operatorname{supp}\widetilde{F}^{\nu}_{\delta}\subseteq {\operatorname{supp}}\chi_{\ell,2\delta}.$

We test \eqref{eq:neweq} with $\mathcal{J}^2U^{\nu}_{\delta}$, which belongs to $H^1(\mathbb{R}^2_a)$, because of the identity \eqref{eq:Ju}  and the elliptic regularity for $U^{\nu}_{\delta}$ as stated in Lemma \ref{lem:ellreg}; moreover, $\mathcal{J}^2U^{\nu}_{\delta}\in H^1_0(\mathbb{R}^2)$ according to \eqref{eq:zero_trace}. Integrating by parts, we get
\begin{align*}
	\int_{\mathbb{R}^2_a}(x\widetilde{\mathbb{A}}&+i\nu\widetilde{\mathbb{T}})\nabla U^{\nu}_{\delta}\cdot \overline{\mathcal{J}^2\nabla U^{\nu}_{\delta}}=-\int_{\mathbb{R}^2_a}\widetilde{F}^{\nu}_{\delta}\overline{\mathcal{J}^2 U^{\nu}_{\delta}}.
\end{align*}
With Lemma \ref{lem:plancherel} on the symmetry of $\mathcal{J}$, we obtain the new identity:
\begin{align}
	\label{eq:ident}
	\begin{split}
		\int_{\mathbb{R}^2_a}(x\widetilde{\mathbb{A}}&+i\nu\widetilde{\mathbb{T}}) \mathcal{J}\nabla U^{\nu}_{\delta}\cdot \overline{\mathcal{J}\nabla U^{\nu}_{\delta}}+
		\overbrace{\int_{\mathbb{R}^2_a}[\mathcal{J}\mathbb{I},(x\widetilde{\mathbb{A}}+i\nu\widetilde{\mathbb{T}})]\nabla U^{\nu}_{\delta}\cdot \overline{\nabla \mathcal{J}U^{\nu}_{\delta}}}^{\mathcal{S}}=-\int_{\mathbb{R}^2_a}\widetilde{F}^{\nu}_{\delta}\overline{\mathcal{J}^2 U^{\nu}_{\delta}}.	
	\end{split}	
\end{align}
Taking the imaginary part of \eqref{eq:ident}, and using the fact that $\widetilde{\mathbb{A}}$ is a Hermitian matrix, and $\widetilde{\mathbb{T}}$ is Hermitian positive definite, yields the following inequality, with some $C>0$,
\begin{align}
	\label{eq:fbound}
	\nu\|\mathcal{J}\nabla U^{\nu}_{\delta}\|_{L^2(\mathbb{R}^2_a)}^2\leq C\left( |\Im \mathcal{S}|_{L^2(\mathbb{R}^2_a)}+\|\widetilde{F}^{\nu}_{\delta}\|_{L^2(\mathbb{R}^2_a)}\|\mathcal{J}^2 U^{\nu}_{\delta}\|_{L^2(\mathbb{R}^2_a)}\right).
\end{align}
The first term $|\Im\mathcal{S}|$ in the right-hand side of the above needs to be treated with care, since a na\"ive bound using Lemma \ref{lem:commutator} and \eqref{eq:stabU} would allow to replace  in \eqref{eq:fbound} $|\Im S|$ by  $\|F^{\nu}_{\delta}\|_{L^2(\mathbb{R}^2_a)}\|\nabla\mathcal{J}U^{\nu}_{\delta}\|_{L^2(\mathbb{R}^2_a)}$, which would not yield \eqref{eq:lemboundnu} because of the loss of a half a power of $\nu$.  

\textbf{Bounding $\Im\mathcal{S}$. }		
We rewrite $
	\mathcal{S}=\sum_{j}\mathcal{S}_j, $
with $\mathcal{S}_1, \mathcal{S}_2, \mathcal{S}_3$ defined as follows. With $\vec{v}=\nabla U^{\nu}_{\delta}$, and recalling that  $\widetilde{\mathbb{A}}=\widetilde{\mathbb{D}}+\widetilde{\mathbb{H}}$, we have that  
\begin{align}
	\label{eq:s123}
	\mathcal{S}_1=\int_{\mathbb{R}^2_a}x[\mathcal{J}\mathbb{I}, \widetilde{\mathbb{D}}]\vec{v}\cdot \overline{\mathcal{J}\vec{v}},\qquad 
	\mathcal{S}_2=\int_{\mathbb{R}^2_a}x[\mathcal{J}\mathbb{I}, \widetilde{\mathbb{H}}]\vec{v}\cdot \overline{\mathcal{J}\vec{v}},\qquad 			\mathcal{S}_3=i\nu \int_{\mathbb{R}^2_a}x[\mathcal{J}\mathbb{I}, \widetilde{\mathbb{T}}]\vec{v}\cdot \overline{\mathcal{J}\vec{v}}.
\end{align} 
\textbf{A bound on $\Im \mathcal{S}_1$. }We have, since $\widetilde{\mathbb{D}}_{11}=1$, see  \eqref{eq:ptilde}, and $\widetilde{\mathbb{D}}$ is a diagonal matrix:
\begin{align}
	\label{eq:I1}
	\Im\mathcal{S}_1=\Im \int_{\mathbb{R}^2_a}x[\mathcal{J}\mathbb{I}, \widetilde{\mathbb{D}}]\vec{v}\cdot \overline{\mathcal{J}\vec{v}}&=	\Im \int_{\mathbb{R}^2_a}x[\mathcal{J},\widetilde{\mathbb{D}}_{22}]v_2\cdot \overline{\mathcal{J}v_2}=\Im \int_{\mathbb{R}^2_a}x[\mathcal{J},\mathbb{E}_{\mathbb{D},22}]v_2\cdot \overline{\mathcal{J}v_2}.
\end{align}
Next, we make use of Lemma \ref{lem:self_adjoint_commutator}, with $\mathcal{A}\vec{v}=x\mathbb{E}_{\mathbb{D},22}v_2$ (we recall that $\mathbb{A}$ is Hermitian, thus $\mathbb{D}$, $\widetilde{\mathbb{D}}$, ${\mathbb{E}}_{\mathbb{D},22}$ are real-valued). This yields
\begin{align*}
	2i\Im \mathcal{S}_1&=\int_{\mathbb{R}^2_a}x[\mathcal{J}^2, \mathbb{E}_{\mathbb{D},22}]v_2\overline{v_2}.
\end{align*}
Next, we employ Lemma \ref{lem:commutator} with compactly supported  $\beta_0=\mathbb{E}_{\mathbb{D},22}=(\mathbb{A}_{22}/\mathbb{A}_{11}-1)\chi_{\ell,2\delta}$, cf. Remark \ref{rem:regularity}, and use the regularity Assumption \ref{assump:matrices}, which results in 
\begin{align*}
	|\Im \mathcal{S}_1|\leq C \|xv_2\|_{L^2(\mathbb{R}^2_a)}\|v_2\|_{L^2(\mathbb{R}^2_a)}.
\end{align*}
Recalling that $v_2=\partial_y U^{\nu}_{\delta}$, the bound \eqref{eq:stabU}, namely $\|x\nabla U^{\nu}_{\delta}\|_{L^2(\mathbb{R}^2_a)}\lesssim \|f\|_{L^2(\Omega)}$, and the bound \eqref{eq:important_mappings} on $\|\partial_y U^{\nu}_{\delta}\|_{L^2(\mathbb{R}^2_a)}$, we conclude that 
\begin{align}
	\label{eq:boundI1}
	\left|\Im \mathcal{S}_1\right|\leq C (\|u^{\nu}\|_{L^2(\Omega)}+\|\partial_y u^{\nu}\|_{L^2(\Omega)})\|f\|_{L^2(\Omega)}.
\end{align}
Remark that it is due to our rewriting \eqref{eq:neweq} that $\widetilde{\mathbb{D}}=1$, and we do not have terms of the type $\|x\partial_x U^{\nu}_{\delta}\|\|\partial_x U^{\nu}_{\delta}\|$ occurring in the bound on $\Im \mathcal{S}_1$.\\
\textbf{A bound on $\Im \mathcal{S}_2$. }
Let us consider
\begin{align}
	\label{eq:I2}
	\Im  \mathcal{S}_2=\operatorname{Im}\int_{\mathbb{R}^2_a}x[\mathcal{J}\mathbb{I}, \widetilde{\mathbb{H}}]\vec{v}\cdot\overline{\mathcal{J}\vec{v}}=\frac{1}{2i}
	\left(x[\mathcal{J}^2\mathbb{I}, \widetilde{\mathbb{H}}]\vec{v}, \vec{v}\right),
\end{align}
where the last identity follows again by Lemma \ref{lem:self_adjoint_commutator} with $\mathcal{A}\vec{v}=\widetilde{\mathbb{H}}\vec{v}$,  and using the fact that $\widetilde{\mathbb{H}}$ is Hermitian. 
Since $\operatorname{diag}\widetilde{\mathbb{H}}=0$, and using the decomposition \eqref{eq:ptilde}, we rewrite the above as follows:
\begin{align*}
	|\Im \mathcal{S}_2|&\lesssim \|x[\mathcal{J}^2, \widetilde{\mathbb{H}}_{21}]v_1\|_{L^2(\mathbb{R}^2_a)}\|v_2\|_{L^2(\mathbb{R}^2_a)}+\|[\mathcal{J}^2, \widetilde{\mathbb{H}}_{12}]v_2\|_{L^2(\mathbb{R}^2_a)}\|xv_1\|_{L^2(\mathbb{R}^2_a)}\\
	&=\|x[\mathcal{J}^2, \mathbb{E}_{\mathbb{H},21}]v_1\|_{L^2(\mathbb{R}^2_a)}\|v_2\|_{L^2(\mathbb{R}^2_a)}+\|[\mathcal{J}^2, \mathbb{E}_{\mathbb{H},12}]v_2\|_{L^2(\mathbb{R}^2_a)}\|xv_1\|_{L^2(\mathbb{R}^2_a)}.
\end{align*}
By the repeated application of Lemma \ref{lem:commutator}, first with $\beta_0=\mathbb{E}_{\mathbb{H},21}$, next with $\beta_0=\mathbb{E}_{\mathbb{H},12}$, and using $\vec{v}=\nabla U^{\nu}_{\delta}$,  
\begin{align}
	\label{eq:boundI2}
	|\Im \mathcal{S}_2|\lesssim \|x\partial_x U^{\nu}_{\delta}\|_{L^2(\mathbb{R}^2_a)}\|\partial_y U^{\nu}_{\delta}\|_{L^2(\mathbb{R}^2_a)}\lesssim \|f\|_{L^2(\Omega)}(\|u^{\nu}\|_{L^2(\Omega)}+\|\partial_y u^{\nu}\|_{L^2(\Omega)}),
\end{align}	
with the last bound obtained like in \eqref{eq:boundI1}.\\
\textbf{A bound on $\Im \mathcal{S}_3$. }It remains to consider the remaining term, namely,  
\begin{align}
	\label{eq:I3}
	\Im\mathcal{S}_3=\nu\Re \int_{\mathbb{R}_a^2}[\mathcal{J}\mathbb{I}, \mathbb{\widetilde{T}}]\vec{v}\cdot\overline{\mathcal{J}\vec{v}}=\nu\Re \int_{\mathbb{R}_a^2}[\mathcal{J}\mathbb{I}, \mathbb{E}_{\mathbb{T}}]\vec{v}\cdot\overline{\mathcal{J}\vec{v}}.
\end{align}
With the Cauchy-Schwarz inequality, we obtain that 
\begin{align*}
	|\Im \mathcal{S}_3|\leq \nu \|[\mathcal{J}\mathbb{I}, \mathbb{E}_{\mathbb{T}}]\vec{v}\|_{L^2(\mathbb{R}^2_a)}\|{\mathcal{J}\vec{v}}\|_{L^2(\mathbb{R}^2_a)}.
\end{align*}
Next we use Lemma \ref{lem:commutator} (where $\beta_0=\mathbb{E}_{\mathbb{T},ij},\; i,j\in\{1,2\}$), which gives, together with $\vec{v}=\nabla U^{\nu}_{\delta}$,  
\begin{align}
	\label{eq:boundI3}
	|\Im \mathcal{S}_3|\lesssim \nu \|\nabla U^{\nu}_{\delta}\|_{L^2(\mathbb{R}^2_a)}\|\mathcal{J}\nabla U^{\nu}_{\delta}\|_{L^2(\mathbb{R}^2_a)}\overset{\eqref{eq:stabU}}{\lesssim} \nu^{1/2}\|f\|_{L^2(\Omega)}\|\mathcal{J}\nabla U^{\nu}_{\delta}\|_{L^2(\mathbb{R}^2_a)}.
\end{align}	
\textbf{The final bound on $\Im \mathcal{S}$. }	Gathering \eqref{eq:boundI1}, \eqref{eq:boundI2}, \eqref{eq:boundI3} into $\mathcal{S}=\sum\limits_{j=1}^3\mathcal{S}_j$, we conclude that 
\begin{align}
	\label{eq:boundImI}
	|\Im \mathcal{S}|\leq \nu^{1/2}\|f\|_{L^2(\Omega)}\|\mathcal{J}\nabla U^{\nu}_{\delta}\|_{L^2(\mathbb{R}^2_a)}+\|f\|_{L^2(\Omega)}(\|u^{\nu}\|_{L^2(\Omega)}+\|\partial_y u^{\nu}\|_{L^2(\Omega)}).
\end{align}
\textbf{Bounding $\mathcal{J}\nabla U^{\nu}_{\delta}$. }We get back to \eqref{eq:fbound}. We employ  \eqref{eq:boundImI}, the bound for $\widetilde{F}^{\nu}_{\delta}$ in \eqref{eq:fnudelta_bound}, as well as \eqref{eq:Ju}, namely, $\|\mathcal{J}^2U^{\nu}_{\delta}\|_{L^2(\mathbb{R}^2_a)}\leq \left(\|U^{\nu}_{\delta}\|_{L^2(\mathbb{R}^2_a)}+\|\partial_y U^{\nu}_{\delta}\|_{L^2(\mathbb{R}^2_a)}\right)$. This results in the following inequality:
\begin{align*}
	\nu\|\mathcal{J}\nabla U^{\nu}_{\delta}\|^2_{L^2(\mathbb{R}^2_a)}&\lesssim \nu^{1/2}\|f\|_{L^2(\Omega)}\|\mathcal{J}\nabla U^{\nu}_{\delta}\|_{L^2(\mathbb{R}^2_a)}+\|f\|_{L^2(\Omega)}\left(\|\partial_y u^{\nu}\|_{L^2(\Omega)}+\|u^{\nu}\|_{L^2(\Omega)}\right)\\
	&+\|f\|_{L^2(\Omega)}(\|U^{\nu}_{\delta}\|_{L^2(\mathbb{R}^2_a)}+\|\partial_y U^{\nu}_{\delta}\|_{L^2(\mathbb{R}^2_a)}).
\end{align*}
Next, we use \eqref{eq:important_mappings}, which allows to replace $U^{\nu}_{\delta}$ by $u^{\nu}$ in the last term:
\begin{align*}
	\nu\|\mathcal{J}\nabla U^{\nu}_{\delta}\|^2_{L^2(\mathbb{R}^2_a)}&\lesssim \nu^{1/2}\|f\|_{L^2(\Omega)}\|\mathcal{J}\nabla U^{\nu}_{\delta}\|_{L^2(\mathbb{R}^2_a)}+\|f\|_{L^2(\Omega)}\left(\|\partial_y u^{\nu}\|_{L^2(\Omega)}+\|u^{\nu}\|_{L^2(\Omega)}\right).
\end{align*}
Finally, with \eqref{eq:est_main_v2}, we deduce that the following inequality holds true with some $C>0$ independent of $\nu$, uniformly in $0<\nu<1$:
\begin{align*}
	\nu\|\mathcal{J}\nabla U^{\nu}_{\delta}\|^2_{L^2(\mathbb{R}^2_a)}\leq C( \|f\|_{L^2(\Omega)}\times \nu^{1/2}\|\mathcal{J}\nabla U^{\nu}_{\delta}\|_{L^2(\mathbb{R}^2_a)}+\|f\|_{L^2(\Omega)}\|\partial_y u^{\nu}\|_{L^2(\Omega)}+\|f\|^2_{L^2(\Omega)})
\end{align*}
and the desired bound follows by applying the Young's inequality to the above. 
\end{proof}
Now we are fully equipped to prove Proposition \ref{proposition:jbound}.
\paragraph{Proof of Proposition \ref{proposition:jbound}}
\label{proof_prop_jbound}
Since $U^{\nu}_{\delta}$ vanishes outside of $\operatorname{supp}\chi_{\ell,2\delta}$, we can define, like in the proof of Proposition \ref{proposition:jbound_basic}, cf. \eqref{eq:ptilde}, 
$
	\widetilde{\mathbb{B}}=\mathbb{B}\chi_{\ell,2\delta}+\mathbb{I}(1-\chi_{\ell,2\delta})=(\mathbb{B}-\mathbb{I})\chi_{\ell,2\delta}+\mathbb{I}, $
so that $\mathbb{B}\nabla U^{\nu}_{\delta}=\widetilde{\mathbb{B}}\nabla U^{\nu}_{\delta}$. Then 
\begin{align*}
	\nu^{1/2}\mathcal{J}(\mathbb{B}\nabla U^{\nu}_{\delta})&=	\nu^{1/2}\mathcal{J}(\widetilde{\mathbb{B}}\nabla U^{\nu}_{\delta})=\nu^{1/2}\widetilde{\mathbb{B}}\mathcal{J}\nabla U^{\nu}_{\delta}+
	[\mathcal{J}\mathbb{I}, \widetilde{\mathbb{B}}]\nabla U^{\nu}_{\delta}=\nu^{1/2}\widetilde{\mathbb{B}}\mathcal{J}\nabla U^{\nu}_{\delta}+\nu^{1/2}[\mathcal{J}\mathbb{I}, \left(\mathbb{B}-\mathbb{I}\right)\chi_{\ell,2\delta}]\nabla U^{\nu}_{\delta}.
\end{align*}
The first term is bounded with Proposition \ref{proposition:jbound_basic}. For the second term we use Lemma \ref{lem:commutator}, cf.  Remark \ref{rem:regularity}: $$\nu^{1/2}\|[\mathcal{J}\mathbb{I}, \left(\mathbb{B}-\mathbb{I}\right)\chi_{\ell,2\delta}]\nabla U^{\nu}_{\delta}\|_{L^2(\mathbb{R}^2_a)}\lesssim \nu^{1/2}\|\nabla U^{\nu}_{\delta}\|_{L^2(\mathbb{R}^2_a)}\overset{\eqref{eq:stabU}}{\lesssim}\|f\|.$$
\subsubsection{Proof of Proposition \ref{proposition:bounds_nu}}
\label{sec:proof_bounds_nu}
Proposition \ref{proposition:bounds_nu} is again a corollary of its counterpart for $u^{\nu}$, with $\mathbb{B}=\mathbb{I}$, namely
\begin{proposition}
	\label{proposition:bounds_nu_appendix}
	There exists $C>0$, s.t. for all $\nu>0$, the following holds true: 
	\begin{align*}	
		\nu\|\nabla \partial_y u^{\nu}\|+\nu^{3/2} \|\nabla \partial_x u^{\nu}\|\leq C\|f\|.
	\end{align*}
\end{proposition}
\begin{proof}
	\textbf{Proof of the bound $\nu\|\nabla \partial_y u^{\nu}\|\lesssim \|f\|$.} We use the Nirenberg's quotient techniques. 
	For $v$ sufficiently regular, and $h\in \mathbb{R}\setminus\{0\}$ small enough, we define the Nirenberg's quotient which takes into account periodic boundary conditions, in the following manner:
	\begin{align*}
		\delta_h^y v=\frac{\tau_h^y v-v}{h}, \quad \tau_h^y v=\left\{
		\begin{array}{ll}
			v(x,y+h), & y+h \text{ and } y\in [-\ell,\ell],\\
			v(x,y-2\ell+h), & y+h>\ell,\, y\in [-\ell,\ell],\\
			v(x,y+2\ell+h), & y+h<-\ell,\, y\in [-\ell,\ell].
		\end{array}
		\right.
	\end{align*}
	Remark that $\delta_h^y v(x, \ell)=\delta_h^y v(x,-\ell)$ for $v$ s.t. $v(x,\ell)=v(x,-\ell)$. 
	
	Next, we test the problem with absorption \eqref{eq:unu_orig} with $\delta_h^y\delta_{-h}^y u^{\nu}$ and integrate by parts, first at the continuous level, and next at the discrete level; we make use of periodic boundary conditions as well. This allows to obtain the following identity:
	\begin{align*}
		\int_{\Omega}( x \mathbb{A}+i\nu\mathbb{T})\nabla \delta_{-h}^yu^{\nu}\cdot\overline{\nabla \delta_{-h}^yu^{\nu}}+\int_{\Omega}\left(\delta_{-h}^y  \left(x  \mathbb{A}+i\nu \mathbb{T}\right)\right)\nabla u^{\nu}\cdot\overline{\nabla \delta_{-h}^y u^{\nu}}=\int_{\Omega}f\,  \overline{\delta_h^y\delta_{-h}^yu^{\nu}}.
	\end{align*}
	Next, we take the imaginary part of the above, use that $\mathbb{A}=\mathbb{A}^*$ and $\mathbb{T}=\mathbb{T}^*$,  and bound the sign-indefinite terms (namely the second term in the left-hand side) with the Cauchy-Schwarz inequality:
	\begin{align}
		\label{eq:dhu}
		\begin{split}
			\nu\|\nabla \delta_{-h}^yu^{\nu}\|^2&\lesssim \left(\max_{i,j}\|\delta_{-h}^y\mathbb{A}_{ij}\|_{L^{\infty}(\Omega)}\|x \, \nabla u^{\nu}\|+\max_{i,j}\|\delta_{-h}^y\mathbb{T}_{ij}\|_{L^{\infty}(\Omega)}\nu \|\nabla u^{\nu}\|\right)\|\nabla \delta_{-h}^yu^{\nu}\|+\|f\|\|\delta_{h}^y\delta_{-h}^yu^{\nu}\|.
		\end{split} 
	\end{align}
	We conclude, using the stability estimate of Theorem \ref{theorem:stability_estimate} and the estimate \eqref{eq:est_main_v2}:
	\begin{align*}
		\nu\|\nabla \delta_{-h}^yu^{\nu}\|^2\lesssim  \|f\|\|\nabla \delta_{-h}^yu^{\nu}\|+\|f\|\|\partial_y\delta_{-h}^yu^{\nu}\|,
	\end{align*}
	where we employed as well \cite[Lemma 4.13]{mclean} to bound $\|\delta_h^y .\|$ by $\|\partial_y.\|$. The above yields immediately 
$
		\nu\|\nabla \delta_{-h}^y u^{\nu}\|\lesssim \|f\|,$
	and next we use again \cite[Lemma 4.13]{mclean} on the connection of difference quotients and derivatives, we obtain the first part of the desired bound in the statement of the proposition.\\
	\textbf{Proof of the bound $\nu^{3/2} \|\nabla \partial_x u^{\nu}\|\lesssim \|f\|$. } 
	%	Since the proof is quite standard and follows the same ideas as obtaining the bound in \textit{Step 1}, we just outline it here, and leave the details to the reader. 
	We start by proving the desired bound in $\Omega_{\Sigma}^{\varepsilon}$, $\varepsilon>0$.
	We consider the problem \eqref{eq:unu_orig} written for  $u^{\nu}_{\varepsilon}(x,y):=u^{\nu}(x,y)\varphi_{\varepsilon}(x)$, where $\varphi_{\varepsilon}$ is the cut-off function in $x$-direction defined in \eqref{eq:cutoff_phi} and $0<\re<a/2$ is fixed. This yields the new problem
	\begin{align}
		\nonumber
		\operatorname{div}\left((x\mathbb{A}+i\nu\mathbb{T})\nabla u^{\nu}_{\varepsilon}\right)&=[	\operatorname{div}\left((x\mathbb{A}+i\nu\mathbb{T})\nabla .\right), \varphi_{\varepsilon}]u^{\nu}+\varphi_{\varepsilon}f\\
		\nonumber
		&=(x\mathbb{A}+i\nu \mathbb{T})\nabla \varphi_{\varepsilon}\cdot \nabla u^{\nu}+\nabla \varphi_{\varepsilon}\cdot (x\mathbb{A}+i\nu \mathbb{T})\nabla u^{\nu}+\operatorname{div}((x\mathbb{A}+i\nu\mathbb{T})\nabla \varphi_{\re})u^{\nu}\\
		\label{eq:newpb}
		&=f_{\varepsilon}^{\nu},\quad \text{ with } \|f^{\nu}_{\re}\|\lesssim \|u^{\nu}\|_{\mathcal{V}_{sing}(\Omega)}+\nu\|\nabla u^{\nu}\|+\|u^{\nu}\|\lesssim \|f\|,		
	\end{align}
	where the last bound follows from the Cauchy-Schwarz inequality, Theorem \ref{theorem:stability_estimate} and \eqref{eq:est_main_v2}. Next we test the above with the Nirenberg's quotient $\delta_h^x \delta_{-h}^x u^{\nu}_{\varepsilon}$, defined via
	\begin{align*}
		\delta_h^x v=\left\{
		\begin{array}{ll}
			\frac{v(x+h,y)-v(x,y)}{h},& \quad |x|, |x+h|<a,\\
			0, & \text{ otherwise, }
		\end{array}
		\right.\quad h\in \mathbb{R}\setminus \{0\},
	\end{align*}
	which, due to the localization, belongs to $\mathcal{H}^1(\Omega)$.  Integrating by parts yields
	\begin{align*}
		\int_{\Omega}( x \mathbb{A}+i\nu\mathbb{T})\nabla \delta_{-h}^xu^{\nu}_{\re}\cdot\overline{\nabla \delta_{-h}^xu^{\nu}_{\re}}+\int_{\Omega}\left(\delta_{-h}^x \left(x  \mathbb{A}+i\nu \mathbb{T}\right)\right)\nabla u^{\nu}_{\re}\cdot\overline{\nabla \delta_{-h}^x u^{\nu}_{\re}}=\int_{\Omega}f^{\nu}_{\re}\,  \overline{\delta_h^x\delta_{-h}^xu^{\nu}_{\re}}.
	\end{align*}
	Proceeding like before, and noting that $\delta^{x}_{-h}x\mathbb{A}=\mathbb{A}+x\delta^x_{-h}\mathbb{A}$, we obtain a counterpart of the bound \eqref{eq:dhu}:
	\begin{align*}
		\nu\|\nabla \delta_{-h}^xu^{\nu}_{\varepsilon}\|^2\lesssim \left(\|\nabla u^{\nu}_{\varepsilon}\|+\|x\nabla u^{\nu}_{\varepsilon}\|\right)\|\nabla \delta_{-h}^xu^{\nu}_{\varepsilon}\|+\|f^{\nu}_{\re}\|\|\delta_{h}^x\delta_{-h}^xu^{\nu}_{\varepsilon}\|.
	\end{align*}
	Next, we employ the inequality \eqref{eq:est_main_v2} to bound $\|\nabla u^{\nu}_{\re}\|$, Theorem \ref{theorem:stability_estimate} for $\|x\nabla u^{\nu}_{\varepsilon}\|$ and \eqref{eq:newpb} for $\|f^{\nu}_{\re}\|$:
	\begin{align*}
		\nu\|\nabla \delta_{-h}^xu^{\nu}_{\varepsilon}\|^2\lesssim \nu^{-1/2}\|f\|\|\nabla \delta_{-h}^xu^{\nu}_{\varepsilon}\|+\|f\|\|\delta_{h}^x\delta_{-h}^xu^{\nu}_{\varepsilon}\|.
	\end{align*}
	Proceeding like in before yields the desired estimate for $\nu\|\nabla \partial_x u^{\nu}_{\varepsilon}\|$. As for the $u^{\nu}_{\varepsilon}(x,y)=u^{\nu}(x,y)(1-\varphi_{\varepsilon}(x,y))$, the interested reader can verify that it satisfies the problem analogous to \eqref{eq:newpb}, which can be posed on $\Omega\setminus \overline{\Omega_{\Sigma}^{\varepsilon/2}}$, with $\left.u^{\nu}(x,y)(1-\varphi_{\varepsilon}(x,y))\right|_{x=\varepsilon/2}=0$. The resulting problem decouples into two independent elliptic  problems, and the respective stability estimate is $\nu$-independent, cf. e.g. \cite[pp. 133-141]{mclean}. 
\end{proof}
Now we can prove Proposition \ref{proposition:bounds_nu}. 
\begin{proof}[Proof of Proposition \ref{proposition:bounds_nu}]
	First of all, remark that 
	\begin{align*}
		\nu\| \partial_y(\mathbb{B}\nabla U^{\nu}_{\delta})\|\lesssim \nu\|\nabla U^{\nu}_{\delta}\|+\nu\|\partial_y \nabla U^{\nu}_{\delta}\|\lesssim \nu(\|\partial_y u^{\nu}\|+\|u^{\nu}\|)+\nu\|\partial_y \nabla u^{\nu}\|,
	\end{align*}
	where we used \eqref{eq:important_mappings}. It remains to apply the bound \eqref{eq:est_main_v2} for $\nu\|u^{\nu}\|_{H^1(\Omega)}\lesssim \nu^{1/2}\|f\|$ and the previous Proposition \ref{proposition:bounds_nu} to obtain the desired estimate in the statement of the proposition.  
\end{proof}
We will also use the above bounds written in the following form:
\begin{align}
	\label{eq:unubtmp2}
	\nu\|\partial_y(\mathbb{B}\nabla u^{\nu})\|_{L^2(\Omega)}\leq C\|f\|_{L^2(\Omega)}.
\end{align}

	\section{Lifting lemma with improved regularity estimates}
\label{appendix:lifting_lemmas}
Below we will present the lifting lemma for functions defined on the real line $\mathbb{R}$. 
We will construct the lifting using the same idea as in Lemma \ref{lemma:lifting_lemma}, however, we need somewhat finer estimates which exploit the cases of low- and high-frequencies. We chose to keep all these results apart, despite the fact that we repeat some of the estimates, since such a generalization makes the elementary proof of Proposition \ref{prop:gnubound_proof} unnecessarily complicated.

Assume that $\psi\in H^{1/2}(\Sigma_{\infty})$, and let us define $\Psi$ equal to $0$ for $|x|>\nu$ and as a solution to the Dirichlet boundary-value problem for the Laplace equation on $(0,\nu)\times \mathbb{R}$:
\begin{align*}
	\Delta \Psi=0 \text{ on }(0, \nu), \qquad \Psi(0,y)=\psi(y), \quad \Psi(\nu, y)=0, \, y\in \mathbb{R}. 
\end{align*}
The above problem is well-posed in $H^1(\mathbb{R}^{2,+}_{\nu})$ due to the Lax-Milgram lemma and the validity of the Poincar\'e inequality in infinite strips. More precisely, it can be shown that
\begin{align*}
	\Psi=\mathcal{F}_y^{-1}\hat{\Psi}, \quad \hat{\Psi}(x,\xi)=\lambda_{\nu}(x,\xi)\hat{\psi}(\xi)\mathbb{1}_{0<x<\nu}, \quad \lambda_{\nu}(x,\xi):=\frac{\mathrm{e}^{\xi (x-\nu)}-\mathrm{e}^{-\xi(x-\nu)}}{\mathrm{e}^{-\xi\nu}-\mathrm{e}^{\xi\nu}}.
\end{align*}
For $0<\nu<a$, we denote by $L^{\nu}$ the operator mapping $H^{1/2}(\Sigma_{\infty})\ni\psi$ into $\Psi\in \{u\in H^1(\mathbb{R}^{2,+}_{a}): \, u(a)=0\}$; we will use the same notation for the operator mapping from $\Psi$ into the restriction of $\Psi$ to $\mathbb{R}^{2,+}_{\nu}$. Some of the properties of this operator are summarized below. 
\begin{lemma}
	Let $0<\nu<a$. The operator $L^{\nu}$ satisfies: 
	\begin{enumerate}
		\item  $L^{\nu}\in \mathcal{L}({H}^{s}(\Sigma_{\infty}), H^{s+1/2}(\mathbb{R}^{2,+}_{\nu}))$, for all $s\geq 1/2$. 
		\item $L^{\nu}\in \mathcal{L}(H^{1/2}(\Sigma_{\infty}), H^1(\mathbb{R}^{2,+}_a))$.
		\item For all $w>0$, $L^{\nu}\mathcal{L}_{w}= \boldsymbol{\mathcal{L}}_{w}L^{\nu}$, $L^{\nu}\mathcal{H}_{w}= \boldsymbol{\mathcal{H}}_{w}L^{\nu}$.
		\item For all $\psi\in H^{3/2}(\Sigma_{\infty})$, $\partial_y L^{\nu}\psi=L^{\nu}\partial_y \psi$. 
	\end{enumerate}  
	Let $0<\re\leq 1/2$ be fixed. Then,  for all $\psi\in H^{1/2}(\Sigma_{\infty})$, with $C_{\re}>0$ independent of $\nu$, but depending on $\re>0$, it holds that
	\begin{align}
		\label{eq:psibound_start0}
		\|\partial_yL^{\nu}\mathcal{L}_{\re\nu^{-1}}\psi\|_{L^2(\mathbb{R}^{2,+}_{\nu})}\leq C_{\re}\|\psi\|_{H^{1/2}(\Sigma_{\infty})},\\
		\label{eq:psibound_start_dx}
		\nu^{1/2}\|\mathcal{J}\partial_x L^{\nu}\mathcal{L}_{\re\nu^{-1}}\psi\|_{L^2(\mathbb{R}^{2,+}_{\nu})}\leq C_{\re}\|\psi\|_{H^{1/2}(\Sigma_{\infty})},\\
		\label{eq:partialy_psibound}
		\|\nabla L^{\nu}\mathcal{H}_{\re\nu^{-1}}\psi\|_{L^2(\mathbb{R}^{2,+}_{\nu})}\leq C_{\re}\|\psi\|_{H^{1/2}(\Sigma_{\infty})}.
	\end{align}
As a corollary $\|\partial_y L^{\nu}\|_{\mathcal{L}(H^{1/2}(\Sigma_{\infty}), L^2(\mathbb{R}^{2,+}_{\nu}))}\leq 2C_{\re}$. 
\end{lemma}
\begin{proof}
The fact that $L^{\nu}\in \mathcal{L}({H}^{s}(\Sigma_{\infty}), H^{s+1/2}(\mathbb{R}^{2,+}_{\nu})), \, s\geq 1/2,$ follows by adapting the elliptic regularity results to unbounded domains (cf. \cite[proof of Theorem 4.18]{mclean}).  The fact that $L^{\nu}\in \mathcal{L}(H^{1/2}(\Sigma_{\infty}), H^1(\mathbb{R}^{2,+}_a))$ stems from its definition. So do the commutator relations $L^{\nu}\mathcal{L}_{w}= \boldsymbol{\mathcal{L}}_{w}L^{\nu}$, $L^{\nu}\mathcal{H}_{w}= \boldsymbol{\mathcal{H}}_{w}L^{\nu}$. Similarly, as $\mathcal{F}_y(\partial_y L^{\nu}\psi)(\xi)=i\xi \lambda_{\nu}(.,\xi)\hat{\psi}(\xi)$, the property $\partial_y L^{\nu}\psi=L^{\nu}\partial_y \psi$ follows.  \\
\textbf{Proof of the operator bounds. }
Let $\psi\in H^{1/2}(\mathbb{R})$. 
We start with preliminary expressions. 
From the definition of $L^{\nu}$ and the Plancherel identity it follows, with $\lambda_{\nu}(x,\xi)=\frac{\mathrm{e}^{\xi (x-\nu)}-\mathrm{e}^{-\xi(x-\nu)}}{\mathrm{e}^{-\xi\nu}-\mathrm{e}^{\xi\nu}}$,  
\begin{align}
	\label{eq:Lnu0}
&	\|\partial_yL^{\nu}\psi\|^2_{L^2(\mathbb{R}^{2+}_{\nu})}=\int_{\mathbb{R}}\xi^2|\hat{\psi}(\xi)|^2\|\lambda_{\nu}(.,\xi)\|^2_{L^2(0,\nu)}d\xi,\\
\label{eq:Lnu1}
&	\|\partial_x L^{\nu}\psi\|^2_{L^2(\mathbb{R}^{2+}_{\nu})}=\int_{\mathbb{R}}|\hat{\psi}(\xi)|^2\|\partial_x\lambda_{\nu}(.,\xi)\|^2_{L^2(0,\nu)}d\xi,\\
\label{eq:Lnu2}
&	\|\mathcal{J}\partial_x L^{\nu}\psi\|^2_{L^2(\mathbb{R}^{2+}_{\nu})}=\int_{\mathbb{R}}(1+|\xi|^2)^{1/2}|\hat{\psi}(\xi)|^2\|\partial_x\lambda_{\nu}(.,\xi)\|^2_{L^2(0,\nu)}d\xi. 
\end{align}
\textbf{Proof of \eqref{eq:psibound_start0}, \eqref{eq:psibound_start_dx}. }
From \eqref{eq:Lnu0} and \eqref{eq:Lnu2}, as well as the definition of $\mathcal{L}_{w}$, we obtain that 
\begin{align}
	\label{eq:bf1}
\|\partial_yL^{\nu}\mathcal{L}_{\re\nu^{-1}}\psi\|^2_{L^2(\mathbb{R}^{2+}_{\nu})}&=\int_{|\xi\nu|<\re}\xi^2|\hat{\psi}(\xi)|^2\|\lambda_{\nu}(.,\xi)\|^2_{L^2(0,\nu)}d\xi\leq \sup_{\xi: |\xi\nu|<\re}(|\xi|\|\lambda_{\nu}(.,\xi)\|^2_{L^2(0,\nu)})\|\psi\|^2_{H^{1/2}(\Sigma_{\infty})},\\
\label{eq:bf2}
\|\mathcal{J}\partial_x L^{\nu}\mathcal{L}_{\re\nu^{-1}}\psi\|^2_{L^2(\mathbb{R}^{2+}_{\nu})}&\leq \sup_{\xi: |\xi\nu|<\re}\|\partial_x\lambda_{\nu}(.,\xi)\|^2_{L^2(0,\nu)}\|\psi\|^2_{H^{1/2}(\Sigma_{\infty})}.
\end{align}
We proceed by bounding $\lambda_{\nu}$ in the case when $|\xi\nu|<\re$. We use the mean-value theorem for the upper bound:
\begin{align*}
	|\mathrm{e}^{\xi(x-\nu)}-\mathrm{e}^{-\xi(x-\nu)}|\leq \xi|x-\nu|\left(\sup_{t\in [-\xi\nu,0]}\mathrm{e}^t+\sup_{t\in [0,\xi\nu]}\mathrm{e}^t\right)\lesssim |\xi\nu|\mathrm{e}^{|\xi\nu|}\leq \mathrm{e}^{\re}|\xi\nu|,
\end{align*}
and the inequality $|\mathrm{e}^{x}-1-x|\leq \frac{e^{\re}|x|^2}{2}$, $\forall x\in [-\re,\re],$ (the Lagrange form of the remainder), which yields  
\begin{align}
	\label{eq:b}
	\left|\mathrm{e}^{\xi\nu}-\mathrm{e}^{-\xi\nu}\right|\geq 2|\xi\nu|-\mathrm{e}^{\re}|\xi\nu|^2\geq |\xi\nu|(2-\mathrm{e}^{\re}\re)\geq |\xi\nu|, \text{ since $\re\leq 1/2$.}
\end{align}
This allows to bound 
\begin{align*}
	\|\lambda_{\nu}(.,\xi)\|_{L^{\infty}(0,\nu)}\lesssim 1  \implies \|\lambda_{\nu}(.,\xi)\|^2_{L^2(0,\nu)}\lesssim \nu\implies \sup_{\xi: |\xi\nu|<\re}|\xi|\|\lambda_{\nu}(.,\xi)\|^2_{L^2(0,\nu)}\lesssim \re\leq 1/2.
\end{align*}
This, together with \eqref{eq:bf1}, proves \eqref{eq:psibound_start0}. In a similar manner, for $|\xi\nu|\leq \re$,
\begin{align*}
	\partial_x \lambda_{\nu}(x,\xi)&=\xi\frac{\mathrm{e}^{\xi(x-\nu)}+\mathrm{e}^{-\xi(x-\nu)}}{\mathrm{e}^{-\xi\nu}-\mathrm{e}^{\xi\nu}}, \text{ therefore, with \eqref{eq:b}, }\\
	\|\partial_x \lambda_{\nu}(.,\xi)\|^2_{L^2(0,\nu)}&\leq |\xi\nu|^{-2} \int_0^{\nu}\xi^2|\mathrm{e}^{\xi(x-\nu)}+\mathrm{e}^{-\xi(x-\nu)}|^2dx\leq 4\nu^{-2}\int_{0}^{\nu}\mathrm{e}^{2|\xi\nu|}dx\leq 4\nu^{-1}\mathrm{e}^{2\re}.
\end{align*}
Plugging in this estimate into \eqref{eq:bf2} yields \eqref{eq:psibound_start_dx}. \\
\textbf{Proof of \eqref{eq:partialy_psibound}. }Using \eqref{eq:Lnu0} and \eqref{eq:Lnu1}, we obtain 
\begin{align}
	\label{eq:nablaLnu}
\|\nabla L^{\nu}\mathcal{H}_{\re\nu^{-1}}\psi\|_{L^2(\mathbb{R}^{2,+}_{\nu})}^2\leq \left(\sup_{\xi: \, |\xi\nu|>\re}\||\xi|^{-1/2}\partial_x\lambda_{\nu}(.,\xi)\|^2_{L^2(0,\nu)}+\sup_{\xi: \, |\xi\nu|>\re}\||\xi|^{1/2}\lambda_{\nu}(.,\xi)\|^2_{L^2(0,\nu)}\right)\|\psi\|^2_{H^{1/2}(\Sigma_{\infty})}.
\end{align}
To bound $\lambda_{\nu}$, we make use of the following bound, valid for $|\xi\nu|>\varepsilon$, $\varepsilon\leq 1/2$, 
\begin{align}
	\label{eq:c}
	|\mathrm{e}^{\xi\nu}-\mathrm{e}^{-\xi\nu}|=\mathrm{e}^{|\xi\nu|}|1-\mathrm{e}^{-2|\xi\nu|}|\geq \mathrm{e}^{|\xi\nu|}(1-\mathrm{e}^{-2\re})\geq \mathrm{e}^{|\xi\nu|}c_{\re}.
	%	(b) \, |\mathrm{e}^{\xi(x-\nu)}-\mathrm{e}^{-\xi(x-\nu)}|\lesssim \mathrm{e}^{|\xi\nu|},
\end{align}
With this bound, for all $|\xi\nu|>\varepsilon$, it holds that, where the hidden constant depends on $\re$, 
\begin{align*}
\|\lambda_{\nu}(.,\xi)\|^2_{L^2(0,\nu)}\lesssim \mathrm{e}^{-2|\xi\nu|}\int_0^{\nu}(\mathrm{e}^{\xi(x-\nu)}-\mathrm{e}^{-\xi(x-\nu)})^2dx
&\lesssim  \mathrm{e}^{-2|\xi\nu|}\int_0^{\nu}\mathrm{e}^{2|\xi||x-\nu|}dx\lesssim \mathrm{e}^{-2|\xi\nu|}\frac{\mathrm{e}^{2|\xi|\nu}-1}{|\xi|}\lesssim |\xi|^{-1}.
%\leq 4c \xi^2\mathrm{e}^{-2|\xi\nu|}\int_0^{\nu}\mathrm{e}^{2|\xi\nu|dx\leq 4c\xi^2\nu
\end{align*}
Similarly, 
\begin{align*}
	\|\partial_x \lambda_{\nu}(.,\xi)\|^2_{L^2(0,\nu)}\lesssim \xi^2\mathrm{e}^{-2|\xi\nu|}\int_0^{\nu}(\mathrm{e}^{\xi(x-\nu)}+\mathrm{e}^{-\xi(x-\nu)})^2dx\lesssim \mathrm{e}^{-2|\xi\nu|}\int_0^{\nu}\mathrm{e}^{2|\xi||x-\nu|}dx\lesssim |\xi|.
\end{align*}
Plugging in the above two bounds into \eqref{eq:nablaLnu} yields the desired estimate \eqref{eq:partialy_psibound}.

The final bound on $\partial_y L^{\nu}$ follows from the following identity, valid for all $w>0$,
\begin{align*}
	\partial_y L^{\nu}=\partial_y L^{\nu}(\mathcal{L}_w+\mathcal{H}_w),
\end{align*}
and the uniform bounds \eqref{eq:partialy_psibound} and \eqref{eq:psibound_start0}.
\end{proof}
\begin{remark}
	One sees easily in the derivation of the above estimates that in the low-frequency case, $|\xi\nu|<\re$, it seems impossible to control $\|\partial_x L^{\nu}\mathcal{L}_{\nu^{-1}\re}\psi\|_{L^2(\mathbb{R}^{2,+}_{(0,\nu)})}$ uniformly in $\nu$, contrary to the derivative tangent to the interface. 
\end{remark}

	\section{Proof of Proposition \ref{prop:lap_reg} }
\label{appendix:lap_auxiliary}
Let us remark that while in view of the regularity result Theorem \ref{theorem:regularity}, the result may seem trivial, in its proof it is crucial that Neumann boundary conditions at $\Sigma$ vanish. In particular, it can be checked on a toy 1D example that the corresponding statement with homogeneous Dirichlet BCs on $\Sigma$ does not hold true. 

We start by proving the corresponding result for $\delta=0$, and next arguing on how to extend it for $\delta>0$. 

\textbf{Proof for $\delta=0$. }The proof is similar to the analogous result of \cite{baouendi_goulaouic}. Since $f\in L^2(\Omega_p)$, by the same arguments that led to Theorem \ref{theorem:stability_estimate}, the estimate \eqref{eq:est_main_v2}, Proposition \ref{proposition:regularity_dy} and the bound \eqref{eq:unubtmp2}, we have 
\begin{align}
	\label{eq:vnuu}
	\|v^{\nu}\|_{L^2(\Omega_p)}+\|\partial_y v^{\nu}\|_{L^2(\Omega_p)}+\|x\nabla v^{\nu}\|_{L^2(\Omega_p)}+\nu^{1/2}\|\nabla v^{\nu}\|_{L^2(\Omega_p)}+\nu\|\partial_y \nabla v^{\nu}\|_{L^2(\Omega_p)}\lesssim \|f\|_{L^2(\Omega_p)},
\end{align}
for all $0<\nu<\nu_0$ (with a fixed $\nu_0>0$). It remains to prove that $\|\partial_x v^{\nu}\|_{\Omega_p}\lesssim \|f\|_{\Omega_p}$. Assume that
\begin{align}
	\label{eq:tpv}
	\|x\partial_y\nabla v^{\nu}\|_{L^2(\Omega_p)}\lesssim \|f\|_{L^2(\Omega_p)}.
\end{align}
We rewrite $\operatorname{div}((x+i\nu r)\mathbb{A}\nabla v^{\nu})=f$ as 
\begin{align*}
\partial_x(\vec{e}_x\cdot (x+i\nu r)\mathbb{A}\nabla v^{\nu})=-x\partial_y(\vec{e}_y\cdot\mathbb{A}\nabla v^{\nu})-i\nu \partial_y(r\vec{e}_y\cdot \mathbb{A}\nabla v^{\nu})+f, 
\end{align*}
and, by using \eqref{eq:vnuu} and \eqref{eq:tpv} to bound the right-hand side of the above, $q^{\nu}:=\vec{e}_x\cdot (x+i\nu r)\mathbb{A}\nabla v^{\nu}$ 
satisfies $
	\|\partial_x q^{\nu}\|_{L^2(\Omega_p)}\lesssim \|f\|_{L^2(\Omega_p)}.$ 
Moreover, by \eqref{eq:vnuu} and \eqref{eq:tpv}, $\|\partial_y q^{\nu}\|_{L^2(\Omega_p)}\lesssim \|f\|_{L^2(\Omega_p)}$. With \eqref{eq:vnuu} and the above, we conclude that 
$
	\|q^{\nu}\|_{H^1(\Omega_p)}\lesssim \|f\|_{L^2(\Omega_p)}.$
From the definition of $q^{\nu}$ it follows that  $\gamma_0^{\Sigma}q^{\nu}=\gamma_{n,\nu}^{\Sigma}v^{\nu}=0$, and thus we can use Hardy's inequality on p. 313 of   \cite{brezis2010functional}: 
\begin{align*}
	\left\|\frac{q^{\nu}}{x}\right\|_{L^2(\Omega_p)}\lesssim \|q^{\nu}\|_{H^1(\Omega_p)}\lesssim \|f\|_{L^2(\Omega_p)}, \text{ which implies }\\
\int_{\Omega_p}\left(1+\frac{\nu^2 r^2(\vec{x})}{x^2}\right)|\mathbb{A}\nabla v^{\nu}(\vec{x})|^2d\vec{x} \lesssim \|f\|^2_{L^2(\Omega_p)},
\end{align*}
hence $\|v^{\nu}\|_{\mathcal{H}^1(\Omega_p)}\lesssim \|f\|_{L^2(\Omega_p)}$ (the desired statement for $\delta=0$). Remark the importance of  $\gamma_{n,\nu}^{\Sigma}v^{\nu}=0$.

It remains to show \eqref{eq:tpv}. Test the problem stated in Proposition \ref{prop:lap_reg} with the Nirenberg's quotient $x\delta_h^y \delta_{-h}^y v^{\nu}$, $\delta_h^y$  defined like in Proposition \ref{proposition:bounds_nu_appendix}, and next integrate by parts, at the continuous and at the discrete level, to obtain
\begin{align*}
	\int_{\Omega_p}\delta^y_{-h}\left((x+i\nu r)\mathbb{A}\nabla v^{\nu}\right)\overline{\nabla (x\delta^y_{-h}v^{\nu})}=\int_{\Omega}f x\overline{\delta_h^y \delta_{-h}^y v^{\nu}}.
\end{align*}
The above rewrites 
\begin{align*}
	&\int_{\Omega_p}(x+i\nu r)\mathbb{A}\delta^y_{-h}\nabla v^{\nu}\,\overline{x\delta^y_{-h}\nabla v^{\nu}}+\int_{\Omega_p}(x+i\nu r)\mathbb{A}\delta^y_{-h}\nabla v^{\nu}\vec{e}_x\overline{\delta_{-h}^y v^{\nu}}\\
	&+
	\int_{\Omega_p}(x\delta_{-h}^y\mathbb{A}+i\nu \delta_{-h}^y(r\, \mathbb{A}))\nabla v^{\nu}\left(x\overline{\nabla \delta_{-h}^y v^{\nu}}+\overline{\vec{e}_x \delta_{-h}^y v^{\nu}}\right)=\int_{\Omega}f x\overline{\delta_h^y \delta_{-h}^y v^{\nu}}.
\end{align*}
Taking the real part of the above, and using that $\mathbb{A}$ is Hermitian, we observe that the first term in the above is sign-definite, while the remaining terms can be bounded using the Cauchy-Schwarz inequality (all the norms below are $\|.\|=\|.\|_{L^2(\Omega_p)}$):
\begin{align}
	\begin{split}
		\int_{\Omega_p}|x\delta_{-h}^y\nabla v^{\nu}|^2&\lesssim 
		(\|x\delta_{-h}^y\nabla v^{\nu}\|+\nu\|\delta_{-h}^y\nabla v^{\nu}\|)\|\delta_{-h}^y v^{\nu}\|\\
		&+\left(\|x	\nabla v^{\nu}\|+\nu\|\nabla v^{\nu}\|\right)(\|x\delta_{-h}^y\nabla v^{\nu}\|+\|\delta_{-h}^y v^{\nu}\|)+\|f\|\|x\delta_{-h}^y\delta_{h}^y v^{\nu}\|.
	\end{split}
\end{align}
Together with \cite[Lemma 4.13]{mclean} (the latter links $\delta_{-h}^y$ and $\partial_y$) and \eqref{eq:vnuu}, we conclude that 
\begin{align*}
	\|x\partial_y\nabla v^{\nu}\|^2\lesssim \|f\|\|x\partial_y\nabla v^{\nu}\|+\|f\|^2.
\end{align*}
The desired bound \eqref{eq:tpv} follows with the Young inequality.
  
\textbf{Proof for $\delta>0$. }We proceed by interpolation. In particular, by testing the problem stated in Proposition \ref{prop:lap_reg} with $v^{\nu}\in \mathcal{H}^1(\Omega_p)$, and using Lemma \ref{lem:f_belongs_vreg}, we remark that for any $\delta<1$, with a hidden constant independent of $\nu$, $
	|v^{\nu}|_{\mathcal{H}^1_{1}(\Omega_p)}\lesssim \|f\|_{L^2_{\delta}(\Omega_p)},$ 
and next proceed like in Proposition \ref{prop:regularity2}.

	\section{Proof of Proposition \ref{proposition:regularity_dy}}
\label{sec:prop_reg_dy}
Let $\nu>0$, and let us define $p^{\nu}\in \mathcal{H}^1(\Omega\setminus\Sigma)$ as a unique solution to the following  problem:
\begin{align}
	\label{eq:divpnu}
	\begin{split}
		&\operatorname{div}(x\mathbb{A}p^{\nu})=\partial_y u^{\nu}, \text{ in }\Omega_p\cup\Omega_n,\\
		&\gamma_{n}^{\Sigma}p^{\nu}=0,\quad
		\gamma_0^{\Gamma_p\cup\Gamma_n}p^{\nu}=0,\\
		&\text{ periodic BCs on }\Gamma^{\pm}_p\cup\Gamma^{\pm}_n.
	\end{split}
\end{align}
The above problem is well-posed due to $\partial_y u^{\nu}\in L^2(\Omega)$ (cf. Lemma \ref{lem:pb_abs_wp}), Theorem \ref{theorem:fl2} (existence and uniqueness of $p^{\nu}\in \mathcal{V}_{reg}(\Omega_p)$), which also affirms that $p^{\nu}\in \mathcal{H}^1(\Omega\setminus\Sigma)$. The solution to \eqref{eq:divpnu} satisfies the following variational formulation:
\begin{align}
	\label{eq:vp0}
	\int_{\Omega_p\cup\Omega_n}\overline{x\mathbb{A}\nabla p^{\nu}}\cdot \nabla v=-\int_{\Omega_p\cup\Omega_n}\overline{\partial_y u^{\nu}} v, \text{ for all }v\in \mathcal{H}^1(\Omega\setminus \Sigma). 
\end{align}
Next, again, due to $\partial_y u^{\nu}\in \mathcal{H}^1(\Omega)$ (by elliptic regularity, cf. Lemma \ref{lem:pb_abs_wp}), and $\gamma_0^{\Gamma_p\cup\Gamma_n}\partial_yu^{\nu}=0$,  we can take in \eqref{eq:vp0}  $v=\partial_y u^{\nu}$ and use the fact that $\mathbb{A}$ is a hermitian matrix; this yields 
\begin{align}
	\label{eq:vp}
	\int_{\Omega_p\cup\Omega_n}\overline{\nabla p^{\nu}}\cdot x\mathbb{A}\nabla \partial_y u^{\nu}=-\int_{\Omega_p\cup\Omega_n}|\partial_y u^{\nu}|^2.
\end{align}
Next, by the elliptic regularity estimate of Theorem \ref{theorem:regularity}, and using the fact that $u^{\nu}\in \mathcal{H}^2(\Omega)$ (Lemma \ref{lem:pb_abs_wp}), it holds that $p^{\nu}\in \mathcal{H}^2(\Omega\setminus\Sigma)$. Remark that $\gamma_0^{\Gamma_p\cup\Gamma_n}\partial_y p^{\nu}=0$.  
This enables us to test the problem with absorption \eqref{eq:unu_orig} with $\partial_y p^{\nu}\in \mathcal{H}^1(\Omega\setminus \Sigma)$. Integration by parts yields (recall \eqref{eq:trace_left_right} for the notation $[\gamma_0^{\Sigma}.]$)
\begin{align*}
	-\int_{\Omega_p\cup\Omega_n}(x\mathbb{A}+i\nu\mathbb{T})\nabla u^{\nu}\cdot\overline{\nabla \partial_y p^{\nu}}-\int_{\Sigma}\gamma_{n, \nu}^{\Sigma}u^{\nu}[\gamma_{0}^{\Sigma}\partial_y \overline{p^{\nu}}]=\int_{\Omega}f\, \overline{\partial_y p^{\nu}}.
\end{align*}
Since $u^{\nu}, p^{\nu}\in \mathcal{H}^2(\Omega\setminus\Sigma)$, in the above expression the integral over the interface $\Sigma$ is a Lebesgue integral. Moreover, $\gamma_0^{\Sigma,\lambda}\partial_y p^{\nu}=\partial_y\gamma_0^{\Sigma,\lambda}p^{\nu}$, $\lambda\in \{n,p\}$, the result being true by density of $\mathcal{C}^{\infty}(\overline{\Omega_{\lambda}})$ functions in $\mathcal{H}^2(\Omega_{\lambda})$ (resp. their traces on $\Sigma$ in $\mathcal{H}^{3/2}(\Sigma)$). 
Integrating by parts on $\Sigma$ (justified in particular by the bound \eqref{eq:gnubound_improved} and can be proven using the usual density argument),  and, next, on $\Omega_p\cup\Omega_n$ we obtain:
\begin{align}
	\label{eq:u_test_p}
	\int_{\Omega_p\cup\Omega_n}(x\mathbb{A}+i\nu\mathbb{T})\nabla \partial_y u^{\nu}\cdot\overline{\nabla p^{\nu}}+\langle \partial_y\gamma_{n, \nu}^{\Sigma}u^{\nu}, [\gamma_{0}^{\Sigma} \overline{p^{\nu}}]\rangle_{\mathcal{H}^{-1/2}(\Sigma), \mathcal{H}^{1/2}(\Sigma)}=\int_{\Omega_p\cup\Omega_n}f\, \overline{\partial_y p^{\nu}}.
\end{align}
Replacing the first term in the above by the right-hand side of \eqref{eq:vp} yields the identity
\begin{align*}
	-\int_{\Omega_p\cup\Omega_n}|\partial_y u^{\nu}|^2=-i\nu \int_{\Omega_p\cup\Omega_n}\mathbb{T}\partial_y \nabla u^{\nu}\, \overline{\nabla p^{\nu}}-\langle \partial_y\gamma_{n, \nu}^{\Sigma}u^{\nu}, [\gamma_{0}^{\Sigma} \overline{p^{\nu}}]\rangle_{\mathcal{H}^{-1/2}(\Sigma), \mathcal{H}^{1/2}(\Sigma)}-\int_{\Omega_p\cup\Omega_n}f\overline{\partial_y p^{\nu}}.
\end{align*}
With the Cauchy-Schwarz inequality and the continuity of the trace operator on $\mathcal{H}^1(\Omega_{\lambda})$, $\lambda\in \{n,p\}$,  
\begin{align*}
	\|\partial_y u^{\nu}\|^2_{L^2(\Omega)}\lesssim \left(\nu\|\partial_y \nabla u^{\nu}\|_{L^2(\Omega)}+\|\partial_y \gamma_{n,\nu}^{\Sigma}u^{\nu}\|_{\mathcal{H}^{-1/2}(\Sigma)}+\|f\|_{L^2(\Omega)}\right)\|p^{\nu}\|_{\mathcal{H}^1(\Omega\setminus\Sigma)}.
\end{align*}
To bound the right-hand side of the above, we use Theorem \ref{theorem:regularity} ($\|p^{\nu}\|_{\mathcal{H}^1(\Omega\setminus\Sigma)}\lesssim\|\partial_y u^{\nu}\|_{L^2(\Omega)}$), Proposition \ref{proposition:bounds_nu} in the form \eqref{eq:unubtmp2} ($\nu\|\partial_y \nabla u^{\nu}\|_{L^2(\Omega)}\lesssim \|f\|_{L^2(\Omega)}$) and Proposition \ref{prop:gnubound_improved} on the control of the conormal trace, namely, $\|\partial_y \gamma_{n,\nu}^{\Sigma}u^{\nu}\|_{\mathcal{H}^{-1/2}(\Sigma)}\lesssim \|f\|_{L^2(\Omega)}+\sqrt{\|f\|_{L^2(\Omega)}\|\partial_y u^{\nu}\|_{L^2(\Omega)}}$. Altogether, this yields
\begin{align*}
	\|\partial_y u^{\nu}\|^2\leq C (\|f\|+\sqrt{\|f\|\|\partial_y u^{\nu}\|})\|\partial_y u^{\nu}\|, \quad C>0,
\end{align*}
uniformly in $\nu>0$. 
Applying the Young inequality to the above yields the desired bound.

	\section{Proofs of the results for general domains}
\label{appendix:proof_general}
\subsection{Proof of Proposition \ref{prop:decomp2}}
%Recall that the proof of Proposition \ref{prop:decomp1}, which is a counterpart of Proposition \ref{prop:decomp2}, follows from the following result. 
We start with the following auxiliary problem. 
Given $f\in L^2(D_p)$, $g\in H^{1/2}(I)$, find $u\in \mathcal{V}_{sing}(D_p)$, s.t. 
\begin{align*}
	&\operatorname{div}(\alpha\mathbb{H}\nabla u)=f \text{ in }D_p,\\
	&\gamma_n^{I}u=g,\quad \gamma_0^{I_p}u=0.
\end{align*}
%Remark that in the vicinity of $I$,  $\alpha(\vec{x})=\operatorname{dist}(\vec{x},I)$, it holds that, see \cite[proof of Lemma 4.16]{gilbrag_trudinger}, 
%\begin{align*}
%	\left.\nabla \alpha\right|_{I}\cdot \vec{n}=1, 
%\end{align*}
The following is a counterpart of Theorem \ref{thm:sp_wp}, combined with Theorem \ref{theorem:regularity}.
\begin{theorem}
	\label{theorem:sp_wp_general}
	If $g=0$, the above problem admits a unique solution  $u\in \mathcal{V}_{sing}(D_p)$. This solution also satisfies the following. For $f\in L^2(D_p)$, $u\in H^1(D_p)$ and $\|u\|_{H^1(D_p)}+\|\alpha u\|_{H^2(D_p)}\lesssim \|f\|_{L^2(D_p)}$.
\end{theorem}
\begin{proof}
	See the proof of Theorem \ref{thm:sp_wp}. Remark that we make use of the fact that the corresponding sesqulinear form is strictly elliptic, in other words, in virtue of the Poincar\`e inequality of  Proposition \ref{prop:poincare_appendix} (which is true in particular due to the homogeneous Dirichlet boundary condition at $I_p\neq \emptyset$), the above problem indeed admits a unique solution in $\mathcal{V}_{reg}(D_p)$. The regularity results of \cite{baouendi_goulaouic}, cf. Theorem \ref{theorem:regularity}, hold true in this case as well, cf. the respective proofs in \citeappendixregularity and the change of coordinates described after Lemma \ref{lem:deftraces_general}.
\end{proof}
At this point we will not need the corresponding result for regularity of $u$ when $f\in H^1(D_p)$, since at the point where it will be needed, we will work with a coordinate-transformed problem, mapped on the rectangular domain $\Omega$. 

With this result, we obtain 
\begin{theorem}
	\label{theorem:decomposition_general}
	The above problem admits a unique solution $u\in \mathcal{V}_{sing}(D_p)$, which writes
	\begin{align*}
		u=u_{sing}+u_{reg}, \qquad u_{sing}=u_h\log|\alpha|,
	\end{align*}
	where $u_h\in H^1(D_p)$ is s.t. 
	\begin{align}
		\label{eq:nablauh}
		\begin{split}
			&\operatorname{div}(\mathbb{H}\nabla u_h)=0 \text{ in }D_p,\\
			&\gamma_0^{I}u_h=h_{I}^{-1}g,\qquad 	\gamma_0^{I_p}u_h=0, 
		\end{split}
	\end{align}
	and $u_{reg}\in \bigcap_{\re>0}\mathcal{H}^{1}_{\re}(D_p)\bigcap \bigcap_{\re>0}H^{1-\re}(D_p)$. In the above, $\gamma_n^I u=\gamma_n^Iu_{sing}$.  
\end{theorem}
\begin{proof}
	The uniqueness follows from Theorem \ref{theorem:sp_wp_general}. The existence follows verbatim by the same argument as in the proof of Theorem \ref{theorem:reg_well_posedness}. Indeed, remark that $u_h$ that solves \eqref{eq:nablauh} satisfies
	\begin{align*}
		\alpha \mathbb{H}\nabla (u_h\log|\alpha|)= u_h\mathbb{H}\nabla \alpha+\alpha\log|\alpha|\mathbb{H}\nabla u_h.
	\end{align*}
With the above we can prove that, in the sense of equality in $H^{-1/2}(I)$,  $\gamma_n^Iu_{sing}=\gamma_0^{I}u_h\vec{n}\cdot \mathbb{H}\vec{n}=h_I\gamma_0^Iu_h=g$, see \eqref{eq:signed_distance}. This can be justified  rigorously like in Theorem \ref{theorem:reg_well_posedness}. 
\end{proof}
Proposition \ref{prop:decomp2} follows immediately from the above.

\subsubsection{The Green's formula}
\label{sec:green_general}
To motivate Definition \ref{def:trace_general}, we need an appropriate Green's formula. 
Let 
\begin{align*}
	\mathcal{V}_{sing}(\operatorname{div}(\alpha \mathbb{H}\nabla .); D_p):=\{v\in \mathcal{V}_{sing}(D_p): \, \operatorname{div}(\alpha\mathbb{H}\nabla v)\in L^2(\Omega_p), \, \gamma_n^{I}v\in H^{1/2}(I)\}.
\end{align*}

Then the following counterpart of Theorem \ref{theorem:green} holds true. 
\begin{theorem}
	\label{theorem:green_general}
	For $u, v\in 	\mathcal{V}_{sing}(\operatorname{div}(\alpha \mathbb{H}\nabla. ); D_p)$, it holds that 
	\begin{align*}
		\int_{D_p}\operatorname{div}(\alpha\mathbb{H}\nabla u)\overline{v}-	\int_{D_p}\overline{\operatorname{div}(\alpha\mathbb{H}\nabla v)}u=-\langle \gamma_n^I u, \overline{\gamma_0^I v}\rangle_{L^2(I)}+\overline{\langle \gamma_n^I v, \overline{\gamma_0^I u}\rangle}_{L^2(I)}.
	\end{align*}
\end{theorem}
The proof of this theorem is technical and relies on a well-chosen change of coordinates and the Green's formula on rectangular domains of Theorem \ref{theorem:green}. By a localization argument and the same argument as in Lemma \ref{lem:deftraces} it is sufficient to prove the above result for $D_p$ replaced by a vicinity of $I$. In particular, the following statement  can be proven like Lemma \ref{lem:deftraces}. 
\begin{lemma}
	\label{lem:deftraces_general}
	Assume that $v\in \mathcal{V}_{sing}(\operatorname{div}(\alpha\mathbb{H}\nabla.); D_p)$ admits two decompositions 
	\begin{align*}
		v=v_{h,j}\log|\alpha|+v_{reg,j},\quad j=1,2, 
	\end{align*}
	where $v_{h,j}\in H^1(D_p)$ with $\gamma_0^{I_p}v_{h,j}=0$, and $v_{reg,j}\in \mathcal{H}^1_{\delta}(D_p)$, $0<\delta<1$, \, $j=1,2$. Then
	\begin{align*}
		\gamma_0^I v_{h,1}=\gamma_0^I v_{h,2}, \qquad \gamma_0^I v_{reg,1}=\gamma_0^I v_{reg,2},
	\end{align*}
	and, in particular, $\gamma_0^Iv_{h,j}=h_{I}^{-1}\gamma_{n}^I v$, $\gamma_0^I v_{reg,j}=\gamma_0^I v$, $j=1,2$.
\end{lemma}
Let us now concentrate on establishing the necessary results in the vicinity of $I$. Since $I$ is compact and the domain $D_p$ is of $C^{3}$ regularity, there exist open sets $\{\mathcal{U}_k\}_{k=1}^N$ s.t. $I\subset \bigcup_{k}\mathcal{U}_k$, and associated local $C^{3}$ charts  $\bpsi_k: \, \overline{\Omega}\rightarrow \overline{\mathcal{U}_k}$, where $\Omega$ is like in Section \ref{sec:simpl}, and $$
\mathcal{U}_{p,k}:=\mathcal{U}_k\cap D_p=\bpsi_k(\Omega_p), \quad \mathcal{U}_k\cap D_n=\bpsi_k(\Omega_n), \quad \mathcal{U}_k\cap I = \bpsi_k(\Sigma),$$ (this is a corollary of the definition \cite[p.94]{gilbrag_trudinger}). Without loss of generality, we can assume that $\{\mathcal{U}_k\}_{k=1}^N\subset U_I$, where $\alpha$ equals the signed distance. Next, we define the subordinate partition of unity $\{\chi_k\}_{k=1}^N\subset C^{\infty}(\mathbb{R}^2)$, see \cite[Corollary 3.22]{mclean}, and functions 
\begin{align*}
	u_k:=\chi_k u,\quad v_k=\chi_k v. 
\end{align*}
By direct computation it follows that for $u, v$ like in Theorem \ref{theorem:green_general}, $u_k,\, v_k\in \mathcal{V}_{sing}(\operatorname{div}(\alpha\mathbb{H}\nabla.); D_p)$, and 
\begin{align}
	\label{eq:support}
	\operatorname{supp}u_k\subsetneq \mathcal{U}_k, \quad 	\operatorname{supp}v_k\subsetneq \mathcal{U}_k.
\end{align}
We then have the following result. 
\begin{proposition}
	\label{prop:green_general_ingredient}
	Let $k,m\in\{1,\ldots, N\}$. Then for  $\mathcal{U}_{p,k}$, $u_k$, $v_m$ as above, it holds that 
	\begin{align}
		\label{eq:ukr}
		\int_{\mathcal{U}_{p,k}}\operatorname{div}(\alpha \mathbb{H}\nabla u_k)\overline{v}_m-	\int_{\mathcal{U}_{p,k}}\overline{\operatorname{div}(\alpha \mathbb{H}\nabla v_m)}{u}_k=-\int_{I}\chi_k\gamma_n^I u\overline{\chi_m \gamma_0^I v}+\int_{I}\overline{\chi_m\gamma_n^I v}{\chi_k \gamma_0^I u}.
	\end{align}
\end{proposition}
This result will be proven by a change of coordinates. Let us introduce several auxiliary results. Let us fix $k\in \{1,\ldots, N\}$. 
For $\vec{x}\in \mathcal{U}_{k}$, it holds that  $\vec{x}:=\bpsi_k(\widetilde{\vec{x}})$, $\widetilde{\vec{x}}\in \Omega$. We define the Jacobian and its determinant:
\begin{align*}
	{\mathbb{J}}=D\boldsymbol{\psi}_k, \quad  j=\operatorname{det}{\mathbb{J}}, \quad \mathbb{J}_{\Sigma}=\left. \mathbb{J}\right|_{\Sigma}, \quad j_{\Sigma}=\left. j\right|_{\Sigma}.
\end{align*}
For any function $h: \, \mathcal{U}_k\rightarrow \mathbb{C}$, we define its pullback  $\widetilde{h}(\tbx)=h({\bpsi}_k(\tbx))$, $\tbx\in \Omega$. Moreover, we denote by $\widetilde{\operatorname{div}}$, $\widetilde{\operatorname{\nabla}}$ etc. differential operators written in $\widetilde{\vec{x}}$-coordinates.

Let $I_k:=\overline{\mathcal{U}_k}\cap I$. Denoting by $\vec{n}_{\Sigma}=(1,0)$ the unit normal to $\Sigma$, we recall that, see \cite[(2.1.62), (2.1.58)]{boffi}.
\begin{lemma}
	\label{lem:transformations}
	For $f\in L^1(I_k)$, it holds that 
$
		\int_{I_k}fd\Gamma=\int_{\Sigma}\widetilde{f}(\widetilde{y})\rho_{\Sigma}(\tilde{y})d\tilde{y}, \quad \rho_{\Sigma}=|\mathbb{J}^{-t}_{\Sigma}\vec{n}_{\Sigma}|j_{\Sigma}=|\left.\partial_{y}\bpsi_k\right|_{\Sigma}|.
$
\end{lemma}
Next, let us see how $\widetilde{\alpha}$ is transformed under $\vec{\bpsi}_k$. 
\begin{lemma}
	\label{lem:alphadef}
	The coefficient $\widetilde{\alpha}(\tbx)=\alpha(\bpsi_k(\tbx))$ satisfies 
	\begin{align*}
	(\tilde{x},\tilde{y})=\tbx\mapsto \widetilde{\alpha}(\tbx)=\tilde{x}\gamma(\tbx),
\end{align*}
where   $\gamma\in C^{2}(\overline{\Omega})$ and  $ \inf_{\tbx\in \Omega}\gamma(\tbx)>0$.  Moreover, $\left. \gamma\right|_{\Sigma}=|\mathbb{J}^{-t}_{\Sigma}\vec{n}_{\Sigma}|^{-1}$. 
\end{lemma}
\begin{proof}
As $\mathcal{U}_k\subset U_I$, we use the expression  \eqref{eq:signed_distance} of $\alpha$ (below $\operatorname{dist}$ is the signed distance to $I$, $\operatorname{dist}>0$ in $\Omega_p\cap \mathcal{U}_I$)
	\begin{align*}
		\alpha(\bpsi_k(\tbx))&=\operatorname{dist}(\bpsi_k(\tbx), I)=\operatorname{dist}(\bpsi_k(\tbx), \bpsi_k(\Sigma)), \text{ and it holds that }\\	
		&|\operatorname{dist}(\bpsi_k(\tbx), \bpsi_k(\Sigma))|=\inf_{\tbx_{\Sigma}\in \Sigma}|\bpsi_k(\tbx)-\bpsi_k(\tbx_{\Sigma})|\geq c\inf_{\tbx_{\Sigma}\in \Sigma}|\tbx-\tbx_{\Sigma}|=c\tilde{x},\\ 
		&|\operatorname{dist}(\bpsi_k(\tbx), \bpsi_k(\Sigma))|=\inf_{\tbx_{\Sigma}\in \Sigma}|\bpsi_k(\tbx)-\bpsi_k(\tbx_{\Sigma})|\leq C\inf_{\tbx_{\Sigma}\in \Sigma}|\tbx-\tbx_{\Sigma}|=C\tilde{x},
	\end{align*}
	for some constants $C, c>0$, since $\bpsi_k$ is bi-Lipschitz. From the above considerations,  regularity of $\bpsi_k$ and of the distance function it follows that 
	\begin{align*}
		\tbx\mapsto \alpha(\bpsi_k(\tbx))\in C^2(\overline{\Omega}),\qquad 
		\text{ and for }\tbx=(\tilde{x},\tilde{y}),\quad \alpha(\bpsi_k(\tbx))=\tilde{x}\gamma(\tbx),	
	\end{align*}
	where $\gamma(\tbx)\geq c$ and $\gamma\in C^{2}(\overline{\Omega})$. One verifies that $ \left.\gamma\right|_{\Sigma}=\left.\widetilde{\nabla}\widetilde{\alpha}\right|_{\Sigma}\cdot \vec{n}_{\Sigma},$
and with the change of coordinates relations, it holds that 
\begin{align*}
	\left.\gamma\right|_{\Sigma}=\mathbb{J}^t_{\Sigma}\left.(\nabla \alpha\circ\bpsi_k)\right|_{\Sigma} \cdot\vec{n}_{\Sigma}=\mathbb{J}^t_{\Sigma}(\vec{n}_I\circ \bpsi_k)\cdot\vec{n}_{\Sigma}=|\mathbb{J}^{-t}_{\Sigma}\vec{n}_{\Sigma}|^{-1}, %\mathbb{J}^t\vec{n}(\bpsi_k(\tbx))|\mathbb{J}^{-t}\vec{n}_{\Sigma}|
\end{align*}
where the identity before the last one follows from \eqref{eq:alphaprop} and the last identity from \cite[(2.1.94)]{boffi} (namely, $\vec{n}_I\circ \vec{\psi}_k=|\mathbb{J}^{-t}_{\Sigma}\vec{n}_{\Sigma}|^{-1}\mathbb{J}^{-t}_{\Sigma}\vec{n}_{\Sigma}$).
\end{proof}
%We will also need the following (quite natural) result. 
%\begin{lemma}
%	\label{lem:localization}
%Let $u\in \mathcal{V}_{sing}(\operatorname{div}(\alpha \mathbb{H}\nabla.); D_p)$. Then $\chi_k u\in \mathcal{V}_{sing}(\operatorname{div}(\alpha \mathbb{H}\nabla.); D_p)$, and, moreover, 
%	\begin{align*}
%		\gamma_n^I (\chi_ku)=\gamma_0^I\chi_k\,\gamma_n^I u.
%	\end{align*}
%\end{lemma}
Equipped with these  results, we can prove Proposition \ref{prop:green_general_ingredient}.
\begin{proof}[Proof of Proposition \ref{prop:green_general_ingredient}] 
\textbf{Rewriting the desired identity in $\Omega_p$. }
Following  the change of coordinates as described in \cite[Section 2.1.3]{boffi}, we have that
	\begin{align*}
		\operatorname{div}(\alpha \mathbb{H}\nabla u_k)(\vec{x})=\frac{1}{j}\widetilde{\operatorname{div}}(j\widetilde{\alpha}\mathbb{J}^{-1}\widetilde{\mathbb{H}}\mathbb{J}^{-t}\widetilde{\nabla}\widetilde{u}_k)(\bpsi_k^{-1}(\vec{x})).
	\end{align*}
Let us introduce, using the notation of Lemma \ref{lem:alphadef}, the positive definite Hermitian matrix, cf. the above:
\begin{align}
	\label{eq:atilde}
	\widetilde{\mathbb{A}}:=\gamma j \mathbb{J}^{-1}\widetilde{\mathbb{H}}\mathbb{J}^{-t}\in C^{1,1}(\overline{\Omega}; \mathbb{C}^{2\times 2}), 
\end{align}
remark the regularity of $\widetilde{\mathbb{A}}$ compared to Assumption \ref{assump:matrices} (which justifies Assumption \ref{assump:matrices_general} and our requirements on the regularity of $D_p, \, D_n, I$). 
 
Parametrizing the integral in the left-hand side, we see that 
	\begin{align}
		\label{eq:lhs_paramerization}
		\int_{\mathcal{U}_{p,k}}\operatorname{div}(\alpha \mathbb{H}\nabla u_k)\overline{{v}}_m=\int_{\Omega_p}\widetilde{\operatorname{div}}(\tilde{x}	\widetilde{\mathbb{A}}\widetilde{\nabla}\widetilde{u}_k)\overline{\widetilde{v}_m},
	\end{align}
	so that the left-hand side of \eqref{eq:ukr} equals to:
	\begin{align}
		\label{eq:I}
		\mathcal{I}:=\int_{\Omega_p}\widetilde{\operatorname{div}}(\tilde{x}	\widetilde{\mathbb{A}}\widetilde{\nabla}\widetilde{u}_k)\overline{\widetilde{v}_m}-\int_{\Omega_p}\overline{\widetilde{\operatorname{div}}(\tilde{x}	\widetilde{\mathbb{A}}\widetilde{\nabla}\widetilde{v}_m)}{\widetilde{u}_k}.
	\end{align}
	We will apply to the above the integration by parts Theorem \ref{theorem:green}, more precisely, its Corollary \ref{cor:green}. Remark that while the matrix  $\widetilde{\mathbb{A}}$ in the above does not satisfy periodicity constraints, the functions $\widetilde{u}_k$, $\widetilde{v}_m$ satisfy
$
		\operatorname{supp}\widetilde{u}_k\subsetneq \Omega, \quad \operatorname{supp}\widetilde{v}_m\subsetneq \Omega,$ 
	due to \eqref{eq:support}. This allows to extend the statement of Theorem \ref{theorem:green} and of Corollary \ref{cor:green} to this case in a trivial manner.  
	
\textbf{Evaluating $\mathcal{I}$. }	
	To evaluate $\mathcal{I}$, by Corollary \ref{cor:green}, it suffices to find a decomposition of $\widetilde{u}_k$, $\widetilde{v}_m$, s.t. 
	\begin{align}
		\label{eq:tildek}
		\widetilde{u}_k=\widetilde{u}_{k,s}\log|\tilde{x}|+\widetilde{u}_{k,r}, \quad		\widetilde{v}_m=\widetilde{v}_{m,s}\log|\tilde{x}|+\widetilde{v}_{m,r},
	\end{align}
where $\widetilde{u}_{k,s}$, $\widetilde{v}_{m,s}\in \mathcal{H}^1(\Omega_p)$ and $\widetilde{u}_{k,r}$ and $\widetilde{v}_{m,r}\in \mathcal{H}^1_{\delta}(\Omega_p)$, $0<\delta<1$. 
%\begin{align}
%		\label{eq:desired_identity}
%		\gamma_0^{\Sigma}\widetilde{u}_{k,s}=\widetilde{a}_{11}^{-1}\gamma_n^{\Sigma}\widetilde{u}_k,\quad \gamma_0^{\Sigma}\widetilde{v}_{k,s}=\widetilde{a}_{11}^{-1}\gamma_n^{\Sigma}\widetilde{v}_k
%\end{align} with the definition of $\gamma_n^{\Sigma}$ adapted to $\widetilde{\mathbb{A}}$, and with $\widetilde{a}_{11}$ defined like $a_{11}$ from Assumption \ref{assump:matrices}.
This decomposition will be constructed with the help of corresponding decomposition for the original functions $u_k, \, v_m$. Using the decomposition defined in Theorem  \ref{theorem:decomposition_general}, we write
	\begin{align*}
		u_k=\chi_k u_h\log|\alpha|+\chi_k u_{reg}, \quad v_m=\chi_m v_h\log|\alpha|+\chi_m v_{reg}, 
	\end{align*}
so that
	\begin{align*}
		\widetilde{u}_k=\widetilde{\chi}_k \widetilde{u}_h\log|\widetilde{\alpha}|+\widetilde{\chi}_k\widetilde{u}_{reg},\quad 	\widetilde{v}_m=\widetilde{\chi}_m \widetilde{v}_h\log|\widetilde{\alpha}|+\widetilde{\chi}_m\widetilde{v}_{reg},
	\end{align*}
	and we rewrite the above to match with \eqref{eq:tildek} and using Lemma \ref{lem:alphadef}:
	\begin{align}
		\label{eq:decomp_uk}
		\begin{split}
	&	\widetilde{u}_k=\widetilde{u}_{k,s}\log|\tilde{x}|+\widetilde{u}_{k,r}, \qquad	\widetilde{u}_{k,s}=\widetilde{\chi}_k\widetilde{u}_h, \quad \widetilde{u}_{k,r}=\widetilde{\chi}_k(\widetilde{u}_{reg}+\widetilde{u}_h\log|\gamma|),
	\end{split}
	\end{align}
and similarly for $\widetilde{v}_m$. The regularity of the above functions allows to apply Corollary \ref{cor:green} to
\eqref{eq:I}: 
\begin{align}
	\label{eq:i_id}
	\mathcal{I}=-\int_{\Sigma}\widetilde{a}_{11}\gamma_0^{\Sigma}\widetilde{u}_{k,s}\, \overline{\gamma_0^{\Sigma}\widetilde{v}_{m,r}}+\int_{\Sigma}\widetilde{a}_{11}\overline{\gamma_0^{\Sigma}\widetilde{v}_{m,s}}\, {\gamma_0^{\Sigma}\widetilde{u}_{k,r}}.
\end{align}
It remains to apply the change of coordinates of Lemma \ref{lem:L1}. We rewrite, recalling the definition \eqref{eq:atilde} of $\widetilde{\mathbb{A}}$, $\left. \gamma\right|_{\Sigma}=|\mathbb{J}^{-t}_{\Sigma}\vec{n}_{\Sigma}|^{-1}$, cf. Lemma \ref{lem:alphadef}, and $\rho_{\Sigma}=|\mathbb{J}^{-t}_{\Sigma}\vec{n}_{\Sigma}|j_{\Sigma}$, cf. Lemma \ref{lem:transformations},
\begin{align*}
\widetilde{a}_{11}=\vec{n}_{\Sigma}\cdot\widetilde{\mathbb{A}}_{\Sigma}\vec{n}_{\Sigma}=|\mathbb{J}_{\Sigma}^{-t}\vec{n}_{\Sigma}|^{-1}j_{\Sigma}\vec{n}_{\Sigma}\cdot \mathbb{J}^{-1}_{\Sigma}\widetilde{\mathbb{H}}\mathbb{J}^{-t}_{\Sigma}\vec{n}_{\Sigma}=\rho_{\Sigma}|\mathbb{J}_{\Sigma}^{-t}\vec{n}_{\Sigma}|^{-2}\vec{n}_{\Sigma}\cdot \mathbb{J}^{-1}_{\Sigma}\widetilde{\mathbb{H}}\mathbb{J}^{-t}_{\Sigma}\vec{n}_{\Sigma}.
\end{align*}
By \cite[(2.1.94)]{boffi}, $\vec{n}_I\circ \bpsi_k=\mathbb{J}^{-t}_{\Sigma}\vec{n}_{\Sigma}|\mathbb{J}^{-t}_{\Sigma}\vec{n}_{\Sigma}|^{-1}$, thus 
\begin{align*}
\widetilde{a}_{11}=\rho_{\Sigma}(\vec{n}_I\cdot \mathbb{H}\vec{n}_I)\circ {\bpsi}_k=\rho_{\Sigma}\widetilde{h_I}.
\end{align*}
Rewriting \eqref{eq:i_id} with Lemma \ref{lem:L1} and recalling the definition \eqref{eq:decomp_uk}, we obtain 
\begin{align*}
	\mathcal{I}&=-\int_{I}h_I \chi_k u_h\, \overline{\chi_m(v_{reg}+v_h\log|\gamma\circ{\bpsi}_k^{-1}|)}+\int_{I}h_I \overline{\chi_m v_h}\, {\chi_k(u_{reg}+u_h\log|\gamma\circ{\bpsi}_k^{-1}|)}\\
	&=-\int_{I}h_I\chi_k u_h\, {\overline{\chi_m}v_{reg}}+\int_{I}h_I\overline{\chi_m v_h}\chi_k u_{reg}	=-\int_{I}\chi_k\gamma_n^I u\overline{\chi_m \gamma_0^I v}+\int_{I}\overline{\chi_m\gamma_n^I v}{\chi_k \gamma_0^I u},
\end{align*} 
where in the last identity we used definitions of $u_h,\,v_h$ in Definition \ref{def:trace_general}, and the fact that $\operatorname{supp}\chi_k\subsetneq \mathcal{U}_k$. 
\end{proof}
Proposition  \ref{prop:green_general_ingredient} enables us to prove Theorem \ref{theorem:green_general}. 
\begin{proof}[Proof of Theorem \ref{theorem:green_general}] 
Let $\chi$ be a regular function $C^{\infty}(\mathbb{R}^2; [0,\,1])$,  equal to $1$ in the vicinity of $I$ that is included into $\bigcup\limits_k\mathcal{U}_k$, and vanishing outside of $\bigcup\limits_{k}\mathcal{U}_k$. We rewrite 
\begin{align*}
\mathcal{I}:=	\int_{D_p}\operatorname{div}(\alpha\mathbb{H}\nabla u)\overline{v}-	\int_{D_p}\overline{\operatorname{div}(\alpha\mathbb{H}\nabla v)}{u}=	\int_{D_p}\operatorname{div}(\alpha\mathbb{H}\nabla (\chi u))\overline{\chi v}-	\int_{D_p}\overline{\operatorname{div}(\alpha\mathbb{H}\nabla(\chi v))}{\chi u}, 
\end{align*}
where the desired identity follows by a classical Green's formula and the fact that $(1-\chi)u$, resp. $(1-\chi)v$ vanishes in the vicinity of $I$. It remains to decompose $\chi u$, $\chi v$ using the partition of unity $\{\chi_k\}_{k=1}^N$, with $\operatorname{supp}\chi_k\subset \mathcal{U}_k$, and use the result of Proposition \ref{prop:green_general_ingredient} (evidently valid with $u,v$ replaced by $\chi u$, $\chi v$).
\end{proof}
\subsubsection{The key stability bound}
The counterpart of Theorem \ref{theorem:stability_estimate} reads. 
\begin{theorem}[The first stability estimate]
	\label{theorem:stability_estimate_general}
	There exists $C>0$, s.t. for all $\nu>0$ sufficiently small, the solution to \eqref{eq:B32} satisfies the following stability bound: $
		\|v^{\nu}\|_{\mathcal{V}_{sing}(D)}\leq C\|f\|_{L^2(D)}.$
\end{theorem}
The proof of this result relies on the following proposition. 
\begin{proposition}
	\label{prop:gnubound_proof_general}
	Given $u^{\nu}$ as in \eqref{eq:B32}, let the co-normal derivative at the interface $\Sigma$ be denoted by 
	$$g^{\nu}:=\gamma_{n, \nu}^{I}v^{\nu}=\left.( \alpha \mathbb{H}+i\nu\mathbb{N})\nabla v^{\nu}\right|_{I}\cdot \vec{n}_I.%=\left.i\nu \vec{e}_x\cdot \mathbb{T}\nabla u^{\nu}\right|_{\Sigma}.
	$$
	There exists $C>0$, s.t. for all $\nu>0$ sufficiently small, it satisfies the following bound:
	\begin{align}
		\label{eq:gnubound_general}
		\|g^{\nu}\|_{{H}^{-1/2}(I)}\leq C\left(  \nu^{1/2}\|f\|+\sqrt{\|f\|\|v^{\nu}\|}\right).
	\end{align}
\end{proposition}
\begin{proof}
	The proof mimics the proof of Proposition \ref{prop:gnubound_proof}, which is based on integration by parts, estimates of Lemma \ref{lem:pb_abs_wp}  (obtained by integration by parts), and an appropriate lifting lemma, which follows from Lemma \ref{lemma:lifting} by a standard localization/change of variables argument (cf. \cite[proof of Lemma 2.7.3]{sayas_tdbie}).
\end{proof}
\begin{proof}[Proof of Theorem \ref{theorem:stability_estimate_general}]
The proof mimics the proof of Theorem \ref{theorem:stability_estimate}. The latter relies on two ingredients: 
\begin{itemize}
	\item Proposition \ref{prop:gnubound_proof}, see its counterpart Proposition \ref{prop:gnubound_proof_general}; 
	\item integration by parts combined with Theorem \ref{theorem:regularity}. 
\end{itemize}
The second result transfers almost verbatim by using the general results of \cite{baouendi_goulaouic}.
\end{proof}
\subsubsection{An improved regularity of the conormal derivative}
The next result is a counterpart of Theorem \ref{theorem:gnubound_improved}.
\begin{theorem}
	\label{theorem:gnubound_improved_general}
	Given $v^{\nu}$ as in \eqref{eq:B32}, let the co-normal derivative at the interface $I$ be denoted by $g^{\nu}:=\gamma_{n, \nu}^{I}v^{\nu}=\left.( \alpha \mathbb{H}+i\nu\mathbb{N})\nabla v^{\nu}\right|_{I}\cdot \vec{n}_I$. Then $g^{\nu}\in {H}^{1/2}(I)$, and there exist $C, \nu_0>0$, s.t. 
	\begin{align}
		\label{eq:gnubound_key_general}
		\|g^{\nu}\|_{{H}^{1/2}(I)}\leq C\|f\|, \quad \text{ for all }0<\nu<\nu_0.
	\end{align}
\end{theorem}
\begin{proof}
This result is proven by recalling the following equivalent expression to the norm in $H^{1/2}(I)$, cf. \cite[(3.29)]{mclean}, see also the notation after Proposition \ref{prop:green_general_ingredient}:
\begin{align}
	\label{eq:gnubound_upper}
	\|g^{\nu}\|^2_{H^{1/2}(I)}= \sum\limits_{k=1}^N\|(\chi_kg^{\nu})\circ {\bpsi}_k\|_{\widetilde{H}^{1/2}(\Sigma)}^2,
\end{align}
where we also recall that $(\chi_kg^{\nu})\circ {\bpsi}_k$ has a support strictly included into $\Sigma$ for all $k=1,\ldots, N$ and define, for $q\in H^{1/2}(\Sigma)$ its extension $q_0$ by $0$ to $\mathbb{R}$ together with the associated norm
$
	\|q\|_{\widetilde{H}^{1/2}(\Sigma)}:=\|q_0\|_{H^{1/2}(\mathbb{R})}.$
To prove Theorem \ref{theorem:gnubound_improved_general} we will rely on localization techniques. Let us fix $k\in \mathbb{N}$. It is straightforward to see that $v^{\nu}_k\in H^1(\mathcal{U}_k)$ satisfies the following problem, cf. \eqref{eq:B32}:
\begin{align*}
	&\operatorname{div}((\alpha\mathbb{H}+i\nu\mathbb{N}\nabla)v^{\nu}_k)=2(\alpha\mathbb{H}+i\nu\mathbb{N})\nabla v^{\nu}\cdot \nabla \chi_k+v^{\nu}\operatorname{div}((\alpha\mathbb{H}+i\nu\mathbb{N})\nabla\chi_k)+\chi_k f =:f^{\nu}_k\text{ in }\mathcal{U}_k,\\
	&v^{\nu}_k=0 \text{ in a vicinity of }\partial \mathcal{U}_k.
\end{align*}
By Theorem \ref{theorem:stability_estimate_general} and Lemma \ref{lem:pb_abs_wp_general}, $
	\|f^{\nu}_k\|_{L^2(D)}\leq C\|f\|_{L^2(D)}.$
Next, we transform the above problem to $\Omega$. In particular, following the proof and the notation of Proposition \ref{prop:green_general_ingredient}, we rewrite the above as follows: 
\begin{align*}
	&\widetilde{\operatorname{div}}((\tilde{x}\widetilde{\mathbb{A}}+i\nu\widetilde{\mathbb{T}})\widetilde{\nabla}\widetilde{v}^{\nu}_k)=j\widetilde{f}^{\nu}_k\text{ in }\Omega,\\
	&\widetilde{v}^{\nu}_k=0 \text{ in a vicinity of }\partial\Omega.
\end{align*}
with $\widetilde{\mathbb{T}}=j\mathbb{J}^{-1}\mathbb{N}\mathbb{J}^{-t}$. We have thus obtained the problem \eqref{eq:unu_orig}, modulo the boundary conditions and the periodicity constraints on the tensors $\widetilde{\mathbb{A}}$ and $\widetilde{\mathbb{T}}$. Using the fact that $\widetilde{v}_k^{\nu}=0$ in a vicinity of $\Omega$, in particular, there exists $\delta>0$, s.t. $\widetilde{v}^{\nu}_k(\tilde{x},\tilde{y})=0$ for $|\tilde{y}|<\ell-\delta$, we fix $\delta$ and recall the definition  \eqref{eq:childelta} of a cutoff function $\chi_{\ell-3\delta/2,\delta}$. Then $\widetilde{v}^{\nu}_k$ satisfies as well the problem where 
\begin{align}
	\label{eq:vnuk_pb}
	\begin{split}
	&\widetilde{\operatorname{div}}((\tilde{x}\widetilde{\mathbb{A}}_{\delta}+i\nu\widetilde{\mathbb{T}}_{\delta})\widetilde{\nabla}\widetilde{v}^{\nu}_k)=j\widetilde{f}^{\nu}_k\text{ in }\Omega,\\
	&\widetilde{v}^{\nu}_k=0 \text{ in a vicinity of }\partial\Omega, 
	\end{split}
\end{align}
where the new matrices satisfy now periodicity constraints, and remain Hermitian and positive-definite:
\begin{align*}
\widetilde{\mathbb{A}}_{\delta}= \chi_{\ell-3\delta/2,\delta}\widetilde{\mathbb{A}}+(1-\chi_{\ell-3\delta/2,\delta})\mathbb{I}, \quad \widetilde{\mathbb{T}}_{\delta}= \chi_{\ell-3\delta/2,\delta}\widetilde{\mathbb{T}}+(1-\chi_{\ell-3\delta/2,\delta})\mathbb{I}.
\end{align*}
Remark that evidently, $\widetilde{v}_k^{\nu}\in \mathcal{H}^1(\Omega)$. From explicit expressions, cf. \eqref{eq:atilde}, it follows that  $\widetilde{\mathbb{A}}_{\delta}$ and $\widetilde{\mathbb{T}}_{\delta}$ satisfy as well regularity constraints in Assumption \ref{assump:matrices}. %Moreover,  $\widetilde{v}_k^{\nu}$ can be extended by $0$ outside of $\Omega$ while preserving its regularity, and $\widetilde{\mathbb{A}}_{\delta}, \widetilde{\mathbb{T}}_{\delta}$ by $\mathbb{I}$; this extension of $\widetilde{v}_{k}^{\nu}$ plays a role of $U^{\nu}_{\delta}$.  
By Theorem \ref{theorem:gnubound_improved}, we conclude that, see also Remark \ref{rem:stab_results},
\begin{align}
	\label{eq:gammannu}
	\|\gamma_{n,\nu}^{\Sigma,\delta}\widetilde{v}_k^{\nu}\|_{\mathcal{H}^{1/2}(\Sigma)}\leq C\|j\widetilde{f}^{\nu}_k\|_{L^2(\Omega)},
\end{align}
where $\gamma_{n,\nu}^{\Sigma,\delta}\widetilde{v}_k^{\nu}=\gamma_0^{\Sigma}(\widetilde{x}\widetilde{\mathbb{A}}_{\delta}+i\nu \widetilde{\mathbb{T}}_{\delta})\widetilde{\nabla}\tilde{v}_k^{\nu}\cdot \vec{n}_{\Sigma}$. 

Because  $\gamma_{n,\nu}^{\Sigma,\delta}\widetilde{v}_k^{\nu}=\gamma_{n,\nu}^{\Sigma}\widetilde{v}_k^{\nu}$ and has a  support strictly included into $\Sigma$, the bound \eqref{eq:gammannu} implies that 
\begin{align}
	\label{eq:vnuk}
	\|\gamma_{n,\nu}^{\Sigma}\widetilde{v}_k^{\nu}\|_{\widetilde{H}^{1/2}(\Sigma)}\leq C'\|j\widetilde{f}^{\nu}_k\|_{L^2(\Omega)},
\end{align}
for some $C'>0$ (see \cite[Theorem 4.2.1]{wendland_hsiao} for the equivalence of different norms on $H^{1/2}(\Sigma)$).
%
 %by following the proof of \ref[Lemma 4.2.5]{wendland_hsiao}, see also the proof of Proposition \ref{prop:apx2}, we conclude that 
%\begin{align*}
%	
%\end{align*}
Finally, to use the above bound in \eqref{eq:gnubound_upper}, it remains to establish a connection between $\gamma_{n,\nu}^{\Sigma}\widetilde{v}_k^{\nu}$ and $\chi_kg^{\nu}\circ \bpsi_k$:
\begin{align*}
\rho^{-1}_{\Sigma}\gamma_{n,\nu}^{\Sigma}\widetilde{v}_k^{\nu}=\chi_kg^{\nu}\circ \bpsi_k,
\end{align*}
with $\rho_{\Sigma}>c>0$ is defined Lemma \ref{lem:transformations}. The above result can be obtained either by direct computations, or recalling that the conormal derivatives can be defined variationally, cf. the expression after \eqref{eq:hdelta}, and next using \eqref{eq:lhs_paramerization}, an analogous result for $\int_{\Omega_p}\alpha\mathbb{H}\nabla u\cdot \nabla v$, and Lemma \ref{lem:transformations}.

With \eqref{eq:vnuk}, and by using the fact that $\rho_{\Sigma}$ is regular, we conclude that 
$
	\|\chi_k g^{\nu}\circ \bpsi_k\|_{\widetilde{H}^{1/2}(\Sigma)}\leq C\|f\|_{L^2(\Omega)},$
which, when inserted into \eqref{eq:gnubound_upper},  implies the desired bound \eqref{eq:gnubound_key_general}.
\end{proof}
%Since $\widetilde{v}_k^{\nu}\in \mathcal{H}^2(\Omega)$ and is compactly supported, 

\subsubsection{An important property of the regular part}
The following result is a counterpart of Theorem \ref{theorem:convergence_decomposition}. 
\begin{theorem}
	\label{theorem:convergence_decomposition_general}
	Let $(v^{\nu})_{\nu>0}\subset {H}^1_0(D)$ be a sequence of solutions to \eqref{eq:B32}. Then there exists a subsequence $(v^{\nu_k})_{k\in\mathbb{N}}$ which converges weakly in $L^2(D)$ to a limit ${v}^*\in \mathcal{V}_{sing}(\operatorname{div}(\alpha\mathbb{H}\nabla.);D)$. This limit necessarily satisfies 
	\begin{align*}
		\operatorname{div}(\alpha\mathbb{H}\nabla {v}^*)= f\text{ in }\Omega,\\
		[\gamma_0^{I}{v}^*]=-i\pi h_I^{-1}\gamma_n^{I}{v}^*.
	\end{align*}
\end{theorem}
\begin{proof}
	The existence and convergence results for the subsequence follow by the same argument as in Theorem \ref{theorem:convergence_decomposition}. 
To prove the desired result about the decomposition, we rely on the localization and change of variable techniques, cf. the proof of Proposition \ref{prop:green_general_ingredient}. By Lemma \ref{lem:deftraces_general}, we see that it is sufficient to prove the result about the jump of $v$ with $v^*$ replaced by $\chi v^*$, where $\chi$ is a cutoff equal to $1$ in the sufficiently small vicinity of the interface $I$ and vanishing otherwise. 
In particular, with the notation of the proof of Proposition \ref{prop:green_general_ingredient}, Proposition \ref{prop:aux} shows that  $\widetilde{v}_{k}^{\nu}:=(v^{\nu}\chi_k)\circ\bpsi_k$ writes 
\begin{align*}
	\widetilde{v}_{k}^{\nu}=\widetilde{v}_{k,h}^{\nu}\log(\widetilde{x}+i\nu\widetilde{r})+\widetilde{v}_{k,cont}^{\nu} \text{ in }\Omega,
\end{align*}
where $\widetilde{r}=\widetilde{\mathbb{T}}_{11}/\widetilde{\mathbb{A}}_{11}$, with $\widetilde{v}_{k,cont}^{\nu}\in \mathcal{H}^1_{\re}(\Omega)$, $0<\re<1$,  $\widetilde{v}_{k,h}^{\nu}\in \mathcal{H}^1(\Omega)$. Passing to the limit like in the proof of Theorem \ref{theorem:convergence_decomposition}, we conclude that 
\begin{align*}
	\widetilde{v}_{k}^{*}=\widetilde{v}_{k,h}^{*}\log(|\widetilde{x}|+i\pi\mathbb{1}_{\widetilde{x}<0})+\widetilde{v}_{k,cont}^{*} \text{ in }\Omega,	
\end{align*}
 with $\widetilde{v}_{k,cont}^*\in \mathcal{H}^1_{\re}(\Omega)$, $0<\re<1$,  $\widetilde{v}_{k,h}^*\in \mathcal{H}^1(\Omega)$. With Lemma \ref{lem:alphadef} the above rewrites 
\begin{align*}
	\widetilde{v}_{k}^{*}=\widetilde{v}_{k,h}^{*}\log(|\widetilde{\alpha}|-\log|\gamma|+i\pi\mathbb{1}_{\widetilde{x}<0})+\widetilde{v}_{k,cont}^{*}\text{ in }\Omega,	
\end{align*} 
and coming back to the original coordinates we rewrite
\begin{align*}
	v_k^*=v^*_{k,h}\log(|\alpha|+i\pi \mathbb{1}_{\mathcal{U}_{k}\cap D_n})+v_{k,cont}^*-{v}^*_{k,h}\log|\gamma\circ\bpsi_k^{-1}|,
\end{align*}
with ${v}_{k,cont}^*\in \mathcal{H}^1_{\re}(D)$, $0<\re<1$,  ${v}_{k,h}^*\in {H}^1_0(D)$. By Lemma \ref{lem:deftraces_general}, using the decomposition of $v_k^*$ from Proposition \ref{prop:decomp2}, namely, $v_k^*=\chi_k(v_h^*\log|\alpha|+v_{reg}^*)$, we conclude that 
\begin{align*}
	\gamma_0^{I}v_{k,h}^*=\gamma_0^I\chi_k v_h^*, \text{ and }\gamma_0^{I
	,\lambda}(v_{k,cont}^*-{v}^*_{k,h}\log|\gamma\circ\bpsi_k|+i\pi \mathbb{1}_{\mathcal{U}_{k}\cap D_n} v^*_{k,h})=\gamma_0^{I,\lambda}\chi_k v_{reg}^*,\quad \lambda\in\{n,p\},
\end{align*}
so that, since $\operatorname{supp}\chi_k\subset\mathcal{U}_k$, 
\begin{align*}
	[\gamma_0^I \chi_k v_{reg}^*]=-i\pi \gamma_0^{I}v^*_{k,h}=-i\pi\gamma_0^I\chi_k v_h^*=-i\pi h_{I}^{-1}\gamma_0^I \chi_k \gamma_n^I v^*.
\end{align*}
By the argument of Lemma \ref{lem:deftraces_general} and $\sum\limits_k\chi_k=1$ on $I$, we conclude that $[\gamma_0^I v^*]\equiv [\gamma_0^I v_{reg}^*]=-i\pi h_I^{-1}\gamma_n v^*$.  
\end{proof}
\subsubsection{Proof of Theorems \ref{theorem:LAP2}, \ref{theorem:main_result2}}
Since all the required tools have been given in the previous sections, the proof repeats verbatim the corresponding proof in Section \ref{sec:LAP_Problem}.

%there exists a $C^4$ mapping $\psi$ from $\Omega$ into an open neighbourdhood of $I$

	\bibliographystyle{alpha} 
	\bibliography{general_bibliography}
\end{document}